\documentclass[11pt,leqno]{article}
\usepackage{a4wide}
\usepackage{amssymb}
\usepackage{amsmath}
\usepackage{amsthm}
\usepackage[curve]{xypic}
\usepackage{color}
\usepackage{hyperref}
\theoremstyle{plain}
\newcommand{\Id}{\operatorname{Id}}
\newcommand{\id}{\operatorname{id}}

\newcommand{\Aut}{\operatorname{Aut}}
\newcommand{\Hom}{\operatorname{Hom}}
\newcommand{\Ind}{\operatorname{Ind}}

\newcommand{\Res}{\operatorname{Res}}
\newcommand{\res}{\operatorname{res}}

\newcommand{\ind}{\operatorname{c-Ind}}

\newcommand{\ev}{\operatorname{ev}}
\newcommand{\End}{\operatorname{End}}
\newcommand{\val}{\operatorname{val}}
\newcommand{\Gal}{\operatorname{Gal}}

\newcommand{\Ker}{\operatorname{Ker}}
\newcommand{\Coker}{\operatorname{Coker}}
\newcommand{\diag}{\operatorname{diag}}
\newcommand{\supp}{\operatorname{supp}}
\newtheorem{theorem}{Theorem}[section]
\newtheorem{corollary}[theorem]{Corollary}
\newtheorem{lemma}[theorem]{Lemma}
\newtheorem{remark}[theorem]{Remark}
\newtheorem{proposition}[theorem]{Proposition}

\newtheorem{definition}[theorem]{Definition}

\begin{document}

 \title{From  \'etale $P_{+}$-representations to \\  $G$-equivariant sheaves on $G/P$}

\author{Peter Schneider, Marie-France Vigneras, Gergely Zabradi \footnote{The third author was partially supported by OTKA Research grant no.\ K-101291} }
 \maketitle

\abstract{} Let $K/\mathbb Q_{p}$ be a finite extension with ring of integers $o$,  let $G$ be a connected reductive split $\mathbb Q_{p}$-group of Borel subgroup $P=TN$ and let $\alpha$ be a simple root of $T$ in $N$. We associate  to a finitely generated module $D$ over the Fontaine ring over $o $ endowed with a semilinear \'etale action of
 the monoid $T_{+} $ (acting on the Fontaine ring via  $\alpha$), a $G(\mathbb Q_{p})$-equivariant sheaf   of $o$-modules on the compact space $G(\mathbb Q_{p})/P(\mathbb Q_{p})$. Our construction  generalizes  the  representation $D\boxtimes \mathbb P^{1} $ of $ GL(2,\mathbb Q_{p})$ associated by Colmez
   to a   $(\varphi,\Gamma)$-module  $D$ endowed with a character  of $\mathbb Q_{p}^{*}$.
\tableofcontents

 \section{Introduction}
\subsection{Notations}  We fix  a finite extension $K/\mathbb Q_{p}$ of ring of integers $o$ and an algebraic closure $\overline {\mathbb Q}_p$ of $K$. We denote by $\mathcal G_p = \Gal (\overline{ \mathbb Q}_p /\mathbb Q_p)$ the absolute Galois group  of $\mathbb Q_p$,   by  $\Lambda(\mathbb Z_p) $
the Iwasawa $o$-algebra of maximal ideal $\mathcal M(\mathbb Z_p)$, and by
 $\mathcal O_{\mathcal E}$ the Fontaine ring which is the $p$-adic completion of the localisation of the Iwasawa $o$-algebra $ \Lambda(\mathbb Z_p)=o[[Z_p]]$ with respect to the elements not in $p \Lambda(\mathbb Z_p)$. We put on  $\mathcal O_{\mathcal E}$ the weak topology inducing  the $ \mathcal M(\mathbb Z_p)$-adic topology on $ \Lambda(\mathbb Z_p)$, a fundamental system of neighborhoods of $0$ being
 $(p^n\mathcal O_{\mathcal E}+  \mathcal M(\mathbb Z_p)^n)_{n\in \mathbb N}.$
The  action of $\mathbb Z_p - \{0\}$ by multiplication on $\mathbb Z_p$
extends to an  action   on $\mathcal O_{\mathcal E}$.

We fix an arbitrary split reductive connected $\mathbb Q_{p}$-group $G$ and a Borel $\mathbb Q_{p}$-subgroup $P=TN$ with maximal $\mathbb Q_{p}$-subtorus $T$ and unipotent radical $N$. We denote by  $w_0$ the longest element of the Weyl group of $T$ in $G$, by $\Phi_+$ the  set  of  roots of $T$ in $N$, and by $u_\alpha :\mathbb G_a \to  N_\alpha$, for $\alpha \in \Phi_+$,  a  $\mathbb Q_p$-homomorphism
 onto the root subgroup $N_\alpha$ of $N$ such that $tu_\alpha (x) t^{-1}=  \alpha (t) x$ for $x\in \mathbb Q_p$ and $t\in T(\mathbb Q_p)$, and   $N_0=\prod_{\alpha\in \Phi_+} u_\alpha (\mathbb Z_p)$  is a subgroup of $N(\mathbb Q_p) $.  We denote
by $T_{+} $ the monoid of dominant elements  $t$ in $T(\mathbb Q_p)$ such that $ \val_p(\alpha(t))\geq 0$ for all $\alpha \in \Phi_+$,   by $T_0\subset T_+$ the maximal subgroup,  by $T_{++}$ the subset of strictly dominant elements, i.e. $\val_p(\alpha(t))> 0$ for all $\alpha \in \Phi_+$, and we put $P_+=N_0T_+, P_0=N_0T_0$.  The natural action of  $T_+$   on $N_0$ extends to an action on the Iwasawa $o$-algebra $\Lambda (N_0)=o[[N_0]]$.
The compact  set $G(\mathbb Q_p)/P(\mathbb Q_p)$ contains the open dense subset $\mathcal C=N(\mathbb Q_p)w_0 P(\mathbb Q_p)/P(\mathbb Q_p)$ homeomorphic to $N(\mathbb Q_p)$ and the compact subset  $ \mathcal C_0=N_0 w_0 P(\mathbb Q_p)/P(\mathbb Q_p)$ homeomorphic to $N_0$. We put $\overline P(\mathbb Q_p)=w_0P(\mathbb Q_p)w_0^{-1}$.

Each simple root $\alpha$ gives a $\mathbb Q_p$-homomorphism
$x_\alpha: N\to \mathbb G_a$ with   section  $u_\alpha$.
We denote  by $\ell_{\alpha}:N_0\to \mathbb Z_p$, resp. $\iota_{\alpha}:  \mathbb Z_p\to N_0$,    the restriction  of $x_\alpha$, resp.  $u_{\alpha}$,     to $N_0$, resp.
$\mathbb Z_p$.

For example,  $G=GL(n)$,  $P$  is  the subgroup of upper triangular matrices,  $N$    consists of the strictly upper diagonal matrices ($1$ on the diagonal),
$T$  is the diagonal subgroup,  $N_0= N(\mathbb Z_p)$,
the simple roots are $\alpha_1, \ldots, \alpha_{n-1}$ where  $\alpha_i(\diag(t_1,\ldots, t_n))= t_i t_{i+1}^{-1}$, $x_{ \alpha_i}$ sends a matrix to its $(i,i+1)$-coefficient, $u_{ \alpha_i}(.) $ is the strictly upper triangular matrix, with $(i,i+1)$-coefficient $.$ and $0$ everywhere else.

We denote by $C^\infty (X,o)$ the $o$-module of locally constant  functions on a locally profinite space $X$ with values in $o$, and by $C_c^\infty (X,o)$ the subspace of compactly supported functions.

\subsection{General overview} Colmez established a correspondence $V\mapsto \Pi(V)$ from the absolutely irreducible $K$-representations  $V$ of dimension $2$ of the Galois group $\mathcal G_p$ to the unitary admissible absolutely irreducible $K$-representations  $\Pi$ of $GL(2, \mathbb Q_p)$ admitting a central character \cite{C}. This correspondence relies on the construction of a representation $D(V)\boxtimes  \mathbb P^1$ of $GL(2, \mathbb Q_p)$ for any representation $V$ (not necessarily of dimension $2$) of $\mathcal G_p$  and any unitary character $\delta:\mathbb Q_{p}^* \to o^*$.
When the dimension of $V$ is $2$ and when $\delta = (x|x|)^{-1}\delta_V$, where $\delta_V$ is the character of $\mathbb Q_{p}^* $ corresponding to the representation $\det V$ by  local class field theory, then $D(V)\boxtimes \mathbb P^1$ is an extension of $\Pi(V)$ by its dual twisted by $\delta \circ \det$.
It is a general belief that the correspondence $V\to \Pi(V)$  should extend to a correspondence  from representations $V$ of dimension $d$   to representations $\Pi$ of  $GL(d,\mathbb Q_p)$.

We generalize here   Colmez's construction of the representation   $D\boxtimes \mathbb P^1$  of $GL(2, \mathbb Q_p)$, replacing  $GL(2)$  by the  arbitrary split reductive connected $\mathbb Q_p$-group $G$. More precisely, we denote by $O_{\mathcal E, \alpha}$   the  ring $\mathcal O_{\mathcal E}$ with the action of    $T_+$  via a simple root $ \alpha\in \Delta$ (if the rank of $G$ is $1$,  $\alpha$ is unique and we omit $\alpha$). For any finitely generated  $\mathcal O_{\mathcal E, \alpha}$-module $D$ with an \'etale semilinear action of $T_+$, we construct a  representation of $G(\mathbb Q_p)$.
It is  realized as  the space of global sections of a  $G(\mathbb Q_p)$-equivariant sheaf on the compact quotient  $G(\mathbb Q_p)/P(\mathbb Q_p)$.
When the rank of $G$ is $1$,  the compact space $G(\mathbb Q_p)/P(\mathbb Q_p)$ is isomorphic to $ \mathbb P^1(\mathbb Q_p)$ and when $G=GL(2)$ we recover  Colmez's sheaf.

We review briefly the main  steps of our construction.

1.   \ We show that the category of  \'etale $T_+$-modules finitely generated over $\mathcal O_{\mathcal E, \alpha}$  is equivalent to the category of
\'etale $T_+$-modules finitely generated over   $\Lambda_{\ell_\alpha}(N_0)$,   for a topological ring $\Lambda_{\ell_\alpha}(N_0)$ generalizing  the Fontaine ring  $\mathcal O_{\mathcal E}$,  which is better adapted to the group $G$,  and depends on the simple root $\alpha$.

2. \  We show that the sections over $\mathcal C_0\simeq N_0 $ of a  $P(\mathbb Q_p)$-equivariant sheaf $\cal S$  of $o$-modules over   $\mathcal C \simeq N$ is an \'etale $o[P_+]$-module $\mathcal S  ({\mathcal C}_0 )$ and that the functor  $\mathcal S\mapsto  \mathcal S (\mathcal C_0 )$ is an equivalence of categories.

3. \  When $ \mathcal S  (\mathcal C_0 )$  is an \'etale  $T_+$-module
finitely generated over $\Lambda_{\ell_\alpha}(N_0)$, and the root system of $G$ is irreducible,  we show that the  $P(\mathbb Q_p)$-equivariant sheaf $\mathcal S$   on  $\mathcal C$ extends to a $G(\mathbb Q_p)$-equivariant sheaf  over  $G(\mathbb Q_p)/P(\mathbb Q_p)$ if  and only if the rank of $G$ is $1$.

4. \ For any strictly dominant element $s\in T_{++} $, we associate functorially  to an \'etale  $T_+$-module $M$
finitely generated over $\Lambda_{\ell_\alpha}(N_0)$,   a $G(\mathbb Q_p)$-equivariant sheaf $\mathfrak Y_s$ of $o$-modules  over  $G(\mathbb Q_p)/P(\mathbb Q_p)$   with sections over $\mathcal C_0$ a dense  \'etale
$\Lambda (N_0)[T_+]$-submodule   $M_s^{bd}$ of $M $. When the rank of $G$ is $1$,  the  sheaf $\mathfrak Y_s$  does not depend on the choice of $s\in T_{++} $,  and $M_s^{bd}=M $; when $G=GL(2)$ we recover the construction of Colmez.
For a general $G$, the   sheaf $\mathfrak Y_s$   depends on the choice of $s\in T_{++} $,   the system $(\mathfrak Y_s)_{s\in T_{++} }$ of sheaves is compatible, and we associate functorially to $M$ the $G(\mathbb Q_p)$-equivariant sheaves $\mathfrak Y_\cup$ and $\mathfrak Y_\cap$ of $o$-modules over  $G(\mathbb Q_p)/P(\mathbb Q_p)$ with sections  over  $\mathcal C_0$ equal to  $\cup_{s\in T_{++}} M_s^{bd}$ and $\cap_{s\in T_{++}} M_s^{bd}$, respectively.

\subsection{The rings $\Lambda_{\ell_\alpha}(N_0)$ and $\mathcal O_{\mathcal E, \alpha}$}

Fixing a simple root $\alpha \in \Delta$,    the  topological local ring $\Lambda_{\ell_\alpha}(N_0)$, generalizing the Fontaine ring $\mathcal O_\mathcal E$,
is defined as \cite{S} with  the  surjective homomorphism $ \ell_\alpha:N_0\to \mathbb Z_p$.

  We denote  by  $\mathcal M (N_{\ell_\alpha})$ the maximal ideal of the   Iwasawa $o$-algebra $\Lambda(N_{\ell_\alpha})=o[[N_{\ell_\alpha}]]$ of the kernel $N_{\ell_\alpha}$ of $\ell_\alpha $.
The ring
  $\Lambda_{\ell_\alpha}(N_0)$  is the $\mathcal M (N_{\ell_\alpha})$-adic completion of the localisation of $\Lambda(N_0)$ with respect to the Ore subset of elements which are not in $ \mathcal M (N_{\ell_\alpha})\Lambda(N_0)$. This is a noetherian  local ring with maximal ideal $\mathcal M_{\ell_\alpha}(N_0)$ generated by $\mathcal M (N_{\ell_\alpha})$.
  We put on  $\Lambda_{\ell_\alpha}(N_0)$ the weak topology with fundamental system of neighborhoods of $0$ equal to $( \mathcal M_{\ell_\alpha}(N_0)^n+ \mathcal M(N_0)^n)_{n\in \mathbb N}$. The action of $T_+$  on $N_0$ extends to an action on $\Lambda_{\ell_\alpha}(N_0)$.
  We denote by $\mathcal O_{\mathcal E, \alpha }$ the ring $\mathcal O_{\mathcal E }$ with the action of $T_+$ induced by $(t,x)\mapsto \alpha(t)x:T_+\times  \mathbb Z_p \to \mathbb Z_p$.
     The homomorphism $\ell_\alpha$ and its section $\iota_\alpha$  induce   $T_+$-equivariant   ring homomorphisms
$$\ell_\alpha: \Lambda_{\ell_\alpha}(N_0) \to \mathcal O_{\mathcal E , \alpha}\ , \ \iota_\alpha:  \mathcal O_{\mathcal E ,\alpha} \to \Lambda_{\ell_\alpha}(N_0) \ , \ \text{such that  $ \ell_\alpha \circ \iota_\alpha =\id.$}$$

 \subsection{Equivalence of categories}

  An \'etale $T_+$-module over $\Lambda_{\ell_\alpha}(N_0)$ is a finitely generated $\Lambda_{\ell_\alpha}(N_0)$-module $M$ with a semi-linear action $T_+\times M \to M$ of $T$ which is \'etale, i.e.
  the action $\varphi_t$ on $M$ of each   $t\in T_+$ is injective and
 $$M=\oplus_{u\in J(N_0/   tN_0t^{-1})} u\varphi_t(M) \ , $$  if $J(N_0/  tN_0t^{-1})\subset N_0$ is a system of representatives of
 the cosets $ N_0/  tN_0t^{-1}$; in particular, the action of each element of the maximal subgroup $T_0$ of $T_+$ is invertible. We denote by $\psi_t$ the left inverse of $\varphi_t$ vanishing on $u\varphi_t(M)$  for $u\in N_0$ not in $tN_0t^{-1}$.  These modules form an abelian category $\mathcal M^{et}_{\Lambda_{\ell_\alpha}(N_0)} (T_+)$.

 We define analogously the abelian category $ \mathcal M^{et}_{\mathcal O_{\mathcal E, \alpha} }(T_+)$ of finitely generated $\mathcal O_{\mathcal E, \alpha}$-modules with an \'etale  semilinear action of $T_+$. The action $\varphi_t$ of each element  $t\in T_+$ such that $\alpha(t)\in \mathbb Z_p^*$ is invertible.
We show that the action $T_+\times D\to D$ of $T_+$ on $D\in  \mathcal M^{et}_{\mathcal O_{\mathcal E, \alpha} }(T_+)$ is continuous for  the weak topology on $D$; the canonical action of the inverse $T_-$ of $T$ is also continuous. Extending the results of \cite{Z}, we show:

 \begin{theorem} The  base change functors  $ \mathcal O_{\mathcal E } \otimes_{\ell_\alpha}- $ and
$ \Lambda_\ell (N_0)  \otimes_{\iota_\alpha} -$ induce  quasi-inverse isomorphisms
$$
\mathbb D: \mathcal M^{et}_{\Lambda_{\ell_\alpha}(N_0)} (T_+)\to
 \mathcal M^{et}_{\mathcal O_{\mathcal E, \alpha} }(T_+)\ , \
\mathbb M: \mathcal M^{et}_{\mathcal O_{\mathcal E, \alpha} }(T_+) \to
\mathcal M^{et}_{\Lambda_{\ell_\alpha}(N_0)} (T_+)\ .
$$

 \end{theorem}

Using this theorem, we  show that  the action of $T_+$ and of $T_-$ on an \'etale $T_+$-module over $\Lambda_{\ell_\alpha}(N_0) $ is continuous for the
weak topology.

   \subsection{$P$-equivariant sheaves on $\mathcal C$}

The $o$-algebra  $C ^\infty(N_0,o)$  is naturally  an \'etale $o[P_+]-$module, and  the monoid $P_+$ acts  on the $o$-algebra $\End_o M$ by $(b,F)\mapsto \varphi_b \circ F\circ  \psi_b$.
 We show that  there exists a unique $o[P_+]$-linear map
$$
\res: C ^\infty(N_0,o)\to \End_o M
$$
 sending the characteristic function $1_{N_0}$ of $N_0$ onto the identity $\id_M$; moreover $\res$ is an algebra homomorphism which sends $1_{b.N_0}$ to $\varphi_b \circ \psi_b$ for all $b\in P_+$ acting on $x\in N_0$ by $(b,x)\mapsto b.x$.

 For the sake of simplicity, we denote  now  by the same letter a group defined over $\mathbb Q_p$ and the group of its $\mathbb Q_p$-rational points.

 Let $M^P$ be the  $o[P]$-module  induced by the canonical action of   the inverse monoid $P_-$ of $P_+$ on $M$; as a representation of
 $N $,  it is isomorphic to the representation induced by   the action of $N_0$ on $M$.  The value at $1$, denoted by
 $\ev_0:M^P\to M$,  is $P_-$-equivariant, and  admits a $P_+$-equivariant splitting $\sigma_0:M\to M^P$ sending $m\in M$ to the function equal to $n \mapsto nm$ on $N_0$ and
 vanishing on $N -N_0$. The $o[P]$-submodule $M^P_c$ of $M^P$
 generated generated by $\sigma_0(M)$ is naturally isomorphic to $A[P] \otimes _{A[P_+]} M$.  When $M=C ^\infty(N_0,o)$ then $M^P_c=C_c ^\infty(N,o)$  and $M^P=C ^\infty(N,o)$ with the natural  $o[P]-$module structure. We have the  natural $o$-algebra embedding
 $$F\mapsto \sigma_0\circ F \circ \ev_0:  \End_o M \to  \End_o M^P \ .$$
 sending  $\id_M$ to the idempotent  $R_0=\sigma_0 \circ ev_0$ in  $\End_o M^P$.

 \begin{proposition} There exists a unique  $o[P]$-linear map
 $$\Res: C_c ^\infty(N,o)\to \End_o M^P \ $$
 sending $1_{N_0}$  to $R_0$;
moreover $\Res$  is an algebra homomorphism.
\end{proposition}
The topology of $N$ is totally disconnected and by a general argument, the functor  of compact global sections is an equivalence of categories from the $P$-equivariant sheaves on  $N\simeq \mathcal C $ to
the non-degenerate modules on the skew group ring
$$ C_c ^\infty( N,o)\# P\ = \oplus_{b\in P } bC ^\infty(N ,o)\ . $$
in which the multiplication is determined by the rule
$(b_1f_1)(b_2f_2)=b_1b_2 f_1^{b_2}f_2$ for $b_i\in P , f_i\in C ^\infty(N ,o)$ and $f_1^{b_2}(.)=f(b_2.)$.

\begin{theorem}
The functor  of sections over $N_0\simeq \mathcal C_0$ from the  $P$-equivariant sheaves on  $N\simeq \mathcal C$ to
the \'etale $o[P_+]$-modules is an equivalence of categories.
\end{theorem}

The global sections on $\mathcal C$ of a $P$-equivarianf sheaf $\cal S$  on $\mathcal C$ is $\mathcal S(C)=\mathcal S(C_0)^P$.

 \subsection{Generalities on $G$-equivariant sheaves on $G/P$}
The functor of global sections from the $G$-equivariant sheaves on $G/P$ to the modules on
the skew group ring $\mathcal A_{G/P}= C ^\infty(G /P ,o)\# G $ is an equivalence of categories. We have the intermediary ring $\mathcal A$   $$\mathcal A_{\mathcal C}=C_c ^\infty( \mathcal C,o)\# P  \ \subset \ \mathcal A =  \oplus_{g\in G } gC_c ^\infty(g^{-1}\mathcal C \cap \mathcal C  ,o)\ \subset \ \mathcal A_{G/P} , $$
and the $o$-module
$$\mathcal Z =  \oplus_{g\in G(\mathbb Q_p)} gC_c ^\infty( \mathcal C  ,o)$$
which is a left ideal of $\mathcal A_{G/P}$ and a right $\mathcal A$-submodule.

\begin{proposition} The functor
$$Z\mapsto Y(Z)=\mathcal Z\otimes_{\mathcal A } Z$$  from the
non-degenerate $\mathcal A$-modules   to the  $\mathcal A_{G/P}$-modules   is an equivalence of categories; moreover   the $G $-sheaf  on $G /P $
corresponding to $Y(Z)$  extends the $P$-equivariant sheaf  on $ {\mathcal C}$
corresponding to $Z|_{\mathcal A_{\mathcal C}} $.
\end{proposition}

Given an \'etale $o[P_+]$-module $M$, we consider the problem of extending  to $\mathcal A$   the   $o$-algebra homomorphism
$$
\Res: A_{\mathcal C}\to \End_o(M_c^P) \quad, \quad \sum_{b\in P}bf_b\mapsto b\circ \Res (f_b) \ .
$$
We introduce the  subrings
\begin{align*}\mathcal A_0 &= 1_{\mathcal C_0}\mathcal A 1_{\mathcal C_0} =
\oplus _{g\in G}gC^\infty(g^{-1}\mathcal C_0\cap \mathcal C_0, o)\   \subset \ \mathcal A \ , \\
 \mathcal A_{\mathcal C 0} &= 1_{\mathcal C_0}\mathcal A_{\mathcal C}1_{\mathcal C_0} =
 \oplus _{b\in P} bC^\infty(b^{-1}\mathcal C_0\cap \mathcal C_0, o)\   \subset \ \mathcal A_{\mathcal C} \ .
 \end{align*}
The  skew monoid ring $\mathcal A_{\mathcal C_0}= C^\infty (\mathcal C_0, o) \# P_+ =\oplus_{b\in P_+}b C^\infty( \mathcal C_0,o)$ is contained in  $\mathcal A_{\mathcal C0}$.   The intersection $g^{-1}\mathcal C_0\cap \mathcal C_0$ is not $0$ if and only if $g\in N_0\overline P N_0$. The subring $\Res ( \mathcal A_{\mathcal C0})$ of $\End_o(M^P)$ necessarily lies in the image of $\End_o (M) $.

The group $P$ acts on  $\mathcal A $ by  $(b,y)\mapsto (b1_{G/P}) y (b1_{G/P})^{-1}$ for $b\in P $, and the map $b\otimes y \mapsto  (b1_{G/P}) y (b1_{G/P})^{-1}$  gives $o[P ]$ isomorphisms
$$
o[P ]\otimes _{o[P_+]} \mathcal A_0\to \mathcal A \quad {\rm and} \quad
o[P ]\otimes _{o[P_+]} \mathcal A_{\mathcal C_0}\to \mathcal A_{\mathcal C}\ .
$$

\begin{proposition} Let $M$ be an \'etale $o[P_+]$-module.  We suppose given, for any $g\in N_0\overline PN_0$, an element $\mathcal H_g \in \End_o(M)$.
The map
$$\mathcal R_0:\mathcal A_0\to \End_o(M)\quad , \quad \sum_{g\in N_0\overline PN_0} gf_g \mapsto \mathcal H_g \circ \res (f_g)$$
is a $P_+$-equivariant $o$-algebra homomorphism which extends $\Res |_{\mathcal A_{\mathcal C0}}$ if and only if,
for all $g,h\in N_0\overline{P}N_0$,   $b\in P \cap N_0\overline{P}N_0$, and all compact open subsets $ \mathcal V \subset \mathcal C_{0}$, the relations
\begin{itemize}
\item[H1.] $ \res (1_{\mathcal V })\circ  {\mathcal H}_{g}   =
  {\mathcal H}_{g} \circ \res (1_{g^{-1}\mathcal V \cap \mathcal{C}_0}) $  ,
\item[H2.]  $  {\mathcal H}_{g} \circ  {\mathcal H}_{h} = {\mathcal H}_{gh} \circ \res ( 1_{h^{-1}\mathcal C_{0} \cap \mathcal C_{0}}) $  ,
  \item[H3.] ${\mathcal H}_{b}  = b \circ \res (1_{b^{-1}\mathcal{C}_0 \cap \mathcal{C}_0})$ .
\end{itemize}
hold true.
In this case,   the unique $o[P]$-equivariant map $\mathcal R:\mathcal A \to \End_A(M_c^P)$ extending   $\mathcal R_0$ is multiplicative.
\end{proposition}

 When these conditions are satisfied,    we obtain a  $G$-equivariant sheaf on $G/P$  with sections on $\mathcal C_0$  equal to $M$.

\subsection{$(s,\res,\mathfrak C)$-integrals   $\mathcal H_g$}
Let $M$ be an \'etale $T_+$-module $M $ over $\Lambda_{\ell_\alpha}(N_0)$
 with the weak topology.
 We denote by $\End_o^{cont}(M)$ the $o$-module of continuous $o$-linear endomorphisms of $M$,
and for $g$ in $ N_0\overline{P}N_0$, by $U_g \subseteq N_0$ the compact open subset such that $$U_g w_0 P/P = g^{-1} \mathcal{C}_0 \cap \mathcal{C}_0\ .$$
For $u\in U_g$, we have a unique element
$\alpha(g,u)\in N_0 T $ such that
 $guw_0N=\alpha(g,u)u w_0 N . $
We consider the map
\begin{align*}
\alpha_{g,0}:N_0\to &\End_o^{cont}(M)  \\
\alpha_{g,0}(u)= \Res(1_{\mathcal C_0})\circ \alpha(g,u) \circ \Res(1_{\mathcal C_0})  \ &{\rm for} \ u\in U_g  \ {\rm and} \ \alpha_{g,0}(u)=0  \ {\rm otherwise}.
\end{align*}
 The   module $M$  is Hausdorff  complete but not compact, also
 we introduce a notion of integrability with respect to a special family $\mathfrak C$ of compact subsets  $C\subset M$, i.e.
 satisfying:
\begin{itemize}
\item[$\mathfrak{C}(1)$] Any  compact subset of  a compact set in
$ \mathfrak{C}$  also lies in $\mathfrak{C}$.
\item[$\mathfrak{C}(2)$] If $C_1,C_2,\dots,C_n\in\mathfrak{C}$ then $\bigcup_{i=1}^nC_i$ is in
  $\mathfrak{C}$, as well.
\item[$\mathfrak{C}(3)$] For all $C\in\mathfrak{C}$ we have $N_0C\in\mathfrak{C}$.
\item[$\mathfrak{C}(4)$] $M(\mathfrak{C}):=\bigcup_{C\in\mathfrak{C}}C$ is an \'etale $o[P_+]$-submodule
of $M$.
\end{itemize}

A map  from $M(\mathfrak{C})$ to $M$ is called $\mathfrak{C}$-continuous if its restriction to any
$C\in \mathfrak{C}$ is continuous.
The $o$-module $\Hom_o^{\mathfrak{C}ont}(M(\mathfrak{C}),M)$ of $\mathfrak{C}$-continuous $o$-linear homomorphisms from $M(\mathfrak{C})$ to $M$  with the $\mathfrak C$-open
topology,  is a
 topological complete $o$-module.

For $s\in T_{++}$,  the open compact subgroups $N_k=s^kN_0s^{-k}\subset N$ for $k\in \mathbb Z$, form a decreasing sequence of union $N$ and intersection $\{1\}$. A map $F\colon N_0\to
\Hom_A^{\mathfrak{C}ont} (M(\mathfrak{C}),M)$ is called $(s, \res, \mathfrak{C})$-integrable  if the limit
\begin{equation*}
    \int_{N_0} Fd\res := \lim_{k \rightarrow \infty} \sum_{u \in J(N_0/N_k)} F(u) \circ \res(1_{uN_k}) \ ,
\end{equation*}
where $J(N_0/N_k) \subseteq N_0$, for any $k \in \mathbb{N}$, is a set of representatives for the cosets in $N_0/N_k$, exists in $ \Hom_A^{\mathfrak{C}ont} (M(\mathfrak{C}),M)$ and is independent of the choice of the sets $J(N_0/N_k)$.
We denote by $\mathcal{H}_{g,J(N_0/N_k)}$ the sum in the right hand side when $F=\alpha_{g,0}(.) |_{M(\mathfrak{C})}$.

\begin{proposition} For all  $g\in N_{0}\overline P N_{0}$, the map $\alpha_{g,0}(.) |_{M(\mathfrak{C})} \colon N_0\to
  \Hom_A^{\mathfrak{C}ont} (M(\mathfrak{C}),M)$ is   ($s$, $\res$, $\mathfrak{C}$)-integrable
 when
 \begin{itemize}
\item[$\mathfrak{C}(5)$]  For any $C\in\mathfrak{C}$ the compact subset
  $ \psi_s(C)\subseteq M$ also lies in $\mathfrak{C}$.
\item[$\mathfrak{T}(1)$]  For any   $C\in\mathfrak{C}$ such that $C=N_{0}C$, any open $A[N_0]$-submodule $\mathcal{M}$ of $M$, and any
compact subset $C_+ \subseteq L_+$ there exists a compact open subgroup $P_{1}=P_1(C,\mathcal{M},C_+)
\subseteq P_0$ and an integer $k(C,\mathcal{M},C_+) \geq 0$ such that
\begin{equation*}
s^{k }(1-P_1)C_+ \psi_s^{k} \subseteq E(C,\mathcal M) \qquad\text{for any $k \geq k(C,\mathcal{M},C_+)$} \ .
\end{equation*}
\end{itemize}
The integrals  $\mathcal H_{g} $ of $ \alpha_{g,0}(.) |_{M(\mathfrak{C})}$
satisfy the    relations H1, H2, H3, when   they  belong  $\End_A (M(\mathfrak{C}))$, and when
\begin{itemize}
\item[$\mathfrak{C}(6)$]  For any $C\in\mathfrak{C}$ the compact subset
  $ \varphi_s(C)\subseteq M$ also lies in $\mathfrak{C}$.
  \item[$\mathfrak{T}(2)$] Given
  a  set $J(N_0/N_k)\subset N_{0}$ of representatives for cosets in $N_0/N_k$, for $k\geq 1$, for any $x\in M(\mathfrak{C})$ and $g\in N_0\overline{P}N_0$ there
  exists a compact $A$-submodule $C_{x,g}\in\mathfrak{C}$ and a positive integer $k_{x,g}$ such that
  $\mathcal{H}_{g,J(N_0/N_k)}(x)\subseteq C_{x,g}$ for any $k\geq  k_{x,g}$.
  \end{itemize}
  \end{proposition}
When $\mathfrak{C} $ satisfies $\mathfrak C (1), \ldots, \mathfrak C (6) $ and  the technical properties $\mathfrak T (1) ,\mathfrak T (2) $ are true,  we obtain a  $G$-equivariant sheaf on $G/P$  with sections on $\mathcal C_0$  equal to $M(\mathfrak C )$.

\subsection{Main theorem}  Let $M$ be an \'etale $T_+$-module $M $ over $\Lambda_{\ell_\alpha}(N_0)$
 with the weak topology and let $s\in T_{++}$.
We have the natural $T_+$-equivariant quotient map
 $$\ell_M:M\to D =\mathcal O_{\mathcal E, \alpha} \otimes_{\ell_\alpha }M \quad, \quad m \mapsto 1\otimes m $$
from $M$  to $D=\mathbb D(M)\in \mathcal M_{\mathcal O_{\mathcal E , \alpha}} (T_+)$, of  $T_+$-equivariant section
$$
\iota_D:D\to M = \Lambda_{\ell_\alpha}(N_0)\otimes_{\iota_\alpha }  D\quad, \quad d \mapsto 1\otimes d  \ . $$
We note that  $o[N_0]\iota_D( D)$ is dense in $M$. A lattice $D_0$ in $D$ is a $\Lambda(\mathbb Z_p)$-submodule generated by a finite set of generators of $D$ over $\mathcal O_{\mathcal E}$.
 When  $D$ is killed by a power of $p$, the $o$-module
\begin{equation*}
M_s^{bd}(D_0):=\{m\in M\mid \ell_{M}(\psi_s^k(u^{-1}m))\in D_0\text{ for all }u\in
N_0\text{ and }k\in \mathbb N \}
\end{equation*}
of $M$ is compact and is a $\Lambda (N_0)$-module. Let   $\mathfrak{C}_{s } $ be  the   family    of compact subsets   of $M $ contained in  $M_{s}^{bd}(D_{0})$ for some lattice $D_{0} $ of $D$, and let  $M_s^{bd}=\cup_{D_0}M_s^{bd}(D_0)$ for all lattices $D_0$ in $D$.
 In general, $M$ is $p$-adically complete, $M/p^nM$ is an \'etale $T_+$-module over $\Lambda_{\ell_\alpha} (N_0)$, and  $D/p^n D = \mathbb D (M/p^nM)$. We denote by $p_n:M\to M/p^n M$ the reduction modulo $p^n$, and by $\mathfrak{C}_{s,n } $ the  family of compact subsets constructed above for $M/p^nM$.
We define     the  family $\mathfrak{C}_{s } $ of compact subsets  $C\subset M$ such that $p_n(C) \in \mathfrak{C}_{s,n }$ for all $n\geq 1$, and  the $o$-module $M_s^{bd}$ of $m\in  M$ such that the  set of $\ell _{M}(\psi_s^{k}(u^{-1}m))$  for $k\in \mathbb N, u\in N_{0}$ is bounded in $D$ for the weak topology.

By reduction to the easier case where $M$ is killed by a power of $p$, we show that  $\mathfrak{C}_{s } $ satisfies $\mathfrak C (1), \ldots, \mathfrak C (6) $ and that the technical properties $\mathfrak T (1) ,\mathfrak T (2) $ are true.

 \begin{proposition} Let $M$ be an \'etale $T_+$-module $M $ over $\Lambda_{\ell_\alpha}(N_0)$ and  let $s\in T_{++}$.

  (i)  $M _s^{bd}$ is a dense $\Lambda (N_0)[T_+]$-\'etale submodule of $M$ containing $\iota_D(D)$.

(ii) For $g\in N_0\overline P N_0$, the  $(s, \res, \mathfrak{C}_s)$-integrals $\mathcal H_{g,s}$ of
 $\alpha_{g,0}|_{M_s^{bd}}$  exist, lie in $\End_o (M _s^{bd})$, and satisfy
the relations H1, H2, H3.

(iii) For $s_1, s_2\in T_{++}$, there exists $s_3\in T_{++}$ such that $M_{s_3}^{bd}$ contains $M_{s_1}^{bd} \cup M_{s_2}^{bd}$ and $\mathcal H_{g,s_1}=\mathcal H_{g,s_2}$ on $M_{s_1}^{bd} \cap M_{s_2}^{bd}$.

\end{proposition}
The    endomorphisms $\mathcal H_{g,s}\in  \End_o (M _s^{bd})$ induce endomorphisms of  $ \cap_{s\in T_{++}}M_s^{bd}$ and of
$ \cup_{s\in T_{++}}M_s^{bd}=\sum_{s\in T_{++}}M_s^{bd}$  satisfying the relations H1, H2, H3. Moreover $ \cup_{s\in T_{++}}M_s^{bd}$ and  $ \cap_{s\in T_{++}}M_s^{bd}$
are  $\Lambda (N_0)[T_+]$-\'etale submodules  of $M$ containing $\iota_D(D)$. Our main theorem is the following:

\begin{theorem} There are  faithful functors
$$
\mathbb Y_{\cap},\ (\mathbb Y_{s})_{s\in T_{++}}, \ \mathbb Y_{\cup} :  \mathcal M^{et}_{\mathcal O_{\mathcal E, \alpha}} (T_+) \ \longrightarrow \ \text { $G$-equivariant sheaves on $G/P$} \ ,
$$
sending $D=\mathbb D(M)$ to a sheaf with sections on $\mathcal C_0$ equal to  the dense $\Lambda(N_0)[T_+]$-submodules of $M$
$$
  \bigcap_{s\in T_{++}} M_s^{bd}, \quad  (M_s^{bd})_{s\in T_{++}}, \ \text{and}\quad  \bigcup_{s\in T_{++}} M_s^{bd} \ ,
$$
respectively.
\end{theorem}

When $G=GL(2, \mathbb Q_p)$,    the sheaves  $\mathbb Y_s(D)$ are all equal to the $G$-equivariant sheaf on $G/P\simeq \mathbb P^1(\mathbb Q_p)$  of global sections $D\boxtimes \mathbb P^1$ constructed by Colmez. When the root system of $G$ is irreducible of rank $>1$, we check that $\cup_{s\in T_{++}}M_s^{bd} $  is never equal to $M$.

 \subsection{Structure of the paper}
In  section   2, we consider a general commutative (unital) ring $A$ and   $A$-modules $M$ with two endomorphisms $\psi, \varphi$ such that $\psi \circ \varphi=\id$.  We show that the induction functor $\Ind^{\mathbb Z}_{\mathbb N, \psi}= \varprojlim_\psi  $ is exact and   that the module $A[\mathbb Z]\otimes_{\mathbb N, \varphi} M$ is isomorphic to the subrepresentation of  $\Ind^{\mathbb Z}_{\mathbb N, \psi}(M)=\varprojlim_\psi M$ generated by the elements of the form $(\varphi^k(m))_{k\in \mathbb N}$.

In  section 3, we consider a general monoid  $P_+=N_0\rtimes L_+$ contained in a  group $P$ with the property that $tN_0t^{-1}\subset N_0$ has a finite index for all $t\in L_+$ and we study the \'etale $A[P_+]$-modules $M$. We show that the inverse monoid $P_-=L_- N_0 \subset P$  acts  on $M$,  the inverse of $t\in L_+ $ acting by the left inverse $\psi_t$ of  the action $\varphi_t$  of $t$ with kernel $\sum u \varphi_t(M)$ for $u\in N_0$ not in $tN_0t^{-1}$. We add the hypothesis   that $L_+$ contains a central element $s$ such that the sequence $( s^kN_0 s^{-k})_{k\in \mathbb N}$ is decreasing of  trivial intersection, of union a group $N$, and that  $P=N\rtimes L$ is the  semi-direct product of $N$ and of $L=\cup_{k\in \mathbb N}L_- s^{k}$. An $A[P_+]$-submodule of  $M$ is \'etale if and only if it is stable by $\psi_s$.
The  representation $M^P$ of  $P$ induced by  $M|_{P_-}$,  restricted to $N$ is the representation induced from $M|_{N_0}$, and restricted to $s^{\mathbb Z}$ is the representation $\varprojlim_{\psi_s} M$ induced from $M|_{s^{-\mathbb N}}$. The natural  $A[P_+]$-embedding $M\to M^P$ generates a subrepresentation
 $M^P_c$ of $M^P$ isomorphic to $A[P]\otimes_{A[P_+]}M.$ When $N$ is a locally profinite group and  $N_0$ an open compact subgroup, we show the existence and the uniqueness of a unit-preserving $A[P_+]$-map $\res: C^{\infty}(N_0,A)\to \End_A(M)$, we  extend it  uniquely to an $A[P]$-map $\Res:C^{\infty}(N,A)\to \End_A(M^P)$,   and we prove our first theorem: the equivalence between the $P$-equivariant sheaves of $A$-modules on $N$ and the \'etale $A[P_+]$-modules on $N_0$.

In   section 4, we suppose that $A$ is a linearly topological commutative ring, that $P$ is a locally profinite group  and that $M$ is a  complete linearly topological $A$-module with a continuous \'etale action of $P_+$ such that the action of $P_-$ is also continuous, or equivalently  $\psi_s$ is continuous (we say that $M$ is a topologically \'etale module). Then $M^P$ is complete for  the compact-open topology  and  $\Res$ is a measure  on $N$ with values in the algebra $E^{cont}$ of continuous endomorphisms of $M^P$. We show  that  $E^{cont}$  is a complete topological ring for the topology defined by the ideals $E^{cont}_{\mathcal L}$ of  endomorphisms with image in an open $A$-submodule $\mathcal L \subset M^P$, and  that  any continuous map  $N\to E^{cont}$ can be integrated with respect to  $\Res$.

In  section 5, we introduce a locally profinite group $G$ containing $P$ as a closed subgroup with compact quotient set $G/P$, such that the double cosets $P\backslash G/P$  admit a finite system $W$ of representatives normalizing  $L$,  of  image  in $N_G(L)/L$  equal to a group, and the image $\mathcal C=Pw_0P/P$ in $G/P$ of a double coset (with $w_0\in W$) is   open dense  and  homeomorphic to $N$ by the map $ n\mapsto nw_0 P/P $. We show that any compact open subset of $G/P$ is a finite disjoint union of $g^{-1}Uw_0P/P$ for $g\in G$ and $U\subset N$ a compact open subgroup.
We prove  the basic result that the  $G$-equivariant sheaves of $A$-modules on $G/P$ identify with modules over the skew group ring $  C^\infty (G/P, A)\# G$, or with    non-degenerate  modules over a (non unital) subring $\mathcal A$, and that    an \'etale $A[P_+]$-module $M$ endowed with endomorphisms $\mathcal H_g\in \End_A(M)$, for $g\in N_0 \overline P N_0$,  satisfying certain relations H1, H2, H3, gives rise to a non-degenerate $\mathcal A$-module. For $g\in G$ we denote $N_g \subset N$  such that $N_gw_0P/P=g^{-1}\mathcal C\cap \mathcal C$.
We study the map $\alpha$ from the set of $(g,u)$ with $g\in G$ and $u\in N_g$    to $P$ defined by  $guw_0N= \alpha(g,u)u w_0 N$. In particular, we show the cocycle relation $\alpha(gh,u)=\alpha(g,h.u)\alpha(h,u)$ when each term makes sense.   When $M$ is compact, then $M^P$ is compact and the action of $P$ on $M^P$ induces a continuous map $P\to E^{cont}$.
We show that the $A$-linear map $\mathcal A \to E^{cont}$ given the integrals of $\alpha(g,.) f(.)$ with respect to  $\Res$,  for $f\in C_c^\infty(N_g,A)$,   is multiplicative. As explained above, we obtain  a  $G$-equivariant sheaf of $A$-modules on $G/P$ with sections $M$ on $\mathcal C_0$.

 In   section  6,  we do not suppose that $M$ is compact and  we introduce the notion of $(s,\res, \mathfrak C)$-integrability
 for a special family $\mathfrak C$ of compact subsets of $M$.
We give  an $(s,\res, \mathfrak C)$-integrability criterion for
 the function $\alpha_{g,0}(u)=\Res(1_{N_0})\alpha(gh,u)\Res(1_{N_0})$ on the open subset $U_g \subset N_0$ such that $U_gw_0P/P=g^{-1}\mathcal C_0\cap \mathcal C_0$,  for $g\in N_0 w_0Pw_0N_0$,    a criterion which ensures that the integrals $\mathcal H_g$ of  $\alpha_{g,0}$ satisfy the relations H1, H2, H3, as well as a method of reduction to the case where $M$ is killed by a power of $p$. When these criterions are satisfied, as explained in section 5, one gets a $G$-equivariant sheaf of $A$-modules on $G/P$ with sections $M$ on $\mathcal C_0$.

The section 7 concerns  classical $(\varphi, \Gamma)$-modules  over $\mathcal O_{\mathcal E}$, seen as \'etale $o[P_+^{(2)}]$-module $D$, where the upper exponent indicates that $P_+^{(2)}$ is the upper triangular monoid $P_+$ of $GL(2,\mathbb Q_p)$. Using the properties of treillis we apply the method explained in section 6   to this case and  we obtain the sheaf constructed by Colmez.

In  section 8 we  consider the  case where $N_0$ is a compact $p$-adic Lie group  endowed with a continuous non-trivial  homomorphism $\ell :N_0\to N_0^{(2)}$ with a section $\iota$, that $L_*\subset L$ is a monoid acting by conjugation on $N_0$ and  $\iota(N_0^{(2)})$, that $\ell$ extends to a continuous homomorphism $\ell :P_*=N_0\rtimes L_*\to  P_+^{(2)}$ sending $L_*$ to $L_+^{(2)}$ and that $\iota$ is $L_*$ equivariant.  We consider the abelian categories of \'etale $L_*$-modules finitely generated over the microlocalized ring $\Lambda_\ell (N_0)$ resp. over $\mathcal O_{\mathcal E}$ (with the action of $L_*$ induced by $\ell$).  Between these categories we have the   base change functors given by  the natural $L_*$-equivariant algebra homomorphisms   $\ell:\Lambda_\ell (N_0)\to\mathcal O_{\mathcal E}$ and  $\iota:\mathcal O_{\mathcal E}\to \Lambda_\ell (N_0)$.  We show  our second theorem: the base change functors    are quasi-inverse equivalences of categories.
When $L_*$ contains an open topologically finitely generated pro-$p$-subgroup, we show that an \'etale $L_*$-module over $\mathcal O_{\mathcal E}$ is   automatically  topologically \'etale for the weak topology; the result  extends to
\'etale $L_*$-modules over $\Lambda_\ell (N_0)$, with the help of this last theorem.

 In the section 9,  we  suppose that  $\ell:P\to P^{(2)}(\mathbb Q_p)$  is a continuous homomorphism with $\ell (L)\subset L^{(2)}(\mathbb Q_p)$,  and that  $\iota:N^{(2)}(\mathbb Q_p)\to N$ is a $L$-equivariant section of $\ell|_N$ (as  $L$ acts on $N^{(2)}(\mathbb Q_p)$ via $\ell$) sending $\ell(N_0)$ in $N_0$. The assumptions of section 8 are satisfied for $L_*=L_+$.  Given an \'etale $L_+$-module $M$ over $\Lambda_\ell (N_0)$, we  exhibit  a special family $\mathfrak C _{s}$ of  compact subsets in $M$ which satisfies the criterions of section 6  with $M(\mathfrak C_{s})$ equal to a dense $\Lambda (N_0)[L_+]$-submodule  $M_s^{bd} \subset M$. We obtain our third theorem: there exists a  faithful functor from the \'etale $L_+$-modules over $\Lambda_\ell (N_0)$ to the $G$-equivariant sheaves  on $G/P$ sending $M$ to the sheaf with sections $M_s^{bd}$ on $\mathcal C_0$.

In   section 10, we check  that   our theory  applies to the group $G(\mathbb Q_p)$ of rational points of a split reductive group of $\mathbb Q_p$,  to a Borel subgroup $P(\mathbb Q_p)$ of maximal split torus $T(\mathbb Q_p)=L$  and to a natural homomorphism $\ell_{\alpha} :P(\mathbb Q_p)\to P^{(2)}(\mathbb Q_p)$ associated to a simple root $\alpha$.  We obtain our main theorem: there are compatible faithful functors from the \'etale $T(\mathbb Q_p)_+$-modules $D$ over $ \mathcal O_{\mathcal E}$  (where $T(\mathbb Q_p)_+$ acts via $\alpha$) to the $G(\mathbb Q_p)$-equivariant sheaves  on $G(\mathbb Q_p)/P(\mathbb Q_p)$ sheaves  with sections $\mathbb M(D)_s^{bd}$ on $\mathcal C_0$, for all strictly dominant $s\in T(\mathbb Q_p)$. When  the root system of $G$ is irreducible of rank $>1$, we show that  $\cup_s M_s^{bd} \neq M=\mathbb M(D) $.

 \bigskip {\bf Acknowledgements:} A part of the work on this article was done when the first and  third authors visited the  Institut Math\'ematique de Jussieu at the Universities of Paris 6 and Paris 7, and the second author visited the Mathematische Institut at the Universit\"at M\"unster.  We express our gratitude to these institutions for their hospitality.
We thank heartily  C.I.R.M.,  I.A.S., the Fields Institute,  as well as Durham, Cordoba and Caen Universities, for their invitations giving us the opportunity to present this work.

\section{Induction  $\Ind_{H}^{G}$ for monoids $H \subset G$}
 A monoid is supposed to have   a unit.

\subsection{Definition and  remarks} \label{Induction} Let $A$ be a commutative  ring,  let $G$ be a monoid and let $H $ be a submonoid of $G$.  We denote by $A[G]$  the monoid  $A$-algebra of $G$ and by  ${\mathfrak M}_{A} (G)$    the category of  left $A[G]$-modules, which has no reason to be equivalent to the category of right $A[G]$-modules. One can construct $A[G]$-modules starting from $A[H]$-modules  in two natural ways, by taking the two adjoints of the restriction  functor
$\Res_{H}^{G}: {\mathfrak M}_{A} (G) \to {\mathfrak M}_{A} (H)$ from $G$ to $H$. For  $M \in  {\mathfrak M} _{A}(H)$  and $V \in  {\mathfrak M}_{A} (G )$ we have the isomorphism
$$
\Hom _{A[G]}( A[G]\otimes_{A[H]}M , V)  \xrightarrow {\ \simeq  \ }  \Hom _{A[H]}(M,V)
$$
and the isomorphism
$$
\Hom _{A[G]}(V,\Hom_{A[H]}(A[G],M))  \xrightarrow {\ \simeq  \ }  \Hom _{A[H]}(V,M) \  .
$$
For monoid algebras, restriction of homomorphisms induces the identification
$$\Hom_{A[H]}(A[G],M)  =  \Ind_{H}^{G}(M)$$
where $\Ind_{H}^{G}(M)$ is
 formed by the functions
$$f:G\to M \  {\rm such \ that }\  f(hg)=hf(g) \  {\rm for \ any }\  \  h\in H, g\in G   \ .$$  The group $G$ acts by right translations, $gf(x)=f(xg) $ for $g,x\in G$,
 and    the isomorphism
 pairs a morphism $\phi$ of the left side   and the  morphism $\Phi$ of   the right side    such that
$$ \phi(v) (g)= \Phi(gv)$$
for  $v$ in $V$ and  $g$ in $G$  (\cite{Vig} I.5.7). It is well known that
 the left and right adjoint functors of  $\Res^{G}_{H}$ are transitive  (for monoids $H\subset K\subset G$), the left adjoint is right exact, the right adjoint is left exact.

\bigskip We observe important differences between  monoids and  groups:

1)  The binary relation $g \sim g'$  if  $g\in   Hg'$ is not    symmetric, there is no ``quotient space''  $H\backslash G$, no notion of
function with finite support modulo $H$ in $ \Ind_{H}^{G}(M)$.


2) When $hM=0$ for some   $h\in H$ such that $hG=G$, then $\Ind_{H}^{G}(M) = 0$.  Indeed $f(hg)=h f(g)$ implies $f(hg)=0$ for any $g\in G$.

3) When $G$ is a group generated, as a monoid, by $H$ and the inverse monoid $H^{-1}  := \{ h\in G \  | \ h^{-1}\in H\}$, and when $M$ in an $A[H]$-module such that the action of any element $h\in H$ on $M$ is invertible, then
$f(g)=g f(1)$ for all $g\in G$ and $f\in \Ind_{H}^{G}(M)$.  This can be seen by induction on the minimal number $m\in \mathbb N$ such that $g=g_{1}\ldots g_{m}$ with   $g_{i} \in H \cup H^{-1}$. Then $g_{1}\in H$ implies $f(g)=g_{1}f(g_{2}\ldots g_{m}) $,
and $g_{1}\in H^{-1}$ implies $f(g_{2}\ldots g_{m}) = f(g_{1}^{-1}g_{1} g_{2}\ldots g_{m})=  g_{1}^{-1}f(g)$.
The representation  $\Ind_{H}^{G}(M)$ is isomorphic by $f\mapsto f(1)$ to  the natural representation of $G$ on $M$.


\subsection{From  $\mathbb N$ to $\mathbb Z$}\label{limpro}

An   $A$-module with an endomorphism $\varphi$ is equivalent to an $A[\mathbb N]$-module, $\varphi$ being the action of   $1 \in \mathbb N$, and  an $A$-module with a bijective endomorphism $\varphi$ is equivalent to an $A[\mathbb Z]$-module.
When  $\varphi$  is bijective,
 $A[\mathbb Z] \otimes_{A[\mathbb N]}M$ and $\Ind_{\mathbb N}^{\mathbb Z}(M) $  are isomorphic  to $M$.

 In general, $A[\mathbb Z] \otimes_{A[\mathbb N]}M$ is the limit of an inductive system  and $\Ind_{\mathbb N}^{\mathbb Z}(M) $ is the limit of  a projective system. The first one is interesting when $\varphi$ is injective, the second one when $\varphi$ is surjective.

   For $r\in \mathbb N$ let $M_{r} =M$. The general element of $M_{r}$ is written  $x_{r} $ with $x\in M$.  Let $\varinjlim \ (M,\varphi)$  be the quotient  of $\sqcup_{r\in \mathbb N}M_{r}$  by the equivalence relation  generated by  $ \varphi(x)_{r+1}  \equiv x_{r}$, with the isomorphism
   induced by the maps $x_{r}\to \varphi(x)_{r} \ : \ M_{r} \to M_{r}$  of inverse  induced by the maps $x_{r}\to x_{r+1}\ : \ M_{r}\to M_{r+1}$.
Let $x\mapsto [x]:\mathbb Z \to A[\mathbb Z] $ be the canonical map. The maps $x_{r}  \to [-r]\otimes x\ : \ M _{r} \to A[\mathbb Z] \otimes_{A[\mathbb N]}M$  for  $r\in \mathbb N$  induce an isomorphism   of $A[\mathbb Z]$-modules $$\varinjlim\  M\quad  \to  \quad A[\mathbb Z] \otimes_{A[\mathbb N]}M \quad .
$$  Let   \begin{equation}
  \varprojlim \ M:=  \{x=(x_m)_{m \in \mathbb N} \in \prod_{m \in \mathbb{N}} M :
     \varphi(x_{m+1}) = x_{m}\ \text{for any } \ m
  \in \mathbb{N}\}  \quad.
   \end{equation}
 with the isomorphism
   $x \mapsto (\varphi(x_{0}), x_{0}, x_{1},  \ldots ) = (\varphi(x_{0}),\varphi(x_{1}), \varphi(x_{2})  \ldots ) $  of inverse    $  x\mapsto ( x_{1}, x_{2}, \ldots ) \ . $
 The map $ f\mapsto  (f(-m))_{m \in \mathbb{N}}$    is an isomorphism  of $A[\mathbb Z]$-modules
  $$\Ind_{\mathbb N}^{\mathbb Z}(M) \quad \to \quad   \varprojlim \ M \quad .$$
The submodules of $M$
$$M^{\varphi^{\infty}=0}\ :=\ \cup_{k\in \mathbb N}M^{\varphi^{k}=0} \quad , \quad  \varphi^{\infty}(M) \ :=  \ \cap_{n\in \mathbb N}\  \varphi^{n}(M)$$
are stable by $\varphi$. The inductive limit sees only the quotient  $M/M^{\varphi^{\infty}=0}$ and the projective limit sees only the submodule  $\varphi^{\infty}(M)$,
$$\varinjlim\  M \ = \ \varinjlim\  (M / M^{\varphi^{\infty}=0}) \quad, \quad   \varprojlim \ M\ =\  \varprojlim \ (\varphi^{\infty}(M)) \quad .$$

 \begin{lemma} \label{2.2}  Let   $0\to  M_{1} \to M_{2}\to M_{3} \to 0 $ be an exact sequence
 of $A$-modules with an endomorphism $\varphi$.

 a) The sequence $$0 \to  \varinjlim \ (M_{1})\to   \varinjlim \ (M_{2})\to  \varinjlim \  (M_{3})\to 0$$
 is exact.

b)  When $\varphi$ is surjective on  $M_{1}$,  the  sequence
$$0 \to  \varprojlim \ (M_{1})\to   \varprojlim \ (M_{2})\to  \varprojlim \  (M_{3})\to 0$$
 is exact.
 \end{lemma}
 \begin{proof} a)   It suffices to show that the map  $\varinjlim \ (M_{1})\to   \varinjlim \ (M_{2})$ is injective.
 Let $x$ in the kernel and let $y_{r}\in M_{r}$ a representative of $x$.  Let $z$ be the image of $y$ in $M_{2}$. Then $z_{r}$ is a representative of the image of $x$ in $\varinjlim \ (M_{2})$.
  There exists  $k\in \mathbb N$ such that  $\varphi^{k}(z)=0$. As
 the map $M_{1}\to M_{2}$ is injective and commutes with $\varphi$ we have
  $\varphi^{k}(y)=0$.  Hence $x=0$.

 b) It suffices to show that  for   $x \in M_{3}, y \in M_{2} $  of image $\varphi(x)$, there exists $z\in M_{2}$ of image $x $ such that $\varphi(z)=y$.
 Take any $z'\in M_{2}$ of image $x$ and $\varphi(z')=y'$. The difference
 $y'-y$ belongs to the image of $M_{1}$ and $\varphi$ being surjective on $M_{1}$ there exists $t\in M_{2}$ in the image of $M_{1}$ such that $\varphi(t)=y-y'$. Take $z=z'+t$.
   \end{proof}

 \subsection{$(\varphi, \psi)$-modules}
 Let   $M$ be  an $A$-module with two  endomorphisms $\psi, \varphi$   such that $\psi \circ \varphi =1 $.
    Then   $\psi$ is surjective,    $\varphi$ is injective,
the endomorphism   $\varphi \circ \psi $ is a projector of $M$ giving the direct  decomposition
 \begin{equation}\label{decom}  \quad  M= \varphi(M) \oplus   M^{\psi=0} \quad , \quad m \ = \ (\varphi \circ \psi)(m) \ + \ m^{\psi =0}
 \end{equation}
for  $m\in M$  and   $m^{\psi =0} \in M^{\psi=0}$  the kernel of $\psi$.
We consider the representation of $\mathbb Z$   induced by $(M,\psi)$ as in (\ref{limpro}),
$$\Ind_{\mathbb N, \psi} ^{\mathbb Z} (M)\  \simeq \   \underset{\psi} \varprojlim \ (M) \ . $$
On the induced representation $\psi$ is an isomorphism and we introduce $\varphi:=\psi^{-1}$.
As $\psi$ is surjective on $M$, the   map $\ev_{0}\ : \ \Ind_{\mathbb N, \psi} ^{\mathbb Z} (M) \to M$, corresponding to the map
 $$  \underset{\psi} \varprojlim \ (M) \ \to \ M  , \quad  (x_{m})_{m\in \mathbb N} \mapsto  x_{0} \quad
  $$
is surjective.  A splitting is  the map
  $\sigma_{0}\ : \ M\to \Ind_{\mathbb N, \psi} ^{\mathbb Z} (M)$ corresponding to
\begin{equation} \label{split} M\ \to\  \underset{\psi} \varprojlim \ (M) \ , \quad   x \mapsto  (\varphi^{m}(x))_{m\in \mathbb N} \quad  .
\end{equation}
  Obviously $\ev_{0}$ is $\psi$-equivariant,  $\sigma_{0}$ is $\varphi$-equivariant,
  $\ev_{0} \ \circ \ \sigma_{0}= \id_{M}$,  and
    $$R_{0}\quad :=\quad  \sigma_{0} \circ \ev_{0}   \ \in  \End_{A}(\Ind_{\mathbb N, \psi} ^{\mathbb Z} (M) )  $$
    is an idempotent  of image $\sigma_{0}(M)$.

    \begin{definition} \label{ci} The representation of $\mathbb Z$ compactly induced from $(M,\psi)$ is the subrepresentation $\ind_{\mathbb N, \psi} ^{\mathbb Z} (M) $     of  $\Ind_{\mathbb N, \psi} ^{\mathbb Z} (M) $ generated by the image of  $\sigma_{0}(M )$.
  \end{definition}

We note that,  for any $k\geq 1$,   the endomorphism $\psi^{k}, \varphi^{k}$ satisfy   the same properties than $\psi, \varphi$ because  $\psi^{k}\circ  \varphi^{k} =1 $.
 For any integer $k\geq 0$,   the value at $k$ is  a   surjective  map $\ev_{k}\ : \ \Ind_{\mathbb N, \psi} ^{\mathbb Z} (M) \to M$, corresponding to the map
     \begin{equation}  \underset{\psi} \varprojlim \ (M) \ \to \ M  , \quad   (x_{m})_{m\in \mathbb N} \mapsto  x_{k} \quad    \end{equation}
of splitting   $\sigma_{k}\ : \ M\to \Ind_{\mathbb N, \psi} ^{\mathbb Z} (M)$ corresponding to the map
    \begin{equation}\label{6}  M\ \to\  \underset{\psi} \varprojlim \ (M) \ , \quad   x \mapsto (\psi^{k}(x),\ldots , \psi(x),  x, \varphi(x), \varphi^{2}(x),  \ldots ) \quad. \end{equation}

The following relations are immediate:
  $$\ev_{k}  \ = \  \ev_{0}  \circ \, \varphi^{k} \  = \  \psi \circ \ev_{k+1}   \  = \   \ev_{k+1} \circ \, \psi  \  ,
  $$
    $$\sigma_{k} \ = \ \psi^{k} \circ \sigma_{0}  \  = \ \sigma_{k+1}\circ \varphi \ =  \ \varphi  \circ \sigma_{k+1}\   . $$
We deduce that $\sigma_{k} (M)\subset \sigma_{k+1} (M)$.
      Since $\sigma_k(M)$ is $\varphi$-invariant we have
\begin{equation}\label{deco}
\ind_{\mathbb N, \psi} ^{\mathbb Z} (M) \ =\ \sum _{k\in \mathbb N} \psi^{k} (\sigma_{0}(M))\ =\ \sum _{k\in \mathbb N} \sigma_{k}(M) \  = \  \bigcup_{k\in \mathbb N} \sigma_{k}(M)  \ .
\end{equation} In $  \underset{\psi} \varprojlim \ (M)$  the subspace  of $(x_{m})_{m\in \mathbb N} $ such that  $x_{k+r} =  \varphi^{k}(x_{r})  $ for all $k\in \mathbb N $ and  for  some $r\in \mathbb N $, is equal to $\ind_{\mathbb N, \psi} ^{\mathbb Z} (M) $.
The definition of $\ind_{\mathbb N, \psi} ^{\mathbb Z} (M)$ is functorial. We get a functor  $\ind_{\mathbb N, \psi} ^{\mathbb Z}$ from the category of $A$-modules with two  endomorphisms $\psi, \varphi$   such that $\psi \circ \varphi =1 $ (a morphism  commutes with $\psi$ and with $\varphi$)  to the category of $A[\mathbb Z]$-modules.

\begin{proposition}\label{pexx}
The  map
\begin{align*}
 A[\mathbb Z]\otimes_{A[\mathbb N], \varphi} M & \ \to \  \Hom_{A[\mathbb N], \psi} ( A[\mathbb Z],M)\ = \  \Ind_{\mathbb N, \psi} ^{\mathbb Z} (M) \\
 [k] \otimes m & \ \mapsto \ (\varphi^k \circ \sigma_0)(m)
\end{align*}
induces an isomorphism from the tensor product  $A[\mathbb Z]\otimes_{A[\mathbb N], \varphi} M$ to  the compactly induced representation  $ \ind_{\mathbb N, \psi} ^{\mathbb Z} (M)$  (note that $\psi$  and $\varphi$ appear).
\end{proposition}
\begin{proof}
  From (\ref{decom}) and the relations between the $\sigma_{k}$ we have  for $m\in M, k\in \mathbb N, k\geq 1$,
 $$
 \sigma_{k}(m)=\sigma_{k-1}(\psi(m))+\sigma_{k}(m^{\psi =0}) \ .
 $$
  By induction  $\sum_{k\in \mathbb N}\sigma_{k}(M)=\sigma_{0}(M)+\sum_{k\geq 1}\sigma_{k}(M^{\psi=0})$. Using (\ref{6}) one checks that  the sum is direct, hence by (\ref{deco}),
$$
\ind_{\mathbb N, \psi} ^{\mathbb Z} (M)\ = \ \sigma_{0}(M) \oplus (\oplus _{k\geq 1}\sigma_{k}(M^{\psi=0}) ) \ .
$$
On the other hand, one deduces from (\ref{decom})  that
$$
A[\mathbb Z]\otimes_{A[\mathbb N], \varphi} M \quad = \quad ([0]\otimes M) \oplus (\oplus_{k\geq 1} ([-k]\otimes M^{\psi = 0}) ) \ .
$$

  \end{proof}

With the lemma \ref{2.2}   we deduce:

\begin{corollary}\label{exx}  The functor  $\ind_{\mathbb N, \psi} ^{\mathbb Z}$   is exact.
\end{corollary}

 We have two kinds of idempotents   in $\End_{A}( \Ind_{\mathbb N, \psi} ^{\mathbb Z} (M))$, for $k\in \mathbb N$, defined by
 \begin{equation}\label{2.9}
 R_{k}\ :=\    \sigma_{0} \circ  \varphi ^{k} \circ \psi^{k}\circ  \ev_{0} \quad, \quad  R_{-k}\  :=\  \psi^{k} \circ R_{0}\circ \varphi^{k} = \sigma_{k} \circ \ev_{k} \quad .
   \end{equation}
 The first ones  are the images of   the idempotents  $r_{k}:=\varphi ^{k} \circ \psi^{k} \in \End_{A}(M)$ via the ring homomorphism
 \begin{equation}\label{evo}  \End_{A}(M) \ \to \  \End_{A} \Ind_{\mathbb N, \psi} ^{\mathbb Z} (M) \quad, \quad  f\mapsto   \sigma_{0} \circ  f\circ  \ev_{0} \quad .
  \end{equation}
 The second ones give an isomorphism from $ \Ind_{\mathbb N, \psi} ^{\mathbb Z} (M)$ to the limit of the projective system $(\sigma_{k}(M), R_{-k}:\sigma_{k+1}(M)\to \sigma_{k}(M))$.
    \begin{lemma} \label{surj} The map
    $f\mapsto (R_{-k}(f))_{k\in \mathbb N}$ is an isomorphism from $ \Ind_{\mathbb N, \psi} ^{\mathbb Z} (M)$ to
    $$\underset{R_{-k}} \varprojlim (\sigma_{k}(M)) \quad : = \quad  \{(f_{k})_{k\in \mathbb N} \ | \ f_{k}\in \sigma_{k}(M) \ , \  f_{k} = R_{-k}(f_{k+1}) \quad {\rm for } \ k \in \mathbb N\
    \}$$
    of inverse $(f_{k})_{k\in \mathbb N} \ \to f$ with $\ev_{k}(f)=\ev_{k}(f_{k})$.
      \end{lemma}

\begin{remark} {\rm  As $\varphi$ is injective,
its restriction to $\cap_{n\in \mathbb N}\varphi^{n}(M)$ is an isomorphism and the  following  $A[\mathbb Z]$-modules are isomorphic  (section \ref{limpro}):
$$\Ind_{\mathbb N, \varphi} ^{\mathbb Z} (M)\quad  \simeq \quad \underset{\varphi} \varprojlim  \ (M) \quad  \simeq \quad  \cap_{n\in \mathbb N}\varphi^{n}(M) \  . $$
  As $\psi$ is surjective, its action on the quotient $M/M^{\psi^{\infty}=0}$ is bijective and the  following  $A[\mathbb Z]$-modules are isomorphic  (section \ref{limpro}):
 $$A[\mathbb Z] \otimes_{A[\mathbb N], \psi} M \quad  \simeq \quad  \underset{\psi} \varinjlim  \ (M) \quad   \simeq \quad  M/M^{\psi^{\infty}=0}.\quad  $$
 }
 \end{remark}

   \begin{remark} When  the $A$-module $M$ is noetherian,  a $\psi$-stable submodule  of $M$ which generates $M$ as a $\varphi$-module is equal to $M$.
 \end{remark}

 \begin{proof} Let  $N$ be a submodule of $M$.  As $M$ is  noetherian  there exists    $k \in \mathbb{N}$ such that the $\varphi$-stable submodule of $M$ generated by $N$ is the submodule $N_{k} \subset M$   generated by $N, \varphi(N), \ldots , \varphi^{k}(N)$.  When $N$ is $\psi $-stable  we have $\psi^{k}(N_{k})=N$ and when $N$ generates $M$ as a $\varphi$-module  we have  $M=N_{k}$. In this case, $M= \psi^{k}(M)=\psi^{k}(N_{k})=N.$
 \end{proof}

 \section{Etale $P_{+}$-module}\label{S3}

{\sl  Let $P=N\rtimes L$ be a semi-direct product of an invariant subgroup $N$ and of a group $L$ and let
 $N_{0}\subset N$ be a subgroup of $N$. For any subgroups $V\subset U \subset N$, the symbol $J(U/V)\subset U$ denotes a set of representatives for the cosets in $U/V$.}

\bigskip  The group $P$ acts on $N$ by
$$(b= n t, x)  \to b.x = ntxt^{-1}$$
for $n,x\in N$ and $t\in L$. The $P$-stabilizer $\{ b\in P \ | \ b.N_{0}\subset N_{0}\}$  of $N_{0}$ is  a monoid   $$P_{+} \ = \ N_{0}L_{+}$$
where    $L_{+} \subset L$   is the $L$-stabilizer of $N_{0}$. Its maximal subgroup $\{ b\in P \ | \ b.N_{0}=N_{0}\}$ is the intersection $P_{0}=N_{0}\rtimes L_{0}$ of $P_{+}$ with the inverse monoid $P_{-}= L_{-}N_{0}$ where $L_{-}$ is the inverse monoid of $L_{+}$ and $L_{0}$ is the maximal subgroup of $L_+$.

\bigskip {\sl We suppose that  the subgroup $t.N_{0} =t N_{0}t^{-1}\subset N_{0}$ has a finite index, for all  $t\in L_{+}$. Let  $A$ be a commutative ring and let $M$ be an $A[P_{+}]$-module,  equivalently an $A[N_{0}]$-module with a semilinear action of $L_{+}$.  }

 \bigskip The action  of  $b  \in P_{+}$  on $M$  is  denoted by $\varphi_{b}$. When $p\in P_{0}$ then $\varphi_{b}$ is invertible and  we write  also $\varphi_{b}(m)=bm \ , \  \varphi_{b}^{-1}(m) =b^{-1}m$ for  $m\in M$.
  The action   $\varphi_{t}\in \End_{A}(M)$ of $t\in L_{+}$
   is $A[N_{0}]$-semilinear:
\begin{equation}\label{produit1}
\varphi_{t}(xm)=\varphi_{t}(x)\varphi_{t}(m) \quad {\rm for } \quad x\in A[N_{0}],  \ m\in M\  .
\end{equation}

 \subsection{Etale}

The  group algebra $A[N_{0}]$ is naturally an $A[P_{+} ]$-module. For $t\in L_{+}$, then  $\varphi_{t}$  is  injective of image $A[tN_{0} t^{-1}]$,  and
 $$A[N_{0}]= \oplus_{u \in  J(N_{0}/tN_{0}t^{-1})}u A[tN_{0} t^{-1}]  \ .$$

 \begin{definition}  We say that  $M $ is \'etale if, for any $t\in L_{+}$,   $\varphi_{t}$ is injective and
  \begin{equation} \label{almostetale} M =\oplus_{u\in  J(N_{0}/tN_{0}t^{-1})} \ u \ \varphi_{t}(M) \  .
  \end{equation}
 An equivalent formulation  is that, for any $t\in L_{+}$,   the   linear map
   $$A[N_{0}]\otimes_{A[N_0], \varphi_{t}}M\to M \ \ , \ \ x\otimes m \mapsto x\varphi_{t}(m)$$
    is bijective.
    For  $M$  \'etale and   $t\in L_{+}$, let $\psi_{t}\in \End_{A}(M)$ be the unique canonical left inverse of $\varphi_{t}$  of kernel $M^{\psi_{t}=0}= \sum_{u\in (N_{0}-tN_{0}t^{-1})}u\varphi_{t}(M) $.

    \end{definition}

 The trivial action of $P_{+}$ on $M$ is not \'etale, and obviously  the restriction to $P_{+}$ of a representation of $P$ is not always \'etale.

 \begin{lemma} \label{idemp} Let $M$ be an \'etale $A[P_{+}]$-module. For  $t\in L_{+}$,  the
  kernel $M^{\psi_{t}=0}$ is an $A[tN_{0}t^{-1}]$-module, the idempotents in  $\End_{A}M$
 $$(u  \circ \varphi_{t} \circ \psi_{t}  \circ  u ^{-1})_{u \in J(N_{0}/t N_{0}t^{-1} ) }$$
 are orthogonal of sum the identity. Any $m\in M$
can be written
\begin{equation}\label{writing}
 m=\sum_{u\in J(N_{0}/tN_{0}t^{-1})} u \varphi_{t}(m_{u,t})
 \end{equation}
for unique elements $m_{u,t}\in M$,  equal to  $m_{u,t}=\psi_t(u^{-1}m) $.
\end{lemma}
\begin{proof} The kernel $M^{\psi_{t}=0}$ is an $A[tN_{0}t^{-1}]$-module because $N_{0}- tN_{0}t^{-1}$ is stable by left multiplication by $tN_{0}t^{-1}$. The endomorphism  $\varphi_{t}\circ \psi_{t}$ is an idempotent  because $\psi_{t}   \circ \varphi_{t}=\id_{M}$.  Then apply (\ref{almostetale}) and notice that $m\in M$ is equal to
  $$m = \sum_{u\in J(N_{0}/tN_{0}t^{-1})} (u\circ  \varphi_{t}\circ  \psi_{t} \circ u^{-1})(m) \ .
  $$
\end{proof}

\begin{remark}\label{cr}  {\rm  1) An $A[P_{+}]$-module $M$ is \'etale when, for any $t\in L_{+}$, the action $\varphi_{t}$ of $t$ admits a left inverse  $f_{t}\in \End_{A}M$  such that  the idempotents $(u   \circ \varphi_{t} \circ f_{t}  \circ u^{-1})_{u \in J(N_{0}/t N_{0}t^{-1} ) } $
 are orthogonal of sum the identity. The endomorphism $f_{t}$ is the canonical left inverse $\psi_{t}$.

2) The  $A[P_{+}]$-module $A[N_{0}]$ is \'etale. As $A[N_{0}]$ is a left and right free $A[tN_{0}t^{-1}] $-module of rank $[N_{0}: tN_{0}t^{-1}]$ we have for $x\in A[N_{0}]$ ,
$$x = \sum_{u\in J(N_{0}/tN_{0}t^{-1})} u \varphi_{t}(x_{u,t}) =  \sum_{u\in J(N_{0}/tN_{0}t^{-1})}
 \varphi_{t}(x'_{u,t}) u^{-1}$$
 where $x_{u,t}=\psi_{t}(u^{-1}x)$, $ x'_{u,t}=\psi_{t}( xu) $ and  $\psi_{t}$ is   the left inverse of $\varphi_{t}$ of kernel
 $$\sum_{u \in N_{0}-tN_{0}t^{-1}} u A[tN_{0}t^{-1}]  = \sum_{u \in N_{0}-tN_{0}t^{-1}}A[tN_{0}t^{-1}] u^{-1}.$$
}
 \end{remark}

 Let $M$ be an \'etale $A[P_{+}]$-module and $t\in L_{+}$.  We denote $m\mapsto m^{\psi_{t}=0}:M\to M^{\psi_{t}=0}$ the projector $\id_{M}-\varphi_{t} \circ \psi_{t}$ along the decomposition $M=\varphi_{t}(M)\oplus M^{\psi_{t}=0}$.

 \begin{lemma} \label{produit}  Let  $x \in A[N_{0}]$ and $m \in M$. We have  $$ \psi_{t} (\varphi_{t}(x )m)\ = \ x \psi_{t} (m)\ , \quad  \psi_{t} (x \varphi_{t}(m))\ = \ \psi_{t} (x) m \ , $$
$$ ( \varphi_{t}(x) m )^{ {\psi_{t}=0} }\ =  \ \varphi_{t}(x) (m^{ {\psi_{t}=0}} ) \quad , \quad (x \varphi_{t}(m) )^{ {\psi_{t}=0} }\ =  \  x^{ {\psi_{t}=0} } \varphi_{t}(m) \quad .
$$
   \end{lemma}

 \begin{proof}  We multiply $m = (\varphi_{t}\circ  \psi_{t})(m) + m^{\psi_{t}=0}$ on the left by   $\varphi_{t}(x)$. By the $A[N_{0}]$-semilinearity of $\varphi_{t}$ we have $\varphi_{t}(x) m =  \varphi_{t} (x \psi_{t} (m)) + \varphi_{t}(x)  (m^{ {\psi_{t}=0} })$.  As $M^{\psi_{t}=0}$ is an $A[tN_{0} t^{-1} ]$-module, the uniqueness of the decomposition  implies
$\psi_{t} ( \varphi_{t}(x) m )= x \psi_{t} (m) $ and $   ( \varphi_{t}(x) m )^{\psi_{t}=0}= \varphi_{t}(x)  (m^{ {\psi_{t}=0} }) $ .

We multiply  $x = (\varphi_{t}\circ  \psi_{t})(x) + x^{\psi_{t}=0}$ on the right by   $\varphi_{t}(m)$. By the semilinearity of $\varphi_{t}$ we have
$x \varphi_{t}(m) = \varphi_{t} (\psi_{t} (x)m) \oplus   x^{\psi_{t}=0}\varphi_{t}(m) $. As $ A[N_{0}]  ^{\psi_{t}=0}\varphi_{t}(M) = M^{\psi_{t}=0}$ the uniqueness of the decomposition   implies
   $\psi_{t}(x \varphi_{t}(m) ) = \psi_{t} (x) m \ , \   (x \varphi_{t}(m) )^{ {\psi_{t}=0} } =   x^{ {\psi_{t}=0} } \varphi_{t}(m)  .$
   \end{proof}

  \begin{lemma} \label{produit2}  Let  $x \in A[N_{0}]$ and $m\in M$.  We have
 $$
  \psi_{t} (xm)=\sum_{u\in J(N_{0}/tN_{0}t^{-1})}\psi_{t} (xu)\psi _{t}(u^{-1}m) \ .
  $$
   \end{lemma}

 \begin{proof} Replace $x$ by  $\sum_{u\in J(N_{0}/tN_{0}t^{-1})}
 \varphi_{t}(x'_{u,t}) u^{-1}$ and $m$ by $\sum_{v\in J(N_{0}/tN_{0}t^{-1})} v \varphi_{t}(m_{v,t})$ in  $\psi_{t} (xm)$. We get
$$
\psi_{t} (xm)= \sum_{u,v\in J(N_{0}/tN_{0}t^{-1})}\psi_{t}( \varphi_{t}(x'_{u,t}) u^{-1}v \varphi_{t}(m_{v,t})) \ .
$$
The kernel of $\psi_{t}$, being an $A[tN_{0}t^{-1}]$-module, is equal to
$$
M^{\psi_{t}=0} = \sum _{u\in N_{0}-t N_{0}t^{-1}} A[tN_{0}t^{-1}]u \varphi_{t}(M) \ . $$
Hence $\psi_{t}( \varphi_{t}(x'_{u,t}) u^{-1}v \varphi_{t}(m_{v,t}))=0$ if $u\neq v$, and $ \psi_{t} (xm)=  \sum_{u\in J(N_{0}/tN_{0}t^{-1})}x'_{u,t}m_{u,t}$.
\end{proof}

\bigskip

  \begin{proposition}  \label{3.2} Let $M$ be an \'etale  $A[P_{+}]$-module. The
  map $$b^{-1}=(ut) ^{-1} \ \mapsto \ \psi_{b} := \psi_{t} \circ u^{-1}\ : \ P_{-}\  \to \ \End_{A}(M) \quad {\rm for} \quad t\in L_{+}\ ,\  u\in N_{0} \quad , $$  defines a canonical  action of $P_{-}$ on $M$.
     \end{proposition}

\begin{proof}
We  check that  $\psi_{b_{1}b_{2}} = \psi_{b_{2}} \circ \psi_{b_{1}} $ for $ b_{1}= u_{1}t_{1},b_{2}=u_{2} t_{2}\in  P_{+}$.  We have  $ \psi_{b_{1}b_{2}} = \psi_{t_{1}t_{2}} \circ (u_{1}t_{1}u_{2}t_{1}^{-1} )^{-1}$ and  $\psi_{b_{2}} \circ \psi_{b_{1}} =  \psi_{t_{2} } \circ u_{2} ^{-1} \circ  \psi_{t_{1} } \circ u_{1} ^{-1} $. As $ u_{2} ^{-1} \circ  \psi_{t_{1} }=  \psi_{t_{1} } \circ t_{1} u_{2} ^{-1} t_{1}^{-1} $. it remains only  to show $\psi_{t_{2}} \psi_{t_{1}}  = \psi_{t_{1}t_{2}} $.  For the sake of simplicity, we note $\varphi_{i} =\varphi_{t_{i}}, \psi_{i}=\psi_{t_{i}}$.
 For $m\in M$ we have
 $m=\varphi_{1}(( \varphi_{2} \circ \psi_{2})(\psi_{1} (m)+ \psi_{1} (m)^{\psi_{2}=0})+ m^{\psi_{1}=0}$.
 This is also $$m= ( \varphi_{t_{1}t_{2}} \circ \psi_{2} \circ \psi_{1} )(m)+  \varphi_{1}( \psi_{1} (m)^{\psi_{2}=0})+ m^{\psi_{1}=0}$$  because
 $\varphi_{t_{1}} \circ  \varphi_{t_{2}} = \varphi_{t_{1}t_{2}}$.
 By the uniqueness of the decomposition
  $m=(\varphi_{t_{1}t_{2}}\circ \psi_{t_{1}t_{2}})(m)+m^{\psi_{t_{1}t_{2}} =0}$ we are reduced to show that
 $$M^{\psi_{t_{1}t_{2}} =0} \quad  =  \quad \varphi_{t_{1}}(M^{\psi_{t_{2}}=0})+  M^{\psi_{t_{1}}=0}  \quad . $$
 It is enough to prove the inclusion $M^{\psi_{t_{1}t_{2}} =0} \ \subset \ \varphi_{t_{1}}(M^{\psi_{t_{2}}=0})+  M^{\psi_{t_{1}}=0} $ to get the equality because $M = \varphi_{t_{1}t_{2}}(M) \oplus V$ with $V$ equal to any of them.  Hence we want to show
\begin{equation}\label{conta} \sum_{u\in N_{0}-t_{1}t_{2}N_{0}(t_{1}t_{2})^{-1} } u  \varphi_{t_{1}t_{2}}(M)   \ \subset  \ \varphi_{1} (\sum_{u\in N_{0}-t_{2}N_{0}t_{2}^{-1} } u  \varphi_{t_{2}}(M))\ +\sum_{u\in N_{0}-t_{1}N_{0}t_{1}^{-1} } u  \varphi_{t_{1}}(M) \ .
\end{equation}
As $\varphi_{t_{1}} \circ u  \circ \varphi_{t_{2}}= t_{1}ut_{1}^{-1}\circ \varphi_{t_{1}t_{2}}$ the right side of (\ref{conta})  is
$$ \sum_{u\in t_{1}N_{0}t_{1}^{-1}-t_{1}t_{2}N_{0}(t_{1}t_{2})^{-1} } u  \varphi_{t_{1}t_{2}}(M) +\sum_{u\in N_{0}-t_{1}N_{0}t_{1}^{-1} } u  \varphi_{t_{1}}(M) \quad .
$$
As $\varphi_{t_{1}t_{2}} =\varphi_{1}\circ \varphi_{t_{2}}$ we have $\varphi_{t_{1}t_{2}} (M)\subset \varphi_{t_{1}} (M)$.  Hence  (\ref{conta}) is true.
\end{proof}

\begin{lemma}\label{-eq}
Let   $f:M\to M'$  be an $A$-morphism between two \'etale  $A[P_{+}]$-modules $M$ and $M'$. Then $f$ is $P_{+}$-equivariant if and only if $f$ is $P_{-}$-equivariant (for the canonical action of  $P_{-}$).
\end{lemma}
\begin{proof}
Let $t\in L_{+}$.  We  suppose  that $f$ is $N_{0}$-equivariant and we show that $f \circ \varphi_{t}=\varphi_{t} \circ f$ is equivalent to $f \circ \psi_{t}=\psi_{t} \circ f$. Our arguments follow the proof of (\cite{Mira} Prop. II.3.4).

a) We suppose $f \circ \varphi_{t}=\varphi_{t} \circ f$. Then $f (\varphi_{t}(M))= \varphi_{t}(f(M))$  is contained in $\varphi_{t}(M')$ and
$f(M^{\psi_{t}=0})=\sum_{u\in N_{0}-tN_{0}t^{-1}} \varphi_{t}(f(M))$ is contained in $M'^{\psi_{t}=0}$. By Lemma \ref{idemp}, this implies $f \circ \varphi_{t}\circ \psi_{t} = \varphi_{t} \circ \psi_{t} \circ f$. As $f \circ \varphi_{t}= \varphi_{t}\circ f$ and $\varphi_{t}$ is injective this is equivalent to $f \circ \psi_{t} =    \psi_{t} \circ f$.

b) We suppose $f \circ \psi_{t}=\psi_{t} \circ f$. Let $m\in M$. Then $f (\varphi _{t}(m))$ belongs to $\varphi _{t}(M)$ because  $\varphi _{t}(M)$ is the subset of $x\in M$ such that $\psi_{t}(u^{-1}x)=0$ for all $u\in N_{0}-tN_{0}t^{-1}$ and we have
$$
\psi_{t}(u^{-1}f (\varphi_{t}(m)))= f (\psi_{t}(u^{-1}(\varphi_{t} (m)))) \ .
$$
Let $x(m)\in M$ be the element such that $f (\varphi _{t}(m))=\varphi_{t}(x(m))$. We have
$$
x(m) = \psi_{t} \varphi_{t}(x(m)) =\psi_{t} (f (\varphi _{t}(m)))= f (\psi_{t} \varphi_{t}(m))= f(m) \ .
$$
Therefore  $f (\varphi _{t}(m))=\varphi_{t}( f(m))$.
\end{proof}

\begin{proposition} \label{abetale} The category ${\mathfrak M}_{A}(P_{+})^{et}$ of \'etale $A[P_{+}]$-modules  is abelian and has a natural fully faithful functor into the abelian category  ${\mathfrak M}_{A}(P_{-})$ of $A[P_{-}]$-modules.
\end{proposition}

\begin{proof} From the proposition \ref{3.2} and the lemma \ref{-eq}, it suffices to show that
 the kernel and the image of  a morphism $f: M\to M'$  between two \'etale modules $M,M'$,  are  \'etale.
 Since the ring homomorphism $\varphi_t$ is flat,  for $t\in L_{+}$,
 the functor $\Phi_{t}:= A[N_{0}]\otimes_{A[N_0], \varphi_{t}}-$ sends  the exact sequence
\begin{equation} (E) \qquad 0\to \Ker f \to M \to M' \to \Coker f \to 0
\end{equation} to an exact sequence
\begin{equation}(\Phi_{t}(E))\qquad    0\to \Phi_{t}(\Ker f ) \to \Phi_{t}(M) \to \Phi_{t}(M') \to \Phi_{t}(\Coker f) \to 0\ ,
\end{equation}
and the natural maps  $j_{-}: \Phi_{t}(-)\to -$ define a map  $\Phi_{t}(E) \to (E)$.
The maps $j_{M}$ and $j_{M'}$ are isomorphisms because $M$ et $M'$ are \'etale, hence the maps $j_{\Ker f}$ and $j_{\Coker f}$ are isomorphisms, i.e. $\Ker f$ and $\Coker f$ are \'etale.
\end{proof}
Note that a subrepresentation of an \'etale representation of $P_{+}$ is not necessarily \'etale and stable by $P_{-}$.

\begin{remark} \label{projlim} An arbitrary direct product or a projective limit of \'etale $A[P_{+}]$-modules is \'etale.
 \end{remark}
\begin{proof} Since the $A[tN_{0}t^{-1}]$-module $A[N_{0}]$ is  free of finite rank, for $t\in L_{+}$, the tensor product $A[N_{0}]\otimes_{A[tN_{0}t^{-1}]} - $ commutes with arbitrary projective limits.
\end{proof}
\subsection{Induced representation $M^{P}$} \label{2.4}

{\sl  Let  $P$ be a locally profinite group, semi-direct product $P=N\rtimes L$ of  closed subgroups $N,L$,  let $N_{0}\subset N$ be an open profinite subgroup, and let $s$ be an element of the centre $Z(L)$ of $L$ such that
$L=L_{-} s^{\mathbb Z}$ (notation of the section \ref{S3})
and  $(N_{k} :=s^{k}N_{0}s^{-k}) _{ k\in \mathbb Z}$
is a  decreasing sequence of union $N$ and  trivial intersection. }

\bigskip  As the conjugation action $L\times N \to N$ of $L$ on $N$ is continuous and $N_{0}$ is compact open in $N$, the subgroups $L_{0}\subset L, P_{0}\subset P$ are open and the monoids $P_{+}, P_{-}$ are open
  in $P$.

 We have
 $$P=P_{-}s^{\mathbb Z}\quad  = \quad s^{\mathbb Z} P_{+}$$
because,
for $n\in N$ and $t\in L$, there exists $k\in \mathbb N$ and $n_{0}\in N_{0}$ such that $n=s^{-k}n_{0}s^{k}$ and  $t s^{-k}\in L_{-}$. Thus
 $  tn= t s^{-k}n_{0}s^{k}  \in P_{-}s^{k}$ and $(tn)^{-1} \in s^{-k}P_{+}$.  In particular $P$ is generated by $P_{+}$ and by its inverse $P_{-}$.

 Let   $M $ be an \'etale left $A[P_{+}]$-module.    We denote by $\varphi$ the action of $s$ on $M$ and by $\psi$   the canonical left inverse of $\varphi$,
 by  $$M^{P} \ : = \ \Ind_{P_{-}}^{P}(M)  $$
the $A[P]$-module induced from the canonical action of $P_{-}$ on $M$ (section \ref{Induction}).

When $f:P\to M$ is an element of $M^{P}$,
 the  values of $f $ on $s^{\mathbb N}$ determine the values  of $f$ on $N$ and reciprocally because,  for any $u\in N_{0}, k\in \mathbb N$,
\begin{equation} \label{2.6}
\begin{split}
     f(s^{-k}us^{k}) & = (\psi^{k} \circ u) ( f(s^{k})) \ , \\
  f(s^{k}) & = \sum_{v\in  J(N_0/N_k)}( v \circ  \varphi^{k}) (f(s^{-k}v ^{-1}s^{k})) \   .
\end{split}
\end{equation}
The first equality is obvious from the definition of  $\Ind_{P_{-}}^{P}$ ,  the second equality  is obvious  by the first equality as the idempotents $(v   \circ \varphi^{k} \circ \psi^{k} \circ  v ^{-1})_{v \in J(N_{0}/ N_{k} ) } $
 are orthogonal of sum the identity, by the lemma \ref{idemp}.

\begin{proposition} \label{3.3}   a) The  restriction to $s^{\mathbb Z}$ is an  $A[s^{\mathbb Z}]$-equivariant isomorphism
$$M^{P}\quad   \to \quad  \Ind^{s^{\mathbb Z}}_{s^{-\mathbb N}} (M)  \quad .  $$
b) The restriction to $N$ is an $N$-equivariant bijection from $ M^{P}$ to
  $\Ind_{N_{0}}^{N} (M)$.
\end{proposition}

\begin{proof} a) As $P=P_{-} s^{\mathbb Z} $ and $s^{-\mathbb N}\subset  P_{-} \cap s^{\mathbb Z}$ (it is an equality if $N$ is not trivial),
the restriction to $s^{\mathbb Z}$ is a $s^{\mathbb Z}$-equivariant injective map
 $ M^{P} \to  \Ind^{s^{\mathbb Z}}_{s^{-\mathbb N}}(M)$.
  To  show that the map is surjective, let  $\phi \in  \Ind^{s^{\mathbb Z}}_{s^{-\mathbb N}}(M)$ and
 $b\in P$. Then, for   $b=b_{-}s^{r}$ with $b_{-}  \in P_{-},   r\in \bf Z$,
 $$ f(b) \quad  := \quad  b_{-} \phi(s^{r}) $$
is well defined because the right side depends only on $b$, and not   on the choice of $(b_{-},r)$.  Indeed for two choices  $b=b_{-}s^{r}=b'_{-}s^{r'}$ with $b_{-}, b'_{-}  \in P_{-},   r \geq r' $ in $\bf Z$,  we have
    $$b_{-} \phi(s^{r})=b'_{-}s^{r'-r} \phi(s^{r})=b'_{-}\phi(s^{r'}) \quad . $$
   The  well defined function $b\mapsto
  f(b) $  on $P$ belongs obviously to  $M^{P}$ and its restriction   to $s^{\mathbb Z}$ is equal to $\phi$.

b)  As $P_{-}\cap N=N_{0}$ the restriction to $N$ is an $N$-equivariant  map $ M^{P}\to \Ind^{N}_{N_{0}}(M) \ .$  The map is injective because the restriction to $N$ of $f\in M^{P} $ determines the restriction of $f$ to $s^{\mathbb N}$ by  (\ref{2.6}) which determines $f$ by a).   We have the natural injective map
 \begin{equation}\label{e1}
  f\  \mapsto \ \phi_f \quad :  \quad\Ind^{s^{\mathbb Z}}_{s^{-\mathbb N}}(M) \to M^{P}\to \Ind_{N_{0}}^{N} (M)
  \end{equation}
    $$ \phi _f(s^{-k}us^{k})=(\psi^{k} \circ u) ( f(s^{k})) \quad {\rm for }\quad k\in \mathbb N, u\in N_{0} \quad, $$
 and we have  the map
  $$ \phi\  \mapsto \  f_\phi \quad : \quad  \Ind^{N}_{N_{0}}(M)\quad \to \quad  \Ind^{s^{\mathbb Z}}_{s^{-\mathbb N}}(M)
  $$
  defined by
$$
 f_\phi(s^{k}) \ = \  \sum_{v\in  J(N_0/N_k)}( v \circ  \varphi^{k}) (\phi (s^{-k}v ^{-1}s^{k}))  \quad {\rm for } \quad k\in \mathbb N.
$$
 Indeed the function $f_\phi$ satisfies $\psi (f_\phi(s^{k+1})) = f_\phi(s^{k})$ : since  $\psi \circ u \circ  \varphi^{k+1}=  s^{-1}u s \circ \varphi^{k}$ when $u\in  N_{1}$ and is  $0$ otherwise, we have
 \begin{align*}
  \psi (f_\phi(s^{k+1})) &
 = \psi (\sum_{v\in  J(N_0/N_{k+1})}( v \circ  \varphi^{k+1}) (\phi (s^{-k-1}v ^{-1}s^{k+1}))   \\
 & = \
\sum_{v\in  N_{1}\cap J(N_0/N_{k+1})}(s^{-1} v s \circ  \varphi^{k}) (\phi (s^{-k-1}v ^{-1}s^{k+1})) \ .
\end{align*}
The last term is
$$
 \sum_{v\in   J(N_0/N_{k})}( v  \circ  \varphi^{k}) (\phi (s^{-k}v ^{-1}s^{k})) \ = \  f_\phi(s^{k})
$$ because   $s ^{-1}(N_{1}\cap J(N_{0}/N_{k+1})) s$ is a system of representatives of $N_{0}/N_{k}$ and each term of the sum does not depend on the representative. Indeed for $u\in N_0$,
\begin{multline*}
    (v s^kus^{-k} \circ  \varphi^{k}) (\phi (s^{-k}(v s^kus^{-k})^{-1}s^{k}) \\
    = (v  \circ  \varphi^{k} \circ u) (\phi (u^{-1}s^{-k}v^{-1}s^k))= ( v  \circ  \varphi^{k}) (\phi (s^{-k}v ^{-1}s^{k})) \ .
\end{multline*}
For $u\in N_{0} , k \in \mathbb N$, we have
\begin{multline*}
    \phi_{f_\phi } (s^{-k}us^{k}) \ = \ (\psi^k\circ u)f_\phi(s^k) \\
     = \  \sum_{v\in  J(N_0/N_k)}(\psi^{k} \circ u  v \circ  \varphi^{k}) (\phi (s^{-k}v ^{-1}s^{k}))\ = \ \phi (s^{-k}us^{k})
\end{multline*}
where the last equality comes from  $\Ker \psi^k= \sum _{u\in N_0-N_k}u\varphi^k(M)$ . Moreover, we have $f_{\phi_f}=f$ as a consequence of Lemma \ref{idemp}.
\end{proof}

\begin{proposition} The induction functor
 $$\Ind_{P_{-}}^{P}    \quad : \quad {\mathcal M}_{A}(P_{+})^{et} \to {\mathcal M}_{A}(P_{-})  \to  {\mathcal M}_{A}(P)$$ is exact.
\end{proposition}

\begin{proof} The canonical action of any element of $P_{-}$ on an \'etale  $A[P_{+}]$-module is surjective.  Apply Lemma \ref{2.2}.
\end{proof}

\begin{proposition} \label{action}
Let $f\in M^{P}$ .  Let $n,n'\in N$ and $t\in L _{+}$ and denote by $k(n)$ the smallest integer $k\in \mathbb N$ such that $n \in N_{-k}$. We have :
\begin{gather*}
( nf )(s^{m})=  (s^{m}n s^{-m} )(f(s^{m}) )  \quad \text{for all $m\geq k(n)$}, \\
  (t^{-1}f) (s^{m})=\psi_{t}(f(s^{m})) \quad\text{and} \quad  (sf)(s^{m}) = f(s^{m+1} )  \quad \text{for all $m \in \mathbb Z$}, \\
  (s^{k}f)(n')= \sum_{v \in  J(N_0/N_k)}v \varphi ^{k}(f( s^{-k}  v^{-1}n' s^{k} )) \quad \text{for all $k \geq 1$}, \\
   (t^{-1}f)(n')= \psi_{t} (f(tn't^{-1})) \quad\text{and}\quad  (nf)(n')=f(n'n) \  .
\end{gather*}
\end{proposition}

\begin{proof}  The formulas $(sf)(s^{m}) = f(s^{m+1} ), (nf)(n')=f(n'n)$ are obvious. It is clear that
$$(t^{-1}f) (s^{m})= f(s^{m}t^{-1})=f(t^{-1}s^{m})=  t^{-1}(f( s^{m}))=\psi_{t} (f( s^{m})) \ ,$$
$$(t^{-1}f) (n')= f(nt^{-1})=f(t^{-1}tn't^{-1})=t^{-1}(f(tn't^{-1}))=\psi_{t} (f(tn't^{-1})) \ .$$
$$nf(s^{m} )= f(s^{m}n)= f(s^{m}ns^{-m}s^{m})=(s^{m}ns^{-m}) f(s^{m}) \ . $$
 Using Lemma \ref{idemp}, we write
$$(s^{k}f)(n')= \sum_{v \in  J(N_0/N_k)}v \varphi ^{k}( \psi ^{k}( v^{-1} ((s^{k}f)(n'))))\  , $$
$$ \psi ^{k}( v^{-1}( (s^{k}f)(n')))=  \psi ^{k}( v^{-1} (f(n's^{k})))= \psi^{k} (f(v^{-1}n's^{k}))=f(s^{-k}v^{-1}n's^{k}) \ .
$$
We obtain
$(s^{k}f)(n')= \sum_{v \in  J(N_0/N_k)}v \varphi (f( s^{-k}  v^{-1}n 's^{k} ))\ . $
 \end{proof}

\begin{definition} The $s$-model  and the $N$-model of $M^{P}$ are the spaces
 $\Ind^{s^{\mathbb Z}}_{s^{-\mathbb N}}(M) \simeq \underset{\psi} \varprojlim \ M $   and    $\Ind^{N}_{N_{0}} (M) $, respectively,  with the action of $P$ described in the proposition \ref{action}.
\end{definition}

\subsection{Compactly induced representation $M_{c}^{P}$}\label{27}

The map
$$
  \ev_{0} \ : \ M^{P}\  \to \ M  \quad, \quad f \ \mapsto f(1)  \quad,
$$
admits a splitting
$$
\sigma_{0}:M\to M^{P}
$$
For $m\in M$,  $\sigma_{0}(m)$ vanishes on $N-N_{0}$ and is equal to $nm$ on $n \in N_0$ and to $\varphi^{k}(m)$ on $s^{k}$ for $k\in \mathbb N$. In particular, by proposition \ref{3.3}.b, $\sigma_0$ is independent of the choice of $s$.

\begin{lemma}\label{equ}
The map $\ev_{0} $ is $P_{-}$-equivariant, the map  $ \sigma_{0} $  is $P_{+}$-equivariant, the $A[P_{+}]$-modules $\sigma_{0}(M)$ and $M$ are isomorphic.
\end{lemma}

\begin{proof}  It is clear on the definition of $M^{P}$ that $\ev_{0}$ is  $P_{-}$-equivariant.
We show that $\sigma_{0}$ is  $L_{+}$-equivariant using the $s$-model.
 Let  $ t\in L_{+}$.
 We choose  $t' \in L_{+}, r\in \mathbb N$ with $t't=s^{r}$.  Then $\varphi_{t'}\varphi_{t}=\varphi^{r}$ and
 $\varphi_{t}= \psi_{t'}\varphi^{r}.$ We  obtain for  $t \sigma_{0}(m) (s^{k}) = \sigma_{0}(m) (s^{k}t)$ the following expression
\begin{multline*}
    \sigma_{0}(m)(t'^{-1}s^{k+r})
    =\psi_{t'}(\sigma_{0}(m)(s^{k+r})) \\
    =\psi_{t'}\varphi ^{r+k}(m)=\varphi_{t}\varphi^{k}(m)
    =\varphi^{k}\varphi_{t}(m)=
    \sigma_{0}(tm)(s^{k}) \ .
\end{multline*}
Hence $t \sigma_{0}(m)= \sigma_{0}(tm)$. We show that $\sigma_{0}$ is $N_{0}$-equivariant using the $N$-model. Let $n_{0}\in N_{0}$ and $m\in M$.  Then  $n_{0}\sigma_{0}(m)=\sigma_{0}(n_{0}m)$, because for  $k\in \mathbb N ,\ u\in N_{0}$,
\begin{multline*}
    n_{0}\sigma_{0}(m) (s^{-k}us^{k}) = \sigma_{0}(m) (s^{-k}us^{k}n_{0}) =  \sigma_{0}(m) (s^{-k}us^{k}n_{0} s^{-k}s^{k}) \\
    = (\psi^{k} \circ u s^{k}n_{0} s^{-k} \circ \varphi ^{k}) (m)=  (\psi^{k} \circ u   \circ \varphi ^{k}) (n_{0}m)=\sigma_{0}(n_{0}m) (s^{-k}us^{k}) \ .
\end{multline*}
\end{proof}

The  compact induction of $M$ from $P_{-}$ to $P$  is defined to be the $A[P]$-submodule $$ \ind_{P_{-}}^{P}(M):=M^{P}_{c} $$  of $M^{P}$ generated by $\sigma_{0}(M)$.  The space  $M^{P}_{c} $  is   the subspace of functions  $f\in M^{P}$ with compact restriction to $N$, equivalently such that  $f(s^{k+r})=\varphi^{k}(f(s^{r}))$ for all $k\in \mathbb N$ and some $r\in \mathbb N$. The restriction to $s^{\mathbb Z}$ is an $s^{\mathbb Z}$-isomorphism  (proposition \ref{3.3})
$$
M^{P}_{c} \quad \to \quad \ind_{s^{-\mathbb N},\psi}^{s^{\mathbb Z}}(M)\quad  .
$$
By proposition \ref{pexx}, the map
\begin{align*}
  A[P]\otimes _{A[P_{+}]}M & \ \to \  \ind_{P_{-}}^{P}(M)\\
 [s^{-k}] \otimes m & \ \mapsto \ (\varphi^{-k} \circ \sigma_0)(m)
\end{align*}
is an isomorphism.

\begin{lemma} \label{ci1} The  compact induction functor   from $P_{-}$ to $P$ is isomorphic to  \begin{equation}   \ind_{P_{-}}^{P} \simeq A[P] \otimes_{A[P_{+}]} \ : \  {\mathcal M}_{A}(P_{+})^{et} \ \to \ {\mathcal M}_{A}(P )  \ ,  \end{equation}
and is exact.  \end{lemma}

\begin{proof} For the exactness  see Corollary \ref{exx}.
\end{proof}

   \subsection{$P$-equivariant  map
$\Res  :  C_{c}^{\infty}(N,A)  \to \End_{A}( M^{P})$}

 Let $C_{c}^{\infty}(N,A)$ be the $A$-module  of locally constant compactly
 supported functions on $N$ with values in $A$, with the  usual product of
 functions and with the natural action of $P$, $$P\times C_{c}^{\infty}(N,A)
 \ \to \ C_{c}^{\infty}(N,A) \quad, \quad (b, f) \mapsto  (bf)(x)=f(b^{-1}.x)
 \quad  . $$
 For any open compact subgroup $U\subset N$,  the subring  $C^{\infty}(U,A) \subset C_{c}^{\infty}(N,A)$  of  functions $f $  supported in $U$, has a unit
 equal to   the characteristic function  $1_{U}$ of $U$,  and  is  stable by the $P$-stabilizer $P_{U}$ of $U$.    We have $b1_{U}=1_{b.U}$. The $A[P_{U}]$-module $C^{\infty}(U,A)$  and
the $A[P]$-module $C_{c}^{\infty}(N,A)$ are cyclic generated by $1_{U}$.
The  monoid $P_{+}=N_{0}L_{+}$  acts on $ \End_{A}(M)$ by
$$P_{+} \times  \End_{A}(M) \ \to \  \End_{A}(M)$$
$$(b, F) \mapsto  \varphi_{b} \circ  F \circ \psi_{b} \quad. $$

\begin{proposition} \label{in} There exists a unique  $P_{+} $-equivariant $A$-linear  map
 $$\res \ : \ C^{\infty}(N_{0},A) \quad \to \quad  \End_{A}(M)$$
 respecting the unit.  It is an homomorphism of $A$-algebras.
 \end{proposition}

\begin{proof} If the map $\res$ exists, it is unique because the $A[P_{+}]$-module $C^{\infty}(N_{0},A)$ is   generated by the unit $1_{N_{0}}$.
 The existence of $\res$ is equivalent to the lemma \ref{idemp}.
 For $b \in P_{+}$ we have the idempotent
\begin{equation}\label{inn}\res (1_{b.N_{0}}) :=   \varphi_{b} \circ \psi_{b} \ \in \  \End_{A}(M) \quad.
\end{equation}
We claim that for any   finite disjoint sum $b.N_{0} =\sqcup_{i\in I }b_{i}.N_{0}$ with $b_{i}   \in P_{+},$   the idempotents  $\res (1_{b_{i}.N_{0}})$ are orthogonal of sum
\begin{equation}\label{a} \res (1_{b.N_{0}}) = \sum_{i}\ \res (1_{b_{i}.N_{0}}) \quad.  \end{equation}
We prove the claim
  by reducing  to the case     $b=1$ and $b_{i}=u_{i} t$ with $u_{i}\in N_{0}, t \in L_{+}$ where the claim follows from the lemma  \ref{idemp}.
To  do this, we write   (\ref{a}) as
$$
u \circ \varphi_{t} \circ \psi_{t} \circ u^{-1} =  \sum_{i\in I}\ u_{i} \circ \varphi_{t_{i}} \circ \psi_{t_{i}} \circ u_{i}^{-1} \quad {\rm for} \quad b=ut, b_{i}=u_{i}t_{i}, \ u,u_{i}\in N_{0}\ , \ t,t_{i}\in L_{+} \ .
$$
Multiplying on the left by $u^{-1}$ and on the right by $u$ we reduce to the case  $u=1$.
 Then we choose  $t'\in L_{+}$ such that $t'\in t_{i}L_{+}$ for all $i\in I$. We reduce to the case    $t_{i}=t'$  constant for $i\in I$, because    $u_{i}t_{i}.N_{0} = \sqcup_{j\in I_{i}} u_{i,j}t'.N_{0}$ is a finite disjoint union with $u_{i,j} \in  N_{0} $,  the equality   (\ref{a}) will be satisfied when both  $\res(1_{t.N_{0}}) \ = \ \sum_{i\in I} \sum_{j\in I_{i}}  \ \res(1_{u_{i,j}t'.N_{0}})    $ and   $\res(1_{u_{i}t_{i}.N_{0}}) \ = \ \sum_{j\in I_{i}}  \ \res(1_{u_{i,j}t'.N_{0}} )  $,  the orthogonality of the idempotents
 $\res(1_{u_{i}t_{i}.N_{0}})$ will be satisfied when   the idempotents $\res(1_{u_{i,j}t'.N_{0}} )$ are orthogonal.
 We are reduced to $b=t , b_{i}=u_{i}t'$ for $i\in I$.
The inclusion   $u_{i} t'N_{0}t'^{-1} \subset  tN_{0}t^{-1}$ implies $t^{-1}t' \in L_{+}$. We write $t'=t \tau$ with $\tau \in L_{+}$. We have $\varphi_{t'}=\varphi_{t  }\circ \varphi_{ \tau}$ and
$\psi_{t'} =\psi_{ \tau }\circ  \psi_{t } $ by Proposition \ref{3.2}.
We have $ tN_{0}t^{-1}=\sqcup_{i\in I} u_{i}t\tau N_{0}\tau^{-1} t^{-1}$ with the $u_{i}$  form a representative system $J(tN_{0}t^{-1}/ t\tau N_{0}\tau^{-1} t^{-1})$   of $tN_{0}t^{-1}/ t\tau N_{0}\tau^{-1} t^{-1}$.  Writing $u_{i}=t v_{i} t^{-1}$
 we write (\ref{a}) under the form
$$  \varphi_{t} \circ \psi_{t}   \quad =  \quad \sum_{v \in J(N_{0}/\tau N_{0}\tau^{-1})}\ t v  t^{-1} \circ \varphi_{t  }\circ \varphi_{ \tau} \circ \psi_{ \tau } \circ \psi_{t }  \circ t v ^{-1}t^{-1}  \quad.
$$
Using (\ref{produit1}) and the lemma \ref{produit} this identity is equivalent to
 $$  \varphi_{t} \circ \psi_{t}   \quad =  \quad \sum_{v\in J(N_{0}/\tau N_{0}\tau^{-1})}\  \varphi_{t  } \circ  v  \circ \varphi_{ \tau} \circ \psi_{ \tau }  \circ  v^{-1}  \circ \psi_{t }
$$
which follows from Lemma \ref{idemp}.
As $\psi_{t}\circ \psi_{t}=\id$, the orthogonality of the  idempotents $  v  \circ \varphi_{ \tau} \circ \psi_{ \tau }  \circ  v ^{-1} $ for
$v \in J(N_{0}/\tau N_{0}\tau^{-1})$  implies   the orthogonality of the  idempotents
$\varphi_{t  } \circ  v \circ \varphi_{ \tau} \circ \psi_{ \tau }  \circ  v^{-1}  \circ \psi_{t }$.

The claim being proved, we get  an $A$-linear map $\res:C^{\infty}(N_{0},A)  \to   \End_{A}(M)$ which  is clearly $P_{+}$-equivariant and respects the unit.
 It respects the product because, for  $f_{1}, f_{2}\in C ^{\infty}(N_{0},A)$,
there exists $t\in L_{+}$ such that $f_{1}$ and $f_{2}$ are constant on each coset  $u tN_{0}t^{-1} \subset N_{0}$.  Hence $\res (f_{1}f_{2})=\sum_{v\in J(N_{0}/tN_{0}t^{-1})}f_{1}(v) f_{2}(v) \res(1_{v t. N_{0}})= \res(f_{1})\circ \res (f_{2})$.
  \end{proof}

The group $P=NL$  acts   on   $\End_{A}( M^{P})$  by conjugation. We have the  canonical injective algebra map
\begin{equation}\label{defs} F\ \mapsto\  \sigma_{0}\circ F\circ \ev_{0} \quad: \quad \End_{A}M \ \to \  \End_{A}(M^{P}) \ .
\end{equation}
It is $P_+$-equivariant since,
by the lemma \ref{equ}   for $b \in P_{+}$,   we have
\begin{equation}\
 b\circ \sigma_{0} \circ F \circ  \ev_{0} \circ b^{-1} \ = \  \sigma_{0} \circ   \varphi_{b} \circ F  \circ \psi_{b}  \circ \ev_{0}  \ .
\end{equation}
We consider the composite  $P_{+}$-equivariant algebra homomorphism
$$ C^{\infty}(N_{0},A)\ \xrightarrow{\ \res \ } \ \End_{A}( M)  \ \longrightarrow \ \End_{A}( M^{P}) \ . $$
 sending $1_{N_{0}} $ to   $R_0:=\sigma_{0} \circ \ev_{0}$ and, more generally, $1_{b.N_0}$ to  $b \circ R_0 \circ b^{-1}$ for $b\in P_+$.

 For $f\in M^{P}$, $R_0(f) \in M^{P}$  vanishes on $N-N_{0}$ and $R_0(f)(s^{k}) \ = \ \varphi^{k}(f(1))$. In the $N$-model,    $R_0$ is the restriction to $N_{0}$.

 We show now that the composite morphism extends to $ C_{c}^{\infty}(N,A)$.

\begin{proposition}  \label{Res} There exists a unique $P$-equivariant $A$-linear  map
$$\Res \ : \ C_{c}^{\infty}(N,A) \ \to \ \End_{A}( M^{P})$$
 such that $ \Res (1_{N_{0}} ) \ = \ R_0   $.
 The map $\Res$ is an algebra homomorphism.
  \end{proposition}

  \begin{proof}  If the map $\Res$ exists, it is unique because the $A[P]$-module $C^{\infty}(N,A)$ is   generated by $1_{N_{0}}$.

For  $b\in P$  we define
$$
\Res (1_{b.N_0})\  := \ b \circ R_0 \circ b^{-1}  \  .
 $$
We prove that
$b \circ R_0 \circ b^{-1}$ depends only on the subset $b.N_{0} \subset N$,
and that  for any finite disjoint decomposition of $b.N_{0} = \sqcup_{i\in I}   b_{i}.N_{0}$ with  $b_{i}\in P$,
the idempotents  $b_{i} \circ R_0 \circ b_{i}^{-1}$ are orthogonal of sum  $b \circ R_0 \circ b^{-1}$.

The equivalence relation $b.N_{0}=b'.N_{0}$ for $b,b'\in P$ is equivalent to $b'P_{0}=bP_{0}$ because the normalizer or of $N_{0}$ in $P$ is $P_{0}$.  We have $b \circ R_0 \circ b^{-1} =R_0$
when $b\in P_{0}$ because $\res(1_{b.N_{0}})=\res(1_{N_{0}})= \id$ (proposition \ref{in}). Hence $b \circ R_0 \circ b^{-1}$ depends only on $b.N_{0}$.
By conjugation by $b^{-1}$, we reduce to prove that the idempotents $b_{i} \circ R_0 \circ b_{i}^{-1}$ are orthogonal of sum $R_0 $
   for  any disjoint decomposition of $N_{0} = \sqcup_{i\in I}   b_{i}.N_{0}$  and $b_{i}\in P$.
    The $b_{i}$   belong to $P_{+}$, and the proposition \ref{in} implies the equality.

To prove that the $A$-linear map $\Res$ respects the product it suffices to check that, for any $t\in L_{+}, k\in \mathbb N$,   the endomorphisms  $\Res(1_{v tN_{0}t^{-1}})\in  \End_{A}( M^{P})$ are orthogonal idempotents, for $v \in J(N_{-k}/ tN_{0}t^{-1})$. We already proved this for $k=0$ and for all $t\in L_{+}$, and $s^{k} J(N_{-k}/ tN_{0}t^{-1}) s^{-k} = J(N_{0}/s^{k}t N_{0}t^{-1}s^{-k})$.
 Hence  we know that
 $$(s^{k}\circ \Res(1_{v tN_{0}t^{-1}}) \circ s^{-k})_{v \in J(N_{-k}/ tN_{0}t^{-1})}$$
  are orthogonal idempotents.
This implies  that $(\Res(1_{v tN_{0}t^{-1}}))_{v \in J(N_{-k}/ tN_{0}t^{-1})}$ are orthogonal idempotents.
\end{proof}

\begin{remark} \begin{itemize}
\item[(i)] The map $\Res$ is the restriction of an algebra homomorphism
 $$ C^{\infty}(N,A) \ \to \ \End_{A}( M^{P}) \ , $$
 where $C^{\infty}(N,A)$ is the algebra of all locally constant functions on $N$. For this we observe
 \begin{enumerate}
\item   The $A[P_{+}]$-module $ C^{\infty}(N_{0},A)$ is \'etale.  For $t\in L_{+}$, the corresponding  $\psi_{t}$
satisfies $(\psi_{t}f)(x)=f(txt^{-1})$.

\item The map $(f,m)\mapsto \res(f)(m) \ : \  C^{\infty}(N_{0},A)\times M \to M$ is $\psi_{t}$-equivariant, hence induces to a pairing  $ C^{\infty}(N_{0},A)^{P}\times M^{P}\to M^{P}$.

\item The $A[P]$-module $ C^{\infty}(N_{0},A)^{P}$ is canonically isomorphic to $C^{\infty}(N,A)$.
 \end{enumerate}

\item[(ii)] The monoid $P_{+}\times P_{+}$ acts on $\End_{A}(M)$ by   $\varphi_{(b_{1},b_{2})}F := \varphi_{b_{1}}\circ F \circ \psi_{b_{2}}$. For this action $\End_{A}(M)$ is an \'etale $A[P_{+}\times P_{+}]$-module,  and we have $\psi_{(b_{1},b_{2})}F =\psi_{b_{1}}\circ F \circ \varphi_{b_{2}}$.

\end{itemize}
\end{remark}

\begin{definition}   \label{not} For any compact open subsets $V\subset U\subset N_{0}$ and $m \in M$, we denote
 $$\res_{U}:=\res (1_{U})\ , \ M_{U}:= \res_{U}(M)\ , \  m_{U}:=\res_{U}(m) \  , \ \res^{U}_{V}:= \res_{V}|_{M_{U}}:M_{U}\to M_{V}  \ . $$
For any compact open subsets $V\subset U\subset N$ and $ f\in M^{P} $
$$ \Res_{U}:=\Res (1_{U})\ , \ M_{U}:= \Res_{U}(M^{P})\ , \  f_{U}:=\Res_{U}(f) \  , \ \Res^{U}_{V}:= \Res_{V}|_{M_{U}}:M_{U}\to M_{V}  .
$$
\end{definition}

\begin{remark} \label{rr}  {\rm The notations are coherent   for $U\subset N_{0}$, as follows from the following  properties.  For $b\in P_{+}$ we have
\begin{itemize}
\item[--]  $\res_{b.U}= \varphi_{b} \circ \res_{U} \circ \psi_{b}$ (proposition \ref{in}) \ ;
\item[--]
 $ b \circ \Res_{U}\   = \  \sigma_{0} \circ \varphi_{b} \circ \res_{U}\circ \ev_{0} $ and   $ \Res_{U} \circ b^{-1}\   = \  \sigma_{0}  \circ \res_{U}\circ   \psi_{b} \circ \ev_{0}$  \  ;
 \item[--]
 $(\Res_{U} f)(1)\   =\ \res_{U}(f(1))  $ .
 \end{itemize}}
 \end{remark}
 We note also that the proposition \ref{Res} implies:

\begin{corollary} \label{Rr}  For any  compact open subset $U\subset N$ equal to a finite disjoint union  $U=\sqcup_{i\in I}U_{i}$ of compact open subsets $U_{i}\subset N$, the idempotents $\Res_{U_{i}} $ are orthogonal   of sum $\Res_{U} $.
\end{corollary}

  \begin{corollary}\label{bN} For $u\in N$, the projector $\Res_{uN_{0}}$ is the restriction to $N_{0}u^{-1}$ in the $N$-model.
  \end{corollary}
  \begin{proof} We have $\Res_{uN_{0}}=u\circ \Res_{N_{0}}\circ u^{-1}$ and  $\Res_{N_{0}}$ is the restriction to $N_{0}$ in the $N$-model. Hence for $x\in N$, $(\Res_{uN_{0}}f)(x)=(\Res_{N_{0}}u^{-1}f)(xu)$ vanishes   for  $x\in N-N_{0}u^{-1}$ and for $ v\in N_{0}$, $(\Res_{uN_{0}}f)(vu^{-1})=(u^{-1}f)(v)=f(vu^{-1})$.
  \end{proof}

The constructions are functorial. A morphism $f:M\to M'$ of $A[P_{+}]$-module, being also $A[P_{-}]$-equivariant induces a morphism $\Ind_{P_{-}}^{P}(f): M^{P}\to M'^{P}$ of $A[P]$-modules.  On the other hand, $M^P$ is a module over the non unital ring $ C^{\infty}_{c}(N,A) $ through the map $\Res$. The morphism $\Ind_{P_{-}}^{P}(f)$ is $ C^{\infty}_{c}(N,A) $-equivariant.  Since $\Res$ is $P$-equivariant , it suffices to prove that $\Ind_{P_{-}}^{P}(f)$ respects $R_0 = \sigma_0\circ ev_0$ which is obvious.

    \subsection{$P$-equivariant sheaf on N}

We formulate now
the   proposition \ref{Res}  in the language of sheaves.

\begin{theorem} \label{the11} One can associate to an \'etale $ A[P_{+}]$-module $M$,  a $P$-equivariant sheaf $\mathcal S_M$ of $A$-modules on the compact open subsets   $U\subset N$, with

 - \  sections $M_{U} $ on $U$,

- \   restrictions
 $\Res^{U}_{V}  $
for any   open compact  subset $V\subset U $,

- \   action $f \mapsto bf =\Res_{b.U}(bf) \ : \ M_{U} \  \to  \ M_{b.U}$ of $b\in P$.
 \end{theorem}
\begin{proof}
 a) $\Res_{U}^{U}$ is the identity on $M_{U}=\Res_{U}(M) $  because $\Res_{U}$ is an idempotent.

 b) $\Res_{W}^{V}\circ \Res_{V} ^{U}= \Res_{W} ^{U}$  for compact open subsets $W\subset V \subset U\subset N$. Write  $V$ as the disjoint union of $W$ and of a compact open subset $W'\subset V$, and use that $\Res_{W} $ and $ \Res_{W'}$ are orthogonal idempotents in $\End_A (M^P)$.

 c) If $U$ is the union of compact open subsets $U_{i}\subset U$ for $i\in I$, and  $f_{i}\in  M_{U_{i}} $  satisfying  $\Res_{U_{i}\cap U_{j}}^{U_{i}} (f_{i})= \Res_{U_{i}\cap U_{j}}^{U_{j}} (f_{j})$ for $i,j\in I$, there exists a unique $f\in
 M_{U} $ such that  $ \Res_{U_{i}} ^{U}(f)=f_{i}$ for all $i\in I$.

  c1) True when $(U_{i})_{i\in I} $ is a partition of $U$ because $I$ is finite and $\Res_{U}$ is the sum of the orthogonal idempotents   $\Res_{U_{i}}$.

c2) True when $I$ is finite because the finite covering defines a finite partition of $U$ by open compact subsets $V_{j}$ for $j\in J$, such that  $V_{j} \cap U_{i}$ is empty or equal to $V_{j}$ for all $i\in I,j\in J$ .
By hypothesis on the $f_{i}$, if   $V_{j} \subset U_{i}$, then the restriction of $f_{i}$ to $V_{j}$ does not depend on the choice of $i$, and is denoted by $\phi_{j}$. Applying c1), there is a unique $f\in  M_{U}$ such that $\Res_{V_{j}}(f)=\phi_{j}$ for all $j\in J$. Note also that  the $V_{j}$ contained in $U_{i}$ form a finite partition of $U_{i}$ and that $f_{i}$ is the unique element of $M_{U_{i}}$ such that $\Res_{V_{j}}(f_{i})=\phi_{j}$ for those $j$,
We deduce that $f$ is the unique element of  $ M_{U}$ such that $\Res_{U_{i}}(f)=f_{i}$ for all $i \in I$.

c3) In general, $U$ being compact, there exists a finite subset $I'\subset I$ such that $U$ is covered by $U_{i}$ for $i\in I'$. By c2), there exists a unique $f_{I'} \in M_{U} $ such that $f_{i}=\Res_{U_{i}}(f_{I'})$ for all $i\in I'$. Let $i'\in I$ not belonging to $I'$.  Then the non empty  intersections $U_{i'} \cap U_{j}$ for $j\in I'$ form a finite covering of $U_{i'}$ by compact open subsets. By c2), $f_{i'}$ is the unique element of  $M_{U_{i'}} $ such that $\Res_{U_{i'}\cap U_{j}}(f_{j})=\Res_{U_{i'} \cap U_{j}}(f_{i'})$ for all non empty  $U_{i'} \cap U_{j}$.
The element $\Res_{U_{i'}}(f) $ has the same property, we deduce by uniqueness that  $f_{i'}= \Res_{U_{i'}}(f)$.

d) Let $f\in M_{U}$. When $b=1$ we have clearly $1(f)= f $.
 For $b,b'\in P$, we have $(bb')(f )= \Res_{(bb').U}((bb') f) =  \Res_{b. ( b'.U)}(b(b' f)) = b (b'f)$.
 For a compact open subset $ V\subset U$, we have $b \circ \Res_{V}\circ \Res_{U} =\Res_{bV}\circ b \circ \Res_{U}$ in $\End_{A}M^{P}$ hence
$b \Res_{V}^{U}= \Res_{b.V} b $.

\end{proof}

\begin{proposition} \label{extension} Let $H$ be a topological group acting continuously on  a locally compact totally disconnected space $X$.
Any $H$-equivariant sheaf  $\mathcal F$ (of $A$-modules)   on  the  compact open subsets of $X$ extends uniquely  to a  $H$-equivariant sheaf on the  open subsets of $X$.

\end{proposition}

\begin{proof}    This is well known. See
\cite{BGR} \S 9.2.3 Prop. 1.

\end{proof}

\begin{remark} {\sl The space of sections on an open subset $U\subset X$ is the projective limit of the sections ${\mathcal F}(V)$ on the compact open subsets $V$ of $U$ for the restriction maps ${\mathcal F}(V)\to {\mathcal F}(V')$ for $V'\subset V$.}
\end{remark}

By this general result, the $P$-equivariant sheaf defined by $M$  on the compact open subsets of $ N$ (theorem \ref{the11}),  extends uniquely to a $P$-equivariant sheaf
${\cal S}_{M}$ on (arbitrary open subsets of) $N$.
We extend the definitions \ref{not} to arbitrary  open subsets $U\subset N$.  We denote by $\Res_{V}^{U}$ the restriction maps for open subsets $V\subset U$ of $N$, by $\Res_{U}=\Res^{N}_{U}$ and by $M_{U}= \Res_{U}(M^{P})$.  In this way we obtain an exact functor $M \to (M_U)_U$ from   ${\mathcal M}_A(P_+)^{et}$ to the category of $P$-equivariant sheaves of $A$-modules on $N$.
Note that for a compact open subset $U$ even the functor $M\to M_U$ is exact.

\begin{proposition}\label{sections} The representation of $P$ on the   global sections of the sheaf $\mathcal S_{M}$  is  canonically   isomorphic to $M^{P}$.
\end{proposition}

\begin{proof} We have the obvious $P$-equivariant homomorphism
$$
M^P\ \xrightarrow{(\Res_U)_U} \ M_N \ = \ \varprojlim _U M_U \ .
$$
The group $N$ is  the  union of $s^{-k}.N_{0} = s^{-k}N_{0}s^{k} $ for $k\in \mathbb N$. Hence
$M_{ N} = \varprojlim _k M_{N_{-k}}$.  In the $s$-model of $M^{P}$ we have
$\Res_{s^{-k}.N_{0} }=R_{-k}$ and by the lemma \ref{surj} the morphism
 $$f\mapsto  (\Res_{s^{-k}.N_{0} } (f))_{k\in \mathbb N} \ : \ M^{P} \quad \to \quad M_{N}$$ is   bijective.
   \end{proof}

 \begin{corollary}   The restriction $\Res^{N}_{U}:M_{N} \to M_{U}$ from the global sections  to the sections on an open   compact subset      $U\subset  N$  is surjective   with a natural splitting.
 \end{corollary}
\begin{proof} It corresponds to an idempotent $\Res_{U}=\Res(1_{U})\in \End_{A}(M^{P})$.
\end{proof}

\subsection{Independence of $N_{0}$}

Let  $U\subset N$ be  a compact open subgroup. For $n\in N$ and $t\in L$, the inclusion $ntUt^{-1}\subset U$ is obviously equivalent to $n\in U$ and $tUt^{-1} \subset U$.
Hence the $P$-stabilizer  $P_U:=\{ b\in P \ | \ b.U\subset U\}$ of $U$ is the semi-direct product of $U$ by  the $L$-stabilizer  where  $L_U$ of $U$.  As the decreasing sequence $(N_k=s^k N_0 s^{-k})_{k\in \mathbb N}$   form a basis of neighborhoods of $1$ in $N$ and $N=\cup_{r\in \mathbb Z}N_{-r}$, the compact open subgroup $U\subset N$ contains some $N_k$ and is contained in some $N_{-r}$. This implies that
the intersection $L_U \cap s^{\mathbb N}$ is not empty hence  is equal to $s _U^{\mathbb N}$ where $s_U=s^{k_U}$ for some $k_U \geq 1$.  The monoid $P_U=UL_U$  and the central element $s_U$ of $L$ satisfy the same conditions as $(P_+=N_0L_+, s)$, given at the beginning of the  section \ref{2.4}. Our theory associates to each  \'etale $A[P_{U}]$-module a $P$-equivariant sheaf on $N$.

The subspace  $M_{U} \subset M^{P}$  (definition \ref{not}) is stable by $P_{U} $ because $b\circ \Res_U = \Res_{b.U} \circ b$ for $b\in P$ and $M_{b.U}= \Res_{b.U}(M) \subset \Res_U(M)=M_U$.
As $M_U =\oplus_{u\in J(U/t.U)} u M_{t.U}$ for $t\in L_U$  the  $A[ P_{U}]$-module  $M_{U}$ is \'etale.

  \begin{proposition}\label{b}  The $P$-equivariant sheaf ${\mathcal S}_{M}$ on $N$ associated to the \'etale $A[P_{+}]$-module $M$  is equal to the  $P$-equivariant sheaf on $N$ associated to the \'etale $A[P_{U}]$-module $M_{U}$.
\end{proposition}
\begin{proof}
For $b\in P_U $ we denote by $\varphi_{U,b}$ the action of $b$ on $M_{U}$ and by $\psi_{U,b}$  the left inverse of $\varphi_{U,b}$ with kernel $M_{U-b.U}$. We have
  $M_U = M_{b.U}\oplus M_{U-b.U}$ and  for $f_U\in M_U$,
\begin{equation}
\varphi_{U,b}(f_U)\ = \ b f_U \ , \ \psi_{U,b}(f_U)\ = \ b^{-1} \Res_{b.U}(f_U)  \ , \ ( \varphi_{U,b} \circ  \psi_{U,b} )(f_U) \ = \ \Res_{b.U}(f_U) \ .
 \end{equation}
 By the last formula and the remark \ref{rr}, the sections   on $b.U$  and the restriction maps from $M_U$ to $M_{b.U}$ in the two sheaves  are the same for any $b\in P_U$. This implies that the two sheaves are equal on (the open subsets of) $U$. By symmetry they are also equal on (the open subsets of)  $N_0$. The same arguments for arbitrary compact open subgroups $U, U'\subset N$  imply that  the $P$-equivariant sheaves on $N$ associated to the \'etale $A[P_{U}]$-module $M_U$ and to  the \'etale $A[P_{U'}]$-module $M_{U'}$ are equal on (the open subsets of)  $U$ and on (the open subsets of)  $U'$. Hence all these sheaves are equal on (the open subsets of)  the compact open subsets of $N$ and also on (the open subsets of) $N$.
 \end{proof}

\subsection{Etale $A[P_{+}]$-module and $P$-equivariant sheaf on $N$}

 \begin{proposition}\label{red1} Let   $M$ be an $A[P_{+}]$-module such that the action $\varphi$ of $s$ on $M$ is \'etale. Then $M$ is an \'etale $A[P_{+}]$-module.
\end{proposition}
\begin{proof} Let $t\in L_{+}$. We have to show that  the action  $\varphi_{t}$ of $t$ on $M$ is \'etale. As $L=L_{+}s^{-\mathbb N}$ with $s$ is central in $L$, there exists $k\in \mathbb N$ such that $s^{k}t^{-1}\in L_{+}$. This implies   $\varphi^{k} = \varphi_{s^{k}t^{-1}}\circ \varphi_{t}$ in $\End_{A}(M)$ and    $s^{k}N_{0}s^{-k}\subset tN_{0}t^{-1}$.
 As  $\varphi$ is injective, $\varphi_{t}$ is also injective. For   any representative system $J( tN_{0}t^{-1}/ s^{k}N_{0}s^{-k})$ of $ tN_{0}t^{-1}/s^{k}N_{0}s^{-k}$ and any representative system $J(N_{0}/tN_{0}t^{-1})$ of $N_{0}/tN_{0}t^{-1}$,  the set of $uv$ for $u\in J(N_{0}/tN_{0}t^{-1})$ and $v\in J(tN_{0}t^{-1}/s^{k}N_{0}s^{-k})$ is a  representative system $J(N_{0}/ s^{k}N_{0}s^{-k})$ of $N_{0}/ s^{k}N_{0}s^{-k}$. Let $\psi$ be the canonical left inverse of $\varphi$. We have
$$\id = \sum_{u\in J(N_{0}/tN_{0}t^{-1})}u \circ \sum _{v\in J(tN_{0}t^{-1}/ s^{k}N_{0}s^{-k})} v \circ \varphi^{k} \circ \psi^{-k} \circ v^{-1} \circ u^{-1}   $$
 $$ = \sum_{u\in J(N_{0}/tN_{0}t^{-1})}u \circ  \sum _{v\in J(tN_{0}t^{-1}/ s^{k}N_{0}s^{-k})} v \circ  \varphi_{t} \circ \varphi_{ {t^{-1}s^{k}}} \circ \psi^{-k} \circ v^{-1} \circ u^{-1}   $$
 $$ = \sum_{u\in J(N_{0}/tN_{0}t^{-1})}u \circ  \varphi_{t} \circ  ( \sum _{v\in J(N_{0}/ t^{-1}s^{k}N_{0}s^{-k}t)} v   \circ \varphi_{ {t^{-1}s^{k}}} \circ \psi^{-k} \circ v^{-1}) \circ u^{-1} \ .  $$
 We deduce that  $\varphi_{t}$ is \'etale of canonical left inverse
 $\psi_{t}$ the expression between parentheses.
 \end{proof}

\begin{corollary}\label{redMF}  An  $A[P_{+}]$-submodule $M' \subset M$ of an \'etale $A[P_{+}]$-module $M$ is \'etale if and only if it is stable by the canonical inverse $\psi$ of $\varphi$.
\end{corollary}

\begin{proof} If $M'$ is $\psi$-stable, for $m'\in M'$ every $m'_{u,s}$ belongs to $M'$ in \eqref{writing}.  Hence the action of $s$ on $M'$ is \'etale, and   $M'$ is \'etale by Proposition  \ref{red1}.
\end{proof}

\begin{corollary} \label{red2} The space ${\mathcal S}(N_{0})$ of global sections of a $P_{+}$-equivariant sheaf ${\mathcal S}$ on $N_{0}$ is an \'etale representation of $P_{+}$, when the action  $\varphi$ of $s$ on  ${\mathcal S}(N_{0})$ is injective.
\end{corollary}

\begin{proof}   By proposition \ref{red1} it suffices to show that
${\cal S}(N_{0}) = \oplus_{u\in J(N_{0}/sN_{0}s^{-1})} us ({\cal S}(N_{0}))$.
But this equality is true because
 $N_{0}$ is the disjoint sum of the open subsets $usN_{0}s^{-1} =us. N_{0}$ and
  ${\cal S}(us.N_{0})=us({\cal S}(N_{0}))$.
\end{proof}

  The canonical  left inverse $\psi$ of the action $\varphi$ of $s$ on ${\cal S}(N_{0})$ vanishes on ${\cal S}(usN_{0}s^{-1})$ for $u\neq 1$ and on $ {\cal S}(sN_{0}s^{-1}) $ is equal to  the isomorphism  $ {\cal S}(sN_{0}s^{-1}) \to  {\cal S}(N_{0})$ induced by  $s^{-1}$.

\begin{theorem}\label{red3}
The functor $M\mapsto {\mathcal S} _{M}$ is an equivalence of categories from the abelian category  of \'etale $A[P_{+}]$-modules to the abelian category of $P$-equivariant sheaves of $A$-modules on $N$, of inverse the functor  ${\mathcal S} \mapsto  {\mathcal S}(N_{0})$ of sections over $N_{0}$.
\end{theorem}
\begin{proof}
Let $ {\cal S}$ be a $P$-equivariant sheaf on $N$.  By the corollary \ref{red2}, the space ${\cal S}(N_{0})$ of sections on $N_{0}$  is an \'etale representation of $P_{+}$ because the action $\varphi $ of $s$ on $ {\cal S}(N_{0})$ is injective.

We show now that  the  representation of $P$  on the space $ {\cal S}(N)_{c}$ of  compact sections on $N$ depends uniquely of the representation of $P_{+} $ on ${\cal S}(N_{0})$.
The representation of $N$ on $ {\cal S}(N)_{c}$ is defined by the representation of $N_{0} $ on   ${\cal S}(N_{0})$, because  $ {\cal S}(N)_{c}  = \oplus_{u\in J( N/N_{0})} {\cal S}(uN_{0})  $ and ${\cal S}(uN_{0})= u {\cal S}(N_{0})$ for $u\in N$. The group $P$ is generated by $N$ and $L_{+}$.
  For  $t\in L_{+}$, the action of $t$ on  $ {\cal S}(N)_{c}$ is defined by the action of $N$ on $ {\cal S}(N)_{c}$ and by the action of $t$ on  ${\cal S}(N_{0})$, because $t{\cal S}(uN_{0})= tut^{-1}t{\cal S}(N_{0})$ with $tut^{-1}\in N$
    for  $u\in  N$.

  We deduce that the $A[P]$-module  $ {\cal S}(N)_{c}$ is equal to the compact induced representation ${\cal S}(N_{0})^{P}_{c}$, and that the sheaves $ {\cal S}$ and ${\cal S} _{{\cal S} (N_{0})}$ are equal.

Conversely, let $M$ be an  \'etale $A[P_{+}]$-module. The  $A[P_{+}]$-module  ${\cal S} _{M}(N_{0})$  of sections on $N_{0}$ of the sheaf  ${\cal S} _{M}$ is equal to $M $   (Theorem \ref{the11}).

\end{proof}

  \section{Topology}

\subsection{Topologically \'etale $A[P_{+}]$-module} \label{S4}

{\sl We add to the hypothesis of the section \ref{2.4} that

a) $A$ is a  linearly topological commutative ring (the open ideals form a basis of  neighborhoods of $0$).

 b)  $M $ is  a  linearly topological $A$-module (the open $A$-submodules form a basis of neighborhoods of $0$), with a continuous action of $P_{+}$
\begin{align*}
    P_{+}\times M & \to M  \\
    (b,x) & \mapsto  \varphi_{b}(x) \ .
\end{align*}
We call such an $M$ a continuous $A[P_+]$-module.
If $M$ is also \'etale in the algebraic sense (definition \ref{almostetale}) and the maps $\psi_t$, for $t\in L_+$, are continuous we   call $M$  a  topologically  \'etale   $A[P_{+}]$-module.  }

\begin{lemma}\label{con1} Let $M$ be  a continuous $A[P_+]$-module which is algebraically \'etale, then:

(i) \ The maps $\psi_t$ for $t\in L_+ $ are open.

(ii) \ If $\psi=\psi_s$ is continuous then $M$ is  topologically  \'etale.
\end{lemma}
\begin{proof}  (i) \ The projection  of $M=M_0 \oplus M_1$ onto the algebraic direct summand $M_0$ (with  the submodule topology) is open. Indeed let $V\subset M$ be an open subset, then $M_0\cap (V+M_1)$ is open in $M_0$ and is equal to the projection of $V$. We apply this to $M=\varphi_t(M) \oplus  \Ker \psi_t$ and to the projection
$\varphi_t \circ \psi_t$. Then we note that $\psi_t (V) = \varphi _t^{-1 } ((\varphi_t \circ \psi_t)(V))$.

(ii) \ Given any $t\in L_+$ we find $t'\in L_+$ and $n\in \mathbb N$ such that $t't=s^n$. Hence $\psi_{t't} = \psi_t \circ \psi_{t'} = \psi ^n$ is continuous by assumption.
As $\psi_{t'}$ is surjective and open, for any open subset $V\subset M$ we have $\psi_t^{-1}(V)=\psi_{t'}((\psi_t \circ \psi_{t'})^{-1}(V))$  which is open.
 \end{proof}

\begin{lemma}\label{cont1}  (i) \ A compact algebraically \'etale $A[P_+]$-module is  topologically  \'etale.

(ii) \ Let  $M $ be a topologically  \'etale $A[P_{+}]$-module. The $P_-$-action $(b^{-1},m)\mapsto \psi_{b}(m) \ : \ P_{-} \times M \to M$ on $M$ is continuous.
\end{lemma}
\begin{proof} (i) \ The compactness of $M$ implies that
$$
M = \varphi_t(M) \oplus \sum_{u \in (N_0 - tN_0 t^{-1})} u \varphi_t (M)
$$
is a topological decomposition of $M$ as the direct sum of finitely many closed submodules. It suffices to check that the restriction of $\psi_t$ to each summand is continuous. On all summands except the first one $\psi_t$ is zero.  By compactness of $M$ the map $\varphi_t$ is a homeomorphism between $M$ and the closed submodule $\varphi_t(M)$. We see that $\psi_t|\varphi_t(M)$ is the inverse of this homeomorphism and hence is continuous.

(ii) \ Since $P_0$ is open in $P_- = L_+^{-1}P_0 $ we only need to show that the restriction of the $P_-$-action to $t^{-1}P_0 \times M \to M$, for any $t \in L_+$, is continuous. We contemplate the commutative diagram
$$
\xymatrix{
  t^{-1}P_0 \times M  \ar[r] \ar[d]_{t\cdot \times \id} & M  \\
  P_0 \times M \ar[r] & M \ar[u]_{\psi_t}   }
$$
where the horizontal arrows are given by the $P_-$-action. The $P_0$-action on $M$ induced by $P_-$ coincides with the one induced by the $P_+$-action. Therefore the bottom horizontal arrow is continuous. The left vertical arrow is trivially continuous, and   $\psi_t$ is continuous by assumption.\end{proof}

\begin{lemma}\label{No} For any compact subgroup  $C\subset P_{+} $, the open $C$-stable $A$-submodules  of $M$ form a basis of neighborhoods of $0$.
\end{lemma}
\begin{proof} We have to show that any open $A$-submodule $\mathcal M$ of $M  $ contains an open $C$-stable $A$-submodule.
By continuity of the action of $P_{+}$ on $M$, there exists for each $c \in C$,
an open $A$-submodule ${\mathcal M}_{c}$ of $M$ and an open neighborhood $H_{c}\subset P_{+}$ of $c$ such that
$\varphi_{x}({\mathcal M}_{c}) \subset {\mathcal M}$ for all $x\in H_{c}$. By the compactness of $C$, there exists a finite subset $I\subset C$ such that
$C= \cup_{c\in I}(H_{c}\cap C)$. By finiteness of $I$, the intersection  ${\mathcal M}'':=\cap_{c\in I}{\mathcal M}_{c} \subset M$ is an open $A$-submodule such that ${\mathcal M}':=\sum_{c\in C} \varphi_{c}({\mathcal M}'') \subset {\mathcal M}$.  The $A$-submodule ${\mathcal M}' $ is $C$-stable and, since ${\mathcal M}'' \subset {\mathcal M}' \subset {\mathcal M}$, also open.
 \end{proof}

Let   $M $ be a  topologically \'etale $A[P_{+}]$-module. Since $P_0$ is open in $P$ the $A$-module $M^{P}$ is a submodule of the $A$-module $C(P,M)$ of all continuous maps from $P$ to $M$. We equip $C(P,M)$ with the compact-open topology which makes it a linear-topological $A$-module. A basis of neighborhoods of zero is given by the submodules
${\mathcal C}(C, {\mathcal M}):=\{f \in C(P,M) \ | f(C)\subset {\mathcal M} \}$  with $C$ and $\mathcal M$ running over all compact subsets in $P$ and over all open submodules in $M$, respectively.
With $M$ also $C(P,M)$ is Hausdorff. Evidently $M^P$ is characterized inside $C(P,M)$ by closed conditions and hence is a closed submodule. Similarly, $\Ind_{s^{-\mathbb N}}^{s^{\mathbb Z}}(M)$ and $\Ind_{N_0}^N(M)$ are closed submodules of $C(s^{\mathbb Z},M)$ and $C(N,M)$, respectively, for the compact-open topologies. Clearly the homomorphisms of restricting maps (proposition \ref{3.3}) $M^P \to \Ind_{s^{-\mathbb N}}^{s^{\mathbb Z}}(M)$ and $M^P \to \Ind_{N_0}^N(M)$  are continuous.

\begin{lemma}\label{compact-open}
The restriction maps $M^P \to \Ind_{s^{-\mathbb N}}^{s^{\mathbb Z}}(M)$ and $M^P \to \Ind_{N_0}^N(M)$  are topological isomorphisms.
\end{lemma}
\begin{proof}
The topology on $M^P$ induced by the compact-open topology on the $s$-model $\Ind_{s^{-\mathbb N}}^{s^{\mathbb Z}} M$ is the  topology with basis of neighborhoods of zero
$$
B_{k,{\mathcal M}}\ := \ \{ f\in M^{P}\  |\  f(s^{m}) \in {\mathcal M} \  { \rm for \ all }\ -k\leq m\leq k  \} \ ,
$$
for all $k\in \mathbb N$ and all open $A$-submodules ${\mathcal M}$ of $M$. One can replace $B_{k,{\mathcal M}}$ by
$$
C_{k,{\mathcal M}}\ := \ \{ f\in M^{P} \ | \ f(s^{k}) \in {\mathcal M} \} \ ,
$$
because $B_{k,{\mathcal M}}\subset C_{k,{\mathcal M}}$ and conversely given $(k,{\mathcal M})$ there exists an open $A$-submodule ${\mathcal M}' \subset {\mathcal M}$ such that  $\psi^{m}({\mathcal M}') \subset {\mathcal M}$ for all $0\leq  m\leq 2k $ as $\psi$ is continuous (lemma \ref{cont1}), hence $C_{k,{\mathcal M}'}\subset B_{k,{\mathcal M}}$.

The topology on $M^P$ induced by the compact-open topology on the $N$-model $\Ind_{N_{0}}^{N}M$ is the topology with basis of neighborhoods of zero
$$
D_{k,{\mathcal M} }\ := \ \{ f \in M^{P}\ | \ f(N_{-k}) \subset  {\mathcal M} \} \ ,
$$
for all $(k,{\mathcal M})$ as above.

We fix an auxiliary compact open subgroup $P_0' \subset P_0$. It then suffices, by Lemma \ref{No}, to let ${\mathcal M}$ run, in the above families, over the open $A[P_0']$-submodules ${\mathcal M}$ of $M$.

Let $C \subset P$ be any compact subset and let ${\mathcal M}$ be an open $A[P_0']$-submodule of $M$. We choose $k \in \mathbb N$ large enough so that $Cs^{-k} \subset P_-$. Since $Cs^{-k}$ is compact and $P_0'$ is an open subgroup of $P$ we find finitely many $b_1, \ldots, b_m \in P_+$ such that $Cs^{-k} \subset b_1^{-1}P_0' \cup \ldots \cup b_m^{-1}P_0'$. The continuity of the maps $\psi_{b_i}$ implies the existence of an open $A[P_0']$-submodule ${\mathcal M}'$ of $M$ such that $\psi_{b_i}({\mathcal M}') \subset {\mathcal M}$ for any $1 \leq i \leq m$. We deduce that
$$
C_{k,{\mathcal M}'} \subset  {\mathcal C} ( \bigcup_i b_i^{-1} P_0' s^k, {\mathcal M}) \subset  {\mathcal C}(C,{\mathcal M}) \ .
$$
Furthermore, by the continuity of the action of $P_{+}$ on $M$, there exists an  open submodule ${\mathcal M}''$ such that $ \sum_{v \in J(N_0/N_k)}v \varphi^{k}({\mathcal M}'') \subset {\mathcal M}'$. The second part of the formula (\ref{2.6}) then implies that
$$
D_{k,{\mathcal M}''} \subset C_{k,{\mathcal M}} \ .
$$
\end{proof}

The maps $\ev_{0} :M^{P}\to M$ and $\sigma_{0}:M\to M^{P}$ are continuous (section \ref{27}). We denote by $\End_{A}^{cont}(M)\subset \End_{A}(M)$ and $E^{cont} \subset E := \End_{A}(M^{P})$ the subalgebra of continuous endomorphisms. We have the canonical injective algebra map (\ref{defs})
$$
f\mapsto \sigma_{0}\circ f \circ \ev_{0} \quad :\quad \End_{A}^{cont}(M) \ \to \ E^{cont}\ .
$$

\begin{proposition}\label{cont2}  Let $M$ be a  topologically  \'etale $A[P_{+}]$-module.

 (i) \  If $M$ is complete, resp. compact, the $A$-module  $M^{P} $  is complete, resp. compact.

 (ii) \ The natural map $P\times M^{P }\  \to  \ M^{P}$ is continuous.

(iii) \  $\Res (f) \in E^{cont}$ for each $f\in C_{c}^{\infty}(N,A) $  (proposition \ref{Res}).

\end{proposition}

\begin{proof} (i) \ If $M$ is complete, by \cite{BTG} TG X.9 Cor. 3 and TG X.25 Th. 2, the compact-open topology on $C(P,M)$ is complete because $P$ is locally compact.
Hence, $M^P$ as a closed submodule is complete as well.

 If $M$ is compact, the  $s$-model of $M^{P}$  is compact as a closed subset of the compact space $M^{\mathbb N} $. Hence by Lemma \ref{compact-open}, $M^P$ is compact.

(ii) It suffices to show that the right translation action of $P$ on $C(P,M)$ is continuous. This is well known: the map in question is the composite of the following three continuous maps
\begin{align*}
    P \times C(P,M) & \longrightarrow P \times C(P \times P,M) \\
    (b,f) & \longmapsto (b, (x,y) \mapsto f(yx)) \ ,
\end{align*}
\begin{align*}
    P \times C(P \times P,M) & \longrightarrow P \times C(P,C(P,M)) \\
    (b, F) & \longmapsto (b, x \mapsto [y \mapsto F(x,y)]) \ ,
\end{align*}
and
\begin{align*}
    P \times C(P,C(P,M)) & \longrightarrow C(P,M) \\
    (b, \Phi) & \longmapsto \Phi(b) \ ,
\end{align*}
where the continuity of the latter relies on the fact that $P$ is locally compact.

(iii) It suffices to consider functions of the form $f = 1_{b.N_0}$ for some $b \in P$. But then $\Res(f) = b \circ \sigma_0 \circ \ev_0 \circ b^{-1}$ is the composite of continuous endomorphisms.
\end{proof}

 \subsection{Integration on $N$ with value in $\End_{A}^{cont}(M^{P})$} \label{Inte}

{\sl We  suppose   that  $M$  is a complete topologically \'etale $A[P_{+}]$-module}.

\bigskip We denote by $E^{cont}  $   the   ring of continuous $A$-endomorphisms of the complete $A$-module $M^{P}$ with the topology defined by the
 right ideals
$$E^{cont}_{\mathcal L} \ :=  \ \Hom_{A}^{cont}(M^{P}, {\mathcal L} ) $$
for all  open $A$-submodules ${\mathcal L} \subset M^{P}$.

\begin{lemma}  $E^{cont}$ is a complete  topological ring.
\end{lemma}
\begin{proof} It is clear that  the maps $(x,y)\mapsto x-y$ and $(x,y)\mapsto x\circ y$  from $E^{cont}\times  E^{cont}$ to $E^{cont}$   are continuous, i.e. that $E^{cont}$ is a   topological ring. The composite of the  natural  morphisms
$$
E^{cont}\ \to\  \varprojlim_{{\mathcal L}} E^{cont }/E_{\mathcal L}^{cont} \ \to \  \varprojlim_{{\mathcal L}} \Hom_{A}^{cont}(M^{P},M^{P}/{\mathcal L})
$$
is an isomorphism (the natural map  $M^{P} \to  \varprojlim_{\mathcal L}M^{P}/{\mathcal L}$ is an isomorphism),  hence the two morphisms are isomorphisms since the kernel of the map $E^{cont} \to \Hom_{A}^{cont}(M^{P},M^{P}/{\mathcal L})$  is $E^{cont}_{\mathcal L}$.
We deduce that $E^{cont}$   is complete.
  \end{proof}

 \begin{definition} An $A$-linear map  $C_{c}^{\infty}(N,A)\to E^{cont}$ is called a measure on $N$ with values in $E^{cont}$.
 \end{definition}

The map $\Res$ is a measure on $N$ with values in  $E^{cont}$ (proposition \ref{cont2}).

\bigskip Let $C_{c}(N,E^{cont})$ be the space of compactly supported {\bf continuous} maps   from $N$ to $E^{cont}$.
  One can ``integrate'' a function in $ C_{c}(N,E^{cont})$ with respect to a measure  on $N$ with values in $E^{cont}$.

 \begin{proposition} \label{int} There is a natural bilinear map
 \begin{align*} C_{c}(N,E^{cont}) \times  \Hom_{A}(C_{c}^{\infty}(N,A), E^{cont}) & \  \to \  E^{cont} \\
  (f,\lambda) &\  \mapsto \ \int_{N}f \ d\lambda \ .
  \end{align*}
  \end{proposition}

\begin{proof} a) \ Every compact subset of $N$ is contained in a compact open subset. It follows that $C_{c}(N,E^{cont})$ is the union of its subspaces  $C(U,E^{cont})$ of functions with support contained in $U$,  for all  compact open subsets $U\subset N$.

b) \  For any open $A$-submodule $\mathcal L$ of $M^{P}$,
a function in $C (U,E^{cont}/E^{cont}_{\mathcal L})$ is locally constant  because
  $E^{cont}/E_{\mathcal L}^{cont} $ is discrete. An upper index $\infty$ means that we consider locally constant functions
 hence
 $$
  C (U,E^{cont}/E^{cont}_{\mathcal L})\  =\   C^{\infty} (U,E^{cont}/E^{cont}_{\mathcal L}) \ = \ C^{\infty} (U,A) \otimes _AE^{cont}/E^{cont}_{\mathcal L} \ .
  $$
There is a natural
 linear pairing
 \begin{align*}
 (C^{\infty} (U,A) \otimes _AE^{cont}/E^{cont}_{\mathcal L} ) \times  \Hom_{A}(C^{\infty}(U,A), E^{cont}) & \ \to \  E^{cont}/E^{cont}_{\mathcal L} \\
(f\otimes x , \lambda) & \ \mapsto \  x\lambda(f) \ .
 \end{align*}
Note that $E^{cont}/E^{cont}_{\mathcal L} $ is a right  $E^{cont}$-module.

 c) Let $f\in C_{c}(N,E^{cont})$ and let $\lambda \in \Hom_{A}(C_{c}^{\infty}(N,A), E^{cont})$. Let $U\subset N$ be an  open compact subset  containing the support of $f$.  For any open $A$-submodule $L$ of $M^{P}$ let $f_{{\mathcal L}} \in C^{\infty}_{c}(U,E^{cont}/E^{cont}_{\mathcal L})$ be the map induced by $f$.   Let
  $$\int_{U} f_{{\mathcal L}} \ d\lambda \quad \in  E^{cont}/E^{cont}_{\mathcal L}$$
 be the image of $(f_{{\mathcal L}}, \lambda)$ by the natural pairing of b).
  The elements $\int_{U} f_{{\mathcal L}} \ d\lambda$ combine in the projective limit $E^{cont}=\varprojlim_{{\mathcal L}} E^{cont}/E^{cont}_{\mathcal L}$ to give an element $\int_{U} f \ d\lambda  \in E^{cont}$.  One checks easily  that $\int_{U} f \ d\lambda  $ does not depend on the choice of $U$. We define
 $$ \int_{N}f \ d\lambda \quad := \quad \int_{U} f \ d\lambda  \quad. $$

\end{proof}

 We recall that  $J(N/V)$ is a system of representatives of $N/V$ when $V \subset N$  is a compact open subgroup.

\begin{corollary} \label{limit} Let $ f\in C_{c}(N,E^{cont})$  and let $\lambda $ be a measure on $N$ with values in $E^{cont}$. Then
$$\lim_{V\to \{1\}}\sum_{v \in  J(N/V)} f(v) \ \lambda (1_{v V})\quad = \quad \int_{N} f \ d\lambda \quad . $$
limit on compact open subgroups $V\subset N$ shrinking to $\{1\}$.
\end{corollary}

\begin{proof} We choose  an open compact subset $U\subset N$ containing the support of $f$.
Let    $L$ be an open $o$-submodule of $M^{P}$ and a compact open subgroup $V\subset N$ such  that $uV\subset U$ and    $f_{{\mathcal L}} $ (proof of the proposition \ref{int}) is constant on $uV$ for all $u\in U$.
Then  $\int_{U} f_{{\mathcal L}} \ d\lambda$ is the image of
  $$ \sum_{v \in  J(N/V)} f(v) \ \lambda (1_{v V})   \quad   $$
 by the quotient map $E^{cont} \to E^{cont}/E_{\mathcal L}^{cont}$.
  \end{proof}

\begin{lemma}\label{lzero}
Let $f \in C_c(N,E^{cont})$ be a continuous map with support in the compact open subset $U \subset N$, let $\lambda $ be a measure on $N$ with values in $E^{cont}$, and let $\mathcal{L} \subset E^{cont}$ be any open $A$-submodule. There is a compact open subgroup $V_{\mathcal{L}} \subset N$ such that $UV_{\mathcal{L}} = U$ and
\begin{equation*}
    \int_N (f1_{uV} - f(u)) d\lambda \in E^{cont}_{\mathcal{L}}
\end{equation*}
for any open subgroup $V \subset V_{\mathcal{L}}$ and any $u \in U$.
\end{lemma}
\begin{proof}
The integral in question is the limit (with respect to open subgroups $V' \subset V$) of the net
\begin{equation*}
    \sum_{v \in J(V/V')} (f(uv) - f(u)) \lambda(1_{uv V'}) \ .
\end{equation*}
Since $E^{cont}_{\mathcal{L}}$ is a right ideal it therefore suffices to find a compact open subgroup $V_{\mathcal{L}} \subset N$ such that $UV_{\mathcal{L}} = U$ and
\begin{equation*}
    f(uv) - f(u) \in E^{cont}_{\mathcal{L}} \qquad \textrm{for any $u \in U$ and $v \in V_{\mathcal{L}}$}.
\end{equation*}
We certainly find a compact open subgroup $\tilde{V} \subset N$ such that $U \tilde{V} = U$. The map
\begin{align*}
    U \times \tilde{V} & \ \to \ E^{cont} \\
    (u,v) & \ \mapsto \ f(uv) - f(u)
\end{align*}
is continuous and maps any $(u,1)$ to zero. Hence, for any $u \in U$, there is an open neighborhood $U_u \subset U$ of $u$ and a compact open subgroup $V_u \subset \tilde{V}$ such that $U_u \times V_u$ is mapped to $E^{cont}_{\mathcal{L}}$. Since $U$ is compact we have $U = U_{u_1} \cup \ldots \cup U_{u_s}$ for finitely many appropriate $u_i \in U$. The compact open subgroup $V_{\mathcal{L}} := V_{u_1} \cap \ldots \cap V_{u_s}$ then is such that $U \times V_{\mathcal{L}}$ is mapped to $E^{cont}_{\mathcal{L}}$.
\end{proof}

 Let $C(N,E^{cont})$ be the space of continuous functions from $N$ to $E^{cont}$. For any continuous function $f\in C(N,E^{cont})$, for any compact open subset $U\subset N$ and for any measure $\lambda$ on $N$ with values in $ E^{cont}$ we denote
  $$\int_{U}f \  d\lambda \ : = \int_{N}  f\ 1_{U} \ d\lambda \quad $$
where $1_{U} \in  C^{\infty} (U,A)$ is the characteristic function of $U$ hence $f1_{U} \in C_{c}(N,E^{cont})$
is the restriction of $f$ to $U$. The ``integral of $f$ on $U$''  (with respect to the measure $\lambda$) is equal to the ``integral of the restriction of $f$ to $U$''.

\begin{remark}\label{multiplic}
For $f \in C_c(N,E^{cont})$ and $\phi \in C_c^\infty(N,A)$ we have
\begin{equation*}
    \int_N f\phi d\Res \ = \ \int_N \phi f d\Res \ = \ \int_N f d\Res \circ \Res(\phi) \ .
\end{equation*}
\end{remark}
\begin{proof}
This is immediate from the construction of the integral and the multiplicativity of $\Res$.
\end{proof}

\section{$G$-equivariant sheaf   on $G/P$}\label{S5}

{\sl  Let $G$ be  a locally profinite group  containing $P=N\rtimes L$ as a closed subgroup  satisfying the assumptions of section \ref{2.4}  such that

a)  $G/P$ is compact.

b) There is a subset $W$ in the $G$-normalizer $N_G(L)$ of   $L$ such that

\begin{itemize}
\item[--] the image of $W$ in $N_G(L)/L$ is a subgroup,

\item[--] $G$ is the disjoint union of $PwP$  for $w\in W$.
\end{itemize}
 We note that $P wP=NwP =PwN$.

c) There exists $w_{0}\in W$  such that  $Nw_{0}P $   is an open  dense subset  of $G$. We call
\begin{equation*}
    \mathcal{C} := Nw_{0}P /P
\end{equation*}
the open cell of $G/P$.

d)   The map $(n,b)\mapsto nw_{0}b$  from $N\times P$  onto $Nw_{0}P $ is a homeomorphism. }

\begin{remark} \label{dense} These conditions imply that
$$G = P \overline P P = C(w_{0}) C(w_{0}^{-1})$$
where $\overline P:=w_{0}Pw_{0}^{-1} $ and $C(g)=PgP$ for $g\in G$.
\end{remark}
\begin{proof}
 The intersection of the two dense open subsets  $g\mathcal{C}$ and $\mathcal{C}$ in $G/P$  is open and not empty, for any $g\in G$.
\end{proof}

The group $G$ acts continuously on the topological space $G/P$,
\begin{align*}
    G \ \times \ G/P & \ \to \ G/P \\
    (g,xP/P) & \ \mapsto \ gxP/P \ .
\end{align*}
For $n,x\in N$ and $t\in L$ we have $ntxw_{0}P/P  = ntxt^{-1}w_{0}P/P  =  (nt.x)w_{0}P/P$ hence
the action of $P$ on the  open cell  corresponds to the action of $P$ on $N$ introduced before the proposition \ref{Res}.

When $M$ is an \'etale $A[P_{+}]$-module, this allows us to systematically view the map $\Res$ in the following as a $P$-equivariant homomorphism of $A$-algebras
\begin{equation*}
    \Res : C^\infty_c(\mathcal{C},A) \ \rightarrow \ \End_A(M^P)
\end{equation*}
and the corresponding sheaf (theorem \ref{the11}) as a sheaf on $\mathcal{C}$. Our purpose is to show that this sheaf extends naturally to a $G$-equivariant sheaf on $G/P$ for certain  \'etale $A[P_{+}]$-modules.  When $M$ is a complete topologically $A[P_{+}]$-module we note that also integration with respect to the measure  $\Res $ (proposition \ref{int}) will be viewed in the following as a map
\begin{align*}
C_{c}(\mathcal{C},E^{cont}) & \  \to \  E^{cont} \\
  f &\  \mapsto \ \int_{\mathcal{C}}f \ d\Res
\end{align*}
on the space $C_{c}(\mathcal{C},E^{cont})$ of compactly supported continuous maps from $\mathcal{C}$ to $E^{cont}$.

 \subsection{Topological $G$-space $G/P$ and the map $\alpha$}

\begin{definition} \label{ap}  An open subset ${\mathcal U}$ of $G/P$ is called standard if there is a $g \in G$ such that $g{\mathcal U}$ is contained in the open cell $\mathcal{C}$.
\end{definition}

The inclusion $g{\mathcal U} \subset Nw_{0}P/P$ is equivalent to ${\mathcal U} = g^{-1}Uw_{0}P/P$ for a unique open subset $U\subset N$. An open subset of a standard open subset is standard. The translates by $G$ of $N_{0}w_{0}P/P$ form a basis of the topology of $G/P$.

\begin{proposition} \label{fpar} A compact open subset ${\mathcal U} \subset G/P$ is a  disjoint union
\begin{equation*}
    {\mathcal U}= \bigsqcup_{g\in I} g^{-1}U w_{0}P/P
\end{equation*}
where $U\subset N$ is a compact open subgroup and $I\subset G$ a finite subset.
\end{proposition}

\begin{proof}  We first observe that any open covering of ${\mathcal U}$ can be refined into a disjoint open covering. In our case, this implies that  ${\mathcal U}$ has a finite disjoint  covering by standard compact open subsets. Let  $g^{-1} Uw_{0}P/P \subset G/P$ be a standard compact open subset. Then $U=\sqcup_{u\in J}u^{-1}V$ (disjoint union) with a finite set $I\subset U$ and  $V\subset N$ is a compact open subgroup. Then
$g^{-1} Uw_{0}P/P= \sqcup_{h\in I} h^{-1}V w_{0}P/P$  (disjoint union)  where     $I=uJ$.
\end{proof}

 For $g\in G$  and  $x$ in the non empty open subset  $g^{-1}\mathcal{C} \cap \mathcal{C}$  of $G/P$ (remark \ref{dense}), there is a unique element $\alpha(g,x) \in P$ such that, if $x = uw_0P/P$ with $u \in N$, then
\begin{equation*}
    guw_0N \ = \ \alpha(g,x)uw_0N \ .
\end{equation*}

We give some properties of the map $\alpha$.

\begin{lemma}\label{c3} Let $g\in G$. Then

(i) \ $g^{-1}\mathcal{C} \cap \mathcal{C} = \mathcal{C}$ if and only if $g\in P$.

(ii) \  The map $\alpha(g,.) : g^{-1}\mathcal{C} \cap \mathcal{C} \to P$ is continuous.

  (iii) \ We have $gx=\alpha(g,x) x$ for $x\in g^{-1}\mathcal{C} \cap \mathcal{C}$ and  we have $\alpha(b,x)=b$ for $b\in P$ and $x\in \mathcal{C}$.
\end{lemma}

\begin{proof}
(i)  We have $g^{-1}\mathcal{C} \cap \mathcal{C} = \mathcal{C}$ if and only if $gNw_{0} P \subset Nw_{0}P $ if and only if $g\in P$. Indeed,  the condition $hPw_{0}P \ \subset \ Pw_{0}P$ on $h\in G$ depends only on $PhP$ and for $ w\in W$, the condition  $ w\, Pw_{0}P \ \subset \ Pw_{0}P$   implies $ww_{0}\in Pw_{0}P$ hence $ww_{0}\in w_{0}L$ by the hypothesis b) hence $w\in L$.

(ii) \ Let $N_g \subset N$ be such that $N_g w_0P/P = g^{-1}\mathcal{C} \cap \mathcal{C}$. It suffices to show that the map $u \to \alpha(g,uw_0 P) u : N_g \to P$ is continuous. This follows from the continuity of the maps $u\mapsto guw_{0}N : N_g \to Pw_{0}P/N = Pw_{0}N/N$  and  $bw_{0}N \mapsto b: Pw_{0}N/N \to P$.

(iii) \ Obvious.
\end{proof}

\begin{lemma}\label{prod!}  Let $g,h\in G$ and $x\in (gh)^{-1}\mathcal{C} \cap h^{-1} \mathcal{C} \cap \mathcal{C}$. Then $hx \in g^{-1}\mathcal{C} \cap \mathcal{C}$ and  we have
\begin{equation*}
    \alpha(gh,x) = \alpha(g,hx) \alpha(h,x) \ .
\end{equation*}
\end{lemma}
\begin{proof} The first part of the assertion is obvious. Let $x = uw_0P$ and $hx = vw_0P$ with $u,v \in N$. We have
\begin{equation*}
    huw_0N = \alpha(h,x)uw_0N,\ gvw_0N = \alpha(g,hx)vw_0N,\ \textrm{and}\ \alpha(gh,x)uw_0N = ghuw_0N \ .
\end{equation*}
The first identity implies $\alpha(h,x)u = vb$ for some $b \in L$. Multiplying the second identity by an appropriate $b' \in L$ we obtain $gvbw_0N = \alpha(g,hx)vbw_0N = \alpha(g,hx)\alpha(h,x)uw_0N$. Finally, by inserting the first identity into the right hand side of the third identity we get
\begin{equation*}
    \alpha(gh,x)uw_0N = g \alpha(h,x)uw_0N = gvbw_0N = \alpha(g,hx)\alpha(h,x)uw_0N
\end{equation*}
which is the assertion.
\end{proof}

  It will be technically convenient later to work on $N$ instead of $\mathcal{C}$. For $g\in G  $ let therefore $N_{g}$ be the  open subset of $N$ such that  $\mathcal{C}\cap g^{-1}\mathcal{C} = N_{g}w_{0}P/P$. We have $N_{g} =N$ if and only if $g\in P$  (lemma \ref{c3} (i)).  We  have the homeomorphism $u \mapsto x_{u} :=uw_{0}P/P :N \xrightarrow{\sim} \mathcal{C}$ and the continuous map (lemma \ref{c3} (ii))
\begin{align*}
    N_g & \longrightarrow P \\
    u & \longmapsto \alpha(g,x_u)
\end{align*}
such that
\begin{equation} \label{f:alpha}
 \begin{split}  gu =  & \ \alpha(g,x_u)u \bar{n}(g,u) \qquad  \text{for some $\bar{n}(g,u) \in \overline{N} := w_0 N w_0^{-1}$} ,  \\
 \alpha(g,x_u) u = & \ n(g,u) t(g,u)   \qquad \quad  \  \text{for some } \ n(g,u) \in N  , t(g,u) \in L  \ .
 \end{split}
\end{equation}

\begin{lemma}\label{covering}
Fix $g \in G$ and let $V \subset g^{-1}\mathcal{C} \cap \mathcal{C}$ be any compact open subset. There exists a disjoint covering $V = V_1 \,\dot{\cup} \ldots \dot{\cup}\, V_m$ by compact open subsets $V_i$ and points $x_i \in V_i$ such that
\begin{equation*}
    \alpha(g,x_i)V_i \subset gV \qquad\textrm{for any $1 \leq i \leq m$}.
\end{equation*}
\end{lemma}
\begin{proof}
We denote the inverse of the homeomorphism $u \mapsto x_{u} :N \xrightarrow{\sim} \mathcal{C}$ by $x \mapsto u_x$. The image $C \subset P$ of $V$ under the continuous map  $x \mapsto \alpha(g,x)u_x \ : \ V \to P$ is compact.  As (lemma \ref{c3} (iii))
$ \alpha(g,x)x = gx \in gV  $ for any $x \in V$,
 under the continuous action of $P$ on $\mathcal{C}$, every element in the compact set $C$ maps the point $w_0P$ into $gV$. It follows that there is an open neighborhood $V_0 \subset \mathcal{C}$ of $w_0P$ such that $C V_0 \subset gV$. This means that
\begin{equation*}
    \alpha(g,x)u_x V_0 \subset gV \qquad\textrm{for any $x \in V$}.
\end{equation*}
Using the proposition \ref{fpar} we find, by appropriately shrinking $V_0$, a disjoint covering of $V$ of the form $V = u_1V_0 \,\dot{\cup} \ldots \dot{\cup}\, u_mV_0$ with $u_i \in N$. We put $x_i := u_i w_0P$.
\end{proof}

We denote by $G_{X}:=\{ x\in G \ | \ xX \subset X \} $  the $G$-stabilizer   of a  subset $X \subset G/P$ and  by
$$ G_{X}^{\dagger}:=\{g\in G \ | \ g \in G_{X} \ , \ g^{-1}\in G_{X} \} \ = \ \{x \in G \ | \ xX\ =\ X\} $$ the subgroup of invertible elements of $G_{X}$.  If $G_{X}$ is open then its inverse monoid  is open hence  $G_{X}^{\dagger}$
is open (and conversely).

\begin{lemma}\label{open}  The $G$-stabilizer $G_{\mathcal U}$  and $G_{\mathcal U}^{\dagger}$  are open in $G$,  for any compact open subset  ${\mathcal U} \subset G/P$.
 \end{lemma}

\begin{proof} By proposition \ref{fpar} it suffices to consider the case where ${\mathcal U}=Uw_{0}P/P$ for some compact open subgroup $U\subset N$. As $Uw_{0}P\subset G$ is an open subset containing $w_{0}$ there exists an open subgroup $K\subset G$ such that $Kw_{0}\subset Uw_{0}P$. The  set $U/(K\cap U)$ is finite because $U$ is compact and $(K\cap U)\subset U$ is an open subgroup. The  finite intersection $K' := \bigcap_{u\in U/(U\cap K)}uKu^{-1} = \bigcap_{u\in U}uKu^{-1}$ is an open subgroup of $K$ which is normalized by $U$. But $K'U = UK'$ implies that  $K'Uw_{0}P= UK'w_{0}P \subset U(Uw_{0}P)P=Uw_{0}P$, and hence that $K'\subset G_{\mathcal U}$. We deduce that $G_{\mathcal U}$ is open. Hence $G_{\mathcal U}^{\dagger}$ is open.
\end{proof}

\begin{remark}\label{sta} The $G$-stabilizer of the open cell $\mathcal{C}$ is the group $P$.
\end{remark}
\begin{proof}
Proof of lemma \ref{c3} (i).
\end{proof}

For  ${\mathcal U} \subset \mathcal{C}$  the map
\begin{equation}\label{c5}
     G_{\mathcal U} \times \mathcal{U} \ \to \ P \quad, \quad (g,x)\ \mapsto \ \alpha(g,x)
\end{equation}
is continuous because, if ${\mathcal U} = Uw_0P/P$ with $U$ open in $N$, then the map $(g,u) \mapsto guw_{0}N : G_{\mathcal{U}} \times U \to Pw_{0}P/N = Pw_{0}N/N$ is continuous (cf.\ the proof of the lemma \ref{c3} (ii)).

\subsection{Equivariant sheaves and modules over skew group rings}

Our construction of the sheaf on $G/P$ will proceed through a module theoretic interpretation of equivariant sheaves. The ring $C^\infty_c(\mathcal{C},A)$ has no unit element. But it has sufficiently many idempotents (the characteristic functions $1_V$ of the compact open subsets $V \subset \mathcal{C}$). A (left) module $Z$ over $C^\infty_c(\mathcal{C},A)$ is called nondegenerate if for any $z \in Z$ there is an idempotent $e \in C^\infty_c(\mathcal{C},A)$ such that $ez = z$.

It is well known that the functor
\begin{equation*}
    \textrm{sheaves of $A$-modules on $\mathcal{C}$} \ \to \  \textrm{nondegenerate $C^\infty_c(\mathcal{C},A)$-modules}
\end{equation*}
which sends a sheaf $\mathcal{S}$ to the $A$-module of global sections with compact support $\mathcal{S}_c(\mathcal{C}) := \bigcup_V \mathcal{S}(V)$, with $V$ running over all compact open subsets in $\mathcal{C}$, is an equivalence of categories. In fact, as we have discussed in the proof of the theorem \ref{the11} a quasi-inverse functor is given by sending the module $Z$ to the sheaf whose sections on the compact open subset $V \subset \mathcal{C}$ are equal to $1_V Z$.

In order to extend this equivalence to equivariant sheaves we note that the group $P$ acts, by left translations, from the right on $C^\infty_c(\mathcal{C},A)$ which we write as $(f,b) \mapsto f^b(.) := f(b.)$. This allows to introduce the skew group ring
\begin{equation*}
    \mathcal{A}_{\mathcal{C}} := C^\infty_c(\mathcal{C},A) \# P = \oplus_{b \in P}\; bC^\infty_c(\mathcal{C},A)
\end{equation*}
in which the multiplication is determined by the rule
\begin{equation*}
    (b_1f_1)(b_2f_2) = b_1b_2 f_1^{b_2}f_2 \qquad \textrm{for $b_i \in P$ and $f_i \in C^\infty_c(\mathcal{C},A)$}.
\end{equation*}
It is easy to see that the above functor extends to an equivalence of categories
\begin{equation*}
    \textrm{$P$-equivariant sheaves of $A$-modules on $\mathcal{C}$} \ \xrightarrow{\simeq} \  \textrm{nondegenerate $\mathcal{A}_{\mathcal{C}}$-modules}.
\end{equation*}

We have the completely analogous formalism for the $G$-space $G/P$. The only small difference is that, since $G/P$ is assumed to be compact, the ring $C^\infty(G/P,A)$ of locally constant $A$-valued functions on $G/P$ is unital. The skew group ring
\begin{equation*}
    \mathcal{A}_{G/P} := C^\infty(G/P,A) \# G = \oplus_{g \in G}\; gC^\infty(G/P,A)
\end{equation*}
therefore is unital as well, and the equivalence of categories reads
\begin{equation*}
    \textrm{$G$-equivariant sheaves of $A$-modules on $G/P$} \ \xrightarrow{\simeq} \  \textrm{unital $\mathcal{A}_{G/P}$-modules}.
\end{equation*}

For any open subset $\mathcal{U} \subset G/P$ the $A$-algebra $C^\infty_c(\mathcal{U},A)$ of $A$-valued locally constant and compactly supported functions on $\mathcal{U}$ is, by extending functions by zero, a subalgebra of $C^\infty(G/P,A)$. It follows in particular that $\mathcal{A}_{\mathcal{C}}$ is a subring of $\mathcal{A}_{G/P}$. There is a for our purposes very important ring   in between  these two rings which is defined to be
\begin{equation*}
    \mathcal{A} := \mathcal{A}_{\mathcal{C} \subset G/P} := \oplus_{g \in G}\; gC^\infty_c(g^{-1}\mathcal{C} \cap \mathcal{C},A) \ .
\end{equation*}
That $\mathcal{A}$ indeed is multiplicatively closed is immediate from the following observation. If $\supp(f)$ denotes the support of a function $f \in C^\infty(G/P,A)$ then we have the formula
\begin{equation}\label{f:supp}
    \supp(f_1^g f_2) = g^{-1}\supp(f_1) \cap \supp(f_2) \qquad\textrm{for $g \in G$ and $f_1, f_2 \in C^\infty(G/P,A)$}.
\end{equation}
In particular, if $f_i \in C^\infty_c(g_i^{-1}\mathcal{C} \cap \mathcal{C},A)$ then
\begin{equation*}
    \supp(f_1^{g_2} f_2) \subset g_2^{-1}(g_1^{-1}\mathcal{C} \cap \mathcal{C}) \cap (g_2^{-1}\mathcal{C} \cap \mathcal{C}) \subset (g_1 g_2 )^{-1}\mathcal{C} \cap \mathcal{C} \ .
\end{equation*}
We also have the $A$-submodule
\begin{equation*}
    \mathcal{Z} := \oplus_{g \in G}\; g C^\infty_c(\mathcal{C},A)
\end{equation*}
of $\mathcal{A}_{G/P}$. Using \eqref{f:supp} again one sees that $\mathcal{Z}$ actually is a left ideal in $\mathcal{A}_{G/P}$ which at the same time is a right $\mathcal{A}$-submodule. This means that we have the well defined functor
\begin{align*}
    \textrm{nondegenerate $\mathcal{A}$-modules} & \ \to \ \textrm{unital $\mathcal{A}_{G/P}$-modules} \\
    Z & \ \mapsto \ \mathcal{Z} \otimes_{\mathcal{A}} Z \ .
\end{align*}

\begin{remark}\label{restrict}
The functor of restricting $G$-equivariant sheaves on $G/P$ to the open cell $\mathcal{C}$ is faithful and detects isomorphisms.
\end{remark}
\begin{proof}
Any sheaf homomorphism which is the zero map, resp.\ an isomorphism, on sections on any compact open subset of $\mathcal{C}$ has, by $G$-equivariance, the same property on any standard compact open subset and hence, by the proposition \ref{fpar}, on any compact open subset of $G/P$.
\end{proof}

\begin{proposition}\label{cat-equiv}
The above functor $Z \mapsto \mathcal{Z} \otimes_{\mathcal{A}} Z$ is an equivalence of categories; a quasi-inverse functor is given by sending the $\mathcal{A}_{G/P}$-module $Y$ to $\bigcup_{V \subset \mathcal{C}} 1_V Y$ where $V$ runs over all compact open subsets in $\mathcal{C}$.
\end{proposition}
\begin{proof}
We abbreviate the asserted candidate for the quasi-inverse functor by $R(Y) := \bigcup_{V \subset \mathcal{C}} 1_V Y$. It immediately follows from the remark \ref{restrict} that the functor $R$, which in terms of sheaves is the functor of restriction, is faithful and detects isomorphisms.

By a slight abuse of notation we identify in the following a function $f \in C^\infty(G/P,A)$ with the element $1f \in \mathcal{A}_{G/P}$, where $1 \in G$ denotes the unit element. Let $V \subset \mathcal{C}$ be a compact open subset. Then $1_V \mathcal{A}_{G/P} 1_V$ is a subring of $\mathcal{A}_{G/P}$ (with the unit element $1_V$), which we compute as follows:
\begin{align*}
    1_V \mathcal{A}_{G/P} 1_V & = \sum_{g \in G} 1_V g C^\infty(V,A) = \sum_{g \in G} g 1_{g^{-1}V} C^\infty(V,A) \\
    & = \sum_{g \in G} gC^\infty(g^{-1}V \cap V,A) \ .
\end{align*}
We note:
\begin{itemize}
  \item[--] If $U \subset V \subset \mathcal{C}$ are two compact open subsets then $1_V \mathcal{A}_{G/P} 1_V  \supset 1_U \mathcal{A}_{G/P} 1_U$.
  \item[--] Let $f \in C^\infty_c(g^{-1} \mathcal{C} \cap \mathcal{C},A)$ be supported on the compact open subset $U \subset g^{-1} \mathcal{C} \cap \mathcal{C}$. Then $V := U \cup gU$ is compact open in $\mathcal{C}$ as well, and $U \subset g^{-1}V \cap V$. This shows that $C^\infty_c(g^{-1} \mathcal{C} \cap \mathcal{C},A) = \bigcup_{V \subset \mathcal{C}} C^\infty(g^{-1}V \cap V,A)$.
\end{itemize}
We deduce that
\begin{equation*}
    \bigcup_{V \subset \mathcal{C}} 1_V \mathcal{A}_{G/P} 1_V = \mathcal{A}_{\mathcal{C} \subset G/P} = \mathcal{A} \ .
\end{equation*}
A completely analogous computation shows that
\begin{equation*}
    1_V \mathcal{Z} = 1_V \mathcal{A} \ .
\end{equation*}
Given a nondegenerate $\mathcal{A}$-module $Z$ the map
\begin{align*}
    1_V (\mathcal{Z} \otimes_{\mathcal{A}} Z) = (1_V \mathcal{Z}) \otimes_{\mathcal{A}} Z = (1_V \mathcal{A}) \otimes_{\mathcal{A}} Z & \ \to \ 1_V Z \\
    1_Va \otimes z = 1_V \otimes 1_V a z & \ \mapsto \ 1_V a z
\end{align*}
therefore is visibly an isomorphism of $1_V \mathcal{A}_{G/P} 1_V$-modules. In the limit with respect to $V$ we obtain a natural isomorphism of $\mathcal{A}$-modules
\begin{equation*}
    R(\mathcal{Z} \otimes_{\mathcal{A}} Z) \ \xrightarrow{\cong} \ Z \ .
\end{equation*}
On the other hand, for any unital $\mathcal{A}_{G/P}$-module $Y$ there is the obvious natural homomorphism of $\mathcal{A}_{G/P}$-modules
\begin{align*}
    \mathcal{Z} \otimes_{\mathcal{A}} R(Y) & \ \to \ Y \\
    a \otimes z & \ \mapsto \ az \ .
\end{align*}
It is an isomorphism because applying the functor $R$, which detects isomorphisms, to it gives the identity map.
\end{proof}

\begin{remark}\label{G-extension}
Let $Z$ be a nondegenerate $\mathcal{A}$-module. Viewed as an $\mathcal{A}_{\mathcal{C}}$-module it corresponds to a $P$-equivariant sheaf $\widetilde{Z}$ on $\mathcal{C}$. On the other hand, the $\mathcal{A}_{G/P}$-module $Y := \mathcal{Z} \otimes_{\mathcal{A}} Z$ corresponds to a $G$-equivariant sheaf $\widetilde{Y}$ on $G/P$. We have $\widetilde{Y} | \mathcal{C} = \widetilde{Z}$, i.\ e., the sheaf $\widetilde{Y}$ extends the sheaf $\widetilde{Z}$.
\end{remark}

We have now seen that the step of going from $\mathcal{A}$ to $\mathcal{A}_{G/P}$ is completely formal. On the other hand, for any topologically \'etale $A[P_{+}]$-module $M$,  the $P$-equivariance of $\Res$ together with the proposition \ref{cont2} imply that  $\Res$ extends to the $A$-algebra homomorphism
\begin{align*}
    \Res \ : \qquad \mathcal{A}_{\mathcal{C}} & \ \to \ \End_{A}^{cont}(M^{P}) \\
    \sum_{b \in P} b f_b & \ \mapsto \ \sum_{b \in P} b \circ \Res(f_b) \ .
\end{align*}

When $M$ is compact it is relatively easy, as we will show in the next section, to further extend this map from $\mathcal{A}_{\mathcal{C}}$ to $\mathcal{A}$.  This makes crucially use of the full topological module $M^P$ and not only its submodule $M^P_c$ of sections with compact support. When $M$ is not compact this extension problem is much more subtle and requires more facts about the ring $\mathcal{A}$.

We introduce the compact open subset $\mathcal{C}_0 := N_0 w_0 P/P$ of $\mathcal{C}$, and we consider the unital subrings
\begin{equation*}
    \mathcal{A}_0 := 1_{\mathcal{C}_0} \mathcal{A}_{G/P} 1_{\mathcal{C}_0} = \sum_{g \in G} gC^\infty (g^{-1}\mathcal{C}_0 \cap \mathcal{C}_0,A)
\end{equation*}
and
\begin{equation*}
    \mathcal{A}_{\mathcal{C}0} := 1_{\mathcal{C}_0} \mathcal{A}_{\mathcal{C}} 1_{\mathcal{C}_0} = \sum_{b \in P} bC^\infty (b^{-1}\mathcal{C}_0 \cap \mathcal{C}_0,A)
\end{equation*}
of $\mathcal{A}$ and $\mathcal{A}_{\mathcal{C}}$, respectively. Obviously $\mathcal{A}_{\mathcal{C} 0} \subseteq \mathcal{A}_0$ with the same unit element $1_{\mathcal{C}_0}$. Since $g^{-1}\mathcal{C}_0 \cap \mathcal{C}_0$ is nonempty if and only if $g \in N_0\overline{P}N_0$ we in fact have
\begin{equation*}
    \mathcal{A}_0 = \sum_{g \in N_0\overline{P}N_0} gC^\infty (g^{-1}\mathcal{C}_0 \cap \mathcal{C}_0,A) \ .
\end{equation*}
The map $A[G] \longrightarrow \mathcal{A}_{G/P}$ sending $g$ to $g1_{G/P}$ is a unital ring homomorphism. Hence we may view $\mathcal{A}_{G/P}$ as an $A[G]$-module  for the adjoint action
\begin{align*}
    G \times \mathcal{A}_{G/P} & \longrightarrow \mathcal{A}_{G/P} \\
    (g, y) & \longmapsto (g 1_{G/P}) y (g 1_{G/P})^{-1} \ .
\end{align*}
One checks that $\mathcal{A}_{\mathcal{C}} \subseteq \mathcal{A}$ are $A[P]$-submodules, that $\mathcal{A}_{\mathcal{C}0} \subseteq \mathcal{A}_0$ are $A[P_+]$-submodules, and that the map $\Res : \mathcal{A}_{\mathcal{C}} \longrightarrow E^{cont}$ is a homomorphism of $A[P]$-modules.

\begin{proposition}\label{P+-P}
The homomorphism of $A[P]$-modules
\begin{align*}
    A[P] \otimes_{A[P_+]} \mathcal{A}_0 & \xrightarrow{\; \cong \;} \mathcal{A} \\
    b \otimes y & \longmapsto (b1_{G/P}) y (b1_{G/P})^{-1}
\end{align*}
is bijective; it restricts to an isomorphism $A[P] \otimes_{A[P_+]} \mathcal{A}_{\mathcal{C}0} \xrightarrow{\; \cong \;} \mathcal{A}_{\mathcal{C}}$.
\end{proposition}
\begin{proof}
Since $P = s^{-\mathbb{N}}P_+$ the assertion amounts to the claim that
\begin{equation*}
    \mathcal{A} = \bigcup_{n \geq 0} (s^{-n}1_{G/P}) \mathcal{A}_0 (s^n 1_{G/P})
\end{equation*}
and correspondingly for $\mathcal{A}_{\mathcal{C}}$. But we have
\begin{equation*}
   (s^{-n}1_{G/P}) \big( gC^\infty(g^{-1} \mathcal{C}_0 \cap \mathcal{C}_0,A) \big) (s^n 1_{G/P}) =  s^{-n}gs^n C^\infty((s^{-n}g^{-1}s^n)s^{-n}\mathcal{C}_0 \cap s^{-n}\mathcal{C}_0,A)
\end{equation*}
for any $n \geq 0$ and any $g \in G$.
\end{proof}

Suppose that we may extend the map $\Res : \mathcal{A}_{\mathcal{C} 0} \longrightarrow \End_{A}^{cont}(M^{P})$ to an $A[P_+]$-equivariant (unital) $A$-algebra homomorphism
\begin{equation*}
    \mathcal{R}_0 : \mathcal{A}_0 \longrightarrow \End_A(M^P) \ .
\end{equation*}
By the above proposition \ref{P+-P} it further extends uniquely to an $A[P]$-equivariant map $\mathcal{R} : \mathcal{A} \longrightarrow \End_A(M^P)$.

\begin{lemma}\label{mult}
The map $\mathcal{R}$ is multiplicative.
\end{lemma}
\begin{proof}
Using proposition \ref{P+-P} we have that two arbitrary elements $y,z \in \mathcal{A}$ are of the form $y = (s^{-m}1_{G/P}) y_0 (s^{m}1_{G/P}), z = (s^{-n}1_{G/P}) z_0 (s^{n}1_{G/P})$ with $m,n \in \mathbb{N}$ and $y_0 , z_0 \in \mathcal{A}_0$.  We can choose $m=n$.  It follows that
\begin{align*}
    yz = (s^{-m}1_{G/P})y_0 z_0  (s^{m}1_{G/P})
    =(s^{-m}1_{G/P}) x_0 (s^{m}1_{G/P})
   \end{align*}
 with $  x_0 := y_0 z_0  \in \mathcal{A}_0 $, and that
 \begin{align*}
    \mathcal{R}(yz) & = \mathcal{R}( (s^{-m}1_{G/P}) x_0 (s^{m}1_{G/P})) = s^{-m} \circ \mathcal{R}_0(x_0) \circ s^{m} \\
      & = s^{-m} \circ \mathcal{R}_0(y_0) \circ \mathcal{R}_0(z_0) \circ s^{n-m} \circ s^{m} \\
    & = (s^{-m} \circ \mathcal{R}_0(y_0) \circ s^{m}) \circ (s^{-m} \circ \mathcal{R}_0(z_0) \circ s^{m}) \\
    & = \mathcal{R}(y) \circ  \mathcal{R}(z) \ .
\end{align*}
 \end{proof}

Note that the images $\Res(\mathcal{A}_{\mathcal{C}0})$ and $\mathcal{R}_0 (\mathcal{A}_0)$ necessarily lie in the image of $\End_A(M) = \End_A(\Res(1_{\mathcal{C}_0})(M^P))$ by the natural embedding into $\End_A(M^P)$. This reduces us to search for an $A[P_+]$-equivariant (unital) $A$-algebra homomorphism
\begin{equation*}
     \mathcal{R}_0 : \mathcal{A}_0 \longrightarrow \End_A(M)
\end{equation*}
which extends $\Res | \mathcal{A}_{\mathcal{C} 0}$. In fact, since for $g \in N_0\overline{P}N_0$ and $f \in C^\infty (g^{-1}\mathcal{C}_0 \cap \mathcal{C}_0,A)$ we have $gf = (g 1_{g^{-1}\mathcal{C}_0 \cap \mathcal{C}_0})(1f)$ with $1f \in \mathcal{A}_{\mathcal{C} 0}$ it suffices to find the elements
\begin{equation*}
    \mathcal{H}_g = \mathcal{R}_0 (g 1_{g^{-1}\mathcal{C}_0 \cap \mathcal{C}_0}) \in \End_A(M) \qquad\text{for $g \in N_0\overline{P}N_0$} \ .
\end{equation*}
 Note that $P_{+}=N_{0}L_{+}$ is contained in $N_0\overline{P}N_0= N_{0}L \overline N N_{0}$.
\begin{proposition}\label{multiplicative}
We suppose given, for any $g \in N_0\overline{P}N_0$, an element $\mathcal{H}_g \in \End_A(M)$. Then the map
\begin{align*}
    \mathcal{R}_0 : \qquad\qquad  \mathcal{A}_0 & \longrightarrow \End_A(M) \\
    \sum_{g \in N_0\overline{P}N_0} g f_g & \longmapsto \sum_{g \in N_0\overline{P}N_0} \mathcal{H}_g \circ \res(f_g)
\end{align*}
is an $A[P_+]$-equivariant (unital) $A$-algebra homomorphism which extends $\Res | \mathcal{A}_{\mathcal{C}0}$ if and only if, for all $g,h\in N_0\overline{P}N_0$,   $b\in P \cap N_0\overline{P}N_0$, and all compact open subsets $ \mathcal V \subset \mathcal C_{0}$, the relations
\begin{itemize}
\item[H1.] $ \res (1_{\mathcal V })\circ  {\mathcal H}_{g}   =
  {\mathcal H}_{g} \circ \res (1_{g^{-1}\mathcal V \cap \mathcal{C}_0}) $  ,
\item[H2.]  ${\mathcal H}_{g} \circ  {\mathcal H}_{h}  = {\mathcal H}_{gh} \circ \res ( 1_{(gh)^{-1}\mathcal C_{0} \cap h^{-1}\mathcal C_{0} \cap \mathcal C_{0}})$  ,
  \item[H3.] ${\mathcal H}_{b}  = b \circ \res (1_{b^{-1}\mathcal{C}_0 \cap \mathcal{C}_0})$ .
\end{itemize}
hold true.   When H1  is true,  H2 can equivalently be replaced by
\begin{equation*}
    {\mathcal H}_{g} \circ  {\mathcal H}_{h} = {\mathcal H}_{gh} \circ \res ( 1_{h^{-1}\mathcal C_{0} \cap \mathcal C_{0}}) \ .
\end{equation*}
\end{proposition}
\begin{proof}
Necessity of the relations is easily checked. Vice versa,
the first two relations imply that $\mathcal{R}_0$ is multiplicative. The third relation says that $\mathcal{R}_0$ extends $\Res | \mathcal{A}_{\mathcal{C}0}$.

 The last sentence of the assertion derives from the fact that we have
\begin{align*}
 {\mathcal H}_{gh} \circ \res ( 1_{(gh)^{-1}\mathcal C_{0} \cap h^{-1}\mathcal C_{0} \cap \mathcal C_{0}})
 & =  {\mathcal H}_{gh} \circ \res ( 1_{(gh)^{-1}\mathcal C_{0} \cap \mathcal C_{0}}) \circ \res ( 1_{h^{-1}\mathcal C_{0} \cap \mathcal C_{0}})\\
&= {\mathcal H}_{gh} \circ \res ( 1_{h^{-1}\mathcal C_{0} \cap \mathcal C_{0}})
\end{align*}
since ${\mathcal H}_{gh} \circ \res ( 1_{(gh)^{-1}\mathcal C_{0} \cap \mathcal C_{0}}) = {\mathcal H}_{gh}$ by the first relation.

The $P_+$-equivariance is equivalent to the identity
\begin{equation*}
    \mathcal{R}_0((c1_{G/P})gf_g (c1_{G/P})^{-1}) = \varphi_c \circ \mathcal{R}_0(gf_g) \circ \psi_c
\end{equation*}
where  $c\in P_{+}$ and $f_g$ is any function in $C^\infty(g^{-1} \mathcal{C}_0 \cap \mathcal{C}_0)$. By the definition of $\mathcal{R}_0$ and the $P_+$-equivariance of $\res$ the left hand side is equal to
\begin{equation*}
    {\mathcal H}_{cgc^{-1}} \circ \varphi_c \circ \res(f_g) \circ \psi_c
\end{equation*}
whereas the right hand side is
\begin{equation*}
    \varphi_c \circ \mathcal{H}_g \circ \res(f_g) \circ \psi_c \ .
\end{equation*}
Since $\psi_c$ is surjective and $\res(f_g) = \res(1_{g^{-1}\mathcal \mathcal{C}_0 \cap \mathcal{C}_0}) \circ \res(f_g)$ we see that the $P_+$-equivariance of $\mathcal{R}_0$ is equivalent to the identity
\begin{equation*}
    {\mathcal H}_{cgc^{-1}} \circ \varphi_c \circ \res(1_{g^{-1}\mathcal \mathcal{C}_0 \cap \mathcal{C}_0}) = \varphi_c \circ \mathcal{H}_g \circ \res(1_{g^{-1}\mathcal \mathcal{C}_0 \cap \mathcal{C}_0}) \ .
\end{equation*}
But as a special case of the first relation we have $\mathcal{H}_g \circ \res(1_{g^{-1}\mathcal \mathcal{C}_0 \cap \mathcal{C}_0}) = \mathcal{H}_g$. Hence the latter identity coincides with the relation
\begin{equation*}
    {\mathcal H}_{cgc^{-1}} \circ \varphi_c \circ \res(1_{g^{-1}\mathcal \mathcal{C}_0 \cap \mathcal{C}_0}) = \varphi_c \circ \mathcal{H}_g   \ .
\end{equation*}
This relation holds true because $\varphi_{c}= {\mathcal H}_{c}$ and by  the second relation ${\mathcal H}_{cgc^{-1}} \circ {\mathcal H}_c ={\mathcal H}_{cg }$ and ${\mathcal H}_c \circ  \mathcal{H}_g =  \mathcal{H}_{cg} \circ  \res(1_{g^{-1}\mathcal \mathcal{C}_0 \cap \mathcal{C}_0})$.
\end{proof}

\subsection{Integrating $\alpha$  when $M$ is compact} \label{6.3}

  Let $M $  be a compact topologically  \'etale $A[P_{+}]$-module. Then $M^{P}$ is compact,  hence the continuous action  of $P$ on $M^{P}$  (proposition \ref{cont2}) induces a continuous
  map $P\to E^{cont}$.

 We will construct an extension $\widetilde{\Res}$ of $\Res$ to $\mathcal{A}_{\mathcal{C} \subset G/P}$ by integration. For any $g\in G$,  we consider the  continuous map
\begin{equation*}
    \alpha_{g} \ : \ g^{-1}\mathcal{C} \cap \mathcal{C} \  \xrightarrow{\alpha(g,.)} \  P \ \to \  E ^{cont} \ .
\end{equation*}
We introduce the $A$-linear maps
\begin{align*}
    \rho \ : \ \mathcal{A} = \mathcal{A}_{\mathcal{C} \subset G/P}  & \ \to \ C_c(\mathcal{C},E^{cont}) \\
    \sum_{g \in G} g f_g & \ \mapsto \ \sum_{g \in G} \alpha_g f_g \ .
\end{align*}
and
\begin{align*}
    \widetilde{\Res} \ : \ \mathcal{A} = \mathcal{A}_{\mathcal{C} \subset G/P} & \ \to \ E^{cont} \\
    a & \ \mapsto \ \int_{\mathcal{C}} \rho(a) d\Res \ .
\end{align*}
For $b \in P$ the map $\alpha_b$ is the constant map on $\mathcal{C}$ with value $b$ (lemma \ref{6.3} iii). It follows that
\begin{equation*}
    \widetilde{\Res} \, | \, \mathcal{A}_{\mathcal{C}} = \Res \ .
\end{equation*}
is an extension as we want it.

\begin{theorem}\label{theo}
$\widetilde{\Res}$ is a homomorphism of $A$-algebras.
\end{theorem}
\begin{proof}
Let $g,h \in G$ and $V_g$ and $V_h$ compact open subsets in $g^{-1} \mathcal{C} \cap \mathcal{C}$ and $h^{-1} \mathcal{C} \cap \mathcal{C}$, respectively. We have to show that
\begin{equation*}
    \widetilde{\Res}((g1_{V_g})(h1_{V_h})) = \widetilde{\Res}(g1_{V_g}) \circ \widetilde{\Res}(h1_{V_h})
\end{equation*}
holds true. This is equivalent to the identity
\begin{equation*}
    \int_{\mathcal{C}} \alpha_{gh} 1_{h^{-1}V_g \cap V_h} d\Res = \int_{\mathcal{C}} \alpha_g 1_{V_g} d\Res \circ \int_{\mathcal{C}} \alpha_h 1_{V_h} d\Res \ .
\end{equation*}
We first treat special cases of this identity.

\textit{Case 1:} We assume that $g=1$ and that $V_1 = h V_h$. In this case we have to show that
\begin{equation*}
    \int_{\mathcal{C}} \alpha_h 1_{V_h} d\Res = \Res(1_{hV_h}) \circ \int_{\mathcal{C}} \alpha_h 1_{V_h} d\Res \ .
\end{equation*}
holds true. The set of all disjoint coverings $V_h = V_1 \,\dot{\cup} \ldots \dot{\cup}\, V_m$ by compact open subsets $V_i$ is partially ordered by refinement. Associating with this covering the element
\begin{equation*}
    \sum_{i=1}^m \alpha(h,x_i) \circ \Res(1_{V_i}) \ ,
\end{equation*}
where the $x_i \in V_i$ are arbitrarily chosen points, defines a net in $E^{cont}$ which converges to $\int_{\mathcal{C}} \alpha_h 1_{V_h} d\Res$. By applying the lemma \ref{covering} to each $V_i$ we obtain a refinement of the given covering which satisfies the assertion of that lemma. In other words, by restricting to a certain cofinal set of coverings and choosing the $x_i$ appropriately, we have $\alpha(h,x_i)V_i \subset hV_h$. But by the $P$-equivariance of $\Res$ we have
\begin{equation*}
    \sum_{i=1}^m \alpha(h,x_i) \circ \Res(1_{V_i}) = \sum_{i=1}^m \Res(1_{\alpha(h,x_i)V_i}) \circ \alpha(h,x_i) \ .
\end{equation*}
Since $1_{hV_h} 1_{\alpha(h,x_i)V_i} = 1_{\alpha(h,x_i)V_i}$ we obtain
\begin{align*}
    \sum_{i=1}^m \alpha(h,x_i) \circ \Res(1_{V_i}) & = \sum_{i=1}^m \Res(1_{\alpha(h,x_i)V_i}) \circ \alpha(h,x_i) \\
    & = \Res(1_{hV_h}) \circ \sum_{i=1}^m \Res(1_{\alpha(h,x_i)V_i}) \circ \alpha(h,x_i) \\
    & = \Res(1_{hV_h}) \circ \sum_{i=1}^m \alpha(h,x_i) \circ \Res(1_{V_i}) \ .
\end{align*}
As the multiplication in $E^{cont}$ is continuous our initial identity follows by passing to the limit.

\textit{Case 2:} We assume that $g=1$ and that $V_1 = h V$ fo some compact open subset $V \subset V_h$. In this case we have to show that
\begin{equation*}
    \int_{\mathcal{C}} \alpha_h 1_V d\Res = \Res(1_{hV}) \circ \int_{\mathcal{C}} \alpha_h 1_{V_h} d\Res \ .
\end{equation*}
holds true. But, applying the first case to $(h,V)$ and $(h, V_h - V)$, we obtain
\begin{align*}
    \Res(1_{hV}) \circ \int_{\mathcal{C}} \alpha_h 1_{V_h} d\Res & = \Res(1_{hV}) \circ \int_{\mathcal{C}} \alpha_h 1_V d\Res + \Res(1_{hV}) \circ \int_{\mathcal{C}} \alpha_h 1_{V_h - V} d\Res \\
    & = \int_{\mathcal{C}} \alpha_h 1_V d\Res + \Res(1_{hV}) \circ \Res(1_{hV_h - hV}) \circ \int_{\mathcal{C}} \alpha_h 1_{V_h - V} d\Res \\
    & = \int_{\mathcal{C}} \alpha_h 1_V d\Res \ .
\end{align*}

\textit{Case 3:} We assume that $V_h \subset (gh)^{-1}\mathcal{C} \cap h^{-1}\mathcal{C} \cap \mathcal{C}$ and that $V_g = hV_h$. In this case we have to show that
\begin{equation*}
    \int_{\mathcal{C}} \alpha_{gh} 1_{V_h} d\Res = \int_{\mathcal{C}} \alpha_g 1_{hV_h} d\Res \circ \int_{\mathcal{C}} \alpha_h 1_{V_h} d\Res \ .
\end{equation*}
holds true. As before we consider the partially ordered set of  disjoint coverings
$V_h = V_1 \,\dot{\cup} \ldots \dot{\cup}\, V_m$ by compact open subsets $V_i$, and we pick points $x_i \in V_i$. The left hand side is the limit of the net
\begin{equation*}
    \sum_{i=1}^m \alpha(gh,x_i) \circ \Res(1_{V_i}) = \sum_{i=1}^m \alpha(g,hx_i) \circ \alpha(h,x_i) \circ \Res(1_{V_i})
\end{equation*}
where we have used the lemma \ref{prod!}. Using the continuity of the product in $E^{cont}$ and the second case we see that the right hand side is the limit of the net
\begin{equation*}
    \sum_{i=1}^m \alpha(g,hx_i) \circ \Res(1_{hV_i}) \circ \int_{\mathcal{C}} \alpha_h 1_{V_h} d\Res = \sum_{i=1}^m \alpha(g,hx_i) \circ \int_{\mathcal{C}} \alpha_h 1_{V_i} d\Res \ .
\end{equation*}
Hence we have to show that the net of differences
\begin{multline*}
    \qquad \sum_{i=1}^m \alpha(g,hx_i) \circ \big( \int_{\mathcal{C}} \alpha_h 1_{V_i} d\Res - \alpha(h,x_i) \circ \Res(1_{V_i}) \big) \\ = \sum_{i=1}^m \alpha(g,hx_i) \circ \int_{\mathcal{C}} (\alpha_h 1_{V_i} - \alpha_h(x_i)) d\Res \qquad
\end{multline*}
converges to zero. This means that, for any open $A$-submodule $\mathcal{L} \subset M^P$ we have to find a disjoint covering of $V_h$ such that for all its refinements we have
\begin{equation*}
    \sum_{i=1}^m \alpha(g,hx_i) \circ \int_{\mathcal{C}} (\alpha_h 1_{V_i} - \alpha_h(x_i)) d\Res \in E^{cont}_{\mathcal{L}} \ .
\end{equation*}
The image $C \subset P$ of $V_h$ under the continuous map $x \mapsto \alpha(g,hx)$ is compact. Hence, by an argument completely analogous to the proof of the lemma \ref{No}, the $C$-invariant open submodules of $M^P$ are cofinal among all open submodules. We therefore may assume that $\mathcal{L}$ is $C$-invariant. This reduces us further to finding a disjoint covering of $V_h$ such that for all its refinements we have
\begin{equation*}
    \int_{\mathcal{C}} (\alpha_h 1_{V_i} - \alpha_h(x_i)) d\Res \in E^{cont}_{\mathcal{L}} \ .
\end{equation*}
This is a special case of the lemma \ref{lzero}.

We now combine these cases to obtain the asserted identity for general $g,h, V_g$, and $V_h$. First of all we note that $h^{-1} V_g \cap V_h \subset (gh)^{-1} \mathcal{C} \cap h^{-1} \mathcal{C} \cap \mathcal{C}$. Hence the third case gives the first equality in the following computation:
\begin{align*}
    \int_{\mathcal{C}} \alpha_{gh} 1_{h^{-1}V_g \cap V_h} d\Res & = \int_{\mathcal{C}} \alpha_g 1_{V_g \cap hV_h} d\Res \circ \int_{\mathcal{C}} \alpha_h 1_{h^{-1}V_g \cap V_h} d\Res \\
    & = \int_{\mathcal{C}} \alpha_g 1_{V_g \cap hV_h} d\Res \circ \Res(1_{V_g \cap hV_h}) \circ \int_{\mathcal{C}} \alpha_h 1_{V_h} d\Res \\
    & = \int_{\mathcal{C}} \alpha_g 1_{V_g} d\Res \circ \int_{\mathcal{C}} \alpha_h 1_{V_h} d\Res
\end{align*}
The second, resp.\ third, equality uses the second case for the right factor, resp.\ the remark \ref{multiplic} for the left factor, on the right hand side.
\end{proof}

\subsection{$G$-equivariant sheaf  on $G/P$ }

Let $M$ be a compact topologically \'etale $A[P_+]$-module. We briefly survey our construction of a $G$-equivariant sheaf on $G/P$ functorially associated with $M$.

From the proposition \ref{Res} we have obtained an $A$-algebra homomorphism
\begin{equation*}
    \Res \ : \ C^\infty_c(\mathcal{C},A) \# P  \ \to \ E^{cont}
\end{equation*}
which gives rise to a $P$-equivariant sheaf on $\mathcal{C}$ as described in detail in the theorem \ref{the11}. In the theorem \ref{theo} we have seen that it extends to an $A$-algebra homomorphism
\begin{equation*}
    \widetilde{\Res} \ : \ \mathcal{A}_{\mathcal{C} \subset G/P} \ \to \ E^{cont} \ .
\end{equation*}
This homomorphism defines on the global sections with compact support $M^P_c$ of the sheaf on $\mathcal{C}$ the structure of a nondegenerate $\mathcal{A}_{\mathcal{C} \subset G/P}$-module. The latter leads, by the proposition \ref{cat-equiv}, to the unital $C^\infty_c(G/P,A) \# G$-module $\mathcal{Z} \otimes_{\mathcal{A}} M^P_c$ which corresponds to a $G$-equivariant sheaf on $G/P$ extending the earlier sheaf on $\mathcal{C}$ (remark \ref{extension}). We will denote the sections of this latter sheaf on an open subset $\mathcal{U} \subset G/P$ by $M \boxtimes \mathcal{U}$. The restriction maps in this sheaf, for open subsets ${\mathcal V} \subset {\mathcal U} \subset G/P$, will simply be written as $\Res_{\mathcal V}^{\mathcal U} \ : \  M\boxtimes {\mathcal U} \ \to \    M\boxtimes {\mathcal V}$.

We observe that for a standard compact open subset $\mathcal{U} \subset G/P$ with $g \in G$ such that $g\mathcal{U} \subset \mathcal{C}$ the action of the element $g$ on the sheaf induces an isomorphism of $A$-modules $M \boxtimes \mathcal{U} \xrightarrow{\cong} M \boxtimes g\mathcal{U} = M_{g\mathcal{U}}$. Being the image of a continuous projector on $M^P$ (proposition \ref{cont2}), $M_{g\mathcal{U}}$ is naturally a compact topological $A$-module. We use the above isomorphism to transport this topology to $M \boxtimes \mathcal{U}$. The result is independent of the choice of $g$ since, if $g\mathcal{U} = h\mathcal{U}$ for some other $h \in G$, then $h\mathcal{U} \subset (gh^{-1})^{-1}\mathcal{C} \cap \mathcal{C}$ and, by construction, the endomorphism $gh^{-1}$ of $M \boxtimes h\mathcal{U}$ is given by the continuous map $\widetilde{\Res}(gh^{-1} 1_{h\mathcal{U}})$.

A general compact open subset $\mathcal{U} \subset G/P$ is the disjoint union $\mathcal{U} = \mathcal{U}_1 \, \dot{\cup} \ldots \dot{\cup} \, \mathcal{U}_m$ of standard compact open subsets $\mathcal{U}_i$ (proposition \ref{fpar}). We equip $M \boxtimes \mathcal{U} = M \boxtimes \mathcal{U}_1 \oplus \ldots \oplus M \boxtimes \mathcal{U}_m$ with the direct product topology. One easily verifies that this is independent of the choice of the covering.

Finally, for an arbitrary open subset $\mathcal{U} \subset G/P$ we have $M \boxtimes \mathcal{U} = \varprojlim M \boxtimes \mathcal{V}$, where $\mathcal{V}$ runs over all compact open subsets $\mathcal{V} \subset \mathcal{U}$, and we equip $M \boxtimes \mathcal{U}$ with the corresponding projective limit topology.

It is straightforward to check that all restriction maps are continuous and that any $g \in G$ acts by continuous homomorphisms. We see that $(M \boxtimes \mathcal{U})_{\mathcal{U}}$ is a $G$-equivariant sheaf of compact  topological $A$-modules.

\begin{lemma}\label{cont4}
For any compact open subset $\mathcal{U} \subset G/P$ the action $G_{\mathcal U}^{\dagger}  \times (M \boxtimes{\mathcal U}) \  \to  \ M \boxtimes{\mathcal U}$ of the open subgroup $G^\dagger_{\mathcal{U}}$ (lemma \ref{open}) on the sections on $\mathcal{U}$ is  continuous.
\end{lemma}
\begin{proof}
 Using the proposition \ref{fpar}, it suffices to consider the case that $\mathcal{U} \subset \mathcal{C}$. Note that  $G_{\mathcal U}^{\dagger}$ acts by continuous automorphisms on $M \boxtimes{\mathcal U}=M _{\mathcal U}$. By \eqref{c5} the map
\begin{align*}
    G_{\mathcal U}^{\dagger} \times \mathcal{U} & \ \to \  E^{cont} \\
    (g,x) & \ \mapsto \ \alpha_g(x)
\end{align*}
is continuous. Hence (\cite{BTG} TG X.28 Th.\ 3) the corresponding map
\begin{equation*}
    G_{\mathcal U}^{\dagger} \ \to \ C(\mathcal{U}, E^{cont})
\end{equation*}
is continuous, where we always equip the module $C(\mathcal{U}, E^{cont})$ of $E^{cont}$-valued continuous maps on $\mathcal{U}$ with the compact-open topology. On the other hand it is easy to see that, for any measure $\lambda$ on $\mathcal{C}$ with values in $E^{cont}$, the map
\begin{equation*}
    \int_{\mathcal{U}} . \, d\lambda \ : \ C(\mathcal{U}, E^{cont}) \ \to \ E^{cont}
\end{equation*}
is continuous. It follows that the map
\begin{align*}
    G_{\mathcal U}^{\dagger} & \ \to \ E^{cont} \\
    g & \ \mapsto \ \widetilde{\Res}(g 1_{\mathcal{U}})
\end{align*}
is continuous. The direct decomposition $M^{P}= M _{\mathcal U}\oplus M _{\mathcal C -\mathcal U}$  gives a natural  inclusion map $\End_A^{cont}(M_{\mathcal{U}}) \to E^{cont}$ through which the above map factorizes.
The  resulting map
\begin{equation*}
    G_{\mathcal U}^{\dagger} \ \to \ \End_A^{cont}(M_{\mathcal{U}})
\end{equation*}
is continuous and coincides with   the $G_{\mathcal U}^{\dagger}$-action on $M_{\mathcal{U}}$. As   $M_{\mathcal{U}}$ is compact  this continuity implies the   continuity of the action $G_{\mathcal U}^{\dagger}\times M_{\mathcal{U}} \to M_{\mathcal{U}}$.
\end{proof}

The same construction can be done, starting from the compact topologically \'etale $A[P_{U}]$-module $M_{U}$, for any compact open subgroup $U\subset N$.
\begin{proposition}
Let $U\subset N$ be a compact open subgroup. The $G$-equivariant sheaves  on $G/P$  associated to $(N_{0},M)$ and to $(U,M_{U})$ are  equal.
 \end{proposition}

\begin{proof}
As the $P$-equivariant sheaves  on the open cell associated to $(N_{0},M)$ and to $(U,M_{U})$ are equal by the proposition \ref{b}, and as the function $\alpha_{g}$ depends only on the open cell, our formal construction gives the same  $G$-equivariant sheaf.
\end{proof}

 \section{Integrating $\alpha$  when $M$ is non compact} \label{fc}

Recall that we have chosen a certain element   $s \in Z(L)$ such that $L=L_{-}s^{\mathbb Z}$ and  $(N_k = s^k N_0 s^{-k}) _{k \in \mathbb{Z}}$  is a decreasing sequence with union $N$ and trivial intersection.
{\sl We now suppose in addition that $(\overline{N}_k := s^{-k} w_0 N_0 w_0^{-1} s^k)_{k \in \mathbb{Z}}$  is a decreasing sequence with union $\overline{N} = w_0 N w_0^{-1}$ and trivial intersection}.

We have chosen $A$ and $M$  in section \ref{S4}. {\sl We suppose now in addition that
 $M$ is a topologically \'etale $A[P_{+}]$-module which is Hausdorff and complete.}

 \begin{definition} \label{mathfrakC}   A special family of compact sets  in $M$
is a family $\mathfrak{C}$ of compact subsets of $M$ satisfying :
\begin{itemize}
\item[$\mathfrak{C}(1)$] Any  compact subset of  a compact set in
$ \mathfrak{C}$  also lies in $\mathfrak{C}$.
\item[$\mathfrak{C}(2)$] If $C_1,C_2,\dots,C_n\in\mathfrak{C}$ then $\bigcup_{i=1}^nC_i$ is in
  $\mathfrak{C}$, as well.
\item[$\mathfrak{C}(3)$] For all $C\in\mathfrak{C}$ we have $N_0C\in\mathfrak{C}$.
\item[$\mathfrak{C}(4)$] $M(\mathfrak{C}):=\bigcup_{C\in\mathfrak{C}}C$ is an \'etale $A[P_+]$-submodule
of $M$.
\end{itemize}
\end{definition}

Note that  $M$ is the union of its compact subsets, and that the family  of all compact subsets of $M$ satisfies these four properties.

\bigskip
Let  $\mathfrak{C}$ be a special family of compact sets  in $M$. A map  from $M(\mathfrak{C})$ to $M$ is called $\mathfrak{C}$-continuous if its restriction to any
$C\in \mathfrak{C}$ is continuous.
We
equip the $A$-module $\Hom_A^{\mathfrak{C}ont}(M(\mathfrak{C}),M)$ of $\mathfrak{C}$-continuous $A$-linear homomorphisms from $M(\mathfrak{C})$ to $M$  with the $\mathfrak C$-open
topology. The $A$-submodules
$E(C,\mathcal{M}):=\{f\in\Hom_A^{\mathfrak{C}ont} (M(\mathfrak{C}),M)\colon f(C)\subseteq
\mathcal{M}\}$, for any $C\in\mathfrak{C}$ and any open $A$-submodule $\mathcal{M} \subseteq M$, form a fundamental system of open neighborhoods of zero in
$\Hom_A^{\mathfrak{C}ont} (M(\mathfrak{C}),M)$. Indeed, this system is   closed for
finite intersection by $\mathfrak{C}(2)$.
Since $N_0$ is compact the $E(C,\mathcal{M})$ for $C$ such that $N_0 C \subseteq C$ and $\mathcal{M}$ an $A[N_0]$-submodule still form a fundamental system of open neighborhoods of zero.
 (Lemma \ref{No} and $\mathfrak{C}(3)$).  We have:
\begin{itemize}
   \item[--]
 $\Hom_A^{\mathfrak{C}ont} (M(\mathfrak{C}),M)$ is a
 topological $A$-module.

  \item[--]   $\Hom_A^{\mathfrak{C}ont} (M(\mathfrak{C}),M)$ is Hausdorff,  since $\mathfrak{C}$ covers
$M(\mathfrak{C})$  by  $\mathfrak{C}(4)$ and $M$ is Hausdorff.

 \item[--] $\Hom_A^{\mathfrak{C}ont} (M(\mathfrak{C}),M)$ is complete (\cite{BTG} TG X.9 Cor.2).
\end{itemize}

\subsection{$(s,\res, \mathfrak{C})$-integrals}

We have the $P_+$-equivariant measure $\res : C^\infty(N_0,A) \longrightarrow \End_A^{cont}(M)$ on $N_0$. If $M$ is not compact then it is no longer possible to integrate any map in the $A$-module $C(N_0, \End_{A}^{cont}(M))$ of all continuous maps on $N_0$ with values in $\End_{A}^{cont}(M)$ against this measure. This forces us to introduce a notion of integrability with respect to a special family of compact sets in $M$.

\begin{definition} \label{integrability}A   map $F\colon N_0\to
  \Hom_A^{\mathfrak{C}ont} (M(\mathfrak{C}),M)$ is called integrable with respect to
  ($s$, $\res$, $\mathfrak{C}$)  if the limit
\begin{equation*}
    \int_{N_0} Fd\res := \lim_{k \rightarrow \infty} \sum_{u \in J(N_0/N_k)} F(u) \circ \res(1_{uN_k}) \ ,
\end{equation*}
where $J(N_0/N_k) \subseteq N_0$, for any $k \in \mathbb{N}$, is a set of representatives for the cosets in $N_0/N_k$, exists in $ \Hom_A^{\mathfrak{C}ont} (M(\mathfrak{C}),M)$ and is independent of the choice of the sets $J(N_0/N_k)$.
\end{definition}

We suppress $\mathfrak{C}$ from the notation when $\mathfrak{C}$ is the family of all compact subsets of  $M$.

\bigskip Note that we regard $\res(1_{uN_{k+1}})$ as an element in
$\End_A^{cont}(M(\mathfrak{C}))$. This makes sense as the algebraically \'etale submodule $M(\mathfrak{C})$ of the topologically \'etale module $M$ is topologically \'etale.

One easily sees that the set
    $C^{int}(N_0,  \Hom_A^{\mathfrak{C}ont} (M(\mathfrak{C}),M)) $ of integrable maps
is an $A$-module.
The $A$-linear map
\begin{equation*}
    \int_{N_0} . d\res : C^{int}(N_0,  \Hom_A^{\mathfrak{C}ont} (M(\mathfrak{C}),M)) \longrightarrow  \Hom_A^{\mathfrak{C}ont} (M(\mathfrak{C}),M)
\end{equation*}
will be called the $(s,\res, \mathfrak{C})$-integral.

\bigskip  We give now a general integrability criterion.
\begin{proposition}\label{integral}
A   map $F : N_0 \longrightarrow  \Hom_A^{\mathfrak{C}ont} (M(\mathfrak{C}),M)$ is   ($s$, $\res$, $\mathfrak{C}$)-integrable  if, for any compact subset $C
  \in\mathfrak{C}$ and any open $A$-submodule $\mathcal{M} \subseteq M$, there
  exists an integer $k_{C,\mathcal{M}} \geq 0$ such that
\begin{equation*}
    (F(u) - F(uv)) \circ \res(1_{uN_{k+1}}) \in E(C,\mathcal{M})
  \qquad\text{for any $k \geq k_{C,\mathcal{M}}$, $u \in N_0$, and $v \in
    N_k$}.
\end{equation*}

\end{proposition}
\begin{proof}
Let $J(N_0/N_k)$ and $J'(N_0/N_k)$, for $k \geq 0$, be two choices of sets of representatives. We put
\begin{equation*}
    s_k(F) := \sum_{u \in J(N_0/N_k)} F(u) \circ \res(1_{uN_k}) \quad\text{and}\quad s'_k(F) := \sum_{u' \in J'(N_0/N_k)} F(u') \circ \res(1_{u'N_k}) \ .
\end{equation*}
Since $ \Hom_A^{\mathfrak{C}ont} (M(\mathfrak{C}),M)$ is Hausdorff and complete it suffices to show that, given any neighborhood of zero $E(C,\mathcal{M})$, there exists an integer $k_0 \geq 0$ such that
\begin{equation*}
    s_k(F) - s_{k+1}(F),\, s_k(F) - s'_k(F) \in E(C,\mathcal{M}) \qquad\text{for any $k \geq k_0$}.
\end{equation*}
For $u \in J(N_0/N_{k+1})$ let $\bar{u} \in J(N_0/N_k)$ and $u' \in J'(N_0/N_{k+1})$ be the unique elements such that $uN_k = \bar{u} N_k$ and $uN_{k+1} = u'N_{k+1}$, respectively. Then
\begin{equation*}
    s_k(F) = \sum_{u \in J(N_0/N_{k+1})} F(\bar{u}) \circ \res(1_{uN_{k+1}})
\end{equation*}
and hence
\begin{equation}\label{Delta}
    s_k(F) - s_{k+1}(F) = \sum_{u \in J(N_0/N_{k+1})} (F(u(u^{-1}\bar{u})) - F(u)) \circ \res(1_{uN_{k+1}}) \ .
\end{equation}
Since $u^{-1}\bar{u} \in N_k$ it follows from our assumption  that the right hand side lies in $E(C,\mathcal{M})$ for $k \geq k_{C,\mathcal{M}}$. Similarly
\begin{equation*}
    s_{k+1}(F) - s'_{k+1}(F) = \sum_{u \in J(N_0/N_{k+1})} (F(u) - F(u(u^{-1}u'))) \circ \res(1_{uN_{k+1}}) \ ;
\end{equation*}
again, as $u^{-1}u' \in N_{k+1} \subseteq N_k$, the right hand sum is contained in $E(C,\mathcal{M})$ for $k \geq k_{C,\mathcal{M}}$.
\end{proof}

\subsection{Integrability criterion for $\alpha$}\label{alpha0}

Let $U_g \subseteq N_0$ be the compact open subset such that $U_g w_0 P/P = g^{-1} \mathcal{C}_0 \cap \mathcal{C}_0$. This intersection is nonempty if and only if $g \in N_0 \overline{P} N_0$, which we therefore assume in the following. We consider the map
\begin{align*}
    \alpha_{g,0} : N_0 & \longrightarrow \End_A^{cont}(M) \\
    u & \longmapsto
     \begin{cases}
     \Res(1_{N_0}) \circ \alpha_g(x_u) \circ \Res(1_{N_0}) & \text{if $u \in U_g$}, \\
     0 & \text{otherwise}
     \end{cases}
\end{align*}
(where we identify $\End_A^{cont}(M)$ with its image in $E^{cont}$ under the natural embedding \eqref{defs}  using that $\Res(1_{N_0})=\sigma_{0} \circ \ev _{0}$). Restricting $\alpha_{g,0}(u)\in \End_{A}^{cont}(M)$ (defined in section \ref{alpha0}) to $M(\mathfrak{C})$ for any $u\in N_0$
we may view $\alpha_{g,0}$ as a map from $N_0$ to
$\End_A^{cont}(M(\mathfrak{C}))$ since $M(\mathfrak{C})$ is an \'etale
$A[P_+]$-submodule of $M$. However, as we do not assume $M(\mathfrak{C})$ to
be complete, it will be more convenient for the purpose of integration to
regard $\alpha_{g,0}$ as a map into $\Hom_A^{\mathfrak{C}ont} (M(\mathfrak{C}),M)$.
We want to establish a criterion for the $(s,\res, \mathfrak{C})$-integrability of the map $\alpha_{g,0}$.

By the argument in the proof of lemma \ref{covering} (applied to $V = g^{-1}\mathcal{C}_0 \cap \mathcal{C}_0$) we may choose an integer $k_g^{(0)} \geq 0$ such that, for any $k \geq k_g^{(0)}$, we have $U_g N_k \subseteq U_g$ and
\begin{equation}\label{f:N0}
    \alpha(g,x_u) uN_{k} \subseteq  gU_{g} \qquad\text{for any $u \in U_g$}.
\end{equation}

\begin{lemma}\label{alpha-int}
 For $u \in U_g$ and $k \geq k_g^{(0)}$ we have
\begin{equation*}
    \alpha_{g,0}(u) \circ \res(1_{uN_k}) = \alpha(g,x_u) \circ \Res(1_{uN_k}) \ .
\end{equation*}
 \end{lemma}
\begin{proof}
 Using the $P$-equivariance of $\Res$ we have
\begin{align*}
    \alpha(g,x_u) \circ \Res(1_{uN_k}) & = \Res(1_{\alpha(g,x_u).uN_k}) \circ \alpha(g,x_u) \circ \Res(1_{uN_k}) \\
    & = \Res(1_{N_0}) \circ  \Res(1_{\alpha(g,x_u).uN_k}) \circ \alpha(g,x_u) \circ \Res(1_{uN_k}) \\
    & = \Res(1_{N_0}) \circ \alpha(g,x_u) \circ \Res(1_{N_0}) \circ \Res(1_{uN_k}) \\
    & = \alpha_{g,0}(u) \circ \res(1_{uN_k})
\end{align*}
where the second identity follows from \eqref{f:N0}.
\end{proof}
For $u \in U_g$ and $k \geq k_g^{(0)}$ we put
\begin{equation}\label{Hgk}
  \mathcal H_{g, J(N_0/N_k) }:=\sum_{u \in U_g \cap J(N_0/N_k)} \alpha(g,x_u) \circ \Res(1_{uN_k}) \ .
\end{equation}
By Lemma \ref{alpha-int}, each summand on the right hand side belongs to    $\End_A(M(\mathfrak{C}))$.   If $\alpha_{g,0}$ is $(s,\res, \mathfrak{C})$-integrable,  the limit
\begin{equation}\label{Hgl}
\mathcal H_{g}:= \lim_{k \geq k_g^{(0)}, k \rightarrow \infty} \mathcal H_{g, J(N_0/N_k) }
\end{equation}
 exists in $\Hom_A^{\mathfrak{C}ont} (M(\mathfrak{C}),M)$ and  is equal to the $(s,\res, \mathfrak{C})$-integral of $\alpha_{g,0}$
   \begin{equation}\label{Hg}
 \int_{N_0} \alpha_{g,0} d\res =\mathcal H_{g} \ .
\end{equation}

We investigate the integrability criterion of Prop. \ref{integral} for the function $\alpha_{g,0}$. We have to consider the elements
 \begin{equation}\label{Deltag}
   \Delta_g(u,k,v)   := (\alpha_{g,0}(u) - \alpha_{g,0}(uv)) \circ \res(1_{uN_{k+1}})  \ ,
\end{equation}
for $u \in U_g$, $k \geq k_g^{(0)}$, and $v \in N_k$. By Lemma \ref{alpha-int}, we have
\begin{align*}
    \Delta_g(u,k,v)   & = (\alpha_{g,0}(u) \circ \res(1_{uN_k}) - \alpha_{g,0}(uv) \circ \res(1_{uvN_k})) \circ \res(1_{uN_{k+1}}) \\
    & = (\alpha(g,x_u) \circ \Res(1_{uN_k}) - \alpha(g,x_{uv}) \circ \Res(1_{uvN_k})) \circ \Res(1_{uN_{k+1}}) \\
    & = (\alpha(g,x_u) - \alpha(g,x_{uv})) \circ \Res(1_{uN_{k+1}}) \\
    & = (\alpha(g,x_u) - \alpha(g,x_{uv})) \circ u \circ \Res(1_{N_{k+1}}) \circ u^{-1}
\end{align*}
Recall that $N_{g}\subset N$ is  the subset such that $N_g w_0 P/P = g^{-1} \mathcal{C} \cap \mathcal{C} $.

\begin{lemma}\label{u-uv}
For $u \in N_g$ and $v \in N$ such that $uv \in N_g$ we have:
\begin{itemize}
  \item[i.] $v \in N_{\bar{n}(g,u)}$;
  \item[ii.] $\alpha(g,x_{uv}) = \alpha(g,x_u) u \alpha(\bar{n}(g,u),x_v) u^{-1}$.
\end{itemize}
\end{lemma}
\begin{proof}
i. Because of $gu = \alpha(g,x_u)u\bar{n}(g,u)$ we have
\begin{equation*}
    \alpha(g,x_u)u\bar{n}(g,u)v = guv \in \alpha(g,x_{uv})uv \overline{N}
\end{equation*}
and hence
\begin{equation*}
    \bar{n}(g,u)vw_0P = u^{-1}\alpha(g,x_u)^{-1} \alpha(g,x_{uv}) uv w_0 P \in Pw_0 P  \ .
\end{equation*}
ii. By i. the equation $\bar{n}(g,u)vw_0N = \alpha (\bar{n} (g,u), x_v)vw_0N$ holds. Hence
\begin{equation*}
    guvw_0N = \alpha(g,x_u)u \bar{n}(g,u)vw_0N = \alpha(g,x_u)u \alpha (\bar{n} (g,u), x_v)vw_0N
\end{equation*}
and therefore $\alpha(g,x_{uv})uv = \alpha(g,x_u)u \alpha (\bar{n} (g,u), x_v)v$.
\end{proof}

We deduce that
\begin{equation*}
    \Delta_g(u,k,v) = \alpha(g,x_u) \circ u \circ (1 - \alpha(\bar{n}(g,u),v)) \circ \Res(1_{N_{k+1}}) \circ u^{-1} \ .
\end{equation*}
For $u \in U_g$ we have
\begin{equation*}
    \alpha(g,x_u)u =  n (g,u)t(g,u) \qquad\text{with $ n (g,u) \in N_0$ and $t(g,u) \in L$},
\end{equation*}
hence   $ \Delta_g(u,k,v) $ is contained in
\begin{equation*}
    N_0 t(g,u)  (1 -  \alpha(\bar{n}(g,u),v) ) \circ \Res(1_{N_{k+1}}) \circ N_0 \ .
\end{equation*}
Since $t(g,U_g) \subseteq L$ and $\bar{n}(g,U_g) \subseteq \overline{N}$ are compact subsets we may choose a $k_g^{(1)} \geq k_g^{(0)}$ such that
\begin{equation}\label{kg1}
    \Lambda_g := t(g,U_g)s^{k_g^{(1)}} \subseteq L_+ \qquad\text{and}\qquad \bar{n}(g,U_g) \subseteq \overline{N}_{-k_g^{(1)}} \ .
\end{equation}
 Writing $t(g,u)= s^{k-k_g^{(1)}} t(g,u)s^{k_g^{(1)}}s^{-k} \subseteq s^{k-k_g^{(1)}} \Lambda_{g} s^{-k}$ we then obtain that $\Delta_g(u,k,v)$ is contained in
\begin{equation*}
    N_0 s^{k-k_g^{(1)}} \big(1 - s^{-(k-k_g^{(1)})}t(g,u)\alpha(\bar{n}(g,u),v)t(g,u)^{-1}s^{k-k_g^{(1)}} \big)\Lambda_g s^{-k} \circ \Res(1_{N_{k+1}}) \circ N_0 \
\end{equation*}
when $k\geq k_g^{(1)}$.  We define
\begin{equation*}
    P_{g,k} := \{ s^{-(k-k_g^{(1)})}t(g,u)\alpha(\bar{n}(g,u),v)t(g,u)^{-1}s^{k-k_g^{(1)}} : u \in U_g, v \in N_k \}
\end{equation*}
which is a subset of $P$.

\begin{lemma}\label{compact}
For any compact open subgroup $P_1 \subseteq P_0$ there is an integer $k_g^{(2)}(P_1) \geq k_g^{(1)}$ such that
\begin{equation*}
    P_{g,k} \subseteq P_1 \qquad\text{for any $k \geq k_g^{(2)}(P_1)$}.
\end{equation*}
\end{lemma}
\begin{proof}
By compactness of the set $\{s^{k_g^{(1)}}t(g,u)u' : u \in U_g, u' \in N_0\}$ we find an open subgroup $P_2 \subseteq P_1$ such that
\begin{equation*}
    s^{k_g^{(1)}}t(g,u)u' P_2 u'^{-1} t(g,u)^{-1} s^{-k_g^{(1)}} \subseteq P_1 \qquad\text{for any $u \in U_g$ and $u' \in N_0$}.
\end{equation*}
Hence we may replace in the assertion $P_1$ by $P_2$ and $P_{g,k}$ by
\begin{equation*}
    P'_{g,k} := \{ (s^{-k}v^{-1}s^k) s^{-k} \alpha(\bar{n}(g,u),v) s^k (s^{-k}v s^k) : u \in U_g, v \in N_k \}.
\end{equation*}
If we multiply the identity
\begin{equation*}
    \bar{n}(g,u)v \overline{N} = \alpha(\bar{n}(g,u),v) v \overline{N}
\end{equation*}
from the left by $v^{-1}$ and conjugate by $s^{-k}$ then we obtain
\begin{align*}
    (s^{-k}v^{-1}s^k)s^{-k}\alpha(\bar{n}(g,u),v)s^k (s^{-k}v s^k) \overline{N} & = (s^{-k}v^{-1}s^k) s^{-k} \bar{n}(g,u) s^k (s^{-k}vs^k) \overline{N} \\
     & \subseteq (s^{-k}v^{-1}s^k) \overline{N}_{k - k_g^{(1)}} (s^{-k}vs^k)  \overline{N}
\end{align*}
and hence
\begin{equation*}
    P'_{g,k} \times \overline{N} \subseteq \bigcup_{u' \in N_0} u' \overline{N}_{k - k_g^{(1)}} u'^{-1}  \overline{N} \ .
\end{equation*}
Since $N_0$ is compact and $P_2 \times \overline{N}$ is an open neighborhood of the unit element in $G$ we find an integer $k_g^{(2)}(P_2) \geq k_g^{(1)}$ such that
\begin{equation*}
    \bigcup_{u' \in N_0} u' \overline{N}_{k - k_g^{(1)}} u'^{-1}  \overline{N} \subseteq P_2 \times \overline{N} \qquad\text{for any $k \geq k_g^{(2)}(P_2)$}.
\end{equation*}
It follows that $P'_{g,k} \subseteq P_2$ for any $k \geq k_g^{(2)}(P_2)$.
\end{proof}

This result says that for $k \geq k_g^{(2)}(P_1)$ we have
\begin{equation*}
    \{\Delta_g(u,k,v) : u \in U_g, v \in N_k\} \subseteq N_0 s^{k-k_g^{(1)}} (1- P_1) \Lambda_g s^{-k} \circ \Res(1_{N_{k+1}}) \circ N_0 \ .
\end{equation*}
Using lemma \ref{equ} we finally observe that
\begin{equation*}
    s^{-(k+1)} \circ \Res(1_{N_{k+1}}) = \Res(1_{N_0}) \circ s^{-(k+1)} = \sigma_0 \circ \ev_0 \circ s^{-(k+1)} = \sigma_0 \circ \psi^{k+1} \circ \ev_0
\end{equation*}
is the image in $\End_{A}^{cont}(M^{P})$ of $\psi^{k+1} \in \End_A^{cont}(M)$. We therefore conclude that,  for any compact open subgroup $P_1 \subseteq P_0$ and for $k \geq k_g^{(2)}(P_1)$, we have
\begin{equation}\label{delta}
    \{\Delta_g(u,k,v) : u \in U_g, v \in N_k\} \subseteq N_0 s^{k-k_{g}^{(1)}} (1- P_1) \Lambda_g s \psi^{k+1} N_0
\end{equation}
in  $\End_A^{cont} (M )$. This leads to an integrability criterion for $\alpha_{g,0}$, which depends only on $(s,M, \mathfrak C)$.

 \begin{proposition}\label{criterion}
 We suppose that $(s,M, \mathfrak C)$ satisfies:

\begin{itemize}
\item[$\mathfrak{C}(5)$]    For any $C\in\mathfrak{C}$ the compact subset
  $ \psi(C)\subseteq M$ also lies in $\mathfrak{C}$.

\item[$\mathfrak{T}(1)$]   For any special compact
subset $C\in\mathfrak{C}$ such that $C=N_{0}C$, any open $A[N_0]$-submodule $\mathcal{M}$ of $M$, and any
compact subset $C_+ \subseteq L_+$ there exists a compact open subgroup $P_{1}=P_1(C,\mathcal{M},C_+)
\subseteq P_0$ and an integer $k(C,\mathcal{M},C_+) \geq 0$ such that
\begin{equation}
s^{k }(1-P_1)C_+ \psi^{k} \subseteq E(C,\mathcal M) \qquad\text{for any $k \geq k(C,\mathcal{M},C_+)$} \ .
\end{equation}
\end{itemize}
Then the map $\alpha_{g,0}\colon N_0\to
  \Hom_A^{\mathfrak{C}ont} (M(\mathfrak{C}),M)$ is   ($s$, $\res$, $\mathfrak{C}$)-integrable
 for all  $g\in N_{0}\overline P N_{0}$.

\end{proposition}

\begin{proof}
By the general integrability criterion of Prop. \ref{integral}, the map $\alpha_{g,0}$  is integrable if for  any $(C,\mathcal{M})$ as  above, there exists $k_{C,\mathcal{M},g}\geq 0$ such that
\begin{equation} \label{delta'}  \Delta_g(u,k,v)\in E(C,\mathcal M)  \qquad\text{for any $k \geq k_{C,\mathcal{M},g}$, $u \in U_{g}$, and $v \in N_k$} .
\end{equation}
By (\ref{delta}),  this is   true if $k_{C,\mathcal{M},g} \geq k_g^{(2)}(P_1)$ and
\begin{align} \label{delta''}  s^{k-k_{g}^{(1)}} (1- P_1) \Lambda_g s \psi^{k+1}  (C) \subset   \mathcal M   \ ,
 \end{align}
\ because $N_{0}\mathcal M  =\mathcal M$ and $N_{0}C=C$.

We note that the  set $C _{+}=\Lambda_{g} s $ is contained in $L_{+}$ by (\ref{kg1}) and is compact,  that
  the set $C' = \psi^{k_{g}^{(1)}+1}(C)\subset M$ is compact and $N_{0}C'= C'$ because
   the map $\psi$ is continuous  and $N_{0}\psi(C)=\psi (sN_{0}s^{-1}C)=\psi(C)$, and that   (\ref{delta''}) is equivalent to
 $$
 s^{k-k_{g}^{(1)}} (1- P_1) C _{+}  \psi^{k-k_{g}^{(1)}}   \subset    E(C',  \mathcal M ) \ .
 $$
By our hypothesis, there exists an open subgroup  $P_{1}\subset P_{0}$ such that this inclusion is satisfied when $k\geq k_{g}^{(1)}+k ( C', \mathcal M , C_{+})$.
For
 \begin{equation}\label{k}
 k_{C,\mathcal{M},g} :=  \max ( k_{g}^{(1)}+k ( C', \mathcal M , C_{+}), k_{g}^{(2)}(P_{1})).
\end{equation}
  (\ref{delta'}) is satisfied. By construction, $P_{1}$ depends on $ \psi^{k_{g}^{(1)}+1}(C),  \mathcal M, \Lambda_{g} s$,  hence only on $C,  \mathcal M, g$.
  \end{proof}

Later, under the assumptions of Prop. \ref{criterion}, we will use the argument in the previous proof in the  following slightly more general form:
for $C,\mathcal{M},C_+ $ as in the proposition and an  integer $k' \geq 0$ we have
\begin{equation}\label{P1}
s^{k-k' }(1-P_1(\psi^{k'}(C),\mathcal{M},C_+ ) )C_+ \psi^{k} \subseteq E(C,\mathcal M)\
\end{equation}
 for any $k \geq k'+k(\psi^{k'}(C),\mathcal{M},C_+)$.

 \subsection{Extension of $\Res$}

\begin{proposition} \label{corspecial}
Suppose that $(s,M, \mathfrak C)$ satisfies  the assumptions of Prop. \ref{criterion} and that the $(s, \res, \mathfrak C)$-integral $\mathcal H_{g} $ of $\alpha_{g,0} $  is contained in $\End_{A}(M(\mathfrak{C}))$ for all $g \in N_0\overline{P}N_0$. In addition we assume that:
\begin{itemize}
\item[$\mathfrak{C}(6)$]    For any $C\in\mathfrak{C}$ the compact subset
  $ \varphi(C)\subseteq M$ also lies in $\mathfrak{C}$.
\item[$\mathfrak{T}(2)$] Given
  a  set $J(N_0/N_k)\subset N_{0}$ of representatives for cosets in $N_0/N_k$, for $k\geq 1$, for any $x\in M(\mathfrak{C})$ and $g\in N_0\overline{P}N_0$ there
  exists a compact $A$-submodule $C_{x,g}\in\mathfrak{C}$ and a positive integer $k_{x,g}$ such that
  $\mathcal{H}_{g,J(N_0/N_k)}(x)\subseteq C_{x,g}$ for any $k\geq  k_{x,g}$.
 \end{itemize}
Then  the $\mathcal H_{g} $ satisfy the    relations H1, H2, H3 of Prop. \ref{multiplicative}.
  \end{proposition}

\begin{remark}\label{resC}  The properties $\mathfrak C(3), \mathfrak C(5),\mathfrak C(6)$ imply that  for any $u\in N_{0}, k\geq 1$,  and $C\in \mathfrak C$ also $\res(1_{uN_{k}}) (C)$ lies in $\mathfrak C$. Indeed, $ \res(1_{uN_{k}})=u\circ \varphi^{k}\circ \psi^{k}\circ u^{-1}$.
\end{remark}

We prove now H1 and H3, which do not use the last assumption. The proof of  ii.    is postponed to the next subsection.

 \begin{proof}
The proof of H1 contains  three steps.

{\sl Step 1:} In this step we establish the relation
\begin{equation*}
    \res (1_{\mathcal{C}_0 \cap g\mathcal V  }) \circ  {\mathcal H}_{g} \circ \res(1_{\mathcal{V}})  = {\mathcal H}_{g} \circ \res (1_{\mathcal V }) .
\end{equation*}
 By additivity we may assume that $\mathcal V =vN_{r}w_{0}P/P$ for $v\in N_{0}$ and  an integer $r$ as large as we wish.

 It suffices, by  \eqref{Hg} and Remark \ref{resC}, to verify that
\begin{equation*}
    \res (1_{\mathcal{C}_0 \cap g\mathcal V  }) \circ \alpha(g,x_u) \circ \Res(1_{uN_k w_0 P/P}) \circ \res(1_{\mathcal V }) = \alpha(g,x_u) \circ \Res(1_{uN_k w_0 P/P}) \circ \res(1_{\mathcal V })
\end{equation*}
for any $u \in N_0$ such that $x_u \in g^{-1} \mathcal{C}_0 \cap \mathcal{C}_0$ and any $k \geq k_g^{(0)}$. By enlarging $k_g^{(0)}$ we have that each set $uN_k w_0 P/P$ either is contained in $\mathcal V $ or is disjoint from $\mathcal V $. This reduces us to verifying
\begin{equation*}
    \res (1_{\mathcal{C}_0 \cap g\mathcal V  }) \circ \alpha(g,x_u) \circ \Res(1_{uN_k w_0 P/P}) = \alpha(g,x_u) \circ \Res(1_{uN_k w_0 P/P})
\end{equation*}
whenever $x_u \in g^{-1} \mathcal{C}_0 \cap \mathcal V $. By the argument in the proof of lemma \ref{covering} (applied to $V := g^{-1} \mathcal{C}_0 \cap \mathcal V $) we may assume, after enlarging $k_g^{(0)}$ further, that
\begin{equation*}
    \alpha(g,x_u) uN_k w_0 P/P \subseteq \mathcal{C}_0 \cap g \mathcal V  \qquad\text{for any $x_u \in g^{-1} \mathcal{C}_0 \cap \mathcal V $}.
\end{equation*}
Using the $P$-equivariance of $\Res$ we then compute
\begin{align*}
    \alpha(g,x_u) \circ \Res(1_{uN_k w_0 P/P}) = & \Res(1_{\alpha(g,x_u)uN_k w_0 P/P}) \circ \alpha(g,x_u) \\
    & = \res (1_{\mathcal{C}_0 \cap g\mathcal V  }) \circ \Res(1_{\alpha(g,x_u)uN_k w_0 P/P}) \circ \alpha(g,x_u) \\
    & = \res (1_{\mathcal{C}_0 \cap g\mathcal V  }) \circ \alpha(g,x_u) \circ \Res(1_{uN_k w_0 P/P}) \ .
\end{align*}

{\sl Step 2:}  By applying Step 1 to  $  \mathcal{V}$ and to $  \mathcal{C}_0 \setminus \mathcal{V}$ we obtain
\begin{align*}
    \res (1_{\mathcal{C}_0 \cap g\mathcal V }) \circ \mathcal{H}_g & = \res (1_{\mathcal{C}_0 \cap g\mathcal V }) \circ \mathcal{H}_g \circ \res(1_{\mathcal{C}_0}) \\
    & = \res (1_{\mathcal{C}_0 \cap g\mathcal V }) \circ \mathcal{H}_g \circ \res(1_{\mathcal{V}}) + \res (1_{\mathcal{C}_0 \cap g\mathcal V }) \circ \mathcal{H}_g \circ \res(1_{\mathcal{C}_0 \setminus \mathcal{V}}) \\
    & = \mathcal{H}_g \circ \res(1_{\mathcal{V}}) + \res (1_{\mathcal{C}_0 \cap g\mathcal V }) \circ \res (1_{\mathcal{C}_0 \cap g(\mathcal{C}_0 \setminus \mathcal V) }) \circ \mathcal{H}_g \circ \res(1_{\mathcal{C}_0 \setminus \mathcal{V}}) \\
    & = \mathcal{H}_g \circ \res(1_{\mathcal{V}}) \ .
\end{align*}
For $\mathcal{V} = \mathcal{C}_0$ we, in particular, get
\begin{equation}\label{hg}
    \res(1_{\mathcal{C}_0 \cap g\mathcal{C}_0}) \circ \mathcal{H}_g = \mathcal{H}_g \ .
\end{equation}

{\sl Step 3:} Using the two identities in Step 2 we finally compute
\begin{align*}
    \mathcal{H}_g \circ \res(1_{g^{-1}\mathcal{V} \cap \mathcal{C}_0}) & = \res(1_{\mathcal{C}_0 \cap \mathcal{V} \cap g\mathcal{C}_0}) \circ \mathcal{H}_g \\
    & = \res(1_{\mathcal{V}}) \circ \res(1_{\mathcal{C}_0 \cap g\mathcal{C}_0}) \circ \mathcal{H}_g \\
    & = \res(1_{\mathcal{V}}) \circ \mathcal{H}_g \ .
\end{align*}

H3. For $b \in P\cap N_{0}\overline P N_{0}$ we have
\begin{equation*}
    \alpha_{b,0} = \ \text{constant map on $N_0$ with value $\res(1_{\mathcal{C}_0}) \circ b \circ \res(1_{\mathcal{C}_0})$}
\end{equation*}
and hence
\begin{equation*}
    \mathcal{H}_b = \res(1_{\mathcal{C}_0}) \circ b \circ \res(1_{\mathcal{C}_0}) = b \circ \res (1_{b^{-1}\mathcal{C}_0 \cap \mathcal{C}_0}) \ .
\end{equation*}

\end{proof}

\subsection{Proof of the product formula}

We  invoke now the full set of assumptions of Prop. \ref{corspecial} and we
   prove the product formula
   \begin{equation*}
{\mathcal H}_{g} \circ  {\mathcal H}_{h} = {\mathcal H}_{gh} \circ \res ( 1_{ h^{-1}\mathcal C_{0} \cap \mathcal C_{0}}) \ .
\end{equation*}
for $g,h\in N_{0}\overline P N_{0}$. This suffices by Prop. \ref{multiplicative}.

\bigskip

Let $k_{0} := \max(k_g^{(0)}, k_h^{(1)},  k_{gh}^{(0)})+1$ and let $k\geq  k_{0}$.

As $k\geq k_h^{(0)}$ (because $k_h^{(1)}\geq k_h^{(0)}$ (\ref{kg1})), the set $U_{h}$ is a disjoint union of cosets $uN_{k}$. We choose a set $J(N_{0}/N_{k})\subset N_{0}$ of representatives of the cosets in $N_{0}/N_{k}$ and  for each $u\in J(N_{0}/N_{k})\cap U_{h}$   a set $J_{u}(N_{0}/N_{k-k_{0}})\subset N_{0}$ of representatives of the cosets  in $N_{0}/N_{k-k_{0}}$ with $n(g,u)\in J_{u}(N_{0}/N_{k-k_{0}})$  (see \eqref{f:alpha}).

 \bigskip We   write
 ${\mathcal H}_{g} \circ  {\mathcal H}_{h} - {\mathcal
      H}_{gh} \circ   \res ( 1_{ h^{-1}\mathcal C_{0} \cap \mathcal C_{0}})$ as the sum over $u\in J(N_0/N_k)\cap U_{h}$ of
\begin{align}\label{so1}
    (\mathcal{H}_{g}\circ\mathcal{H}_{h}-\mathcal{H}_{gh}\circ
   \Res(1_{  U_{h}}))\circ\Res(1_{uN_k} )=a_{k,u}+b_{k,u}+c_{k,u}\ ,
  \end{align}
 where
   \begin{align*}
    a_{k,u}:= &(\mathcal{H}_g\circ\mathcal{H}_h-\mathcal{H}_{g,J_u(N_0/N_{k-k_0})}\circ\mathcal{H}_{h,J(N_0/N_k)})\circ\Res(1_{uN_k} )\\
 b_{k,u}:=&(\mathcal{H}_{g,J_u(N_0/N_{k-k_0})}\circ\mathcal{H}_{h,J(N_0/N_k)}-\mathcal{H}_{gh,J(N_0/N_k)})\circ \Res(1_{  U_{h}})\circ\Res(1_{uN_k} )\\
c_{k,u}:=&(\mathcal{H}_{gh,J(N_0/N_k)}-\mathcal{H}_{gh})\circ \Res(1_{  U_{h}})\circ\Res(1_{uN_k} ).
\end{align*}

The product formula follows from the claim that $b_{k,u}=0$  and that  for   an arbitrary compact subset  $C\in \mathfrak C $ such that $N_{0} C=C$,  and  an arbitrary open $A[N_{0}]$-module  $\mathcal M \subset M$,
 $a_{k,u}$ and $c_{k,u}$ lies in   $E(C,\mathcal{M})$  when $k$ is very large, independently of $u$.

\bigskip
The claim results from the following three  propositions.

 \bigskip  Because  $(s,M,\mathfrak C)$ satisfies Prop. \ref{criterion}, we associate to $(C, \mathcal M, g)$  the integer $k_{C, \mathcal M, g }$   defined in (\ref{k}) which is independent of the choice of the $J(N_{0}/N_{k})$.   For the sake of simplicity, we write
\begin{equation}\label{Hgs} \mathcal{H}_{g}^{(k)} := \mathcal{H}_{g,J(N_0/N_k)} \ ,  \  s_{g}^{(k)}:= \mathcal{H}_{g}^{(k+1)}-\mathcal{H}_{g}^{(k)} .
\end{equation}
By \eqref{Delta}, we have, for $k\geq k_{g}^{(0)}$,
 $$ s_{g}^{(k)} = \sum _{u\in U_{g}\cap J(N_{0} /N_{k+1})} \Delta_{g}(u,k,v_{u})$$
for some $v_{u}\in N_{k}$. It follows from \eqref{delta} that, for any given compact open subgroup $P_{1}\subset P_{0}$, we have
\begin{equation} \label{sgk}
 s_{g}^{(k)} \in \  <  N_0 s^{k-k_{g}^{(1)}} (1- P_1) \Lambda_g s \psi^{k+1} N_0>_{A} \quad \text{ for $k \geq k_{g}^{(2)}(P_{1})$ ,}
\end{equation}
where we use the notation $<X>_{A}$ for the $A$-submodule in $\End_{A}(M)$ generated by $X$.
We deduce from the proof
 of Prop. \ref{criterion}, that
 $ s_{g}^{(k)} \in E(C, \mathcal M)$  for any $k\geq k_{C, \mathcal M, g }$.

 \begin{proposition}\label{extendcv}
 $( \mathcal{H}_g  -   \mathcal{H}_{g,J(N_0/N_k)}) \circ \Res(1_{uN_k} )\in E(C, \mathcal M)
$
for any $k\geq k_{C, \mathcal M, g }$.
  \end{proposition}

 \begin{proof}
When  $k\geq 0$,  $k_{2} \geq   \max (k-1, k_{g}^{(0)})$,  $u'\in U_{g}, v\in N_{k}$ we have that
 $ \Delta_{g}(u',k_{2},v )  \circ \Res(1_{uN_k} )$ is equal either to $\Delta_{g}(u',k_{2},v) $ or to  $0$.
 If follows that
  \begin{equation*}
s_{g }^{(k_{2})} \circ \Res(1_{uN_k} )
\subseteq
  E(C,\mathcal{M}) \qquad\text{ for any $k_{2} \geq \max (k-1,  k_{ C, \mathcal M,  g})$ and $k\geq 0$},  \end{equation*}

Now we fix $k \geq k_{ C, \mathcal M,  g}$. Note that $\Res(1_{uN_k} )(C)$ is contained in  $\mathfrak C$ by the stability of $\mathfrak C$ by  $\psi$, $\varphi$,
 and $ u^{\pm 1}$. Therefore the  sequence
 $(\mathcal{H}_{g }  ^{(k_{2})} \circ\Res(1_{uN_k} ))_{k_{2}}$
 converges to $\mathcal{H}_g \circ\Res(1_{uN_k} )$ in $\Hom_{A}^{\mathfrak Cont}(M(\mathfrak C),M)$. In particular, we have
\begin{equation*}
(\mathcal{H}_g-\mathcal{H}_{g }  ^{(k_{2})})\circ\Res(1_{uN_k} )\subseteq
  E(C,\mathcal{M}) \qquad\text{ for any $k_{2} \geq \max (k-1,  k_{ C, \mathcal M,  g})$ and $k\geq 0$}. \end{equation*}
  The statement follows by taking $k_2=k$.
\end{proof}

This establishes that $c_{k,u}$ lies in   $E(C,\mathcal{M})$  when $k \geq k_{C,\mathcal{M}, gh}$.

\bigskip Note that  the  proposition is true also for   any other system $J'(N_{0}/N_{k}) \subset N_{0}$ of representatives for the cosets in $N_{0} /N_{k} $  for the same integer $k_{C,\mathcal{M}, g}$.  We write  $\mathcal{H}_g'^{(k )}$  and $s_g'^{(k )}$ for the elements defined in \eqref{Hgs} for $J'(N_{0}/N_{k})$.

\begin{proposition}\label{complimit}
There exists an integer $k_{C, \mathcal M, g,h, k_{0}}\in \mathbb N$, independent of the choices of  $J(N_{0}/N_{k})$ and $J'(N_{0}/N_{k})$, such that:

\begin{itemize}
\item[i.]
$
\mathcal{H}_g'^{(k+1-k_0)}\circ\mathcal{H}_h^{(k+1)}-\mathcal{H}_g'^{(k-k_0)}\circ\mathcal{H}_h^{(k)}\in E(C, \mathcal M),
$ for all $k\geq k _{C, \mathcal M, g,h, k_{0}}$, and
the sequence $(\mathcal{H}_g'^{(k-k_0)}\circ\mathcal{H}_h^{(k)})$ converges to
$\mathcal{H}_g\circ\mathcal{H}_h$ in $\Hom_{A}^{\mathfrak Cont}(M(\mathfrak C), M)$.
\item[ii.]

$(\mathcal{H}_g\circ\mathcal{H}_h -   \mathcal{H}_g'^{{(k-k_{0} )}}\circ \mathcal{H}_{h}^{(k)}) \circ \Res(1_{uN_k} )\in E(C, \mathcal M),
 $ for all $k\geq k_{C, \mathcal M, g,h, k_{0}}$.
\end{itemize}
\end{proposition}

\begin{proof}
i.  To prove the first assertion, we write
\begin{align}\label{gh0}
\mathcal{H}_g'^{(k+1-k_0)}\circ\mathcal{H}_h^{(k+1)}-\mathcal{H}_g'^{(k-k_0)}\circ\mathcal{H}_h^{(k)}=\mathcal{H}_g'^{(k+1-k_0)}\circ s_h^{(k)}+s_g'^{(k-k_0)}\circ\mathcal{H}_h^{(k)}.
\end{align}
Note that, when $k\geq k_{g}^{(1)}$, the endomorphisms   $\mathcal{H}_g'^{(k)}$ and  $\mathcal{H}_g^{(k)}$ are contained in the $A$-module $<N_{0}s^{k-k_{g}^{(1)}}\Lambda_{g} \psi^{k}N_{0}>_{A}$,
   because
\begin{equation*}
\alpha(g,x_{u}) \circ \Res(1_{uN_{k}}) =n(g,u)t(g,u)u^{-1}u s^{k}\psi^{k}u^{-1}\subset N_{0}s^{k-k_{g}^{(1)}} \Lambda_{g}\psi^{k}N_{0} \qquad\text{ for $u\in U_{g}$}.
\end{equation*}
We consider any compact open subgroup $P_{1}\subset P_{0}$ and we assume $k \geq \max (k_{g}^{(2)}(P_{1})+k_{0}, k_{h}^{(2)}(P_{1}))$. With (\ref{sgk}) we obtain  that (\ref{gh0}) is contained in
\begin{align*}
& \ <N_0s^{k+1-k_0-k_g^{(1)}}\Lambda_g\psi^{k+1-k_0}N_0 s^{k-k_h^{(1)}}(1-P_1)\Lambda_hs\psi^{k+1}N_0>_{A} \\
+& \ <N_0 s^{k-k_0-k_g^{(1)}}(1-P_1)\Lambda_gs\psi^{k-k_0+1}N_0s^{k-k_h^{(1)}}\Lambda_h\psi^{k}N_0>_{A} .
\end{align*}
  Recalling that  $\psi^{a }(N_0 \varphi ^{a+b}(m))=   \psi^{a}(N_{0}) \varphi ^{b}(m) =  N_{0}\varphi ^{b}(m) $ for $a,b\in \mathbb N$ and $m\in M$,  we see that  this is contained in
\begin{align*}
& \  <N_0s^{k+1-k_0-k_g^{(1)}}\Lambda_gN_0s^{k_0-k_h^{(1)}-1}(1-P_1)\Lambda_hs\psi^{k+1}N_0>_{A}\\
+ & \ <N_0s^{k-k_0-k_g^{(1)}}(1-P_1)\Lambda_gN_0s^{k_0-k_h^{(1)}}\Lambda_h\psi^kN_0>_{A}.
\end{align*}
As  $k+1-k_0-k_g^{(1)} \geq k_{g}^{(2)}(P_{1}) +1 -k_g^{(1)} \geq 1$ and as $\Lambda_{g} \subset L_{+}$, we have
$$
N_{0}s^{k+1-k_0-k_g^{(1)}}\Lambda_g N_{0}\subset N_{0}s^{k+1-k_0-k_g^{(1)}}\Lambda_g \ ,
$$
and this is  contained in
\begin{align*}
& \ <N_0s^{k+1-k_0-k_g^{(1)}}\Lambda_g s^{k_0-k_h^{(1)}-1}(1-P_1)\Lambda_hs\psi^{k+1}N_0>_{A} \\
+ & \ <N_0s^{k-k_0-k_g^{(1)}}(1-P_1)\Lambda_g s^{k_0-k_h^{(1)}}\Lambda_h\psi^kN_0>_{A}.
\end{align*}
We assume, as we may,  that the compact open subgroup $P_{1}$ of $P_{0}$ satisfies $tP_1t^{-1}\subseteq P_1$ for all $t$ in the compact  set
$\Lambda_g s^{k_0-k_h^{(1)}-1} $ of $L_+$. Then we finally obtain that  (\ref{gh0} is contained in
 \begin{align*} & \ <N_0s^{k+1-k_0-k_g^{(1)}}(1-P_1)\Lambda_g s^{k_0-k_h^{(1)}}\Lambda_h\psi^{k+1}N_0>_{A} \\
+ & \ < N_0s^{k-k_0-k_g^{(1)}} (1-P_1)\Lambda_g s^{k_0-k_h^{(1)}}\Lambda_h\psi^kN_0>_{A}.
\end{align*}
This subset of $\End_{A}(M)$ is contained in $E(C, \mathcal M)$  when
\begin{align*} s^{k+1-k_0-k_g^{(1)}}(1-P_1)\Lambda_g s^{k_0-k_h^{(1)}}\Lambda_h\psi^{k+1}(C) \quad\text{and} \quad s^{k-k_0-k_g^{(1)}} (1-P_1)\Lambda_g s^{k_0-k_h^{(1)}}\Lambda_h\psi^k(C)
\end{align*} are contained in $E(C,\mathcal M)$ because $N_{0}C=C$ and $\mathcal M$ is an $A[N_{0}]$-module.
By \eqref{P1}, this is true when
 $P_{1}$ is contained in $P_{1}(\psi^{k_{0}+k_{g}^{(1)}}(C) , \mathcal M,\Lambda_g s^{k_0-k_h^{(1)}}\Lambda_h ) $ and   $k\geq k_{C, \mathcal M, g,h, k_{0}}$ where
\begin{equation}\label{k0}  k_{C, \mathcal M, g,h, k_{0}}:= \max(k_{g}^{(2)}(P_{1})+k_{0}, k_{h}^{(2)}(P_{1}),
k( \psi^{k_{0}+k_{g}^{(1)}}(C) , \mathcal M,\Lambda_g s^{k_0-k_h^{(1)}}\Lambda_h ) ).
\end{equation}
The first assertion of i. is proved.  We deduce the second assertion from the following claim and the last assumption of Prop. \ref{corspecial}:

\bigskip Let $(A_n)_{n \in \mathbb{N}}$ and $(B_n)_{n \in \mathbb{N}}$ be two convergent sequences in $\Hom_A^{\mathfrak{C}ont} (M (\mathfrak C), M)$ with limits $A$ and $B$, respectively;  assume that $(B_n)_{n \in \mathbb{N}}$ and $B$  are in $ \End_{A}(M (\mathfrak C)) $ and  that, for any $x\in  \mathfrak C$ there exists an $A$-submodule
$C \in \mathfrak C$ such that $B_{n}(x)\in C$ for any large $n$.
Then, if  the sequence $(A_n \circ B_n)_{n \in \mathbb{N}}$ is convergent,  its limit   is $A \circ B$.

Let $D$ be the limit of the sequence $(A_n \circ B_n)_n$. It suffices to show that, for any open $A$-submodule $\mathcal{M} \subseteq M$ and any element $x \in M (\mathfrak C)$ we have $(D - A \circ B)(x) \in \mathcal{M}$. We write
\begin{equation*}
     D- A \circ B = (D - A_n \circ B_n) - (A - A_n) \circ B_n - A \circ (B - B_n) \ .
\end{equation*}
Obviously $(D - A_n \circ B_n)(x) \in \mathcal{M}$ for large $n$. Secondly,    the elements  $B_n(x) $   for  any large $n$ are contained in some compact $A$-submodule $C \in \mathfrak C$, hence also $(B - B_n)(x)$. Moreover $A - A_n \in E(C,\mathcal{M})$ for large $n$. Hence $(A - A_n) \circ B_n (x) \in \mathcal{M}$ for large $n$. Finally, $A$ being $\mathfrak C$-continuous there is an open $A$-submodule $\mathcal{M}' \subseteq M$ such that $A(\mathcal{M}' \cap  C) \subseteq \mathcal{M}$. Furthermore $(B - B_n)(x) \in \mathcal{M}' \cap C$ for large $n$.    Hence $A \circ (B - B_n)(x) \in \mathcal{M}$ for large $n$.

\bigskip

ii.    This follows from the second assertion in i. together with remark \ref{resC}.
\end{proof}

  We have now proved that $a_{k,u}\in E(C, \mathcal M)$ when $k\geq k_{C, \mathcal M, g,h, k_{0}}$.

\begin{proposition}  \label{f:prod}  For
$u\in J(N_0/N_k)\cap U_{h}$, we have
 \begin{equation}\label{zero}
\mathcal{H}_{g,J_u(N_0/N_{k-k_0})}\circ\mathcal{H}_{h,J(N_0/N_{k})}\circ\Res(1_{uN_k} )=\mathcal{H}_{gh,J(N_0/N_k)}\circ\Res(1_{uN_k} ).
\end{equation}
 \end{proposition}
\begin{proof}
  The left side of (\ref{zero}) is
 \begin{align*}
     \sum_{v \in U_g \cap J_{u}(N_0/N_{k-k_0})}
    \alpha(g,x_v) \circ \Res(1_{vN_{k-k_0} }) \circ \, \alpha(h,x_u) \circ
    \Res(1_{uN_k })\ .
\end{align*}
The right side of (\ref{zero}) is $\alpha(gh,x_u) \circ
    \Res(1_{uN_k })$ if $u\in J(N_0/N_k)\cap U_{h}\cap U_{gh}$ and is $0$ if $u$ does not belong to $U_{gh}$.
We recall that
\begin{equation*}
    \alpha(h,x_u)u = n(h,u)t(h,u) \qquad\text{with $n(h,u) \in N_0$ and $t(h,u) \in L_+ s^{-k_h^{(1)}}$}.
\end{equation*}
It follows that
\begin{equation*}
    \alpha(h,x_u)uN_k w_0 P \subseteq n(h,u) N_{k-k_h^{(1)}} w_0 P \subset n(h,u) N_{k-k_ 0} w_0 P .
\end{equation*}
 We obtain
\begin{equation*}
    \Res(1_{vN_{k-k_0} }) \circ \alpha(h,x_u) \circ
    \Res(1_{uN_k })     =
    \begin{cases}
    \alpha(h,x_u) \circ  \Res(1_{uN_k }) & \text{if $vN_{k-k_0} = n(h,u)N_{k-k_0}$}, \\
    0 & \text{otherwise}.
    \end{cases}
\end{equation*}
We check now that $u\in U_{gh}\cap U_{h}$ if and only if   $n(h,u) \in U_{g}$.
Indeed $x_{u}= uw_{0} P/P$ belongs to $h^{-1} \mathcal{C}_0 \cap \mathcal{C}_0=U_{h}w_{0}P/P$,
\begin{equation*}
x_u \in (gh)^{-1} \mathcal{C}_0 \cap h^{-1} \mathcal{C}_0 \cap \mathcal{C}_0 \quad\text{ if and only if}\quad  hx_{u}\in g^{-1} \mathcal{C}_0 \cap   \mathcal{C}_0
\end{equation*}
and $hx_{u}=\alpha(h,x_{u})x_{u}=n(h,u)w_{0}P/P$. It follows that   $u\in U_{gh}\cap U_{h}$ if and only if   $n(h,u) \in U_{g}$.
As  $J_u(N_0/N_{k-k_0})$ contains
$n(h,u)$, we have $v= n(h,u)$ when $vN_{k-k_0} = n(h,u)N_{k-k_0}$.
We deduce that the left side of (\ref{zero}) is  $0$ when $u$ does not belong to $U_{gh} $ and otherwise  is equal to
$$ \alpha(g,hx_u)\circ \alpha(h,x_u) \circ  \Res(1_{uN_k })  =  \alpha(gh,x_u) \circ  \Res(1_{uN_k }) \ ,
$$
where the last equality  follows from the product formula for $\alpha$ (Lemma \ref{prod!}).
     \end{proof}

 We have proved that $b_{k,u}=0$, therefore ending the proof of the product formula.

\subsection{Reduction modulo $p^{n}$}\label{topAM}

We investigate now the situation that will appear
for generalized $(\varphi, \Gamma)$-modules $M$, where the reduction modulo a power of $p$ allows  to reduce to the simpler case where
$M$ is killed by a power of $p$. We will use later this section to get a special family $\mathfrak C_{s}$ in $M$ such that  the $(s,\res, \mathfrak C_{s})$-integrals $\mathcal H_{g}$ exist for all $g\in N_{0}\overline P N_{0}$ and  satisfy the relations H1, H2, H3 of Prop.  \ref{multiplicative}.

\bigskip We assume now that $(A,M)$  satisfies:
{\sl \begin{itemize}
 \item[a.]   $A$ is a commutative ring with the $p$-adic topology (the ideals $p^{n}A$  for $n\geq 1$ form a basis of neighborhoods of $0$) and   is Hausdorff.

  \item[b.]  $M$ is a linearly topological $A$-module with  a topology weaker than the $p$-adic topology (a neighborhood of $0$  contains some  $p^n M$) and     $M$ is a Hausdorff  and  topological  $A[P_{+}]$-module   as in section \ref{fc} (we do not suppose that $M$ is complete).
   \item[c.] The submodules $p^n M$, for  $n \geq 1$, are closed in $M$.
  \item[d.] $M$ is $p$-adically complete: the linear map $
    M  \to \varprojlim_{n\geq 1} (M/p^n M)
$ is  bijective.
  \end{itemize}}
 For all $n \geq 1$,  we equip $M/p^{n}M$ with  the quotient topology so that  the quotient map $p_{n}:M\to M/p^{n}M$ is continuous.
  The natural  homomorphism
 \begin{equation*}
  M \xrightarrow{\; \cong \;} \varprojlim_{n \geq 1} (M/p^n M)
  \end{equation*}
is an homeomorphism, and the natural  homomorphism
\begin{equation*}
       \End_A^{cont}(M) \xrightarrow{\; \cong \;} \varprojlim_{n \geq 1}\End_A^{cont}(M/p^n M)
\end{equation*}
is bijective. We  have:

\begin{itemize} \label{spintegr}

  \item[-] For a  subset $C$ of $M$,  let $\overline C $ be the  closure  of $C$. Then $\overline C= \varprojlim_{n \geq 1}\overline {p_{n}(C)}$ and if $C$ is closed, $C=\varprojlim_{n \geq 1}p_{n}(C)$.  If  $C$ is $p$-compact  (i.e. $p_{n}(C)$ are compact for all $n\geq 1$), then $C$ is compact, and conversely (\cite{TG} I.29  Cor. and I.64  Prop.8).

\item[-] An endomorphism $f$ of $M$ which is $p$-continuous (i.e. the endomorphism $f_{n}$  induced by $f$ on $M/p^{n}M$ is continuous for all $n\geq 1$) is continuous, and conversely.

\item[-] An action of a topological group $H$   on $M$  which is $p$-continuous (i.e. the induced action of $H$ on $M/p^{n}M$ is continuous for all $n\geq 1$)   is continuous, and conversely.

\item[-]  If the  $M/p^{n}M$ are complete  for all $n\geq 1$, then $M$ is complete.

\item[-] The image $\mathfrak C_{n}$  in $M/p^{n}M$, for all $n\geq 1$,  of a special family $\mathfrak C$ of compact subsets in $M$  such that, for all positive integers  $n$,
\begin{equation*}   p^{n} M\cap M(\mathfrak C)= p^{n}  M(\mathfrak C)
\end{equation*}
is a special family. In this case, one has $M(\mathfrak C_{n})=M(\mathfrak C )/p^{n}M(\mathfrak C )$.
 \item[-]    $M$ is  a topologically \'etale $A[P_{+}]$-module if and only if $M/p^{n}M$ is a topologically  \'etale $A[P_{+}]$-module, for all $n\geq 1$.
   If we replace ``topologically'' by ``algebraically'', this is the same proof  as for classical $(\varphi, \Gamma)$-modules (see subsection \ref{CPG}).
      The canonical inverse $\psi_{s}$ of the action $\varphi_{s}$ of $s$ is continuous if and only if it is $p$-continuous.

   \end{itemize}
We introduce now our setting which will be discussed in this section.

\bigskip {\sl We suppose that :
 \begin{itemize}
 \item[-] $M$ is  a   topologically \'etale $A[P_{+}]$-module, and  $M/p^{n}M$ is complete  for all $n\geq 1$.
 \item[-] We are given, for $n\geq 1$,  a special family $ \mathfrak C_{n }$
of compact  subsets in  $M_{n}=M/p^{n}M$ such that $\mathfrak C_{n }$ contains the image of $\mathfrak C_{n+1 }$ in $M_{n}$ for all $n\geq 1$.
    \end{itemize}}
  Let $ \mathfrak C$ be the set of compact subsets   $C\subset M$ such that $p_{n}(C)\in  \mathfrak C_{n }$ for all $n\geq 1$.

  \begin{lemma}\label{Limspe}
   $ \mathfrak C$ is a  special family in $M$ and $M( \mathfrak C) =\varprojlim_{n\geq 1} M(\mathfrak{C}_{n})$.
  \end{lemma}
   \begin{proof}
    $ \mathfrak C(1)$ It is obvious that a compact subset $C'$ of $C\in \mathfrak{C} $ is in  $\mathfrak{C} $ because $p_{n}$ is continuous and $p_{n}(C')$ is compact.

$ \mathfrak C(2)$  $p_{n}$ commutes with finite union hence $ \mathfrak{C} $ is stable by finite union.

$ \mathfrak C(3)$ $p_{n}$ commutes with the action of $N_{0}$ hence  $C \in \mathfrak{C} $  implies $N_{0}C \in \mathfrak{C} $.

$ \mathfrak C(4)$ By definition $x\in M(\mathfrak{C}) $ if and only if $p_{n}(x)\in M(\mathfrak{C_{n}}) $ for all $n>1$.   The compatibility of the   $\mathfrak C_{n }$ implies that the $M(\mathfrak{C}_{n})$ form a projective system. We deduce   $M(\mathfrak{C})= \varprojlim_{n\geq 1} M(\mathfrak{C}_{n})$.
 As the latter ones are topologically \'etale,   the topological $A[P_{+}]$-module $M(\mathfrak{C})$ is topologically \'etale
by Remark \ref{projlim}.
\end{proof}
We have the natural map
\begin{equation*}
\varprojlim_{n} \Hom_{A} (M(\mathfrak C_{n}),  M/p^{n}M)\to \Hom_{A} (\varprojlim_{n}M(\mathfrak C_{n}),  \varprojlim_{n}M/p^{n}M)= \Hom_{A} ( M(\mathfrak C ),   M ) \ .
\end{equation*}

\begin{lemma}  The above map induces a continuous map
\begin{equation}\label{f:pl}
\varprojlim_{n} \Hom_{A}^{\mathfrak C_{n}cont}(M(\mathfrak C_{n}),  M/p^{n}M) \to \Hom_{A}^{\mathfrak Ccont}( M(\mathfrak C ),   M ) \ ,
\end{equation}
for the projective limit of the $\mathfrak C_{n}$-open topologies on the left hand side.
  \end{lemma}

   \begin{proof} Let $f=\varprojlim f_{n}$ be a map in the image, and let $C\in \mathfrak C $. Then $f|_{C}$ is the projective limit of the $f_{n}|_{p_{n}(C)}$ hence is continuous. This means that the map in the  assertion is well defined. For the continuity, let $C\in  \mathfrak C $ and $\mathcal M \subset M$
   be an open $A$-submodule. The preimage of $E(C, \mathcal M)$ is equal to
 \begin{equation*}
 \big ( \varprojlim_{n} \ \Hom_{A}^{\mathfrak C_{n}cont}(M(\mathfrak C_{n}),  M/p^{n}M) \big )\  \cap  \ \big ( \prod_{n} E(p_{n}(C), \mathcal M+p^{n}M/p^{n}M) \big ) \ .
  \end{equation*}
Since $\mathcal M$ contains some $p^{n_{o}} M$, this intersection is equal to the open submodule
 \begin{equation*}
  \{ (f_{n}) \in \varprojlim_{n} \ \Hom_{A}^{\mathfrak C_{n}cont}(M(\mathfrak C_{n}),  M/p^{n}M)  :  f_{n  } \in E(p_{n}(C), \mathcal M+p^{n}M/p^{n}M)  \ \text{for} \  n \leq n_{0} \}.
  \end{equation*}
\end{proof}

   \begin{proposition} \label{basicspecial}   In the above setting assuming that    all the assumptions of Prop. \ref{corspecial} are satisfied for $(s, M/p^{n}M, \mathfrak C_{n } )$ and  for all $n\geq 1$.
 Then,  for all $g\in N_{0}\overline P N_{0}$,
  the functions
  $$\alpha_{g,0}:N_{0}\to \Hom_A^{\mathfrak{C}ont} (M(\mathfrak{C}),M)$$ are $(s,res, \mathfrak C)$-integrable,  their $(s,res, \mathfrak C)$-integrals
  $\mathcal H_{g} $ belong to $\End_{A}(M(\mathfrak{C}))$ and satisfy  the relations H1, H2, H3 of Prop. \ref{multiplicative}.
  \end{proposition}

 \begin{proof}  In the following we indicate with an extra index $n$ that the corresponding notation is meant for the module $M/p^{n} M $ with the special family
$\mathfrak C_{n } $. Then  $\alpha_{g,0}(u)$ is the image of $(\alpha_{g,0,n}(u))_{n}$ by the map \eqref{f:pl},  for $u\in N_{0}$. It follows that  $\mathcal H_{g, J(N_0/N_k) }$ is the image of  $(\mathcal H_{g, J(N_0/N_k) ,n})_{n}$ for $g\in N_{0}\overline P N_{0}$. By assumption the integral
$\mathcal H_{g,   n }=\lim_{k\to \infty}\mathcal H_{g, J(N_0/N_k) ,n} $ exists, lies in  $\Hom_A^{\mathfrak{C}_{n}ont} (M(\mathfrak{C}_{n}),M/p^{n}M)$, and satisfies the relations  H1, H2, H3 of Prop. \ref{multiplicative}.

The continuity of the map \eqref{f:pl} implies that the image  of $(\mathcal H_{g,   n })_{n}$ is  equal to the limit $ \lim_{k\to \infty}\mathcal H_{g, J(N_0/N_k) } $, therefore is
the integral $\mathcal H_{g}$ of $\alpha_{g,0}$. The additional properties for   $\mathcal H_{g}$ are inherited from the corresponding properties of the $\mathcal H_{g,   n }$.
 \end{proof}

Under the assumptions of Prop. \ref{basicspecial}, we  associate to $(s,M, \mathfrak C  )$,  an $A$-algebra homomorphism
\begin{equation*}
    \widetilde{\Res} \ : \ \mathcal{A}_{\mathcal{C} \subset G/P} \ \to \ \End_{A} (M(\mathfrak C)^{P}) \ .
\end{equation*}
via the propositions \ref{multiplicative} , \ref{P+-P}, which extends  the  $A$-algebra homomorphism
\begin{equation*}
    \Res \ : \ C^\infty_c(\mathcal{C},A) \# P  \ \to \ \End_{A} (M(\mathfrak C)^{P})
\end{equation*}
constructed in the proposition \ref{Res}.
  The  homomorphism  $\Res$ gives rise to a $P$-equivariant sheaf on $\mathcal{C}$ as described in detail in the theorem \ref{the11}.
The  homomorphism  $\widetilde{\Res}$ defines on the global sections with compact support $M(\mathfrak C)^{P}_c$ of the sheaf on $\mathcal{C}$ the structure of a nondegenerate $\mathcal{A}_{\mathcal{C} \subset G/P}$-module. The latter leads, by the proposition \ref{cat-equiv}, to the unital $C^\infty_c(G/P,A) \# G$-module $\mathcal{Z} \otimes_{\mathcal{A}} M(\mathfrak C)^{P}_c$ which corresponds to a $G$-equivariant sheaf on $G/P$ extending the earlier sheaf on $\mathcal{C}$ (remark \ref{G-extension}).

\section{Classical $(\varphi, \Gamma)$-modules on $\mathcal O_{\mathcal E}$}

\subsection{The Fontaine ring ${\cal O}_{\cal E}$} Let $K/\mathbb Q_{p}$  be  a finite extension of ring of integers $o$, of uniformizer $p_{K}$ and residue field $k$. By definition the Fontaine ring  ${\cal O}_{\cal E}$ over $o$ is the $p$-adic completion of the localisation of the Iwasawa $o$-algebra $\Lambda (\mathbb Z_{p}):=o[[\mathbb Z_{p}]]$ with respect to the multiplicative set of elements which are not divisible by $p$. We choose a generator $\gamma$ of $\mathbb Z_{p}$ of image $[\gamma]$ in ${\cal O}_{\cal E}$ and we denote $X=[\gamma]-1 \in {\cal O}_{\cal E}$.  The Iwasawa $o$-algebra $\Lambda (\mathbb Z_{p})$ is a local noetherian ring of maximal ideal $\mathcal M (\mathbb Z_{p})$ generated by $p_{K}, X$. It is a compact ring for the $\mathcal M (\mathbb Z_{p})$-adic topology.
The ring ${\cal O}_{\cal E}$  can be viewed as the ring of infinite Laurent series $\sum_{n \in \mathbb Z}a_{n}X^{n}$ over $o$ in the variable $X$ with $\lim_{n\to -\infty}a_{n}=0$, and $\Lambda (\mathbb Z_{p} )$ as the  subring  $o[[X]]$ of Taylor series. The Fontaine ring  ${\cal O}_{\cal E}$ is a local noetherian ring of maximal ideal $ p_{K}{\cal O}_{\cal E}$ and residue field isomorphic to $k((X))$; it is  a pseudo-compact  ring for the $p$-adic  ($=$ strong) topology and a complete ring  (with continuous multiplication) for the weak topology. A fundamental system of open neighborhoods of $0$ for the weak topology of ${\cal O}_{\cal E}$ is given by
$$
 (  O_{n,k} =p^{n}{\cal O}_{\cal E} + \mathcal M (\mathbb Z_{p})^{k})_{n,k\in \mathbb N}
$$
or by
$$
   (B_{n,k} =p^{n}{\cal O}_{\cal E} + X^{k}\Lambda (\mathbb Z_{p})_{n,k\in \mathbb N}
$$
Other fundamental system of  neighborhoods of $0$ for the weak topology  are
$$(O_{n}:=O_{n,n})_{n\geq 1} \quad \text{ or  } \quad (B_{n}:=B_{n,n})_{n\geq 1} \ . $$

\subsection{The group $GL(2, \mathbb Q_{p})$}
We consider the group $G=GL(2, \mathbb Q_{p})$ and
$$
N_{0 }:=\begin{pmatrix} 1&  \mathbb Z_{p} \cr 0 & 1
\end{pmatrix}\ , \
 \Gamma := \begin{pmatrix} \mathbb Z_{p} ^{*} & 0 \cr 0 & 1 \end{pmatrix} \ ,
\   L_{0} := \begin{pmatrix} \mathbb Z_{p} ^{*}& 0 \cr 0 & \mathbb Z_{p} ^{*}
\end{pmatrix}\ ,  \  L_{* }:= \begin{pmatrix} \mathbb Z_{p}-\{0\} & 0 \cr 0 & 1
\end{pmatrix}  \ ,  $$
$$
N_{k }:=\begin{pmatrix} 1& p^{k} \mathbb Z_{p} \cr 0 & 1
\end{pmatrix}\  , \ L_k:= \begin{pmatrix}1+p^{k} \mathbb Z_{p}  & 0 \cr 0 & 1+p^{k}\mathbb Z_{p}
\end{pmatrix}\    {\rm for}  \ k\geq 1\ ,
$$
$ P_k=L_kN_{k}$ for $k\in \mathbb N  $,  the upper triangular subgroup $P$, the diagonal subgroup $L$, the  upper unipotent subgroup $N$,  the center $Z$, the mirabolic monoid $P_{* } =N_{0}  L_{* } $, and the monoids  $L_{+}= L_{*}Z$ , $P_{+}=N_{0}L_{+} $.
The subset of non invertible elements in  the monoid $L_{*}$ is
$$\Gamma s_{p}^{\mathbb N-\{0\}}  = \{s_{a}: =\begin{pmatrix} a&0 \\ 0 & 1 \end{pmatrix} \ \ {\rm for } \ \ a\in p\mathbb Z_{p}- \{0\}  \} \  .
$$
An element $s\in \Gamma s_{p}^{\mathbb N-\{0\}}Z$  is called strictly dominant.
 In the following we   identify the group $\mathbb{Z}_p$  with $N_{0}$. The action of  $P_{+} $ on $N_{0}$ induces an \'etale ring action of $P_{ +} $ trivial on $Z$ on $\Lambda ( N_{0})$  which respects the ideal generated by $p$. This action  extends first to the localisation and then to the completion to give an \'etale ring action of $P_{ + }$ on  ${\cal O}_{\cal E}$ determined by its restriction to $P_{*}$. For   the
weak topology (and not for the $p$-adic topology), the  action $P_{ +} \times  {\cal O}_{\cal E}\to  {\cal O}_{\cal E}$ of the monoid $P_{+}$ on $ {\cal O}_{\cal E}$ is continuous (see  Lemma 8.24.i in \cite{SVig}). For $t\in L_+$ the  canonical left inverse $\psi_t$ of the action $\varphi _t$  of $t$ is continuous (this is proved in a more general setting later in Prop. \ref{contphipsi-general}).

\subsection{Classical \'etale $(\varphi, \Gamma)$-module}\label{CPG}

Let $s\in  \Gamma s_{p}^{\mathbb N-\{0\}}Z$. A finitely generated \'etale   $\varphi_s$-module $D$ over ${\cal O}_{\cal E}$ is a finitely generated ${\cal O}_{\cal E}$-module with an \'etale semilinear endomorphism   $\varphi_s$.  These modules form an abelian category $\mathfrak M_{{\cal O}_{\cal E}}^{et}(\varphi_s)  $. We fix such a module $D$.

\bigskip
In the following,  the topology of   $D$ is its weak topology.  For any surjective ${\cal O}_{\cal E}$-linear map $f: \oplus ^{d}{\cal O}_{\cal E}\to D$, the image  in $D$ of a fundamental system of  neighborhoods of $0$ in $ \oplus ^{d}{\cal O}_{\cal E}$ for the weak topology is a fundamental system of  neighborhoods of $0$ in $D$. Finitely generated $\Lambda (N_{0})$-submodules  of $D$ generating the ${\cal O}_{\cal E}$-module $D$  will be called lattices.
The   map $f$  sends $\oplus  ^{d}\Lambda (\mathbb{Z}_{p})$  onto a lattice $D^{0}$ of $D$. We note $\mathcal O_{n,k}:=p^{n}D +  \mathcal M (\mathbb Z_{p})^{k} D^{0}$  and  $\mathcal B_{n,k}:=p^{n}D +  X^{k} D^{0}$. Writing $\mathcal O_{n}:=\mathcal O_{n,n}$ and $\mathcal B_{n}:=\mathcal B_{n,n}$,
  $(\mathcal O_{n})_{n}$ and $(\mathcal B_{n})_{n}$ are two  fundamental systems of neighborhoods of $0$ in $D$. The  topological ${\cal O}_{\cal E}$-module $D$ is Hausdorff and complete.

A treillis $D _{0}$ in   $D$ is a compact $\Lambda (N_{0})$-submodule $D _{0}$  such that the image of $D_{0}$ in  the finite dimensional $k((X))$-vector space $D/p_K D$   is a $k[[X]]$-lattice (\cite{Mira} D\'ef. I.1.1).
 A lattice is a treillis and a treillis contains a lattice.

\bigskip

For $n  \geq 1$, the reduction modulo $p^{n} $ of $D$ is the finitely generated ${\cal O}_{\cal E}$-module $D/p^{n}D$ with the induced action of   $\varphi_s$. The action remains \'etale, because the multiplication by $p^{n}$ being a morphism in $\mathfrak M_{{\cal O}_{\cal E}}^{et}(\varphi_s)  $ its cokernel belongs to the category. The reduction modulo $p^{n}$ of   $\psi_s$ is the canonical left inverse of the reduction modulo $p^{n}$ of   $\varphi_s$.  The reduction modulo $p^{n}$ of a treillis of $D$ is a treillis of $D/p^{n}D$.

Conversely, if the reduction modulo $p^{n} $ of a finitely generated    $\varphi_s$-module $D$ over ${\cal O}_{\cal E}$  is \'etale for all $n\geq 1$, then $D$ is an \'etale   $\varphi_s$-module over ${\cal O}_{\cal E}$ because $D=\varprojlim_{n} D/p^{n}D$.

The weak topology of $D$ is the projective limit of the weak topologies of $D/p^{n}D$.

\bigskip When  $D$ is killed by a power of $p$ and  $D _{0}$ is a treillis of $D$, we have :

\begin{enumerate}
\item $D_{0}$ is  open and closed in $D$.

\item  $ (\mathcal M (\mathbb Z_{p})^{n} D_{0})_{n\in \mathbb N}$  and $(X^{n}D_{0})_{n\in \mathbb N}$ form two fundamental systems of open neighborhoods of zero in $D$.

\item Any  treillis of $D$ is contained in $X^{-n}D_{0}$ for some $n\in \mathbb N$.

\item $D=\bigcup_{k\in \mathbb N}X^{-k}D_{0}$.

\item $D_{0}$ is a lattice.

\end{enumerate}

The first   four properties are easy;  a reference  is   \cite{Mira} Prop. I.1.2. To show that $D_{0}$ is a lattice, we pick some lattice $D^0$ then $D_{0}$  is contained in the lattice  $X^{-n} D^0$ for some $n \in \mathbb{N}$ by the property 3.  Since the ring $\Lambda(N_0)$ is noetherian the assertion follows.

When  $D$ is killed by a power of $p$, the weak topology of $D$ is locally compact (by properties 2 and 5).

\begin{proposition}\label{contphipsi}
  Let $D$ be a finitely generated \'etale   $\varphi_s$-module  over ${\cal O}_{\cal E}$. Then   $\varphi_s$ and its canonical inverse   $\psi_s$ are continuous.
\end{proposition}

\begin{proof}
 a)  The above ${\cal O}_{\cal E}$-linear surjective  map $f:\oplus^{d}{\cal O}_{\cal E}\to D$ sends $(a_{i})_{i}$ onto $\sum_{i}a_{i}d_{i}$ for some elements $d_{i}\in D$.
As   $\varphi_s$ is \'etale, the map  $(a_{i})_{i}\mapsto \sum_{i}a_{i}\varphi _s(d_{i})$
 also gives an ${\cal O}_{\cal E}$-linear surjective map $\oplus^{d}{\cal O}_{\cal E} \to  D$. Both surjections are topological quotient maps by the definition of the  topology on $D$, and  the morphism   $\varphi_s$ of ${\cal O}_{\cal E}$ is continuous. We deduce that the morphism   $\varphi_s$ of $D$ is continuous.

b) The image  $\oplus^{d}\Lambda(N_{0})$ by $f$ is a lattice $D ^{0}$ of $D$. For any $k \in \mathbb{N}$ the $\Lambda(N_{0})$-submodule $D_{0,k}$ of $D$ generated by $(  \varphi_s(X^{k}e_{i}))_{ 1\leq i \leq d}$ also is a treillis of $D$ because   $\varphi_s$ is \'etale. We have $\psi_s (D_{0,k})= X^{k}D_{0}$ (cf.\ lemma \ref{produit}).

c) When $D$ is killed by a power of $p$, we deduce that   $\psi_s$ is continuous by the properties 1 and 2 of the treillis. When $D$ is not killed by a power of $p$,  we deduce that  the reduction modulo $p^{n} $ of   $\psi_s$ is continuous for all $n$; this implies that   $\psi_s$ is continuous  because $(A=o,D)$ satisfy the properties a, b, c, d of section \ref{topAM},
 and $D/p^{n}  D$ is a (finitely generated) \'etale   $\varphi_s$-module  over ${\cal O}_{\cal E}$.
 \end{proof}

We put
\begin{equation*}
    D^{+} := \{x \in D :\ \text{the sequence $(\varphi_s^k(x))_{k \in \mathbb{N}}$ is bounded in $D$} \}
\end{equation*}
(cf.\ \ref{subsec:bounded}) and
\begin{equation} \label{D++}
D^{++} :=\{x\in D \ | \ \lim_{k\to \infty}\varphi_s^{k}(x)=0\} \ .
\end{equation}

\begin{proposition} \label{Mira}
\begin{itemize}
  \item[(i)] When  $D$ is killed by a power of $p$, then $D^+$ and $D^{++}$ are lattices in $D$.
  \item[(ii)] There exists a  unique maximal treillis $D^{\sharp}$ such that $\psi_s(D^{\sharp})=D^{\sharp}$.
  \item[(iii)] The set of   $\psi_s$-stable treillis in $D$ has a unique minimal element $D^\natural$; it satisfies $\psi_s(D^\natural)=D^\natural$.
  \item[(iv)] $X^{-k} D^\sharp$ is a treillis stable by   $\psi_s$ for all $k\in \mathbb N$.
\end{itemize}
\end{proposition}

\begin{proof}
  The references given in the following are stated for \'etale $(\varphi_{s_p}, \Gamma)$-modules but the proofs never use that there exists an action of $\Gamma$ and they are valid for \'etale $\varphi _{s_p}$-modules.

(i) For $s=s_p$ this is \cite{Mira} Prop. II.2.2(iii) and Lemma II.2.3. The properties of $s_p$ which are needed for the argument are still satisfied for general $s$ in the following form:
\begin{itemize}
  \item[--] $\varphi_s(X) \in \varphi_{s_p}^m(X) \Lambda(\mathbb{Z}_p)^\times$ where $s = s_0s_p^m z$ with $s_0 \in \Gamma$, $m \geq 1$, and $z \in Z$.
  \item[--] $(\varphi_s(X) X^{-1} )^{p^k} \in p^{k+1}\Lambda(\mathbb{Z}_p) + X^{(p-1)p^k} \Lambda(\mathbb{Z}_p)$ for any $k \in \mathbb{N}$.
\end{itemize}

(ii) and (iii) For any finitely generated ${\cal O}_{\cal E}$-torsion module $M$ we denote its Pontrjagin dual of continuous $o$-linear maps from $M$ to $K/o$ by $M^\vee := \Hom_{o}^{cont}(M,K/o)$. Obviously, $M^\vee$ again is an ${\cal O}_{\cal E}$-module by $(\lambda f)(x) := f(\lambda x)$ for $\lambda \in {\cal O}_{\cal E}$, $f \in M^\vee$, and $x \in M$. It is shown in \cite{Mira} Lemma I.2.4 that:
\begin{itemize}
  \item[--] $M^\vee$ is a finitely generated ${\cal O}_{\cal E}$-torsion module,
  \item[--] the topology of pointwise convergence on $M^\vee$ coincides with its weak topology as an ${\cal O}_{\cal E}$-module, and
  \item[--] $M^{\vee\vee} = M$.
\end{itemize}
Now let $D$ be as in the assertion but killed by a power of $p$. One checks that $D^\vee$ also belongs to $\mathfrak M_{{\cal O}_{\cal E}}^{et}(\varphi_s)$ with respect to the semilinear map $\varphi_s(f) := f \circ \psi_s $ for $f \in D^\vee$; moreover, the canonical left inverse is $\psi_s(f) = f \circ \varphi_s$. next, \cite{Mira} Lemma I.2.8 shows that:
\begin{itemize}
  \item[--] If $D_0 \subset D$ is a lattice then $D_0^\perp := \{ d \in D^\vee : f(D_0) = 0 \}$ is a lattice in $D^\vee$, and $D_0^{\vee\vee} = D_0$.
\end{itemize}
We now define $D^\natural := (D^\vee)^+$ and $D^\sharp := (D^\vee)^{++}$. The purely formal arguments in the proofs of \cite{Mira} Prop.\ II.6.1, Lemma II.6.2, and Prop.\ II.6.3 show that $D^\natural$ and $D^\sharp$ have the asserted properties.

For a general $D$ in $\mathfrak M_{{\cal O}_{\cal E}}^{et}(\varphi_s)$ the (formal) arguments in the proof of \cite{Mira} Prop.\ II.6.5 show that $((D/p^n D)^\natural)_{n \in \mathbb{N}}$ and $((D/p^n D)^\sharp)_{n \in \mathbb{N}}$ are well defined projective systems of compact $\Lambda(\mathbb{Z}_p)$-modules (with surjective transition maps). Hence
\begin{equation*}
    D^\natural := \varprojlim (D/p^n D)^\natural \quad\text{and}\quad D^\sharp := \varprojlim (D/p^n D)^\sharp
\end{equation*}
have the asserted properties.

(iv) $X^{-k} D^\sharp$ is clearly a treillis. As
$X$ divides $\varphi_s(X)= (1+X)^a -1$ in $\Lambda (\mathbb Z_p)=o[[X]]$,  there exists $f(X) \in o[[X]]$ such that  $\varphi_s(X^k)=X^kf(X)^k$. So we
have
$\psi_s(X^{-k}D^{\sharp})=\psi_s( \varphi_s(X^{-k}) f(X)^kD^{\sharp})=X^{-k}\psi_s(f(X)^kD^{\sharp})
\subset X^{-k}\psi_s(D^{\sharp}) \subset X^{-k}D^{\sharp}$
since $D^{\sharp}$ is   $\psi_s$-stable  by definition.
\end{proof}

\begin{proposition} \label{bd} Let $D$ be  a finitely generated \'etale   $\varphi_s$-module over ${\cal O}_{\cal E}$. For any compact subset $C \subseteq D$ and any $n\in \mathbb N$, there exists  $k_{0} \in \mathbb N$ such that
\begin{equation*}
    \bigcup_{k \geq k_0} \psi_s^{k}(N_0 C) \subseteq D^{\sharp}+p^{n}D \ .
\end{equation*}
\end{proposition}

\begin{proof}
 We choose a treillis $D^{0}$ containing $C$, as we may. A treillis is a $\Lambda(N_{0})$-module hence $N_0 C \subseteq D^{0}$.
By the formal argument in the proof of \cite{Mira} Prop.\ II.6.4 we find a $k_0$ such that $\bigcup_{k \geq k_0} \psi_s^{k}(D^0) \subseteq D^{\sharp}+p^{n}D$.
 \end{proof}

\begin{corollary}\label{cbd}
Let $D$ be  a finitely generated \'etale   $\varphi_s$-module over ${\cal O}_{\cal E}$ killed by a power of $p$. For any compact subset $C \subseteq D$,  there exists  $k_{0}, r \in \mathbb N$ such that
\begin{equation*}
    \bigcup_{k \geq k_0} \psi_s^{k}(N_0 C) \subseteq X^{-r}D^{++}  \ .
\end{equation*}
\end{corollary}

 For any submonoid $ L' \subset L_{+}$ containing a strictly dominant element, an \'etale $L'$-module over ${\cal O}_{\cal E}$   is a finitely generated  ${\cal O}_{\cal E}$-module with an \'etale semilinear action of $L'$.

A topologically   \'etale $L'$-module   over ${\cal O}_{\cal E}$ will be an \'etale $L'$-module  $D$  over ${\cal O}_{\cal E}$  such that the action $ L'\times D\to D$ of $ L'$ on $D$ is continuous.  This terminology is provisional since we will show later on (Cor. \ref{top-etale}) in a more general context that any \'etale $L'$-module over ${\cal O}_{\cal E}$ in fact is topologically \'etale and, in particular, is a complete topologically \'etale  $o[N_{0}L']$-module in our previous sense.

Let $D$ be  a topologically  \'etale $L_{+}$-module over  ${\cal O}_{\cal E}$. When the matrix  $g=\begin{pmatrix} a_g& b_g \cr c_g & d_g
\end{pmatrix}\in GL(2,\mathbb Q_{p})$  belongs to $N_{0}\overline P N_{0}$, the set  $X_g\subset \mathbb Z_p$ of $r$ such that
$$\begin{pmatrix} a_g& b_g \cr c_g & d_g\end{pmatrix} \begin{pmatrix} 1& r \cr 0 & 1\end{pmatrix}= \begin{pmatrix} 1& b(g,r) \cr 0 & 1\end{pmatrix} \begin{pmatrix} a(g,r)& 0 \cr 0 & d(g,r)\end{pmatrix}\begin{pmatrix} 1& 0\cr c(g,r) & 1\end{pmatrix}
$$ with $b(g,r)\in \mathbb Z_p$, is not an empty set. We denote by $t(g,r)$ the diagonal matrix in the right hand side.
For $s\in L_+$ strictly dominant, i.e. $s=s_a z $  with $a\in p\mathbb Z_{p} -\{0\}$ and $z\in Z$,  and a large positive integer $k_{g,s}$, we have $t(g,r)s^k \in L_+$. For   $k\geq k_{g,s}$, and
a system of representatives $J(\mathbb Z_p/a\mathbb Z_p)\subset \mathbb Z_p$ for the cosets $\mathbb Z_p/a\mathbb Z_p$, we set
$$
\mathcal H_{g,s,J(\mathbb Z_p/a^k\mathbb Z_p)}( . )=\sum_{r\in  X_g \cap J(\mathbb Z_p/a^k\mathbb Z_p)}  (1+X)^{b(g,r)}  \varphi_{t(g,r)s^k } \psi_s^k  ((1+X)^{-r } . )) $$
in $\End^{cont}_o(D)$.

\begin{proposition} \label{3}
Let $D$ be  a topologically  \'etale $L_{+}$-module over  ${\cal O}_{\cal E}$. For  the compact open topology in $\End^{cont}_o(D)$,   the maps $\alpha_{g,0} : N_0 \longrightarrow \End_o^{cont}(D)$, for $g\in  N_{0}\overline P N_{0}$, are integrable with respect to $s $ and $\res$, for all $s\in L_+$ strictly dominant, i.e. $s=s_a z $  with $a\in p\mathbb Z_{p} -\{0\}$ and $z\in Z$,
their  integrals
$$\mathcal H_{g}=\int_{N_{0}}\alpha_{g,0} d\res = \lim_{k\to \infty}\mathcal H_{g,s,J(\mathbb Z_p/a^k\mathbb Z_p)}$$
for any choices of $J(\mathbb Z_p/a^k\mathbb Z_p)\subset \mathbb Z_p$,   do not depend on the choice of $s$ and satisfy the relations H1, H2, H3 of Proposition \ref{multiplicative}.
\end{proposition}

\begin{proof}   By Prop. \ref{basicspecial}, we  reduce to the case  that $D$ is killed by a power of $p$ and to showing the assumptions of Prop. \ref{corspecial} for the family of all compact subsets of $D$.  The axioms $\mathfrak C_{i}$, for $1\leq i \leq 6$, are obviously satisfied by continuity of $ \varphi_s,  \psi_s$, and of the action of $ n\in N_{0}$  on $D$.

i. We show  first the convergence criterion of Proposition \ref{criterion},  using   the theory of treillis, i.e. of lattices, in $D$.

Given a   lattice  $\mathcal M \subseteq D$, a compact subset $C \subseteq D$ such that $N_0 C \subseteq C$, and a compact subset $C_+ \subseteq L_{+}$,    we want  to find a compact open subgroup $P_1 \subset P_{0}$ and an integer $k_{0} \in \mathbb N$ such that
\begin{equation}\label{cv}
     s^{k }(1- P_1) C_+  \psi_s^{k} \subseteq E(C,\mathcal M)
\end{equation}
for all $ k\geq k_{0}$.

We choose $r_{0}\in \mathbb N$ with $ \varphi_s^{k}(D^{++}) \subset \mathcal M$ for all $k\geq r_{0}$, as we may by properties 5 and  6 of treillis.
We choose $r, k_{0}\in \mathbb N$ such that $k_{0}\geq r_{0}$ and
\begin{equation*}
    \bigcup_{k \geq k_0} \psi_s^{k}(C) \subseteq X^{-r}D^{++} \ ,
\end{equation*}
as we may by Cor.\ref{cbd}.  Applying $C_{+}$ we obtain
\begin{equation*}
    \bigcup_{k \geq k_0} C_{+}  \psi_s^{k}(C) \subseteq C_{+}( X^{-r}D^{++})  \ . \end{equation*}
The continuity of the action of $P_{+}$  on $D$ implies that $C_{+}( X^{-r}D^{++})$ is compact. Hence we can choose $r' \in \mathbb N$  such that $C_{+}( X^{-r}D^{++}) \subseteq X^{-r'}D^{++}$ and we obtain
\begin{equation*}
    \bigcup_{k \geq k_0} C_{+}  \psi_s^{k}(C)   \subseteq X^{-r'}D^{++} \ .
\end{equation*}
  As $X^{-r'}D^{++}$ is compact and $D^{++}$ an open neighborhood of $0$, the continuity of the action of $P_{+}$  on $D$
there exists a compact open subgroup  $P_1 \subseteq P_+$ such that
\begin{equation*}(1- P_1)  X^{-r'}D^{++}\subseteq  D^{++} \ . \end{equation*}
 Hence we have
$s^{k }(1- P_1) C_{+}  \psi_s^{k}  (C) \subset \varphi_s^{k }( D^{++}) \subset \mathcal M
$ for all $ k\geq k_{0}$.

ii.    To obtain all the  assumptions of Prop. \ref{corspecial} for the family of all compact subsets of $D$, it remains to prove that, given $ x \in D$ and  $g\in N_{0}\overline P N_{0}$, $s=s_a z $  with $a\in p\mathbb Z_{p} -\{0\}$ and $z\in Z$,  and  $(J(\mathbb Z_p/a^k \mathbb Z_p))_k$,  there exists a compact $C_{x,g,s}\subset D$ and a positive integer $k_{x,g,s}$ such that $\mathcal H_{g,s,J(\mathbb Z_p/a^k\mathbb Z_p)}(x)\in C_{x,g,s}$ for any $k\geq k_{x,g,s}$. This is clear because $D$ is locally compact (by hypothesis $D$ is killed by a power of $p$) and  the sequence $(\mathcal H_{g,s,J(\mathbb Z_p/a^k\mathbb Z_p)}(x))_k$ converges.

iii. The  independence of the choice of $s\in L_+$ strictly dominant results from the fact that,   for $z\in Z$,  $e\in \mathbb Z_p^*$, and  a positive integer  $r$, we have $(zs_{p^r e})^k N_0(zs_{p^r e} )^{-k}=
 s_p^{kr}N_0s_p^{kr}$ and $\varphi_{zs_{p^r e}}^k\psi_{zs_{p^r e}}^k = \varphi_{s_p}^{rk}\psi_{s_p}^{kr}$ as $\psi_{z s_e}$ is the right and left inverse of $\varphi_{z s_e}$.
 \end{proof}

\begin{remark}{\rm
 For a topologically \'etale $L _{+}$-module $D$ finitely generated over ${\cal O}_{\cal E}$ on which $Z$ acts through a character $\omega$ the pointwise convergence of the integrals $\int_{N_0} \alpha_{g,0} d\res$ is  a basic theorem of  Colmez, allowing him the construction of the representation of $GL(2,\mathbb Q_{p})$  that he denotes $D \boxtimes_{\omega}{\mathbb P}^{1}$.

 Our construction coincides with Colmez's construction because  our
  $\mathcal H_{g}\in \End_{o}^{cont}(D)$ are the same than the $H_{g}$ of Colmez given in \cite{C} Lemma II.1.2  (ii). To see this, we denote
$$
w_{0} \ = \   \begin{pmatrix}  0 &1 \cr 1 & 0 \end{pmatrix} , u(b)=\begin{pmatrix}  1&b \cr 0 & 1 \end{pmatrix}, \overline u(c)=\begin{pmatrix}  1&0 \cr c & 1 \end{pmatrix}, \delta(a,d) = \begin{pmatrix}  a&0 \cr 0 & d \end{pmatrix} , g = \begin{pmatrix}  a&b \cr c & d \end{pmatrix} .
$$
A matrix  $g= \begin{pmatrix}  a &b \cr c & d \end{pmatrix} $ belongs to $Pw_{0} P$ if and only if $c\neq 0$. When $c\neq 0$ we have
\begin{equation*}
  g\ =\ u(ac^{-1})\ \delta(b-adc^{-1},c ) \ w_{0}\ u(dc^{-1}) .
\end{equation*}
From
$$
gu(x)w_{0}= \begin{pmatrix}  ax+b&a \cr c x+d& c \end{pmatrix}   ,
$$
we see that $u( \{ x\in \mathbb Q_{p} , cx+d \ \neq 0 \} ) $ is the subset $ N_{g}\subset N$  of $u(x)$ such that $gu(x)w_{0}\in  Pw_{0} P$. We suppose  $cx+d\neq 0$ and we write
$$
g[x]=\frac{ax+b}{cx+d} \ , \ g'[x]= \frac{ad-bc}{(cx+d)^{2}} .
$$
Then we have
$$
gu(x)=  \delta(cx+d, cx+d)  \begin{pmatrix} g'[x] & g[x] \cr 0 & 1 \end{pmatrix}\overline n(g,u(x))
$$
with $\overline n(g,u(x))\in N$.  We deduce that the   subset  of $u(x)\in N_{g}$ such that $gu(x)w_{0}\in  N_{0}w_{0} P$ is
$$
U_{g}= u (X_{g}) \quad {\rm where} \quad X_{g}=\{ x\in \mathbb Z_{p}, cx+d \neq 0, \frac{ax+b}{cx+d} \in  \mathbb Z_{p}\}  ,
$$
and that we have,
$$
\mathcal H_{g,s_p,J(\mathbb Z_p/p^k\mathbb Z_p)} =\sum_{x\in  X_g \cap J(\mathbb Z_p/p^k\mathbb Z_p)}  \delta(cx+d, cx+d)  \begin{pmatrix} g'[x] & g[x] \cr 0 & 1 \end{pmatrix} \varphi_{s_p} ^{k} \psi_{s_p}^{k} u(-x)\  ,
$$
By Colmez's formula in \cite{C} Lemma II.1.2  (ii),  $H_g=\lim_{k\to \infty} \mathcal H_{g,p,J(\mathbb Z_p/p^k\mathbb Z_p)}$.  Hence $H_g=\mathcal H_g$.}
\end{remark}

The major goal of the paper is to generalize Prop. \ref{3}. See Prop. \ref{912}.

 \section{A generalisation of $(\varphi, \Gamma)$-modules}\label{GM}

 We return to a general group $G$.   We denote
 $G^{(2)}:=GL(2,\mathbb Q_{p})$  and the objects relative to $G^{(2)}$  will be  affected with an upper index $^{(2)}$.

\bigskip  {\sl  a)  We suppose that $N_{0}$ has the structure of a $p$-adic Lie group  and  that we have a  continuous surjective homomorphism
   $$\ell : N_{0}\to N _{0}^{(2)}  \ .     $$
   We choose  a  continuous homomorphism $\iota:N _{0}^{(2)}  \to N_{0} $ which is a section of $\ell $} (which is possible because $N_{0}^{(2)} \simeq \mathbb Z_{p}$).

  \bigskip    We  have   $N_{0}= N_{\ell}\,\iota( N _{0}^{(2)})$ where $N_{\ell}$ is the kernel of $\ell$.

  We denote by
  $L_{\ell, +}:=\{t\in L  \ | \ tN_{\ell}t^{-1} \subset  N_{\ell}\ , \  tN_{0}t^{-1} \subset N_{0}\}$   the stabilizer of $N_{\ell}$ in the $L$-stabilizer of $N_{0}$, and
      by  $ L_{\ell ,\iota}:=\{t\in L \ | \ tN_{\ell}t^{-1} \subset N_{\ell}\ , \ t\iota( N _{0}^{(2)})t^{-1} \subset \iota(N _{0}^{(2)})\}$
 the stabilizer of $N_{\ell}$ in the $L$-stabilizer of $\iota( N _{0}^{(2)})$. We have
  $ L_{\ell ,\iota} \subset L_{\ell, +} $.

\bigskip

{\sl  b) We suppose  given a submonoid $L_{*}$ of  $ L_{\ell ,\iota}$ containing $s$ and   a continuous homomorphism $\ell: L_{*} \to L_{+}^{(2)} $   such  that $(\ell, \iota)$ satisfies
$$
\ell (t ut^{-1} )=\ell (t)\ell (u) \ell (t)^{-1}  \ , \ \ t \iota(y)t^{-1} =\iota (\ell (t) y \ell (t)^{-1}) \ , \ {\rm for} \ u\in N_{0},y\in N_{0}^{(2)}, t\in L_{*} \ .
$$}

The sequence $\ell (s^{n}N_{0}s^{-n})= \ell(s)^{n}N_{0}^{(2)}\ell (s)^{-n}$ in $N^{(2)}$ is decreasing with trivial intersection.
 The maps $\ell$ in a) and b)  combine to a unique continuous homomorphism $$\ell \ : \ P_{*}:=N_{0}\rtimes L_{*}\to P_{+}^{(2)}.$$

\subsection{The microlocalized ring $\Lambda_{\ell}(N_{0})  $}
 The ring $ \Lambda_{\ell}(N_{0})$, denoted by $ \Lambda_{N_{\ell}}(N_{0})$ in \cite{SVig},
 is a generalisation of the ring  $\mathcal O_{\mathcal E}$, which corresponds to  $\Lambda_{\id} (N_{0}^{(2)})$.

The maximal ideal ${\cal M }(N_{\ell})$ of  the completed group $o$-algebra $ \Lambda(N_{\ell})=o[[N_{\ell}]]$ is generated by $p_{K}$ and by the kernel of the augmentation map $o[N_{\ell}]\to o$.

The ring $ \Lambda_{\ell}(N_{0})$  is the ${\cal M }(N_{\ell})$-adic completion of the localisation of  $ \Lambda (N_{0})$ with respect to the Ore subset $S_{\ell} (N_{0})$
of elements which are not in $ {\cal M }(N_{\ell})  \Lambda (N_{0})$.  The   ring $\Lambda ( N_{0})$
 can be viewed as the ring  $\Lambda (N_{\ell})[[X]]$ of skew Taylor series over $\Lambda (N_{\ell})$ in the variable $X=[\gamma]-1$  where  $\gamma \in N_{0}$ and $\ell(\gamma)$ is  a topological generator of $\ell (N_{0})$.  Then   $ \Lambda_{\ell}(N_{0})$ is viewed as
the ring of infinite skew Laurent series $\sum_{n \in \mathbb Z}a_{n}X^{n}$ over $\Lambda (N_{\ell})$ in the variable $X$ with $\lim_{n\to -\infty}a_{n}=0$ for the pseudo-compact topology of $\Lambda (N_{\ell})$.

 The ring $ \Lambda_{\ell}(N_{0})$
 is  strict-local noetherian  of maximal ideal  ${\mathcal M }_{\ell}(N_{0})$ generated by   ${\mathcal M }(N_{\ell})$.  It is  a pseudocompact  ring for the ${\cal M }(N_{\ell})$-adic topology (called the strong topology). It is  a complete Hausdorff ring for the weak topology (\cite{SVig} Lemma 8.2)
with fundamental system of open neighborhoods of $0$ given by
$$O_{n,k} :={\mathcal M }_{\ell}(N_{0})^{n}  +{\cal M}(N_{0})  ^{k}  \ \ {\rm for } \ \ n \in \mathbb N , k\in \mathbb N \ .
$$
In the computations it is sometimes better to use the fundamental systems of open neighborhoods of $0$ defined by
$$B_{n,k}:= {\mathcal M }_{\ell}(N_{0})^{n}  +X ^{k} \Lambda(N_{0}) \ \ {\rm for } \ \ n \in \mathbb N , k\in \mathbb N \ , $$
and
$$
C_{n,k}:=  \mathcal M_{\ell} (N_{0})^{n}   + \Lambda(N_{0})X^{k}  \ \ {\rm for } \ \ n \in \mathbb N , k\in \mathbb N \ ,
$$ which are equivalent due to   the two formulae
  \begin{equation*} X^k \Lambda(N_0)  \subseteq  \Lambda(N_0) X^k + \mathcal{M}(N_0)^k
\quad {\rm and} \quad
  \Lambda(N_0) X^k  \subseteq  X^k\Lambda(N_0) + \mathcal{M}(N_0)^k \ ,
  \end{equation*}
 We  write $O_{n}:=O_{n,n},B_{n}:=B_{n,n},$ and $C_{n}=C_{n,n}$. Then $(O_{n})_{n} ,  (B_{n})_{n}$, and $(C_{n})_{n}$
are also  fundamental system of open neighborhoods of $0$ in $ \Lambda_{\ell}(N_{0})$.

The  action $(b=ut,n_{0})\mapsto b.n_{0}=utn_{0}t^{-1}$ of  the monoid $P_{\ell, +}=N_{0}\rtimes L_{\ell, +}$ on $N_{0}$  induces a  ring  action $(t,x)\mapsto \varphi_{t}(x)$  of $L_{\ell, +}$ on the $o$-algebra  $\Lambda (N_{0})$   respecting  the ideal $ \Lambda (N_{0}) {\cal M }(N_{\ell})$, and the Ore set  $S_{\ell}(N_{0})$  hence defines a ring action  of $L_{\ell, +}$ on the $o$-algebra  $ \Lambda_{\ell}(N_{0})$. This actions respects the maximal ideals ${\cal M}(N_{0})$  and ${\mathcal M }_{\ell}(N_{0})$  of the rings $\Lambda (N_{0})$   and  $ \Lambda_{\ell}(N_{0})$ and hence the open neighborhoods of zero $O_{n,k}$.

\begin{lemma} \label{tOnk} For $t\in L_{\ell, +}$, a  fundamental system of open neighborhoods of $0$ in $\Lambda_{\ell }(N_{0})$ is given by $$
(\varphi_{t}(O_{n,k} ) \Lambda(N_{0}))_{n,k\in \mathbb N}\ .
$$

\end{lemma}
\begin{proof} Obviously $\varphi_{t}(O_{n,k} ) \Lambda(N_{0})\subset  O_{n,k} $
because $\varphi_{t}(\mathcal M(H))= \mathcal M(tHt^{-1}) \subset \mathcal M(H) $ for $H$ equal to $N_{0}$ or $N_{\ell}$. Conversely given $ n,k\in \mathbb N$, we have to find $n',k'\in \mathbb N$ such that $ O_{n',k'} \subset \varphi_{t}(O_{n,k} ) \Lambda(N_{0})$. This can be deduced  from the following fact. Let  $H' \subset H$  be an open subgroup. Then given $k'\in \mathbb N$, there is   $k\in \mathbb N$ such that
$$\mathcal M (H') ^{k'}\Lambda(H) \supset \mathcal M (H) ^{k} \ .
$$
Indeed by taking a smaller $H'$ we can suppose that $H'\subset H$ is  open normal.
Then $\mathcal M (H') ^{k'}\Lambda(H)$ is a two-sided ideal in $\Lambda(H)$
and the factor ring $\Lambda(H)/\mathcal M (H') \Lambda(H)$ is an artinian local ring
with maximal ideal  $\mathcal M(H)/\mathcal M (H') \Lambda(H)$. It remains to observe that in  any artinian local ring
the maximal ideal is nilpotent.
\end{proof}
\begin{proposition} \label{sw}
The   action of $L_{\ell, +}$ on $ \Lambda_{\ell}(N_{0})$  is  \'etale :  for any $t\in L_{\ell, +}$, the map
$$
(\lambda, x) \mapsto \lambda \varphi_{t}(x) : \Lambda(N_{0} )\otimes_{\Lambda (N_{0}) , \varphi_{t}}\Lambda_{\ell}(N_{0}) \to \Lambda_{\ell}(N_{0})
$$
is bijective.
\end{proposition}
\begin{proof}  a) We  follow  (\cite{SVig} Prop. 9.6, Proof, Step 1).

a1) The conjugation by $t$ gives   a natural isomorphism
 $$
 \Lambda_{\ell}(N_{0})\to \Lambda_{tN_{\ell}t^{-1}}(tN_{0}t^{-1}) \ .
 $$

 a2) Obviously $\Lambda_{tN_{\ell}t^{-1}}(tN_{0}t^{-1}) = \Lambda(tN_{0}t^{-1}) \otimes_{\Lambda(tN_{0}t^{-1})}\Lambda_{tN_{\ell}t^{-1}}(tN_{0}t^{-1})$, and the map
$$\Lambda(tN_{0}t^{-1}) \otimes_{\Lambda(tN_{0}t^{-1})}\Lambda_{tN_{\ell}t^{-1}}(tN_{0}t^{-1})\to \Lambda (N_{0} )\otimes_{\Lambda (tN_{0} t^{-1})}\Lambda_{tN_{\ell}t^{-1}}(tN_{0}t^{-1})$$  is injective  because $\Lambda_{tN_{\ell}t^{-1}}(tN_{0}t^{-1})$ is flat on $\Lambda(tN_{0}t^{-1})$.

a3) The natural map
$$\Lambda (N_{0} )\otimes_{\Lambda (tN_{0} t^{-1})}\Lambda_{tN_{\ell}t^{-1}}(tN_{0}t^{-1})\to \Lambda_{\ell}(N_{0})$$
is bijective.

a4) The ring action $\varphi_{t}: \Lambda_{\ell}(N_{0})\to \Lambda_{\ell}(N_{0})$ of $t$ on $\Lambda_{\ell}(N_{0})$  is the composite of the maps of a1), a2), a3),  hence is injective.

a5) The  proposition is equivalent to a3) and $\varphi_{t}$ injective.
 \end{proof}

\begin{remark} {\rm The proposition is equivalent to :  for any $t\in L_{\ell, +}$, the map
$$(u, x) \mapsto u \varphi_{t}(x): o[N_{0}]\otimes_{o[N_{0}], \varphi_{t}}\Lambda_{\ell}(N_{0}) \to \Lambda_{\ell}(N_{0})$$
is bijective. }
\end{remark}

\subsection{The categories  $ {\mathfrak M}_{\Lambda_{\ell}(N_{0})}^{et}( L_{*})$ and  $ {\mathfrak M}^{et}_{{\cal O}_{\cal E}, \ell}(L_*)$}

  By the universal properties of  localisation and adic completion the  continuous homomorphisms $\ell$  and $\iota$  between $N_{0}$ and $N_{0}^{(2)}$ extend to  continuous   $o$-linear homomorphisms  of pseudocompact rings,
 \begin{equation}\label{elliota}
  \ell: \Lambda_{\ell}(N_{0})\to  {\cal O}_{\cal E} \ , \ \iota: {\cal O}_{\cal E}\to  \Lambda_{\ell}(N_{0}) \ \ , \ \ \ell \circ \iota = \id \ .
\end{equation}
If we view the rings  as Laurent series, $\ell (X)=X^{(2)}$, $\iota (X^{(2)})=X$, and $\ell$ is  the augmentation map $\Lambda (N_{\ell})\to o$ and $\iota$ is  the natural injection $o\to \Lambda (N_{\ell})$, on the coefficients. We have for $n,k\in \mathbb N$,
\begin{equation}\label{elioeq}
\begin{split}
\ell ({\mathcal M}_{\ell}(N_{0}) )= p_{K} {\cal O}_{\cal E} \ \ &, \ \ \ell (B_{n,k}) =B_{n,k}^{(2)} \ , \\
\iota(p_{K} {\cal O}_{\cal E})={\mathcal M}_{\ell}(N_{0})\cap \iota({\cal O}_{\cal E})
 \ \ &, \ \ \iota(B_{n,k}^{(2)})=B_{n,k}\cap \iota({\cal O}_{\cal E}) \ .
 \end{split}
\end{equation}
We denote  by $J(N_{0})$ the kernel of  $\ell:\Lambda(N_{0})\to \Lambda(N_{0}^{(2)})$  and by   $J_{\ell }(N_{0})$ the kernel of $\ell: \Lambda_{\ell}(N_{0})\to  {\cal O}_{\cal E} $. They are the closed two-sided ideals generated (as left or right ideals) by the kernel of the augmentation map $o[ N_{\ell}]\to o$.
We have
\begin{equation}\label{Ldeco}
\begin{split} \Lambda_{\ell}(N_{0}) = \iota ({\cal O}_{\cal E}) \oplus   J_{\ell }(N_{0}) \ \ & , \ \    \mathcal M_{\ell}(N_{0}) = p_{K}\iota ({\cal O}_{\cal E}) \oplus   J_{\ell }(N_{0})  \ , \\
 X^{k}\Lambda (N_{0})= \iota ((X^{(2)})^{k}\Lambda(N_{0}^{(2)}) \oplus X^{k}J(N_{0}) \ \ &, \ \ B_{n,k}= \iota (B_{n,k}^{(2)}) \oplus (J_{\ell}(N_{0})\cap B_{n,k}) \ .
\end{split}
\end{equation}
The maps $\ell$ and $\iota$  are $L_{*}$-equivariant: for $t\in L_{*}$,
\begin{equation} \label{Leq}
\ell \circ \varphi _{t}  = \varphi _{\ell (t)}\circ \ell   \quad,  \quad \iota \circ \varphi _{\ell (t) }=\varphi_{ t} \circ \iota  \ ,
\end{equation}
thanks to the hypothesis b) made at the beginning of this chapter.
The map $\iota$ is  equivariant for the canonical action of the inverse monoid $L_{*}^{-1}$, but not the map $\ell$.

  \begin{lemma} \label{3.6} For $t\in L_{*}$, we have $\iota\circ \psi_{\ell (t)} =\psi_{t} \circ  \iota $.
  We have  $\ell \circ \psi_{t} =\psi_{\ell (t)} \circ  \ell  $ if and only if $N_{\ell} = tN_{\ell}t^{-1}$.
  \end{lemma}

\begin{proof}
Clearly  $N_{0}= N_{\ell}\rtimes \iota (N_{0}^{(2)}) $ and  $tN_{0}t^{-1} =tN_{\ell}t^{-1} \rtimes  t\iota(N_{0}^{(2)})t^{-1}$ for $t\in L $.   We  choose, as we may, for $t\in L_{\ell, \iota}$, a system $J(N_{0}/tN_{0}t^{-1})$ of representatives of $N_{0}/tN_{0}t^{-1}$ containing $1$ such  that
\begin{equation}\label{J}
J(N_{0}/tN_{0}t^{-1})=\{u\iota (v) \ | u\in J(N_{\ell}/tN_{\ell}t^{-1}) \ , \ v\in J(N_{0}^{(2)}/\ell(t)N_{0}^{(2)}\ell (t^{-1}))\} \ .
\end{equation}

We have
 $\iota\circ \psi_{\ell (t)} =\psi_{t} \circ  \iota $ because, for $\lambda \in {\cal O}_{\cal E}$, we have on one hand  \eqref{writing}
 $$\lambda=  \sum _{v\in J(N_{0}^{(2)}/\ell(t) N_{0}^{(2)}\ell (t)^{-1})} v   \varphi_{\ell (t)}( \lambda _{v, \ell (t)}) \ \  , \ \   \lambda _{v, \ell (t)}=\psi_{\ell (t)}(v^{-1}\lambda)  \ , $$
 $$
    \iota(\lambda) =
    \sum _{v\in J(N_{0}^{(2)}/\ell(t) N_{0}^{(2)} \ell (t)^{-1})}\iota (v)  \varphi_{t} (\iota (\lambda _{v, \ell (t)})) \ ,
     $$
and on the other hand   \eqref{writing}
$$   \iota(\lambda) =
    \sum _{u\in J(N_{\ell} /t  N_{\ell}t^{-1}), v\in J(N_{0}^{(2)}/\ell(t) N_{0}^{(2)})\ell(t)^{-1})} u\iota(v) \varphi_{t} (\iota (\lambda) _{u \iota (v),t})
\    , $$
where
$ \iota (\lambda )_{u \iota(v),t} = \psi_{t}(\iota(v)^{-1}u^{-1}\iota (\lambda))$.
By the uniqueness of the decomposition,
$$\iota (\lambda) _{\iota (v),t}=\iota (\lambda _{v, \ell(t)}) \ , \ \iota (\lambda) _{u\iota(v),t}=0 \ \ {\rm if} \  u\neq 1 \ .
$$ Taking $u=1,v=1$, we get
$\psi_{t}( \iota (\lambda))=\iota (\psi_{\ell(t)}(\lambda))$.

A similar argument shows that  $\ell \circ \psi_{t} =\psi_{\ell (t)} \circ  \ell  $ if and only if $N_{\ell} = tN_{\ell}t^{-1}$. For $\lambda \in \Lambda_{\ell}(N_{0})$,
$$
 \lambda =\sum _{u  \in J(N_{0}/t  N_{0}t^{-1})}  u \varphi_{t} (\lambda _{u,t}) \ \ , \ \  \lambda _{u,t} =\psi_{t} (u^{-1}\lambda)\ ,
  $$
$$
  \ell (\lambda) =\sum _{u  \in J(N_{0}/t  N_{0}t^{-1})} \ell (u) \varphi_{\ell (t)}(\ell( \lambda _{u,t}))= \sum _{v\in J(N_{0}^{(2)}/ \ell (t) N_{0}^{(2)})}v \varphi_{\ell(t)}( \ell (\lambda)_{v, \ell(t)})
$$
By the uniqueness of the decomposition,
$$\ell (\lambda)_{v,\ell (t)} = \sum_{u\in J(N_{\ell} /t  N_{\ell}t^{-1})}  \ell( \lambda _{u\iota(v),t}) \ . $$
We deduce that $\ell \circ \psi_{t} =\psi_{\ell (t)}\circ \ell  $ if and only if   $ N_{\ell} =t  N_{\ell}t^{-1}$.
\end{proof}

\begin{remark}\label{3.7}{\rm  $\ell \circ \psi_{s}\neq \psi_{\ell (s)}\circ \ell$,   except in the trivial case where $\ell:N_{0} \to N_{0}^{(2)}$ is  an isomorphism, because $sN_{\ell}s^{-1}\neq N_{\ell}$ as the intersection of the decreasing sequence  $s^{k}N_{\ell}s^{-k}$  for $k\in \mathbb N$ is trivial.  }

\end{remark}

For future use, we note:

\begin{lemma} \label{lista}   The left or right $o[N_{0}]$-submodule   generated by $ \iota ({\cal O}_{\cal E})$  in $\Lambda_{\ell}(N_{0})$ is dense.
\end{lemma}

\begin{proof} As $o[N_{0}]$ is dense in $\Lambda (N_{0})$ it suffices to show that the  left or right $\Lambda(N_{0})$-submodule   generated by $ \iota ({\cal O}_{\cal E})$  in $\Lambda_{\ell}(N_{0})$ is dense. This will be shown even with respect to the $\mathcal M_{\ell}(N_{0})$-adic topology.

 Viewing $\lambda \in \Lambda_{\ell}(N_{0})$  as an infinite  Laurent series $\lambda =\sum_{n\in \mathbb Z} \lambda_{n}X^{n}$ with $\lambda_{n} \in \Lambda(N_{\ell})$  and $\lim_{n\to-\infty}\lambda_{n}=0$ in the  $\mathcal M (N_{\ell})$-adic  topology of $\Lambda(N_{\ell})$, and noting that the left,  resp. right,  $\Lambda (N_{0})$-submodule  of $ \Lambda_{\ell}(N_{0})$ generated by $ \iota ({\cal O}_{\cal E})$ contains $\Lambda (N_{0})X^{-m}$, resp. $X^{-m}\Lambda (N_{0})$, for any positive integer $m$,
we use that for each $n\in \mathbb N$ there exists $\mu_{n}$ in $\Lambda (N_{0})X^{-m}$, resp. $  X^{-m}\Lambda (N_{0})$, for some large $m$, such that  $\lambda - \mu_{n} \in {\mathcal M}_{\ell}(N_{0})^{n}$.
\end{proof}

Let $M$ be a finitely generated  $\Lambda_{\ell}(N_{0})$-module and
let $f: \oplus_{i=1}^n \Lambda_{\ell} (N_{0})\to M$ be  $\Lambda_{\ell} (N_{0})$-linear surjective map. We put on $M$ the quotient topology of the weak topology
on $ \oplus_{i=1}^n \Lambda_{\ell} (N_{0})$; this is  independent of the choice of $f$. Then $M$ is a Hausdorff and complete topological $\Lambda_{\ell}(N_{0})$-module, every submodule is closed (\cite{SVig} Lemma 8.22). In the same way we can equip $M$ with the pseudocompact topology. Again $M$ is Hausdorff and complete and  every submodule is closed in the pseudocompact topology, because $\Lambda_{\ell} (N_{0})$ is noetherian. The weak topology on $M$ is weaker than the pseudocompact topology which is weaker that the $p$-adic topology. In particular the intersection of the submodules $p^{n}M$ for $n\in \mathbb N$ is $0$. By \cite{Gab} IV.3.Prop. 10, $M$ is $p$-adically complete, i.e., the natural map
$M\to \varprojlim_{n} M/p^{n}M$
is bijective.

\bigskip Unless otherwise indicated, $M$ is always understood to carry the weak topology.

\begin{lemma}\label{redmod}
The  properties  a,b,c,d  of section \ref{topAM} are satisfied by $(o,M)$ and $M$ is complete.
\end{lemma}

\begin{definition}  A finitely generated module $M$ over $\Lambda_{\ell}(N_{0})$  with an  \'etale semilinear action of a submonoid $L'$ of $L_{\ell, +}$ is called an \'etale $L'$-module over $\Lambda_{\ell}(N_{0})$.
\end{definition}

 We denote by   $ {\mathfrak M}^{et}_{\Lambda_{\ell}(N_{0})}(L')$ the category of \'etale $L'$-modules on $\Lambda_{\ell}(N_{0})$.

\begin{lemma} The category  $ {\mathfrak M}^{et}_{\Lambda_{\ell}(N_{0})}(L') $  is
abelian.
\end{lemma}
\begin{proof}
As in the proof  of  Proposition \ref{abetale} and using  that the ring $\Lambda_{\ell}(N_{0})$ is noetherian.
\end{proof}

\bigskip The continuous homomorphism $\ell :L_{*}\to L_{+}^{(2)}$ defines an \'etale semilinear action of $L_{*} $ on the   ring $\Lambda_{id}(N_{0}^{(2)})$ isomorphic to $\cal O_{\cal E}$.

\begin{definition}  A finitely generated module $D$ over $\cal O_{\cal E}$  with an \'etale semilinear action of $L_{*}$ is called an \'etale $L_{*}$-module over $\cal O_{\cal E}$.
\end{definition}

An element $t\in L_*$ in the kernel $L_*^{\ell =1}$ of $\ell$ acts  trivially on $\mathcal O_{\mathcal E} $ hence bijectively on an \'etale $L_{*}$-module over $\cal O_{\cal E}$.

\begin{remark} {\rm The action of $L_*^{\ell =1}$ on $D$ extends to an action of  the subgroup of $L$ generated by $ L_*$ if $L_*^{\ell =1}$ is commutative or if we assume
that for each $t\in L_*^{\ell =1}$ there exists an integer $k>0$ such that $s^kt^{-1}\in L_* $. The assumption is trivially satisfied whenever $L_*=H\cap L_+$ for some subgroup $H\subset L$.

Indeed, the subgroup generated by $L_*^{\ell =1}$ is the set of words of the form $x_1^{\pm 1}\dots x_n^{\pm 1}$ with $x_i\in L_*^{\ell =1}$ for $i=1,\dots,n$. So if we have an action of all the elements and all the inverses, then we can take the products of these, as well. We need to show that this action is well defined, i.e., whenever we have a relation
\begin{equation}\label{relation}
x_1^{\pm 1}\dots x_n^{\pm 1}=y_1^{\pm 1}\dots y_r^{\pm 1}
\end{equation}
in the group then the action we just defined is the same using the $x$'s or the $y$'s.
If $L_*^{\ell =1}$ is commutative, this is easily checked.
In the second case, we can choose a big enough $k=\sum_{i=1}^n k_i+\sum_{j=1}^r k_j$ such that $s^{k_i}x_i^{-1}\in L_*$ and $s^{k_j}y_j^{-1}\in L_*$. Then multiplying the relation \eqref{relation} by $s^k$ we obtain a relation in $L_*$ so the two sides will define the same action on $D$. This shows that the actions defined using the two sides of \eqref{relation} are equal on $\varphi_s^k(D)\subset D$. However, they are also equal on group elements $u\in N_0^{(2)}$ hence on the whole $D=\bigoplus_{u\in J(N_0^{(2)}/\varphi_s^k(N_0^{(2)}))}u\varphi_s^k(D)$.} \end{remark}

We denote by   $ {\mathfrak M}^{et}_{{\cal O}_{\cal E, \ell}}(L_*)$ the category of  \'etale $L_{*}$-modules on ${\mathcal O}_{\mathcal E}$.

\begin{lemma} The category  $ {\mathfrak M}^{et}_{{\cal O}_{\cal E, \ell}}(L_*)$  is
abelian.
\end{lemma}
\begin{proof}
As in the proof  of  Proposition \ref{abetale} and using  that the ring $\cal O_{\cal E}$ is noetherian.
\end{proof}

We will prove later that the categories $ {\mathfrak M}^{et}_{{\cal O}_{\cal E, \ell}}(L_*)$ and $ {\mathfrak M}^{et}_{\Lambda_{\ell}(N_{0})}(L_{*}) $ are equivalent.

 \subsection{Base change functors  }

We recall a general argument of semilinear algebra (see \cite{SVig}). Let  $A$ be a ring with a ring endomorphism $\varphi_{A}$,  let $B$ be another ring with a ring endomorphism $\varphi_{B}$, and let  $f:A\to B$ be a ring homomorphism such that $f\circ \varphi_{A}=\varphi_{B}\circ f$. When  $M$ is an $A$-module  with a semilinear endomorphism $\varphi_{M} $,  its image by base change  is the     $B$-module $B \otimes_{A,f} M$ with the semilinear endomorphism  $\varphi_{B} \otimes \varphi_{M}$.
The endomorphism $\varphi_{M}$ of $M$ is called \'etale if the natural map
$$
a\otimes m \mapsto a \varphi_{M}(m): A \otimes_{A,\varphi_{A}} M\to M
$$
is bijective.

\begin{lemma}
When $\varphi_{M}$ is \'etale, then $\varphi_{B}\otimes \varphi_{M}$
is \'etale.
\end{lemma}
\begin{proof}
We have
\begin{equation*}
    B \otimes_{B,\varphi_{B}} (B\otimes_{A,f} M) =
    B \otimes_{A, \varphi_B \circ f} M = B \otimes_{f \circ \varphi_s} M = B \otimes_{A,f} (A \otimes_{A,\varphi_s} M) \cong B\otimes_{A,f} M .
\end{equation*}
\end{proof}

Applying these general considerations to the  $L_{*}$-equivariant maps $\ell: \Lambda_{\ell}(N_{0})\to {\cal O}_{\cal E} $ and $\iota:{\cal O}_{\cal E} \to  \Lambda_{\ell}(N_{0})$ satisfying
$\ell \circ \iota = \id$ (see \eqref{elliota}, \eqref{Leq}). We   have the base change functors
$$M \mapsto  \mathbb D(M):= {\cal O}_{\cal E} \otimes_{\Lambda_{\ell}(N_{0}), \ell }M
$$
from the category of $\Lambda_{\ell}(N_{0})$-modules to the category of ${\cal O}_{\cal E}$-modules, and
$$D \mapsto \mathbb M(D):= \Lambda_{\ell} (N_{0}) \otimes_{ {\cal O}_{\cal E} , \iota }D
$$
in the opposite direction.  Obviously these base change functors respect the property of being finitely generated. By the general lemma we obtain:

\begin{proposition}
The above functors  restrict to  functors
$$\mathbb D\colon {\mathfrak M}^{et}_{\Lambda_{\ell}(N_{0})}(L_{*})\to  {\mathfrak M}^{et}_{{\cal O}_{\cal E, \ell}}(L_{*})
 \quad \text{and} \quad \mathbb M\colon {\mathfrak M}^{et}_{{\cal O}_{\cal E, \ell}}(L_{*}) \to  {\mathfrak M}^{et}_{\Lambda_{\ell}(N_{0})}(L_{*}) \ .
$$
\end{proposition}

 When $M \in {\mathfrak M}^{et}_{\Lambda_{\ell}(N_{0})}(L_{*})$, the diagonal action of $L_{*}$ on $ \mathbb D(M)$ is:
\begin{equation}\label{A1}
 \varphi_{t}(\mu \otimes m)=
\varphi_{\ell (t)}(\mu) \otimes \varphi_{t}(m)  \ \ {\rm for}\ t\in L_{*}, \mu\in {\cal O}_{\cal E} ,m\in M,
\end{equation}
When  $D\in {\mathfrak M}^{et}_{{\cal O}_{\cal E, \ell}}(L_{*})$, the diagonal action  of $M_{*}$ on $ \mathbb M(D)$ is:
 \begin{equation}\label{A2}
 \varphi_{t} (\lambda \otimes d)=  \varphi_{t}(\lambda) \otimes  \varphi_{t} (d )
 \ \ {\rm for} \ t\in L_{*}, \lambda\in \Lambda_{\ell} (N_{0}), d\in D \ .
\end{equation}
  The  natural map
$$\ell_{M} :M\to \mathbb D(M) \ \ , \ \  \ell_{M} (m) = 1\otimes m
$$  is surjective,  $L_{*}$-equivariant,  with a  $P_{*}$-stable kernel  $M_{\ell}:=J_{\ell}(N_{0})M$.
 The injective  $L_{*}$-equivariant map
 $$\iota_{D}: D\to \mathbb M(D) \ \ , \ \  \iota_{D} (d) = 1\otimes d$$
  is  $\psi_{t} $-equivariant for $t\in L_{*}$ (same proof as Lemma \ref{3.6}).

\bigskip

For future use we note the following property.

\begin{lemma}\label{coe}
Let $d\in D$ and $t\in L_{*}$. We have
\begin{equation*}
    \psi_{t}(u^{-1} \iota_{D}( d)) =
    \begin{cases}
    \iota _{D}(\psi_{t}(v^{-1}d) ) &  \text{if $u = \iota(v)$ with $v \in N_{0}^{(2)}$}, \\
    0 & \text{if $u\in N_{0} \setminus \iota(N_{0}^{(2)})tN_{0}t^{-1}$}.
    \end{cases}
\end{equation*}
\end{lemma}
\begin{proof}
We choose a set $J \subset N_{0}^{(2)}$ of representatives  for the cosets in $N_{0}^{(2)}/\ell (t) N_{0}^{(2)}\ell(t)^{-1}$. The semilinear endomorphism $\varphi_{t}$ of $D$ is \'etale hence
$$
  d = \sum_{v \in J} v \varphi_{t}(d_{v,t}) \quad \text{where $d_{v,t}=\psi_{t}(v^{-1}d) $ }.
$$
Applying $\iota_{D}$ we obtain
$$
 \iota_{D}( d)=\sum_{v} \iota(v)  \iota _{D}(\varphi_{t}(d_{v,t}))= \sum_{v} \iota(v) \varphi_{t}( \iota _{D}(d_{v,t})) =
 \sum_{v} \iota(v) \varphi_{t}(\psi_t( \iota _{D}(v^{-1}d))) \ .
$$
The map $\iota $ induces an injective map from $J$ into $N_{0}/tN_{0}t^{-1}$ with image included in a set $J(N_{0}/tN_{0}t^{-1})\subset N_{0}$ of representatives  for the cosets in $N_{0}/tN_{0}t^{-1}$. As the action  $\varphi_{t}$ of $t$ in $\mathbb M(D)$ is \'etale, we have \eqref{writing}
$$
 m=\sum_{u\in J(N_{0}/tN_{0}t^{-1})} u \varphi_{t}(m_{u,t})) \quad \text{where $m_{u,t}=\psi_t(u^{-1}m) $  }
$$
for any $m \in \mathbb M(D)$. We deduce that $\psi_{t}(\iota (v^{-1}) \iota_{D}( d))=\iota _{D}(d_{v,t})$ when $v\in J$ and  $\psi_{t}(u^{-1}\iota_{D}( d))=0$ when $u \in J(N_{0}/tN_{0}t^{-1}) \setminus \iota(J)$.
As any element of $N_{0}^{(2)}$ can belong to a set of  representatives  of $N_{0}^{(2)}/\ell (t) N_{0}^{(2)}\ell(t)^{-1}$, we deduce that
$\psi_{t}(\iota (v^{-1}) \iota_{D}( d))=\iota _{D}(d_{v,t})$ for any $v\in N_{0}^{(2)}$.  For the same reason
$\psi_{t}(\iota (u^{-1}) \iota_{D}( d))=0$ for any $u\in N_{0}$ which does not belong to $\iota(N_{0}^{(2)})tN_{0}t^{-1}$.
\end{proof}

\subsection{Equivalence of categories}

Let  $D\in  {\mathfrak M}^{et}_{{\cal O}_{\cal E},\ell}(L_{*})$.    By definition $\mathbb D(\mathbb M(D))={\cal O}_{\cal E} \otimes_{\Lambda_{\ell} (N_{0}), \ell} (\Lambda_{\ell} (N_{0})\otimes_{ {\cal O}_{\cal E}, \iota} D)$, and we have a natural map
$$\mu \otimes (\lambda \otimes d)\mapsto \mu \ell(\lambda) d :   {\cal O}_{\cal E} \otimes_{\Lambda_{\ell} (N_{0}), \ell} (\Lambda_{\ell} (N_{0})\otimes_{ {\cal O}_{\cal E}, \iota} D)  \to D  \ . $$

\begin{proposition} \label{easy}  The natural  map $\mathbb D(\mathbb M(D)) \to D$ is an isomorphism
in $ {\mathfrak M}^{et}_{{\cal O}_{\cal E},\ell}(L_{*})$.
\end{proposition}
\begin{proof}   The natural map
 is  bijective because  $\ell \circ \iota = \id: {\cal O}_{\cal E}\to \Lambda_{\ell} (N_{0}) \to  {\cal O}_{\cal E}$,  and $L_{*}$-equivariant because
the action of $t\in L_{*}$ satisfies
$$\varphi_{t}(\mu \otimes (\lambda \otimes d))=\varphi_{\ell (t)}(\mu)   \otimes  \varphi_{t} (\lambda \otimes d) = \varphi_{\ell (t)}(\mu)   \otimes (\varphi_{t}(\lambda) \otimes \varphi_{t}(d)) \ ,
$$
$$
\varphi_{t}(\mu \ell(\lambda) d)= \varphi_{\ell (t)}( \mu (\ell(\lambda))\varphi_{t}( d) = \varphi_{\ell (t)}( \mu) \ell(\varphi_{ t}(\lambda))\varphi_{t}( d) \ ,$$
by (\ref{A1}), (\ref{A2}).
\end{proof}

 The kernel $N_{\ell}$ of
$\ell:N_{0}\to \mathbb Z_{p}$  being  a closed subgroup of $N_{0}$ is also a $p$-adic Lie group, hence  contains an  open pro-$p$-subgroup $H$  with the following property (\cite{S} Remark 26.9 and Thm. 27.1):

 For any integer $n\geq 1$, the map $h\mapsto h^{p^{n}}$ is an homeomorphism of $H$ onto an open subgroup $H_{n}\subseteq H$,
and $(H_{n})_{n\geq 1}$ is a fundamental system of open neighborhoods of $1$ in $H$.

The groups $s^{k}N_{\ell}s^{-k}$ for $k\geq 1$ are open and form a fundamental system of neighborhoods of $1$ in $N_{\ell}$.
  For any integer $n \geq 1$ there exists a positive  integer $k$ such that any element in
$ s^{k}N_{\ell}s^{-k}$   is  contained in $H_{n}$, hence is a  $p^{n}$-th power  of some element in $N_{\ell}$. We denote by $k_{n}$ the smallest positive  integer such that
any element in
$ s^{k_{n}}N_{\ell}s^{-k_{n}}$  is a  $p^{n}$-th power  of some element in $N_{\ell}$.

\begin{lemma} \label{4.6} For any positive integers $n $ and $k\geq k_{n}  $, we have
   $$
  \varphi ^{k}(J_{\ell}(N_{0})) \subset {\mathcal M}_{\ell}(N_{0})^{n+1}
  \ .
  $$
\end{lemma}

\begin{proof}   For $u\in N_{\ell}$,
 and  $j\in \mathbb N$, the value at $u$ of the
  $p^{j}$-th cyclotomic polynomial $\Phi_{p^{j}}(u)$ lies in  ${\mathcal M}_{\ell}(N_{0})$ and
  $$u^{p^{n}}-1 =\prod_{j=0}^{n}\Phi_{p^{j}}(u)$$
lies in   ${\mathcal M}_{\ell}(N_{0})^{n+1}$. An element  $v \in  s^{k}N_{\ell}s^{-k}$ is a $p^{n}$-th power  of some element in $N_{\ell}$ hence $v-1$ lies in   ${\mathcal M}_{\ell}(N_{0})^{n+1}$.
The ideal $J_{\ell}(N_{0})$ of $\Lambda_{\ell}(N_{0})$ is generated by $u-1$ for $u\in N_{\ell}$ and  $\varphi ^{k}(J_{\ell}(N_{0}))$ is contained in the ideal generated by $v-1$ for
  $v \in  s^{k}N_{\ell}s^{-k}$.   As  ${\mathcal M}_{\ell}(N_{0}) $ is an ideal of $\Lambda_{\ell}(N_{0})$ we deduce that  $\varphi ^{k}(J_{\ell}(N_{0}))\subset {\mathcal M}_{\ell}(N_{0})^{n+1} $.
\end{proof}

\begin{lemma}\label{fully-faithful}
\begin{itemize}
  \item[i.] The functor $\mathbb{D}$ is faithful.
  \item[ii.] The functor $\mathbb{M}$ is fully faithful.
\end{itemize}
\end{lemma}

\begin{proof}
Obviously ii. follows from i. by proposition \ref{easy}. To prove i. let $f: M_{1} \to M_{2}$ be a morphism in ${\mathfrak M}^{et}_{\Lambda_{\ell} (N_{0})}(L_*)$ such that $\mathbb D(f) =0$, i.\ e., such that $f(M_1) \in J_{\ell}(N_{0}) M_{2}$. Since $M_1$ is \'etale we deduce that $f(M_1) \subseteq \bigcap_{k} \varphi^k(J_{\ell}(N_{0})) M_{2}$ and hence, by lemma \ref{4.6}, in $\bigcap_{n} {\mathcal M}_{\ell}(N_{0})^{n} M_2$. Since the pseudocompact topology on $M_2$ is Hausdorff we have $\bigcap_{n} {\mathcal M}_{\ell}(N_{0})^{n} M_2 = 0$. It follows that $f = 0$.
\end{proof}

Let $M \in {\mathfrak M}^{et}_{\Lambda_{\ell} (N_{0})}(L_{*})$. By definition,
$$
\mathbb  M \mathbb D(M)= \Lambda_{\ell} (N_{0}) \otimes_{{\mathcal O}_{\mathcal E}, \iota} ({\mathcal O}_{\mathcal E}\otimes_{\Lambda_{\ell} (N_{0}), \ell} M ) =\Lambda_{\ell} (N_{0}) \otimes_{\Lambda_{\ell} (N_{0}), \iota \circ \ell} M  \ .
$$

In the particular case where $L_{*}=s^{\mathbb N}$ is the monoid generated by $s$,  we denote the category  $ {\mathfrak M}^{et}_{\Lambda_{\ell}(N_{0})}(L_*)$  (resp.  $ {\mathfrak M}^{et}_{{\cal O}_{\cal E, \ell}}(L_*)$),
by $ {\mathfrak M}^{et}_{\Lambda_{\ell}(N_{0})}(\varphi) $ (resp. $ {\mathfrak M}^{et}_{{\cal O}_{\cal E, \ell}}(\varphi)$).
The category  $ {\mathfrak M}^{et}_{\Lambda_{\ell}(N_{0})}(L_*)$  (resp.  $ {\mathfrak M}^{et}_{{\cal O}_{\cal E, \ell}}(L_*)$) is a subcategory of ${\mathfrak M}^{et}_{\Lambda_{\ell}(N_{0})}(\varphi)$ (resp.  $ {\mathfrak M}^{et}_{{\cal O}_{\cal E, \ell}}(\varphi)$).

\begin{proposition}\label{hard}
For any $M \in {\mathfrak M}^{et}_{\Lambda_{\ell}(N_{0})}(\varphi)$ there is a unique morphism
\begin{equation*}
    \Theta_M : M \to \mathbb{M}\mathbb{D}(M) \qquad\text{in ${\mathfrak M}^{et}_{\Lambda_{\ell}(N_{0})}(\varphi)$}
\end{equation*}
such that the composed map $\mathbb{D}'(\Theta_M) : \mathbb{D}(M) \xrightarrow{\mathbb{D}(\Theta_M)} \mathbb{D}\mathbb{M}\mathbb{D}(M) \cong \mathbb{D}(M)$ is the identity. The morphism $\Theta_M$, in fact, is an isomorphism.
\end{proposition}
\begin{proof}
The uniqueness follows immediately from Lemma \ref{fully-faithful}.i. The construction of  such an isomorphism $\Theta_M$ will be done in three steps.

\textit{Step 1:} We assume that $M$ is free over $\Lambda_{\ell} (N_{0})$, and we start with an arbitrary finite $\Lambda_{\ell} (N_{0})$-basis $(\epsilon_{i})_{i\in I}$ of $M$.    By (\ref{Ldeco}), we have
$$
 M=(\oplus_{i\in I} \iota ({\mathcal O}_{\mathcal E})\epsilon_{i})\oplus (\oplus_{i\in I} J_{\ell}(N_{0})\epsilon_{i}) \ .
$$
The $\Lambda_{\ell} (N_{0})$-linear map from $M$ to $\mathbb  M \mathbb D(M)$ sending $\epsilon_{i}$ to $1\otimes (1\otimes  \epsilon_{i})$   is bijective.  If  $\oplus_{i\in I} \iota ({\mathcal O}_{\mathcal E})\epsilon_{i}$ is $\varphi$-stable, the map is also  $\varphi$-equivariant and is an isomorphism in the category ${\mathfrak M}^{et}_{\Lambda_{\ell} (N_{0})}(\varphi)$. We will construct a   $\Lambda_{\ell} (N_{0})$-basis $(\eta_{i})_{i\in I}$ of $M$ such that $\oplus_{i\in I} \iota ({\mathcal O}_{\mathcal E})\eta_{i}$ is $\varphi$-stable.

We have
$$
\varphi(\epsilon_{i})=\sum_{j\in I}(a_{i,j} +b_{i,j}) \epsilon_{j} \ \ {\rm where } \ a_{i,j}
\in \iota ({\cal O}_{\cal E}) \ , \ b_{i,j}\in J_{\ell}(N_{0})  \ .
$$
If  the $b_{i,j}$ are not all $0$, we will show that there exist elements $ x_{i,j}\in J_{\ell} (N_{0})$  such that  $(\eta_{i})_{i\in I}$ defined by
$$
\eta_{i} := \epsilon_{i}     +\sum_{j\in I}  x_{i,j} \epsilon_{j} \ ,
$$
satisfies $\varphi (\eta_{i}) = \sum_{j\in I}a_{i,j}\eta_{j}  $ for $ i\in I$. By the Nakayama lemma (\cite{BAC} II \S3.2 Prop.\ 5), the set $(\eta_{i})_{i\in I}$ is  a $\Lambda_{\ell} (N_{0})$-basis of $M$, and we obtain an isomorphism   in ${\mathfrak M}^{et}_{\Lambda_{\ell} (N_{0})}(\varphi)$,
$$
\Theta_M \ : \ M \to \mathbb  M\mathbb D(M) \  \ , \ \ \Theta (\eta_{i})=1\otimes (1 \otimes \eta_{i}) \ \ {\rm for} \ i\in I \  , $$
such that $\mathbb D'(\Theta_M)$ is the identity morphism of $\mathbb D(M)$.

The conditions on the matrix    $X:=(x_{i,j})_{i,j\in I}$  are :
    $$\varphi (\Id+X) (A+B) = A (\Id +X) \   $$
   for the matrices $A:=(a_{i,j})_{i,j\in I} \ , B:=(b_{i,j})_{i,j\in I}$.
The coefficients  of  $A$ belong to the commutative ring $\iota ({\cal O}_{\cal E}) $. The matrix $A$ is invertible because the $\Lambda_{\ell} (N_{0})$-endomorphism $f$ of $M$ defined by
$$
  f(\epsilon_{i})=\varphi (\epsilon_{i})  \ \ {\rm for } \ i\in I
$$
is an automorphism of $M$ as $\varphi$ is \'etale. We  have to solve the equation
$$
  A^{-1}B+A^{-1} \varphi (X) (A+B ) = X \ .
$$
For any $k \geq 0$ define
$$
U_{k} =A^{-1}\varphi  (A^{-1})\ldots   \varphi ^{k-1}(A^{-1}) \, \varphi ^{k}(A^{-1}B) \,   \varphi ^{k-1}(A+B) \ldots  \varphi  (A+B)(A+ B)  \ .
$$
We have
$$
A^{-1} \varphi(U_k) (A+B) = U_{k+1} \ .
$$
Hence $X := \sum_{k \geq 0} U_k$ is a solution of our equation provided this series converges with respect to the pseudocompact topology of $\Lambda_{\ell} (N_{0})$. The coefficients of $A^{-1}B $ belong to the two-sided ideal $J_{\ell}(N_{0})$ of $\Lambda_{\ell} (N_{0})$. Therefore the coefficients of  $U_{k}$
belong to the two-sided ideal generated by $\varphi^{k}(J_{\ell}(N_{0}))$.  Hence the series converges (Lemma \ref{4.6}). The coefficients of every term in the series belong to $J_{\ell}(N_{0})$ and  $J_{\ell}(N_{0})$ is closed in $\Lambda_{\ell} (N_{0})$, hence $x_{i,j}\in J_{\ell}(N_{0})$ for $i,j\in I$.

\textit{Step 2:} We show that any module $M$ in ${\mathfrak M}^{et}_{\Lambda_{\ell} (N_{0})}(\varphi)$ is the quotient of another module $M_1$ in ${\mathfrak M}^{et}_{\Lambda_{\ell} (N_{0})}(\varphi)$ which is  free over $\Lambda_{\ell}(N_{0})$ .

Let $(m_{i})_{i\in I}$ be a minimal finite system of generators of the $\Lambda_{\ell} (N_{0})$-module $M$. As $\varphi$ is \'etale, $(\varphi(m_{i}))_{i\in I}$ is also a minimal system of generators. We denote by $(e_{i})_{i\in I}$  the canonical $\Lambda_{\ell} (N_{0})$-basis of $\oplus_{i\in I}\Lambda_{\ell} (N_{0})$, and we consider the two surjective $\Lambda_{\ell} (N_{0})$-linear maps
$$f , g   : \oplus_{i\in I}\Lambda_{\ell} (N_{0}) \to M \ \ ,  \ \ f(e_{i})=m_{i}\ , \ g(e_{i})=\varphi (m_{i}) \ .
$$
In particular, we find elements $m'_i \in M$, for $i \in I$, such that $g(m'_i) = \varphi(m_i)$. By the Nakayama lemma (\cite{BAC} II \S3.2 Prop.\ 5) the $(m'_i)_{i \in I}$ form another $\Lambda_{\ell} (N_{0})$-basis of $\oplus_{i\in I} \Lambda_{\ell} (N_{0})$. The $\varphi$-linear map
$$
  \oplus_{i\in I}\Lambda_{\ell} (N_{0}) \to  \oplus_{i\in I}\Lambda_{\ell} (N_{0}) \ \ , \ \
\varphi  (\sum_{i\in I}\lambda_{i}e_{i}):= \sum_{i\in I}\varphi(\lambda_{i}) m'_i
$$
therefore is \'etale. With this map,  $M_1 := \oplus_{i\in I}\Lambda_{\ell} (N_{0}) $ is a module in ${\mathfrak M}^{et}_{\Lambda_{\ell} (N_{0})}(\varphi)$ which is  free over $\Lambda_{\ell}(N_{0})$, and the surjective map $f$ is a morphism in ${\mathfrak M}^{et}_{\Lambda_{\ell} (N_{0})}(\varphi)$.

\textit{Step 3:} As $\Lambda_{\ell} (N_{0})$ is noetherian,  we deduce from Step 2 that for any module  $M$ in $ {\mathfrak M}^{et}_{\Lambda_{\ell} (N_{0})}(\varphi)$ we have an exact sequence
\begin{equation*}
    M_2 \xrightarrow{f} M_1 \xrightarrow{f'} M \to 0
\end{equation*}
in $ {\mathfrak M}^{et}_{\Lambda_{\ell} (N_{0})}(\varphi)$ such that $M_1$ and $M_2$ are free over $\Lambda_{\ell} (N_{0})$. We now consider the diagram
$$
\xymatrix{
\mathbb  M  \mathbb D (M_2) \ar[r]^{\mathbb  M \mathbb D(f)}&
\mathbb  M  \mathbb D ( M_1 ) \ar[r] ^{\mathbb  M  \mathbb D  (f')}&
\mathbb  M  \mathbb D ( M  ) \ar[r] &0  \\
M_{2} \ar[u]^{\Theta_{M_2}}_{\cong} \ar[r]^{f} & M_{1} \ar[u]^{\Theta_{M_1}}_{\cong} \ar[r]^{f'}&M \ar@{-->}[u]^{\Theta_M} \ar[r] &0}.
 $$
Since the functors $\mathbb M$ and $\mathbb D$ are right exact both rows of the diagram are exact. By Step 1 the left two vertical maps exist and are isomorphisms. Since
$$
\mathbb D(\mathbb  M \mathbb D(f) \circ \Theta_{M_2} - \Theta_{M_1} \circ f) =
\mathbb{D}(f) \circ \mathbb D'(\Theta_{M_2}) - \mathbb D'(\Theta_{M_1}) \circ \mathbb{D}(f) = 0
$$
it follows from lemma \ref{fully-faithful}.i that the left square of the diagram commutes. Hence we obtain an induced isomorphism $\Theta_M$ as indicated, which moreover by construction satisfies $\mathbb{D}'(\Theta_M) = \id_{\mathbb{D}(M)}$.
\end{proof}

\begin{theorem}\label{eq}  The  functors
$$ \mathbb M:  {\mathfrak M}^{et}_{{\cal O}_{\cal E}, \ell}(L_*)\to  {\mathfrak M}^{et}_{\Lambda_{\ell}(N_{0})}(L_*) \ \  ,
  \ \  \mathbb D:
 {\mathfrak M}^{et}_{\Lambda_{\ell}(N_{0})}(L_*)\to  {\mathfrak M}^{et}_{{\cal O}_{\cal E}, \ell}(L_*)\ , $$
are quasi-inverse equivalences of categories.
\end{theorem}

\begin{proof} By proposition \ref{easy} and lemma \ref{fully-faithful}.ii it remains to show that the functor $\mathbb{M}$ is essentially surjective. Let $M \in {\mathfrak M}^{et}_{\Lambda_{\ell} (N_{0})}(L_{*})$. We have to find a $D \in  {\mathfrak M}^{et}_{{\cal O}_{\cal E, \ell}}(L_*)$ together with an isomorphism $M \cong \mathbb{M}(D)$ in ${\mathfrak M}^{et}_{\Lambda_{\ell} (N_{0})}(L_{*})$. It suffices to show that the morphism $\Theta_M$ in proposition \ref{hard} is $L_*$-equivariant.

We want to prove that $(\Theta_M \circ \varphi_{t} - \varphi_{t}\circ \Theta_M)(m)=0$ for any $m\in M$ and $t\in L_{*}$. Since $\mathbb D'(\Theta_M) = \id_{\mathbb{D}(M)}$ we certainly have $(\Theta\circ \varphi_{t} - \varphi_{t}\circ \Theta)(m)\in J_{\ell}(N_{0}) \mathbb{M}\mathbb{D}(M)$ for any $m \in M$ and $t\in L_{*}$. We choose for any positive integer $r $ a set $J(N_{0}/N_{r})\subseteq N_{0}$ of representatives for the cosets in $N_{0}/N_{r}$. Writing \eqref{writing}
\begin{equation*}
m=\sum_{u\in J(N_{0}/N_{r})}u \varphi^{r}(m_{u,s^r})  \ ,\ m_{u,s^r}=\psi^{r}(u^{-1}m)
\end{equation*}
and using that $st=ts$ we see that
$$
(\Theta_M\circ \varphi_{t} - \varphi_{t}\circ \Theta_M)(m)=\sum_{u\in J(N_{0}/N_{r})}\varphi_{t}(u) \varphi^{r} ((\Theta_M\circ \varphi_{t} - \varphi_{t}\circ \Theta_M)(m_{u,s^r}))
$$
lies, for any $r$, in the  $\Lambda_{\ell}(N_{0} )$-submodule of $\mathbb{M}\mathbb{D}(M)$ generated by $\varphi^{r}(J_{\ell}(N_{0})) \mathbb{M}\mathbb{D}(M)$. As in the proof of lemma \ref{fully-faithful}.ii we obtain $\bigcap_{r > 0} \varphi^{r}(J_{\ell}(N_{0})) \mathbb{M}\mathbb{D}(M) = 0$.
\end{proof}

Since the functors  $ \mathbb M$ and $ \mathbb D$ are right exact they commute with the reduction modulo $p^{n}$, for any integer $n\geq 1$.

\subsection{Continuity}

 {\sl In this section we assume that $L_*$  contains a subgroup $L_1$ which is open in $L_*$ and is a topologically finitely generated pro-$p$-group.}

We will show that the $L_*$-action on any \'etale $L_*$-module over $\Lambda_\ell(N_0)$ is automatically continuous.  Our proof is highly indirect so that we temporarily we will have to make some definitions. But first a few partial results can be established directly.

\bigskip
Let $M$ be a finitely generated $\Lambda_{\ell}(N_{0})$-module.

\begin{definition} A   lattice in $M$ is a $\Lambda (N_{0})$-submodule of  $M$ generated by a finite system of generators of the $ \Lambda_{\ell} (N_{0})$-module $M$.
\end{definition}

The lattices of $M$  are of the form $M^{0} =  \sum_{i=1}^{r}\Lambda (N_{0})m_{i}$ for a   set $(m_{i})_{1\leq i \leq r}$ of generators of the $\Lambda (N_{0})$-module $M$.

   We have the three fundamental systems of neighborhoods of $0$  in $M$ :
\begin{align}\label{SFV0}
  (\sum_{i=1}^{r} O_{n,k}m_{i} &=\mathcal M_{\ell} (N_{0})^{n} M + \mathcal M (N_{0})^{k}M^{0})_{n,k\in \mathbb N}  \ , \\
 \label{SFV1}  (\sum_{i=1}^{r} B_{n,k}m_{i} & = \mathcal M_{\ell} (N_{0})^{n} M +X^{k}M^{0})_{n,k\in \mathbb N}  \ , \\
\label{SFV3}  (\sum_{i=1}^{r} C_{n,k}m_{i} & =\mathcal M_{\ell} (N_{0})^{n} M + M_{k}^{0})_{n,k\in \mathbb N}  \ ,
\end{align}
where $M^{0} _{k} $ is the lattice $\sum_{i=1} ^{r} \Lambda (N_{0})X^{k}m_{i}$, and    is different from  the set $X^{k}M_{0}$ when $N_{0}$ is not commutative.

If $M$ is an \'etale $L_*$-module over $\Lambda_\ell(N_0)$, for any fixed $t\in L_{\ell, +}$
 we have a fourth fundamental system of neighborhoods of $0$  in $M$ :
   \begin{equation*}
( \sum_{i=1}^{r} \varphi_{t}(O_{n,k}) \Lambda (N_{0}) \varphi_{t}(m_{i}))_{n,k\in \mathbb N}  \ ,
\end{equation*}
given by Lemma \ref{tOnk},  because $( \varphi_{t}( m_{i})_{1\leq i \leq r}$ is also a system of generators of the $\Lambda_{\ell} (N_{0} )$-module $M$.

\begin{proposition}\label{contphipsi-general}
Let $L'$ be a submonoid of $L_{\ell,+}$.  Let $M$ be an \'etale $L'$-module over $\Lambda_\ell(N_0)$. Then the maps $\varphi_t$ and $\psi_t$, for any $t \in L'$, are continuous on $M$.
\end{proposition}
\begin{proof}   The ring endomorphisms $\varphi_t$ of $\Lambda_\ell(N_0)$ are continuous since they preserve $\mathcal M (N_0)$ and $\mathcal M (N_\ell)$.
 The continuity of the $\varphi_t$ on $M$ follows as in part a) of the proof of proposition \ref{contphipsi}.
The continuity of the $\psi_t$ follows from
$$
\psi_{t}(\sum_{i=1}^{r} \varphi_{t}(O_{n,k}) \Lambda (N_{0}) \varphi_{t}(m_{i}) )=
\sum_{i=1}^{r} O_{n,k}\psi_{t} (\Lambda (N_{0})) \varphi_{t}(m_{i}) = \sum_{i=1}^{r}  O_{n,k}  m_{i} \ .
$$
\end{proof}

The same proof shows that, for any  $D \in {\mathfrak M}^{et}_{{\cal O}_{\cal E}, \ell}(L_*)$,   the maps $\varphi_t$ and $\psi_t$, for any $t \in L_*$, are continuous on $D$.

\begin{proposition}\label{pre-auto-cont} The $L_*$-action $L_*\times D\to D$ on an \'etale $L_*$-module $D$ over ${\cal O}_{\cal E}$    is continuous.
\end{proposition}
\begin{proof}
Let $D$ be in ${\mathfrak M}^{et}_{{\cal O}_{\cal E}, \ell}(L_*)$. Since we know already from Prop.\ \ref{contphipsi-general} that each individual $\varphi_t$, for $t \in L_*$, is a continuous map on $D$ and since $L_1$ is open in $L_*$ it suffices to show that the action $L_1 \times D \rightarrow D$ of $L_1$ on $D$ is continuous. As $D$ is $p$-adically complete with its weak topology being the projective limit of the weak topologies on the $D/p^n D$ we may further assume that $D$ is killed by a power of $p$. In this situation the weak topology on $D$ is locally compact. By Ellis' theorem (\cite{Ell} Thm.\ 1) we therefore are reduced to showing that the map $L_1 \times D \rightarrow D$ is separately continuous. Because of Prop.\ \ref{contphipsi-general} it, in fact, remains to prove that, for any $d \in D$, the map
$$
L_1 \longrightarrow D \ ,\ g \longmapsto gd
$$ is continuous at $1 \in L_1$. This amounts to finding, for any $d \in D$ and any lattice $D_0 \subset D$, an open subgroup $H \subset L_1$ such that $(H - 1)d \subset D_0$. We observe that $(X^m D_{++})_{m \in \mathbb{Z}}$ is a fundamental system of $L_1$-stable open neighbourhoods of zero in $D$ such that $\bigcup_{m} X^m D_{++} = D$. We now choose an $m \geq 0$ large enough such that $d \in X^{-m}D_{++}$ and $X^m D_{++} \subset D_0$. The $L_1$ action on $D$ induces an $L_1$-action on $X^{-m}D_{++} / X^{m}D_{++}$ which is $o$-linear hence given by a group homomorphism $L_1 \rightarrow \Aut_o(X^{-m}D_{++} / X^{m}D_{++})$. Since $D_{++}$ is a finitely generated $o[[X]]$-module which is killed by a power of $p$ we see that $X^{-m}D_{++} / X^{m}D_{++}$ is finite. It follows that the kernel $H$ of the above homomorphism is of finite index in $L_1$. Our assumption that $L_1$ is a topologically finitely generated pro-$p$-group finally implies, by a theorem of Serre (\cite{DDMS} Thm.\ 1.17), that $H$ is open in $L_1$. We obtain
$$
(H-1)d \subset (H-1)X^{-m} D_{++} \subset X^{m} D_{++} \subset D_0 \ .
$$
\end{proof}

 In the special case of classical $(\varphi, \Gamma)$-modules on ${\cal O}_{\cal E}$ the  proposition is stated as Exercise 2.4.6 in \cite{Ked} (with the indication of a totally different proof).

\begin{proposition}\label{cont-ring}
Let  $L'$  be a submonoid of $L_{\ell, +}$ containing an open subgroup $L_2$  which is a topologically finitely generated pro-$p$-group. Then the  $L'$-action  $L'\times \Lambda_\ell (N_0)\to \Lambda_\ell (N_0)$ on $\Lambda_\ell (N_0)$ is continuous.
\end{proposition}
\begin{proof}
 Since we know already from Prop. \ref{sw} and   \ref{contphipsi-general} that each individual $\varphi_t$, for $t \in L'$, is a continuous map on $\Lambda_\ell (N_0)$ and since $L_2$ is open in $L'$ it suffices to show that the action $L_2 \times  \Lambda_\ell (N_0) \rightarrow  \Lambda_\ell (N_0)$ of $L_2$ on $ \Lambda_\ell (N_0)$ is continuous. The ring $ \Lambda_\ell (N_0)$ is $ \mathcal M_\ell (N_0)$-adically complete with its weak topology being the projective limit of the weak topologies on the $ \Lambda_\ell (N_0)/ \mathcal M_\ell (N_0)^n  \Lambda_\ell (N_0)$. It suffices to prove that the induced action of $L_2$ on $\Lambda' = \Lambda_\ell (N_0)/ \mathcal M_\ell (N_0)^n  $ is continuous.
 The weak topology on $\Lambda'$ is locally compact since   $(B_k'=( X^k \Lambda (N_0)+\mathcal M_\ell (N_0)^n)/ \mathcal M_\ell (N_0)^n   )_{k\in \mathbb Z}$ forms a fundamental system of compact neighborhoods of $0$.
 By Ellis' theorem (\cite{Ell} Thm.\ 1) we therefore are reduced to showing that the map $L_2 \times \Lambda'\rightarrow \Lambda'$ is separately continuous. Because of Prop.\ \ref{contphipsi-general} it, in fact, remains to prove that, for any $x \in \Lambda'$, the map
$$
L_2 \longrightarrow \Lambda' \ ,\ g \longmapsto gx
$$ is continuous at $1 \in L_2$. This amounts to finding, for any $x \in \Lambda'$ and any large  $k\geq 1 $, an open subgroup $H \subset L_2$ such that $(H - 1)x\subset B'_k$. We observe that  the $ B'_k $, for $k \in \mathbb{Z}$, are $L_2$-stable   of union $\Lambda'$. We now choose an $m \geq k $ large enough such that $x \in B'_{-m}$. The $L_2$-action on $\Lambda'$ induces an $L_2$-action on $B'_{-m}  /  B'_m$ which is $o$-linear hence given by a group homomorphism $L_2 \rightarrow \Aut_o(B'_{-m}  /  B'_m)$. Since $B'_0$ is isomorphic to $o[[X]]\otimes_o
\Lambda (N_\ell)/\mathcal M(N_\ell)^n$
 as an $o[[X]]$-module, and $\Lambda (N_\ell)/\mathcal M(N_\ell)^n$ is finite,
 we see that $B'_{-m}  /  B'_m$ is finite. It follows that the kernel $H$ of the above homomorphism is of finite index in $L_2$. Our assumption that $L_2$ is a topologically finitely generated pro-$p$-group finally implies, by a theorem of Serre (\cite{DDMS} Thm.\ 1.17), that $H$ is open in $L_2$. We obtain
$$
(H-1)x \subset (H-1) B'_{-m}  \subset B'_m \subset B'_k \ .
$$
\end{proof}

\begin{lemma}\label{cont-li}
\begin{itemize}
   \item[i.] For any $M \in {\mathfrak M}^{et}_{\Lambda_\ell(N_0)}(L_*)$ the weak topology on $\mathbb{D}(M)$ is the quotient topology, via the surjection $\ell_M : M \to \mathbb{D}(M)$, of the weak topology on $M$.
   \item[ii.] For any $D \in {\mathfrak M}^{et}_{{\cal O}_{\cal E}, \ell}(L_*)$ the weak topology on $\mathbb{M}(D)$ induces, via the injection $\iota_D : D \to \mathbb{M}(D)$, the weak topology on $D$.
 \end{itemize}
\end{lemma}
\begin{proof}
i. If we write $M$ as a quotient of a finitely generated free $\Lambda_\ell(N_0)$-module then we obtain an exact commutative diagram of surjective maps of the form
$$
\xymatrix{
  \oplus_{i=1}^n \Lambda_\ell(N_0)  \ar@{>>}[d]_{\oplus_i \ell} \ar@{>>}[r] & M \ar@{>>}[d]^{\ell_M}  \\
  \oplus_{i=1}^n {\cal O}_{\cal E} \ar@{>>}[r] & \mathbb{D}(M)  }.
$$
The horizontal maps are continuous and open by the definition of the weak topology. The left vertical map is continuous and open by direct inspection of the open zero neighbourhoods $B_{n,k}$ (see \eqref{elioeq}). Hence the right vertical map $\ell_M$ is continuous and open.

ii. An analogous argument as for i. shows that $\iota_D$ is continuous. Moreover $\iota_D$ has the continuous left inverse $\ell_{\mathbb{M}(D)}$. Any continuous map with a continuous left inverse is a topological inclusion.
\end{proof}

An \'etale $L_*$-module $M$ over $\Lambda_\ell(N_0)$, resp.\ over ${\cal O}_{\cal E}$, will be called topologically \'etale if the $L_*$-action $L_* \times M \to M$ is continuous. Let ${\mathfrak M}^{et,c}_{\Lambda_\ell(N_0)}(L_*)$ and ${\mathfrak M}^{et,c}_{{\cal O}_{\cal E}, \ell}(L_*)$ denote the corresponding full subcategories of ${\mathfrak M}^{et}_{\Lambda_\ell(N_0)}(L_*)$ and ${\mathfrak M}^{et}_{{\cal O}_{\cal E}, \ell}(L_*)$, respectively. Note that, by construction, all morphisms in ${\mathfrak M}^{et}_{\Lambda_\ell(N_0)}(L_*)$ and in ${\mathfrak M}^{et}_{{\cal O}_{\cal E}, \ell}(L_*)$ are automatically continuous. Also note that by proposition \ref{contphipsi-general} any object in this categories is a complete topologically \'etale $o[N_0L_*]$-module in our earlier sense.

\begin{proposition}\label{cont-eq}
The  functors $\mathbb{M}$ and $\mathbb{D}$ restrict to quasi-inverse equivalences of categories
$$
\mathbb M:  {\mathfrak M}^{et,c}_{{\cal O}_{\cal E}, \ell}(L_*)\to  {\mathfrak M}^{et,c}_{\Lambda_{\ell}(N_{0})}(L_*) \ \  ,
  \ \  \mathbb D:
 {\mathfrak M}^{et,c}_{\Lambda_{\ell}(N_{0})}(L_*)\to  {\mathfrak M}^{et,c}_{{\cal O}_{\cal E}, \ell}(L_*) \ .
$$
\end{proposition}
\begin{proof}
It is immediate from lemma \ref{cont-li}.i that if $L_*$ acts continuously on $M \in {\mathfrak M}^{et}_{\Lambda_{\ell}(N_{0})}(L_*)$ then it also acts continuously on $\mathbb{D}(M)$.

 On the other hand,  let $D \in {\mathfrak M}^{et}_{{\cal O}_{\cal E}, \ell}(L_*)$ such that the action of $L_{*}$ on $D$ is continuous. We choose  a lattice $D_{0}$ in $D$ with a finite system $(d_{i})$  of generators. Given $t\in L_{*}$ we introduce $D_{t}:=\sum_{i}\Lambda (N_{0}^{(2)})t.d_{i} $ which is a lattice in $D$ since the action of $t$ on $D$ is \'etale.
Also $D_{0}+D_{t}$ is a lattice in $D$.
 The $\Lambda_\ell(N_0)$-module $\mathbb{M}(D)$ is  generated by $\iota_{D}(D _{0})$  as well as by $\iota_{D}(D _{0}+D_{t})$  and both
  $$
 (C_{n}\iota_{D}(D_{0}))_{n \in \mathbb N}\quad {\rm and} \quad (C_{n}\iota_{D}(D_{0}+D_{t}))_{n \in \mathbb N}
  $$
 are fundamental systems of neighbourhoods of $0$  in $\mathbb{M}(D)$ for the weak topology.  To show that the action of $L_{*}$ on $\mathbb{M}(D)$ is continuous, it suffices to find for any $t\in L_{*}, \lambda_{0}\in\Lambda_{\ell}(N_{0}), d_{0}\in D_{0}, n \in \mathbb N$ a neighborhood  $L_{t}\subset L_{*}$ of $t$ and  $n'\in \mathbb N$ such that
 \begin{equation}\label{0t}
L_{t}. (\lambda_{0} \iota_{D} (d_{0}) + C_{n'} \iota_{D }(D_{0} ))\subset t.\lambda_{0} \iota_{D} (d_{0}) + C_{n}\iota_{D }(D_{0}+D_{t} )\ .
 \end{equation}
   The three maps
   \begin{align*}
   \lambda \mapsto \lambda  \iota_{D} (d_{0}) & : \Lambda_{\ell}(N_{0} )\to \mathbb M (D) \  \\
    d\mapsto \lambda_{0} \iota_{D} (d) & : D \to  \mathbb M (D) \    \\
   (\lambda, d) \mapsto \lambda \iota_{D}(d) & : \Lambda_{\ell}(N_{0} \times D) \to  \mathbb M (D)\
     \end{align*}
are continuous  because $\iota_{D} $ is continuous.
  The action of $L_{*}$  on $D$  and on $\Lambda_{\ell}(N_{0})$  is continuous (Prop. \ref{cont-ring}). Altogether this implies   that  we can find a small  $L_{t}$ such  that
  $$L_{t}.\lambda_{0}  \iota_{D} (d_{0}) \subset t.\lambda_{0}  \iota_{D} (d_{0}) + C_{n}\iota_{D }(D_{0}+D_{t} )  \ .
  $$
  Since $\iota_{D}$ is $L_{*}$-equivariant we have, for any $n' \in \mathbb N$,
\begin{equation*}
L_{t}. C_{n'}\iota_{D}(D_{0} )=(L_{t}.  C_{n'}) \,  \iota_{D}(L_{t}.D_{0})\ .
\end{equation*}
 The continuity of the action of $L_{*}$ on $\Lambda_{\ell}(N_{0})$ shows that  $L_{t}.   C_{n'}\subset
  C_{n}$ when $L_{t}$ is small enough and $n'$ is large enough.

For $d \in D_{0}$ we have $L_{t}. \Lambda(N_{0}^{(2)})d  \subset \Lambda(N_{0}^{(2)}) (L_{t}.d ) $. The action of $ L_{*} $ on $D$ is continuous hence, for any $n'$, we can choose a small $L_{t}$ such that $L_{t}.d \subset t.d +C_{n'} ^{(2)}D _{0}$.
We can choose the same $L_{t} $ for each $d_{i}$ and we obtain
 $$L_{t}.D_{0} \subset  \sum_{i}\Lambda (N_{0}^{(2)})t.d_{i} +C_{n'} ^{(2)}D _{0} \ . $$
 Applying $\iota_{D}$, we obtain
$$\iota_{D}(L_{t}.D_{0}) \subset \iota_{D}(D_{t}) + C_{n'}\iota_{D}(D_{0})
$$
and then
$$
(L_{t}.  C_{n'}) \,  \iota_{D}(L_{t}.D_{0}) \subset C_{n}\iota_{D}(D_{t})+ C_{n}C_{n'}\iota_{D}(D_{0}) \ .
$$
We check that $C_{n}C_{n'} \subset C_{n,n+n'}\subset C_{n}$ when $n'\geq n $. Hence when $n'$ is large enough,
$$
L_{t}. (C_{n'}\iota_{D}(D_{0}) ) \subset C_{n}\iota_{D}(D_{t}+D_{0}) \ . $$
This ends the proof of \eqref{0t}.
\end{proof}

\begin{proposition}\label{auto-cont}
We have ${\mathfrak M}^{et,c}_{{\cal O}_{\cal E}, \ell}(L_*) = {\mathfrak M}^{et}_{{\cal O}_{\cal E}, \ell}(L_*)$ and ${\mathfrak M}^{et,c}_{\Lambda_{\ell}(N_{0})}(L_*) = {\mathfrak M}^{et,c}_{\Lambda_{\ell}(N_{0})}(L_*)$.
\end{proposition}
\begin{proof}
The first identity was shown in proposition \ref{pre-auto-cont}. The second identity follows from the first one together with theorem \ref{eq} and proposition \ref{cont-eq}.
\end{proof}

\begin{corollary}\label{top-etale}
Any \'etale $L_*$-module over $\Lambda_{\ell}(N_{0})$, resp.\ over ${\cal O}_{\cal E}$, is a complete topologically \'etale $o[N_0L_*]$-module in our sense.
\end{corollary}
\begin{proof} Use propositions \ref{contphipsi-general} and \ref{auto-cont}.
\end{proof}

\section{Convergence in  $L_{+}$-modules on $\Lambda_{\ell}(N_{0})$}\label{S9}

  {\sl In this section, we use the  notations of sections \ref{GM}  where we assume that $N$ is a $p$-adic Lie group. We assume that  $\ell$ and $\iota$ are continuous  group  homomorphisms
$$
\ell: P\to P^{(2)} \ , \ \iota :  N^{(2)}\to N \ , \  \ell \circ  \iota = \id \ , $$
such that
$  \ell(L_{+}) \subset L_{+}^{(2)} , \ \ell(N) = N^{(2)}
  $,  $(\iota\circ   \ell )(N_{0}) \subset N_{0}$, and
  \begin{equation}\label{fe} t \iota(y) t^{-1} = \iota (\ell (t)y\ell(t)^{-1}) \quad \text{for $y\in N^{(2)}, t\in L$} \ .
 \end{equation}

  }

 \bigskip  The assumptions  of Chapter \ref{GM} are naturally satisfied with  $L_{*}=L_{+}$. Indeed, the  compact open subgroup $N_{0}$ of $N$ is a compact $p$-adic Lie group, the  group $\ell (N_{0})$
  is a compact non-trivial subgroup $N_{0}^{(2)}$ of $N^{(2)}\simeq \mathbb Q_{p}$ hence  $N_{0}^{(2)}$  is isomorphic to $\mathbb Z_{p}$ and is open in $N^{(2)}$, the kernel of $\ell|_{N_{0}}$ is normalized by $L_{\ell,+}$.
 Note that $L_{+}$ normalizes $\iota(N_{0}^{(2)})$ since $\ell (L_{+}) $ normalises  $N_{0}^{(2)}$ and   \eqref{fe}.

\bigskip
Let  $M\in  {\mathfrak M}^{et}_{\Lambda_{\ell}(N_{0})}(L_{+})$ and  $D\in  {\mathfrak M}^{et}_{{\cal O}_{\cal E}, \ell}(L_{+})$ related by   the equivalence of categories (Thm. \ref{eq}),
$$ M =\Lambda_{\ell}(N_{0})\otimes_{\mathcal O_{\mathcal E}, \iota}D  =\Lambda_{\ell}(N_{0} )\iota_{D}(D)\ . $$

\bigskip We will exhibit in this chapter
  a special family $\mathfrak C _{s}$ of  compact subsets in $M$ such that $M(\mathfrak C_{s})$ is a dense $o$-submodule   of  $M$, and such that
  the $P$-equivariant sheaf on $\mathcal C$ associated to the \'etale $o[P_{+}]$-module $M(\mathfrak C_{s})$ by the theorem \ref{the11} extends to a $G$-equivariant sheaf on $G/P$.  We will follow the   method  explained   in subsection \ref{topAM} which reduces   the most technical part
  to the easier case where  $M$ is killed by a power of $p$.

\subsection{Bounded sets}\label{subsec:bounded}

   \begin{definition} A subset $A$ of $M$ is  called bounded  if for any open neighborhood $\cal B$  of $0$ in $M$ there exists an open neighborhood $B$  of $0$ in  $\Lambda_{\ell}(N_{0})$ such that
$$B  A   \subset {\cal B}  \ . $$
\end{definition}
Compare with (\cite{SVig} Def. 8.5.
The  properties satisfied by   bounded subsets of $M$  can be proved directly or deduced   from the properties of bounded subsets of
$\Lambda_{\ell}(N_{0})$ (\cite{War}  \S 12).
Using  the fundamental system   \eqref{SFV1} of neighborhoods of $0$, the set $A$ is bounded if and only if
 for any  large $n$ there exists $n' > n$ such that
 $$( {\mathcal M}_{\ell} (N_{0})^{n'} +X^{n'} \Lambda (N_{0} ))A   \subset {\mathcal M}_{\ell} (N_{0})^{n} M + X^{n} M^{0} \ , $$
 equivalently $X^{n'-n}A\subset  {\mathcal M}_{\ell} (N_{0})^{n} M +M^{0}$. We obtain (compare with (\cite{SVig}  Lemma 8.8):

\begin{lemma}\label{bdP}  A subset $A$ of $M$ is   bounded  if and only if for any large positive $n$ there exists a positive integer $n'$ such that
$$A \subset {\mathcal M}_{\ell} (N_{0})^{n} M + X^{-n'}M^{0} \ . $$
\end{lemma}

The following properties of bounded subsets will be used in the construction of a special family $\mathfrak C_{s}$ in the next subsection.
 \begin{itemize}
  \item[--]  Let $f:\oplus _{i=1}^{r}\Lambda_{\ell}(N_{0}) \to M$ be a surjective homomorphism of $\Lambda_{\ell}(N_{0})$-modules.
The image by $f$ of a bounded subset of $\oplus _{i=1}^{r}\Lambda_{\ell}(N_{0}) $
is a bounded subset of $M$.
  For $1\leq i \leq r$, the $i$-th projections $A_{i} \subset \Lambda_{\ell}(N_{0}) $ of a   subset  $A$ of   $\oplus_{i=1} ^{r}\Lambda_{\ell}(N_{0}) $ are all   bounded  if and only if $A$ is bounded.

 \item[--]  A compact subset   is bounded.

\item[--] The $\Lambda(N_{0}  )$-module generated by  a bounded subset is bounded.

\item[--] The closure  of a bounded subset   is bounded.

\item[--] Given a compact subset $C$ in $\Lambda_{\ell}(N_{0})$ and a bounded subset $A$ of $M$, the subset  $CA$ of $M$ is bounded.

\item[--]  The image of a bounded subset by  $f\in \End_{o}^{cont}(M)$  is bounded.    The image by $\ell_{M}$  of a bounded subset in $M$  is bounded in $D$.

\item[--]   A  subset $A$ of $D$ is bounded if and only if the  image $A_{n}$ of $A$ in $D/p^{n}D$ is bounded for all large  $n$.

\item[--]    When $D$ is killed by a power of $p$,   a  subset $A$ of $D$ is bounded if and only if
$A$ is contained in a lattice, i.e. if $A$ is contained in a compact subset (by  the properties of  lattices  given in Section \ref{CPG}).

 \end{itemize}

\begin{lemma}
The image by $\iota_{D}$  of a bounded subset   in $D$  is bounded in $M$.
\end{lemma}
\begin{proof} Let $A\subset D$ be a bounded subset and let $D^{0}$ be a fixed lattice in $D$. For all $n\in \mathbb N$ there exists $n'\in \mathbb N$ such that $A \subset p^{n}D + (X^{(2)})^{-n'}D^{0}$ by Lemma \ref{bdP}. Applying $\iota_{D}$
we obtain
$$\iota_{D}(A) \subset p^{n}\iota_{D}(D) + X^{-n'}\iota_{D}(D^{0})\subset \mathcal M_{\ell}(N_{0})^{n}M+ X^{-n'} M^{0}$$
where $M^{0}= \Lambda(N_{0})\iota_{D}(D^{0})$ is a lattice in $M$. By  the same lemma, this means that $\iota_{D}(A)$ is bounded in $M$.
\end{proof}

 \subsection{The module $M_{s}^{bd}$}

\begin{definition}  $M_{s}^{bd}  $ is the set of $m\in  M$ such that the  set of $\ell _{M}(\psi^{k}(u^{-1}m))$  for $k\in \mathbb N, u\in N_{0}$ is bounded in $D$.
\end{definition}

The definition of  $M_{s}^{bd}  $ depends on $s $ because   $\psi$  is the canonical left inverse of the action $\varphi$ of $s$ on $M$. We recognize $m_{u,s^k}=\psi^{k}(u^{-1}m)$ appearing in the expansion $(\ref{writing})$.

  \begin{proposition} \label{Lbd}  $M_{s}^{bd} $ is  an   \'etale $o[P_{+}]$-submodule  of $M$.
\end{proposition}

 \begin{proof}
  a)  We check first that   $M_{s}^{bd} $ is $P_{+}$-stable. As $M_{s}^{bd}$ is   $N_{0}$-stable and $P_{+}=N_{0}L_{+}$, it suffices to   show that  $tm=\varphi_{t}(m)\in M_{s}^{bd}$ when
$t\in L_{+} $ and  $m\in M_{s}^{bd}$.
 Using the expansion \eqref{writing} of $m$ and $st=ts$, for $k\in \mathbb N$ and  $ n_{0}\in N_{0}$,  we write $ \psi^{k}(n_{0}^{-1}  tm)$ as  the sum  over $u\in J(N_{0}/N_{k})$ of
$$
   \psi^{k}(n_{0}^{-1}tu \varphi^{k} (m_{u,s^k})) =  \psi^{k}(n_{0}^{-1} tut^{-1} \varphi^{k} (\varphi_{t}(m_{u,s^k})))
  =    \psi^{k}(n_{0}^{-1} tut^{-1})\varphi_{t}(m_{u,s^k}) \ ,
  $$
and  $\ell _{M}( \psi^{k}(n_{0}^{-1} \varphi_{t} (m)))$ as the sum  over $u\in J(N_{0}/N_{k})$ of
$$
\ell_{M} (\psi^{k}(n_{0}^{-1} tut^{-1}) \varphi_{t}(m_{u,s^k})) = v_{k,n_{0}}\ell _{M}( \varphi_{t}(m_{u,s^k}))=  v_{k,n_{0}}\varphi _{  t} (\ell_{M} (m_{u,s^k}) ) \ ,
$$
where
$v_{k,n_{0}}:=\ell (\psi^{k}(n_{0}^{-1} tut^{-1}))$ belongs to $N_{0}^{(2)}$ or is $0$. As $m\in M_{s}^{bd}$,
the set of $\ell _{M}(m_{u,s^k} )$ for $k\in \mathbb N$ and $u\in N_{0}$ is bounded in $D$.
Its image by the continuous map $\varphi _{ t}$ is bounded and generates a bounded $o[N_{0}^{(2)}]$-submodule of $D$. Hence $\varphi_{t}(m)\in M_{s}^{bd}$.

  b) The $o[P_{+}]$-module $M_{s}^{bd} $ is $\psi$-stable (hence
   $M_{s}^{bd} $ is \'etale by Corollary \ref{redMF}) because we have,
for $m\in M_{s}^{bd}, u \in N_{0},k\in \mathbb N $,
 \begin{equation}\label{Lbdb}
 \psi ^{k}(u^{-1} \psi(m) )= \psi^{k+1}( \varphi (u^{-1}) m)   \ .
\end{equation}

   \end{proof}

The goal of this section is to show that the $P$-equivariant sheaf on $\mathcal C$ associated to the \'etale $o[P_{+}]$-module $M_{s}^{bd}$  extends to a $G$-equivariant sheaf on $G/P$.
We will follow the method explained in subsection \ref{topAM}.

\bigskip  Put $p_{n}:M\to M/p^{n}M$ for the reduction modulo $p^{n}$ for a positive integer $n$. Recall that $M$ is $p$-adically complete.

 \begin{lemma}\label{lred1} The $o$-submodule $ M_{s}^{bd}\subset M$ is closed  for the $p$-adic topology, in  particular
 $$M_{s}^{bd} = \varprojlim_{n} ( M_{s}^{bd}/p^{n} M_{s}^{bd}) \ . $$
Moreover  $M_{s}^{bd} $ is the set of $m\in M$ such that $p_{n}(m)$ belongs to $(M/p^{n}M)_{s}^{bd}$ for all $n\in \mathbb N$, and we have
$$M_{s}^{bd} = \varprojlim_{n} ( M /p^{n} M)_{s}^{bd} \ . $$

\end{lemma}

  \begin{proof} a)  Let  $m$ be an element in the closure of $M_{s}^{bd}$ in $M$  for the $p$-adic topology.  For any $r\in \mathbb N$, we choose   $m'_{r} \in  M_{s}^{bd}$ with $m-m'_{r}\in p^{r}M$.  For each $r$, we   choose  $r'\geq 1$ such that $\ell _{M}(\psi^{k}(u^{-1}m'_{r}))\in p^{r}D+X^{-r'}D^{0}$  for all  $k\in \mathbb N, u\in N_{0},$ applying Lemma \ref{bdP}.
We have
$$\ell_{M} (\psi^{k}(u^{-1}m)) \in  \ell _{M}(\psi^{k}(u^{-1}m'_{r})+ p^{r}M) =  \ell_{M} (\psi^{k}(u^{-1}m'_{r}) )+p^{r}D \subset  p^{r}D+X^{-r'}D^{0} \  . $$
  By the same lemma,  $ m\in M_{s}^{bd}$.  This proves that    $ M_{s}^{bd} $ is closed  in $M$ hence $p$-adically complete.

 b) The reduction modulo $p^{n}$ commutes with $\ell_{M}, \psi,$ and the action of $N_{0}$.       The following properties are equivalent :

 $m\in M_{s}^{bd} $,

 $\{\ell_{M} (\psi^{k}(u^{-1}m ))$ for  $k\in \mathbb N, u\in N_{0}\}  \subset D $ is bounded,

 $\{\ell_{M/p^{n}M} (\psi^{k}(u^{-1}p_{n} (m)))$ for  $k\in \mathbb N, u\in N_{0}\}  \subset D/p^{n}D$ is bounded for all  positive integers $n $,

 $p_{n}(m)\in (M/p^{n} M)_{s}^{bd}$ for all positive integers $n$.

\noindent  We deduce that  $m\mapsto (p_{n}(m))_{n}:M^{bd}_{s}\to  \varprojlim_{n} (M/p^{n}M)^{bd}_{s}$  is an isomorphism.
     \end{proof}

\begin{proposition} \label{io}
$ D=D_{s}^{bd}$ and $M_{s}^{bd}$ contains $\iota_{D}(D)$.
 \end{proposition}

\begin{proof}
i) We show that $D=D_{s}^{bd}$.  By Lemma \ref{lred1}, we can suppose that $D$ is killed by a power of $p$. Let $d\in D$.
 By Prop. \ref{bd},
 for $n\in \mathbb N$, there exists $k_{0}\in \mathbb N$ such that $ \psi^{k}(v^{-1}d)\in D^{\sharp} $ for $k\geq k_{0}, v\in N_{0}^{(2)}$.  As $D^{\sharp} \subset D$   is bounded,  and as the   set of $ \psi^{k}(v^{-1}d)$  for all $0\leq k <k_{0}, v\in N_{0}^{(2)},$ is also bounded because
the set of $v^{-1}d$ for $v\in N_{0}^{(2)}$ is bounded and $\psi^{k}$ is continuous, we deduce that
 $d\in D_{s}^{bd}$.

ii) We show that $M_{s}^{bd}$ contains $\iota_{D}(D)$ by showing
$$\{\ell _{M}(\psi^{k}( u^{-1} \iota_{D}(d))) \ { \rm for } \ k\in \mathbb N, u\in N_{0}\}  = \{\psi^{k}( v^{-1}d)\  {\rm  for } \ k\in \mathbb N, v\in  N_{0}^{(2)} \}
$$
when $d\in D$ (the right hand side is bounded in $D$ by i)).
We write an element  $N_{0}$   as  $\iota (v)u$ for   $u $ in $N_{\ell} $  and $v\in  N_{0}^{(2)} $. By Lemma \ref{coe},
  $$
  \psi^{k}( u^{-1}\iota (v)^{-1}\iota_{D}(d)) = \psi^{k}(u^{-1}\iota_{D}(v^{-1}d))=
s^{-k}u^{-1}s^{k} \psi^{k}( \iota_{D}(v^{-1}d))
$$  when $u\in s^{k}N_{\ell}s^{-k}$ and is $0$
  when    $u$  is not  in $s^{k}N_{\ell}s^{-k}$.  When $u\in s^{k}N_{\ell}s^{-k}$ we have $\ell _{M}(s^{-k}u^{-1}s^{k} \psi^{k}(\iota_{D}( v^{-1}d)))=\psi^{k}( v^{-1}d)$ as $\iota_{D}$ is $\psi$-equivariant.
  \end{proof}

\begin{proposition} \label{list}
$M_{s}^{bd}  $ is  dense in $M$
\end{proposition}
\begin{proof}
 $M_{s}^{bd} \subset M $ is an  $o[N_{0}]$-submodule, which by Proposition \ref{io} contains $\iota_{D}(D)$. The  $o[N_{0}]$-submodule of $M$ generated by $\iota_{D}(D)$ is dense by Lemma \ref{lista}.
\end{proof}

We summarize: we proved that $M_{s}^{bd}\subset M$ is a dense $o[N_{0}]$-submodule, stable by $L_{+} $, and the action of $L_{+}$ on $M_{s}^{bd}$ is \'etale.

\begin{remark}\label{Landa} {\rm It follows  from Lemma \ref{lred1} and the subsequent proposition \ref{sfp} and that $M_{s}^{bd}$  is a $\Lambda (N_{0})$-submodule of $M$.
}
\end{remark}

\subsection{The  special family $\mathfrak{C}_{s } $  when $M $ is killed by a power of $p$}

{\sl  We suppose that $M$ is killed by a power of $p$}.

\begin{proposition} \label{sfp}

1. For any  lattice $D_{0}$ in $D$, the $o$-submodule
\begin{equation*}
M_s^{bd}(D_0):=\{m\in M\mid \ell_{M}(\psi^k(u^{-1}m))\in D_0\text{ for all }u\in
N_0\text{ and }k\in \mathbb N \}.
\end{equation*}
of $M$ is compact, and is a $\psi$-stable $\Lambda (N_{0})$-submodule.

2. The   family    $\mathfrak{C}_{s } $   of compact subsets   of $M $ contained in  $M_{s}^{bd}(D_{0})$ for some lattice $D_{0} $ of $D$, is special  (Def. \ref{mathfrakC}), satisfies $\mathfrak C (5)$ (Prop. \ref{criterion}) and $\mathfrak C (6)$ (Prop. \ref{corspecial}), and  $M(\mathfrak{C}_{s } )=  M_s^{bd}$  is a $\Lambda (N_{0})$-submodule of $M$.

\end{proposition}

\begin{proof} 1.  a)
 As $\ell$ and $\psi$ are continuous (Proposition \ref{contphipsi-general}) and $D_0 \subset D$ is
closed, it follows that $M_s^{bd}(D_0)$ is an intersection of closed subsets in $M$, hence $M_s^{bd}(D_0)$ is closed in  $M$. As $M_s^{bd}(D_0)$ is an $o[N_{0}]$-submodule of $M$ and $o[N_{0}]$ is dense in $\Lambda(N_{0})$ we deduce that $M_s^{bd}(D_0)$ is a $\Lambda (N_{0})$-submodule. It is $\psi$-stable by \eqref{Lbdb}.
The weak topology on $M$ is  the projective limit of the
weak topologies on $M/\mathcal{M}_\ell(N_0)^nM$, and we have (\cite{TG} I.29 Corollary)
 \begin{equation*}
M_s^{bd}(D_0)=\varprojlim_{n\geq 1}
(M_s^{bd}(D_0)+\mathcal{M}_\ell(N_0)^nM)/\mathcal{M}_\ell(N_0)^nM\ .
\end{equation*}
Therefore it suffices to show
that
\begin{equation*}
(M_s^{bd}(D_0)+\mathcal{M}_\ell(N_0)^nM)/\mathcal{M}_\ell(N_0)^nM
\end{equation*}
is compact for each large $n $. We will show the stronger property that it is a  finitely generated  $\Lambda(N_0)$-module.

b)
We prove first that $M_s^{bd}(D_0)$ is the intersection of the $\Lambda(N_0)$-modules generated by the image by $\varphi^k$ of the inverse image $\ell_{M}^{-1}(D_0)$ of $D_{0}$ in $M$, for $k\in \mathbb N$,
\begin{equation}\label{Msbd}
M_s^{bd}(D_0)=\bigcap_{k\in \mathbb N}\Lambda(N_0)\varphi^k(\ell_{M}^{-1}(D_0))\ .
\end{equation}
The inclusion from left to right  follows from the expansion (\ref{writing}), as $m\in M_s^{bd}(D_0)$ is equivalent to $m_{u,s^k} =\psi^k(u^{-1}m)\in  \ell_{M}^{-1}(D_0)$ for all  $u\in N_{0}$  and  $k\in \mathbb N$. The inclusion from right to left follows
$$
\ell_{M} \psi^{k}u^{-1}(\Lambda(N_0)\varphi^k(\ell_{M}^{-1}(D_0)))=D_{0} \ .
$$
c) We pick a lattice $M_{0}$ of $M$ such that $\ell_{M}^{-1}(D_{0})=M_{0}+J_{\ell}(N_{0})M$, as $J_{\ell}(N_{0})M$ is the kernel of $\ell_{M} $. By  Lemma \ref{4.6} we can choose  for each $n\in \mathbb N$ a large integer $r$  such that
$\varphi^r(J_\ell(N_0)M)\subseteq
\mathcal{M}_\ell(N_0)^nM$. Therefore we have
\begin{align*}
M_s^{bd}(D_0)\subseteq \Lambda(N_0)\varphi^r(M_0+ J_\ell(N_0)M)\subseteq \Lambda(N_0)\varphi^r(M_0)+\mathcal{M}_\ell(N_0)^nM \ .
\end{align*}
We deduce
\begin{align*}
(M_s^{bd}(D_0)+\mathcal{M}_\ell(N_0)^nM)/\mathcal{M}_\ell(N_0)^nM\subseteq  (\Lambda(N_0)\varphi^r(M_0)+\mathcal{M}_\ell(N_0)^nM)/\mathcal{M}_\ell(N_0)^nM \ .
\end{align*}
The right term is a finitely generated  $\Lambda(N_0)$-module hence the left term is  finitely generated  as a $\Lambda(N_0)$-module
 since $\Lambda(N_0)$ is noetherian.

 2. The family is stable by finite union because a finite sum of lattices is a lattice.
If $C \in \mathfrak{C}_{s }$ then $N_{0}C \in \mathfrak{C}_{s }$ because $M_{s}^{bd}(D_{0})$ is a $\Lambda (N_{0})$-module. We have
$$M(\mathfrak{C}_{s } )=\cup_{D_{0}}M_{s}^{bd}(D_{0}) =  M_s^{bd}  \ , $$
 when $D_{0}$ runs over the lattices of $D$, the last follows from the fact that
 a bounded subset of $D$ is contained in a lattice (this is the only part in the proof where the assumption that $M$ is killed by a power of $p$ is used). Apply Prop. \ref{Lbd}.

Property $\mathfrak C (5)$  is immediate because $M_{s}^{bd}(D_{0})$ is $\psi$-stable.
Property $\mathfrak C (6)$ follows from $\varphi (M_{s}^{bd}(D_{0}))\subset M_{s}^{bd}(D_{s})$ where $D_{s}$ is the lattice of $D$ generated by $\varphi (D_{0})$ (this uses the part a) of the proof of  Prop. \ref{Lbd}).
 \end{proof}

\begin{proposition}\label{912} All the assumptions of   Prop. \ref{corspecial} are satisfied
\end{proposition}
\begin{proof}

   a) Proof of the  convergence criterion.

The lattice $M^{++}:= \Lambda (N_{0}) \iota_{D}(D^{++})$  of $M$ satisfies $\ell_{M} (M^{++})=D^{++}$, and   is $\varphi$-stable (because $\Lambda (N_{0})$ is $\varphi$-stable, $\iota_{D}$ and $\varphi$ commute, and $D^{++} $ is $\varphi$-stable).

A lattice in $D$ is contained in $X^{-n}D^{++}$  for some $n\in \mathbb N$, and    $C \in \mathfrak{C}_{s}$   is contained in $M_{s}^{bd}(X^{-n}D ^{++})$ for some $n\in \mathbb N$.
An  open $o[N_{0}] $-submodule $\mathcal M$ of $M$ contains  ${\mathcal M}_{\ell}(\Lambda_{0})^{r}M+X^{r} M^{++} $ for some $r\in \mathbb N$. Let   $C_{+}$ be a compact subset of $L_{+}$. We want to   find a compact open subgroup $P_{1}\subset P_{+}$ and an integer $k_{0}\geq 0$ such that, for $k\geq k_{0}$,
$$s^{k}(1-P_{1})C_{+} \psi^{k} (C) \subset  {\mathcal M}\ . $$
It suffices to find $P_{1}$ and $k_{0}$ when $$C=M_{s}^{bd}(X^{-r}D ^{++})\ , \ \mathcal M= {\mathcal M}_{\ell}(\Lambda_{0})^{r}M+X^{r}M^{++} $$ for large $r$.

Let $r\in \mathbb N$.
As $\ell$ is continuous, $\ell (C_{+})$ is a compact subset of $L_{+}^{(2)}$.  By the continuity of the action of $P_{+}^{(2)}$ on $D$ there exists   a compact open subgroup $P_{1}^{(2)}\subset P_{+}^{(2)}$ such that
$$ (1-P_{1}^{(2)})\ell (C_{+})  X^{-r}D^{++}\subset X^{r}D^{++} \ . $$
We deduce that for all $k\in \mathbb N$ we have
$$ s^{k} (1-P_{1}^{(2)})\ell (C_{+})  X^{-r}D^{++}\subset \varphi^{k}(X^{r}D^{++}) = \varphi^{k}(X^{r})  \varphi^{k}(D^{++}) \subset X^{r} D^{++}\ , $$
where the last inclusion follows from $\varphi(X^{r})\Lambda(N_{0}) \subset X^{r}\Lambda (N_{0})$ and $\varphi (D^{++})\subset D^{++}$.

We choose, as we can, a compact open subgroup $P_{1}$ of $P_{+}$ such that $\ell (P_{1})
\subset P_{1}^{(2)}$.   For any $k'\in \mathbb N$, we have
$$\ell_{M} (s^{k}(1-P_{1})C_{+} \psi^{k'} (M_{s}^{bd}(X^{-r}D ^{++}))) \subset  X^{r}D^{++} \ . $$
The inverse image of  $X^{r}D^{++}$ by $\ell_{M}$ is $X^{r}M^{++} +  J_{\ell}(N_{0})M$,  and we have
\begin{equation}\label{subs}
s^{k}(1-P_{1})C_{+} \psi^{k'} (M_{s}^{bd}(X^{-r}D ^{++})) \subset  X^{r}M^{++} +  J_{\ell}(N_{0})M \ .
\end{equation}
  The module    $X^{r }M^{++}$  is $\varphi$-stable because  $$\varphi(X^{r}M^{++})= \varphi(X^{r})\varphi (M^{++})\subset \varphi(X^{r})M^{++}
 \subset X^{r}M^{++} \ .$$
We choose $k_{0}\in \mathbb N$ such that   (Prop. \ref{4.6}), for $k\geq k_{0} $.
$$\varphi ^{k}(J_{\ell}(N_{0})) \subset {\mathcal M}_{\ell}(N_{0})^{r} \ .$$
Applying $s^{k_{0}}$ to (\ref{subs}) for any $k\in \mathbb N$ we obtain
$$s^{k+k_{0}}(1-P_{1})C_{+} \psi^{k'} (M_{s}^{bd}(X^{-r}D ^{++})) \subset {\mathcal M}_{\ell}(N_{0})^{r} + X^{r}M^{++}  \ .
$$
 Then, taking $k'=k+k_{0}$,  we obtain
$$s^{k }(1-P_{1})C_{+} \psi^{k} (M_{s}^{bd}(X^{-r}D ^{++})) \subset    {\mathcal M}_{\ell}( N_{0})^{r}  + X^{r}M^{++} \
$$
for all $k\geq k_{0}$. This ends the proof of the  convergence criterion.

\bigskip b) Proof that $\mathcal H_{g}(m)$ belongs to $M_{s}^{bd}$ when $m\in M_{s}^{bd}$ and $g\in N_{0}\overline P N_{0}$.

 We have  to show that
 the set   $$
\bigcup_{x\in \mathbb N, n_{0}\in N_{0}} \ell _{M}(\psi^{x}(n_{0}^{-1} \mathcal  H_{g}(m)))  $$  is  bounded.
In general  $M_{s}^{bd} $ is not complete and $M$ is complete,    the  convergence criterion  implies that  the sequence   $(\mathcal  H_{g  }^{(k)}(m) )_{k\geq k_{g}^{(0)}}$  (see \eqref{Hgk}, \eqref{Hgs})
 converges to   $ \mathcal  H_{g}(m)\in M$.

Given an integer $k_{g}$ and we write, for $x\in \mathbb N$,
$$
  \mathcal H_{g}(m)=\mathcal H_{g}^{(x+ k_{g} )}(m) + \sum_{k\geq  x+k_{g} }s_{g}^{(k)}(m) \ .
  $$
As $\psi , \ell _{M}$ are continuous, the action of $n_{0}^{-1}\in N_{0}$ is continuous, we have
$$
\ell _{M}(\psi^{x}(n_{0}^{-1}\mathcal H_{g}(m))) =\ell_{M} (\psi^{x}(n_{0}^{-1}\mathcal H_{g}^{(x+ k_{g} )}(m)) )+ \sum_{k\geq  x+k_{g} }\ell_{M} (\psi^{x}(n_{0}^{-1}s_{g}^{(k)}(m))) \ .
$$
Let $r\in \mathbb N$ such that $p_{K}^{r}M=p_{K}^{r}D=0$.  It suffices to find  an integer $k_{g} $ such that
$$\bigcup_{ x\in \mathbb N, n_{0}\in N_{0}} \ell_{M} (\psi^{x}(n_{0}^{-1}\mathcal H_{g}^{(x+ k_{g} )}(m) ))$$
is bounded and such that,  for all $x\in \mathbb N$ and $ k\geq  x+k_{g} , n_{0}\in N_{0}$,
$$  \psi ^{x}( n_{0}^{-1} s_{g}^{(k)}(m)) \ \subset \   {\mathcal M}_{\ell}(N_{0})^{r}M+ M^{++} \ , $$
 (because $\ell_{M}( {\mathcal M}_{\ell}(N_{0})^{r}M+ M^{++})= D^{++}$ is bounded).

\bigskip We explain   how one chooses $k_{g}$ in order to ensure the inclusion. First, we choose a lattice $D_{0}$ of $D$ such that $m  \in M_{s}^{bd}(D_{0})$. By \eqref{sgk},  it suffices to show that for a compact open subgroup $P_{1}\subset P_{0} $ we have  $$\psi^{x} N_{0}  s^{k-k_{g}^{(1)}} (1-P_{1}) \Lambda_{g}s\psi^{k+1}N_{0} M_{s}^{bd}(D_{0})
\subset {\mathcal M}_{\ell}(N_{0})^{r}M+ M^{++} \ ,$$
 for    $x\in \mathbb N$ and for
$k\geq k_{g}^{(2)}(P_{1})\geq k_{g}^{(1)}$ (Lemma \ref{compact}). When $k-k_{g} ^{(1)}\geq x$ the left hand side is contained in
  $$ N_{0}  s^{k-k_{g}^{(1)}-x} (1-P_{1}) C_{+}  M_{s}^{bd}(D_{0})
  $$
where  $\Lambda_{g}s= C_{+}$ is a compact subset of $L_{+}$ because $\psi^{x} N_{0}  s^{k-k_{g}^{(1)}} (m) \subset N_{0}s^{k-k_{g}^{(1)}-x} (m) \cup \{0\}$ for $m\in M $ and
$\psi^{k+1}N_{0}(M_{s}^{bd}(D_{0})  ) \subset M_{s}^{bd}(D_{0})  $. By the continuity of the action of $P_{0}$ on $M$ and the compactness of $C_{+}  M_{s}^{bd}(D_{0})$,  we choose $P_{1}$ such that
 $$ (1-P_{1}) C_{+}  M_{s}^{bd}(D_{0})   \subset {\mathcal M}_{\ell}(N_{0})^{r}M+ M^{++} \ ,$$
and  we choose $k_{g}:=k_{g} ^{(2)}(P_{1})$. We have
\begin{align*}
N_{0}  s^{k-k_{g}^{(1)}-x} (1-P_{1}) C_{+}  M_{s}^{bd}(D_{0}) & \subset N_{0}  s^{k-k_{g}^{(1)}-x}({\mathcal M}_{\ell}(N_{0})^{r}M+ M^{++})\\
&\subset {\mathcal M}_{\ell}(N_{0})^{r}M+ M^{++} \ .
\end{align*}
   such that for $k\geq k_{g}+ x$ this inclusion is satisfied.

\bigskip  We show now that the
     set of
$\ell _{M}(\psi^{x}(n_{0}^{-1}\mathcal H_{g}^{(x+ k_{g} )}(m) )$  for $n_{0}\in N_{0}, x\in \mathbb N$   is  bounded in $D$.
Indeed, applying $\psi^{x} n_{0}^{-1}$ to
\begin{align*}
\mathcal H_{g}^{(x+ k_{g} )}(m)  & = \sum_{u\in J(U_{g}/N_{x+k_{g}})} n(g,u) t(g,u) \varphi^{x+ k_{g} }(m_{u, x+ k_{g}})  \\
& = \sum_{u\in J(U_{g}/N_{x+k_{g}})} n(g,u)  \varphi^{x+ k_{g} }(t(g,u) m_{u, x+ k_{g}}) \ ,
\end{align*}
we obtain
$$
\psi^{x}(n_{0}^{-1}\mathcal H_{g}^{(x+ k_{g} )}(m))  =  \sum_{u \in J(U_{g}/N_{x+k_{g}})}
\psi ^{x}(n_{0} ^{-1}n(g,u)) \varphi^{k_{g} } (t(g,u)m_{u,x+k_{g}}) \ .
$$
Each  summand  in the right hand side  is contained in  the compact set $N_{0 }C'_{+}  M_{s}^{bd }(D_{0})$
where $C_{+}'=s^{k_{g}}t(g,U_{g})$ is compact in $L_{+}$ because $k_{g}\geq k_{g}^{(1)}$; the  image by $\ell_{M}$ of  $N_{0 }C'_{+}  M_{s}^{bd }(D_{0})$ is compact hence bounded in $D$. The $o$-submodule of $D$ generated by $\ell_{M}(N_{0 }C'_{+}  M_{s}^{bd }(D_{0}))$ is also bounded.

\bigskip c) The family $\mathfrak{C}_{s}$ satisfies the last assumption of Prop. \ref{corspecial}

We
have seen in b) that  there exists an integer $k_g=k_g(D_0)$
such that
\begin{equation}\label{uni1}
\bigcup_{(x,k), 0\leq x\leq k-k_g}\ell_{M}\circ\psi^x(N_0\mathcal{H}_g^{(k)}(M_s^{bd}(D_0))
\end{equation}
is bounded. Hence it suffices to show that the set
\begin{equation}\label{uni2}
\bigcup_{(x,k), x\geq k-k_g}\ell_{M}\circ\psi^x(N_0\mathcal{H}_g^{(k)}(M_s^{bd}(D_0))
\end{equation}
is bounded. Then the above two sets \eqref{uni1} and \eqref{uni2} will  lie in a common lattice $D_{1}$ of $D$, hence $\mathcal{H}_g^{(k)}(M_s^{bd}(D_0)) \subset M_s^{bd}(D_1)$ for all $k\geq k_{g}$.

Let $m$ be in $M_s^{bd}(D_0)$ and $n_{0}\in N_{0}$. For $x\geq k-k_g$,  we write
\begin{align*}
 \psi^x(n_0^{-1}\mathcal{H}_g^{(k)}(m))&=\sum_{u\in
  J(U_g/N_k)}\psi^x(n_0^{-1}n(g,u)\varphi^{k-k_g}(t(g,u)s^{k_g}m_{u,s^k})) \\
&=\sum_{u\in
  J(U_g/N_k)} \psi^{x-k+k_g}(\psi^{k-k_g}(n_0^{-1}n(g,u))(t(g,u)s^{k_g}m_{u,s^k}))\ .
  \end{align*}
 Each  summand  in the right hand side  is contained in   $ \psi^{x-k+k_g}(\Lambda(N_0) C'_{+}M_s^{bd}(D_0))$. By (\ref{Msbd}),
  $$M_s^{bd}(D_0)  \subset \Lambda(N_0)\varphi^{x-k+k_{g}}\ell_{M}^{-1}(D_0) \ , $$
  hence
\begin{align*}
  \psi^{x-k+k_g}(\Lambda(N_0) C_{+}M_s^{bd}(D_0))
&     \subset  \psi^{x-k+k_g} (\Lambda(N_0)\varphi ^{x-k+k_{g} } C'_{+}  \ell_{M}^{-1} (D_0))
\\
&= \Lambda(N_0) C'_{+}  \ell_{M}^{-1} (D_0)
 \ .
\end{align*}
The image by $\ell_{M}$ of $\Lambda(N_0)C_{+} \ell_{M}^{-1} (D_0)$  is
$\Lambda(N_0^{(2)})\ell ( C'_{+})(D_0)$ which is bounded in $D$ because  $\ell (C'_{+}) \subset L_{+}^{(2)}$ is  compact.
\end{proof}

  \subsection{Functoriality and dependence on $s$}

  Let $Z(L)_{\dagger \dagger}\subset Z(L)$ be the subset  of elements $s $  such that
 $L=L_{-}s^{\mathbb N}$ and $(s^{k}N_{0}s^{-k})_{k\in \mathbb Z}$ and $(s^{-k}w_0N_{0}w_0 ^{-1}s^{k})_{k\in \mathbb Z}$ are decreasing sequences of trivial intersection and union $N$ and $w_0Nw_0^{-1}$, respectively (see section \ref{fc}).

Let $M$ be a topologically etale $L_+$-module over $\Lambda_\ell(N_0)$ and let $D:= \mathbb D(M)$. We have $D/p^{n}D=  \mathbb D (M/p^{n}M)$ for $n\geq 1$. By Lemma \ref{redmod}, $M$ satisfies the properties a,b,c,d of subsection \ref{topAM} and is complete (the same is true for $M/p^{n}M$). The image  $D_{0,n}$ in $D/p^{n} D$ of any lattice $D_{0, n+1}$  in $D/p^{n+1} D$ is a lattice and the maps $\ell$ and $\psi$ commute with the reduction modulo $p^{n}$,  hence  $(M/p^{n+1} M)_{s}^{bd}(D_{0, n+1})$  maps into $(M/p^{n} M)_{s}^{bd}(D_{0, n})$. Therefore the special family $\mathfrak C_{s, n+1}$ in $M/p^{n+1}M$ maps to the special family $\mathfrak C_{s, n}$ in $M/p^{n}M$. As in Lemma \ref{Limspe} we define the special family $\mathfrak C_{s}$ in $M$ to consist of all compact subsets $C\subset M$ such that $p_{n}(C)\in \mathfrak C_{s, n}$ for all $n\geq 1$. By Prop. \ref{sfp} and Lemma \ref{lred1} we have
$$ M(\mathfrak C_{s})=M_{s}^{bd}\ .$$

\begin{theorem} \label{the}
Let   $s\in Z(L)_{\dagger \dagger}$ and $M\in \mathcal M^{et}_{\Lambda_{\ell}(N_{0})}(L_+)$.
\begin{itemize}
\item[(i)] The  $(s, \res,  \mathfrak C_{s})$-integrals $\mathcal H_{g,s}$ of the functions $\alpha_{g,0}|_{M_{s}^{bd}}$  for $g\in N_{0}\overline P N_{0}$ exist,  lie in $\End_{o} (M_{s}^{bd})$, and satisfy the relations H1, H2, H3 of Prop. \ref{multiplicative}.

\item[(ii)]The map $M\mapsto (M_s^{bd}, (\mathcal H_{g,s})_{g\in N_{0}\overline P N_{0}})$ is functorial.
\end{itemize}
\end{theorem}

\begin{proof} (i) By Prop. \ref{912} the assumptions of Prop. \ref{basicspecial} are satisfied.

(ii) Let $f:M\to M'$ be a morphism in $\mathcal M^{et}_{\Lambda_{\ell}(N_{0})}(L_+)$. For $m \in M$  we denote $E_s(m)=\{\lambda_{M}(\psi_s^k u^{-1} m) \ \text{ for $u\in N_0, k\in \mathbb N$}\}$. We have
\begin{equation}\label{Esm}
 \mathbb D(f) (E_s(m))= E_s (f(m)) \ \text{when $m\in M$} \ ,
 \end{equation}
because the  maps $\lambda_M:M\to D  $ and
$\lambda_{M'}:M'\to D'  $ sending $x$ to $1\otimes x$ for $x\in M$ or $x\in M'$ satisfy $\lambda_M' \circ f = \mathbb D(f)  \circ \lambda_M $, and
$f$ is $P^-$-equivariant by  Lemma \ref{-eq}.
Any morphism between finitely generated modules on
$\mathcal O_{\mathcal E}$ is continuous for the weak topology (cf.\ \cite{SVig} Lemma 8.22). The image of a bounded subset by a continuous map is bounded. We deduce
from \eqref{Esm} that $E_s(m)$ bounded  implies $E_s (f(m)) $ bounded, equivalently
 $m\in M_s^{bd}$ implies  $f(m)\in  {M'_{s}}^{bd}$.
For $m\in M_s^{bd}$ we have  $f (\mathcal H_{g,s} (m))=\mathcal H_{g,s} (f(m))$ where
$$\mathcal H_{g,s}(.) = \lim_{k\to \infty}\sum_{u\in J (N_0/ s^k N_0 s^{-k})}   n(g,u)\varphi_{ t(g,u) s^k} \psi_s^k u^{-1}(.) \  ,$$
because $f$ is $P_+$ and $P_-$-equivariant by  Lemma \ref{-eq}.
\end{proof}

We investigate now the dependence on $s\in Z(L)_{\dagger \dagger}$ of the dense subset
$M_s^{bd}\subseteq M$ and of the  $(s, \res,  \mathfrak C_{s})$-integrals $\mathcal H_{g,s}  $.

\begin{lemma}  $Z(L)_{\dagger \dagger} $ is stable by product.
\end{lemma}
\begin{proof}  Let $s,s'\in Z(L)_{\dagger \dagger}$.  Clearly $L_- s'^n = L_- s^{-n } s^ns'^n \subset L_-(ss')^n $ because  $L_-$ is a monoid and $s^{-1}\in Z(L)_- = Z(L)\cap L_-$. Therefore  $L=L_{-}(ss')^{\mathbb N}$.
The sequence  $((ss')^{k}N_{0}(ss')^{-k})_{k\in \mathbb Z}$ is decreasing because
$$s'^{k+1}s^ {k+1}N_{0} s^{-k-1}s'^{-k-1}\subset s'^k s^{k+1}N_0 s^{-k-1} s'^{-k} \subset
s'^k s^{k}N_0s^{-k}s'^{-k}  \ . $$
The intersection is trivial and  the union is $N$ because $s'^k s^{k}N_0s^{-k}s'^{-k}\subset  s^{k}N_0s^{-k} $   when $k\in \mathbb N$ and $s'^k s^{k}N_0s^{-k}s'^{-k}\supset  s^{k}N_0s^{-k} $ when $-k\in \mathbb N$.   One makes the same argument  with   $w_0N_0w_0^{-1}$.
 \end{proof}

\begin{lemma}\label{above}
\begin{itemize}
\item[(i)] The action of
 $t_{0}\in \ell^{-1}(L_{0}^{(2)})\cap L_{+}$ on $D $  is invertible.

\item[(ii)]
 There exists a  treillis $D_{0}$ in $D$ which is stable by  $\ell^{-1}(L_{0}^{(2)})\cap L_{+}$.

\end{itemize}
 \end{lemma}

\begin{proof}   (i) is true because the action of $t_{0}$ on $D$  is \'etale and   $N_0^{(2)}=\ell(t_0)N_0^{(2)}\ell(t_0)^{-1}$.

(ii) Let $s \in Z(L)_{\dagger \dagger}$ and let $\psi_{s}$ be the canonical inverse of the \'etale action $\varphi_{s}$ of $s$ on $D$.    We show that    the minimal  $\psi_s$-stable treillis $D^\natural$  of $D$ (Prop.\ \ref{Mira}(iii)) is stable by  $\ell^{-1}(L_{0}^{(2)})\cap L_{+}$.

For  $t_{0}\in \ell^{-1}(L_{0}^{(2)})\cap L_{+}$ we claim that $\varphi_{t_0}(D^{\natural})$ is also a
$\psi_s$-stable treillis in $D$. We have $\psi_{s} \psi_{t_{0}}= \psi_{t_{0}} \psi_{s}$  as $t_{0}\in Z(L)$.  Multiplying by  $\varphi_{t_{0}}$ on both sides, one gets
$\varphi_{t_{0}}\psi_{s} \psi_{t_{0}}\varphi_{t_{0}}= \varphi_{t_{0}}\psi_{t_{0}} \psi_{s}\varphi_{t_{0}}$. Since $\psi_{t_{0}}$  is the two-sided inverse of $\varphi_{t_{0}} $   by (i) we get that $\varphi_{t_{0}}$  and $\psi_{s}$ commute.  Hence $\varphi_{t_0}(D^{\natural})$  is  a compact $o$-module which is $\psi_{s}$-stable. It is a $\Lambda (N_{0}^{(2)})$-module because any $\lambda \in \Lambda (N_{0}^{(2)})$ is of the form $\lambda=\varphi_{\ell (t_{0})}(\mu)$ for some $\mu\in  \Lambda (N_{0}^{(2)})$ and
$\lambda \varphi_{t_0}(d)=\varphi_{t_{0}}(\mu d)$ for all $d\in D$. As $D^{\natural}$ contains a lattice and $\varphi_{t_0}$ is \'etale, we deduce that $\varphi_{t_0}(D^{\natural})$ contains a lattice and therefore is a treillis. By the minimality of $D^{\natural}$ we must have
$$D^{\natural} \subset \varphi_{t_0}(D^{\natural})  \ . $$
Similarly one checks that $\psi_{t_0}(D^{\natural})$ is a treillis. It is  $\psi_s$-stable because  $\psi_s$ and $\psi_{t_0}$ commute. Hence
$$D^{\natural} \subset \psi_{t_0}(D^{\natural})  \ . $$
Applying  $\varphi_{t_{0}} $ which is the two-sided inverse of $\psi_{t_0}$ we obtain
$\varphi_{t_0}(D^{\natural}) \subset D^{\natural}$ hence
$D^{\natural} = \varphi_{t_0}(D^{\natural})\ .$
\end{proof}

 We denote by $Z(L)_{\dagger} \subset Z(L)$ the monoid of $z\in Z(L)_{+}=Z(L)\cap L_{+}$  such that $z^{-1}w_{0}N_{0}w_{0}^{-1}z \subset w_{0}N_{0}w_{0}^{-1}$. We have $Z(L)_{\dagger \dagger}Z(L)_{\dagger}\subset Z(L)_{\dagger \dagger}$.

\bigskip  Note that  $L_{0}^{(2)}$ contains the center of $GL(2,\mathbb Q_{p})$ and that   $Z(L^{(2)})_{\dagger} = L^{(2)}_{+}$.

For $m\in M,t\in L_{+}, u\in U,$ and a system of representatives $J(N_{0}/tN_{0}t^{-1})\subset N_{0}$  for the cosets in $N_{0}/tN_{0}t^{-1}$
we have \eqref{writing}
 \begin{equation}\label{writingmu}
 m = \sum_{u\in J(N_{0}/tN_{0}t^{-1}) } u \mu_{t,u} \quad , \quad  \mu_{t,u} := \varphi_{t}  \psi_{t} (u^{-1}m)  \ .
 \end{equation}
  For $g\in  N_{0}\overline P N_{0}$ and $ s\in Z(L)_{\dagger \dagger}$,   we have the smallest positive integer  $k_{g,s}^{(0)}$  as in \eqref{f:N0}.   For    $k\geq k_{g,s}^{(0)}$, we have $\mathcal H_{g,s,J(N_{0}/N_{k})}\in \End_{o}^{cont}(M)$ where (compare with \eqref{Hgk})
 \begin{equation}\label{Hgkmu}
   \mathcal H_{g,s,J(N_{0}/N_{k})} (m)=  \sum_{u \in J(U_{g}/N_{k })} n(g,u)t(g,u)\mu_{s^{k},u} \ .
  \end{equation}
   When $m\in M_s^{bd}$, the integral $\mathcal H_{g,s}(m)$ is the limit of $\mathcal H_{g,s,J(N_{0}/N_{k})}(m)$ by Theorem \ref{the} and \eqref{Hgl}.

 \begin{proposition}\label{powers}  Let  $s\in Z(L)_{\dagger \dagger}$, $t_0 \in  \ell^{-1}(L_{0}^{(2)})\cap Z(L)_{\dagger}$ and $r$ a positive integer.

 (i) We have $M_{st_0} ^{bd} \subseteq M_{s}^{bd}  = M_{s^r}^{bd}$.

 (ii) For $g\in  N_{0}\overline P N_{0}$ we have  $\mathcal H_{g,s}= \mathcal H_{g,st_{0}}$ on $M_{st_0} ^{bd}$ and $\mathcal H_{g,s}= \mathcal H_{g,s^{r}}$ on $M_s^{bd}$.
\end{proposition}

\begin{proof} a) Note that  $st_{0}$ and $s^{r}$  in the proposition belong also to $Z(L)_{\dagger \dagger}$.

For a treillis $D_{0}$ in $D$ which is stable by  $
\ell^{-1}(L_{0}^{(2)})\cap L_{+}$ (Lemma \ref{above}),
$(X^{(2)})^{-r}D_{0}$ is a treillis in $D$; it is also stable by $t_0 \in
\ell^{-1}(L_{0}^{(2)})\cap L_{+}$ because
$$
\varphi _{\ell (t_{0})}((X^{(2)})^{-r}\Lambda (N_{0}^{(2)}) )= \varphi _{\ell (t_{0})}((X^{(2)})^{-r})\varphi _{\ell (t_{0})}(\Lambda (N_{0}^{(2)}) )=(X^{(2)})^{-r} \Lambda (N_{0}^{(2)}) \ . $$

When $M$ is killed by a power of $p$, this implies with Prop. \ref{sfp}  that  $M_{s}^{bd}$ is the union of $M_s^{bd}(D_{0})$ when $D_{0}$ runs over the lattices of $D$ which are stable by $\ell^{-1}(L_{0}^{(2)})\cap L_{+}$.

b)  We suppose from now on, as we can by Lemma \ref{lred1}, that $M$ is killed by a power of $p$ to prove $M_{st_{0}}^{bd}\subset M_{s}^{bd} =M_{s^r}^{bd}$.
  Let $m\in M_{s t_{0}}^{bd}(D_{0})$ where
 $D_{0}$ is a  $\ell^{-1}(L_{0}^{(2)})\cap L_{+}$-stable  lattice of $D$. For $u\in N_{0}$ and $k\in \mathbb N$, using \eqref{writing} for $t=t_0^k$ we obtain that
\begin{align*}
\ell_{M}(\psi_{s}^k(u^{-1}m))=
\ell_{M}(\sum_{v\in J(N_0/t_0^kN_0t_0^{-k})}v\circ \varphi_{t_0}^k\circ\psi_{t_0}^k\circ
v^{-1}\circ\psi_{s}^{k}(u^{-1}m))=\\
=\sum_{v\in J(N_0/t_0^kN_0t_0^{-k})}\ell(v)\varphi_{t_0}^k(\ell_{M}(\psi_{st_{0}}^k(\varphi_{s}^k(v^{-1})u^{-1}m)))
\end{align*}
lies in $D_0$, since $D_0$ is both $N_{0}^{(2)}$- and $\varphi_{t_0}$-invariant and
$\ell_{M}(\psi_{st_{0}}^k( u'm))\in D_{0}$ for $u'\in N_{0}$.  Therefore
$M_{s t_{0}}^{bd}(D_{0}) \subset M_{s}^{bd}(D_{0})$ and by a) we deduce $M_{st_{0}}^{bd}\subset M_{s}^{bd} $.

 For any $m\in M$ we observe that
$$\{\ell_{M}(\psi_{s^{r}}^{k}(u^{-1}m) \ \text{for $k\in \mathbb N,u\in N_{0}$}\} \subset \{\ell_{M}(\psi_{s}^{k}(u^{-1}m) \ \text{for $k\in \mathbb N,u\in N_{0}$}\} \ ,$$
as $\psi_{s^{r}}^{k}= \psi_{s}^{rk}$.
We deduce that $M_{s }^{bd}(D_0)\subset M_{s^{r}}^{bd}(D_0)$ for any lattice $D_0$ of $D$ hence  $M_s^{bd}\subset  M_{s^r}^{bd}$. Conversely, for $k_{1}\in \mathbb N$
we write $k_1=rk-k_2$  with $k\in \mathbb N$ and
 $0\leq k_2<r$ and we observe that
 \begin{align*}
\ell_{M}(\psi_s^{k_1}(u^{-1}m))&=\ell_{M}(\sum_{v\in
  J(N_0/s^{k_2}N_0s^{-k_2})}v\circ \varphi_s^{k_2}\circ\psi_{s}^{rk}(\varphi_s^{k_1}(v^{-1})u^{-1}m))\\
&=\sum_{v\in
  J(N_0/s^{k_2}N_0s^{-k_2})}\ell(v)\varphi_s^{k_2}(\ell_{M}(\psi_{s^r}^{k}(\varphi_s^{k_1}(v^{-1})u^{-1}m))) \ .
\end{align*}
The $\Lambda (N_{0}^{(2)})$-submodule  $D_{r}$ generated by  $\sum _{i=1}^{r-1}\varphi_s^{i }(D_0)$ is  a lattice because  the action $\varphi_{s}  $ of $s$ on $D$ is \'etale. We deduce that  $ M_{s^{r}}^{bd}(D_{0})\subset M_{s}^{bd}(D_{r})$ since $\ell_{M}(\psi_{s^r}^{k}(u'm))\in D_{0}$ for $u'\in N_{0}, m \in M_{s^{r}}^{bd}(D_{0}) $. Therefore
$M_{s ^{r}}^{bd}(D_{0}) \subset M_{s}^{bd}(D_{r})$ hence $M_{s ^{r}}^{bd}  \subset M_{s}^{bd} $. It is obvious that $\mathcal H_{g,s}= \mathcal H_{g,s^{r}}$ on $M_s^{bd}$.

c)   Let  $g\in  N_{0}\overline P N_{0} , k\geq k_{g,s}^{(0)} , t_{0} \in  \ell^{-1}(L_{0}^{(2)})\cap Z(L)_{\dagger}$ and $r\geq 1$.
 We have
 $$k_{g,st_{0}}^{(0)} \leq k_{g,s}^{(0)} \quad, \quad k_{g,s^{r}} ^{(0)}\leq k_{g,s}^{(0)}$$ because $(st_0)^k N_0 (st_0)^{-k} \subset N_{k} $ and $ (s^{r})^k N_0 (s^{r})^{-k}=N_{kr} \subset N_k$.

 Let $d$ in $D$ and $v\in N_0$.  By \eqref{writing} we have
 \begin{align*}
d&= \sum_{u\in  J(N_0^{(2)}/\ell(st_0)^kN_0^{(2)}\ell(st_0)^{-k})}u \varphi_{st_0}^k\circ\psi_{st_0}^k(u^{-1}d) \\
&=\sum_{u\in  J(N_0^{(2)} /\ell(s)^k N_0^{(2)} \ell(s )^{-k})}
u\varphi_{s}^k\circ\psi_{s}^k (u^{-1}d) \ ,
\end{align*}
with the second equality holding true summand per summand, because $\psi_{t_0}$ is the left and right inverse of  $\varphi_{t_0}$ on $D$ (Lemma \ref{above} (i)) and  $\ell(t_0)N_0^{(2)}\ell(t_0)^{-1}=N_0^{(2)}$.
Since $\iota_D$ commutes with $\varphi_t$ and $\psi_t$ for $t\in L_+$, this implies
\begin{align*}
v \iota_D(d)&=\sum_{u\in  J(N_0^{(2)} /\ell(st_0)^k N_0^{(2)} \ell(st_0 )^{-k})}
v\iota(u)\varphi_{st_0}^k\circ\psi_{st_0}^k(\iota(u)^{-1}\iota_D(d)) \\
&=\sum_{u\in  J(N_0^{(2)} /\ell(s)^k N_0^{(2)} \ell(s )^{-k})}
v\iota(u)\varphi_{s}^k\circ\psi_{s}^k(\iota(u)^{-1}\iota_D(d)) \ ,
\end{align*}
again with the second equality holding true summand per summand. We choose, as we can,
 system of representatives   $J(N_{0}/(st_0)^k N_0 (st_0)^{-k}) $ and $J(N_{0}/ s^kN_0 s^{-k})$  containing $ \iota(J(N_0^{(2)} /\ell(s)^k N_0^{(2)} \ell(s )^{-k}))$.
  For $k\geq k_{g,s}^{(0)} \geq k_{g,st_0}^{(0)}$, we   obtain
$$\mathcal H_{g,st_0,vJ(N_{0}/(st_0)^k N_0 (st_0)^{-k}) }(v\iota_D(d)) =
\mathcal H_{g,s,vJ(N_{0}/ s^kN_0 s^{-k}) }(v\iota_D(d)) \ . $$
Passing to the limit when $k$ goes to infinity, and using linearity we deduce that $\mathcal H_{g,st_0}=
\mathcal H_{g,s}$ on the $o[N_0]$-submodule $<N_0\iota_D (D)>_o$ generated by $\iota_D (D)$ in $M_{st_0}^{bd}$.

 d) Let $m \in M_s^{bd}(D_1)$ with $D_1 \subset D$ a $\psi_s$-stable lattice (Prop. \ref{Mira} (iv)). For  a positive integer $k$, and a set of representatives  $J(N_0/s^{k}N_0s^{-k })$, we write $m$ in the form \eqref{writing}
\begin{align*}
m=\sum_{u\in
  J(N_0/s^{k}N_0s^{-k })}u\varphi_{s }^{k }(\iota_D(d(s,u))+m(s,u))
\end{align*}
with $m(s,u)$ in $J_\ell(N_0)M$ and   $d(s,u)=\ell_M(\psi_{s }^{k}(u^{-1}m))$ in $D_{1}$.
Then
\begin{equation*} m(s ):=\sum_{u\in
   J(N_0/s^{k}N_0s^{-k })} u\varphi_{s }^{k }(\iota_D(d (s,u))) \ \text{  lies in $<N_0\iota_D (D)>_o$}
    \end{equation*}
  because $\iota_D$ is $L_+$-equivariant. Moreover
 $m-m(s)$ is contained in the $o[N_0]$-submodule
$N_0 \varphi_s^{k}(J_\ell(N_0)M)$ generated by $\varphi_s^{k}(J_\ell(N_0)M)$.
 We show that
\begin{equation}
m (s) \in M_s^{bd } (D_1) \ .
\end{equation}
For $v\in N_{0}$ and $r\leq k $  we have
\begin{align*}
\psi_{s}^{r}(v^{-1}(m-m (s))) &=
\psi_{s}^{r}(v^{-1}\sum_{u\in J(N_0/s^{k}N_0s^{-k })}u\varphi_{s }^{k }(m(s,u))) \\
&=\sum_{u\in J(N_0/s^{k}N_0s^{-k })}\psi_{s}^r(v^{-1}u)\varphi_{s}^{k-r}(m(s,u))
\end{align*}
which lies in $J_\ell(N_0)M$ since $m(s,u)$ is in $J_\ell(N_0)M$ and  $J_\ell(N_0)M$ is $N_0$ and $\varphi_s$-stable. This shows that
$\ell_M(\psi_{s}^{r}(v^{-1}m(s)))=\ell_M(\psi_{s}^{r}(v^{-1}m))$ lies in $D_1$.
On the other hand, for $r>k$  we have
\begin{align*}
\ell_M(\psi_{s}^{r}(v^{-1}m (s)))&=\ell_M(\psi_{s}^{r}(v^{-1}
\sum_{u\in J(N_0/s^{k}N_0s^{-k })}u\varphi_s^{k}(\iota_D(d(s,u))))) \\
& =\sum_{u\in J(N_0/s^{k}N_0s^{-k })}\ell_M(\psi_{s}^{r-k}(\psi_{s}^{k}(v^{-1}u)\iota_D(d(s,u))))
\end{align*}
which lies in $D_1$. Indeed, since $D_1$ is $\psi_s$-stable the formula in part ii) of the proof of Prop.\ \ref{io} implies that $\iota_D(D_1) \subseteq M_s^{bd}(D_1)$; hence the  $\iota_D(d(s,u))$ lie in the $\psi_{s}$- and
$N_0$-invariant subspace $M_{s}^{bd}(D_1)$. We conclude that  $m(s)\in M_{s}^{bd}(D_1)$.

Therefore,  for any $\psi_{st_0}$-stable lattice $D_1 \subset D$, any $k\geq 1 $, and  any set of representatives $J(N_0/(st_0)^{k}N_0(st_0)^{-k })$, we have defined an $o$-linear homomorphism
 \begin{equation*}
 m\mapsto m (st_0) \quad  M_{st_0}^{bd}(D_1) \to M_{st_0}^{bd}(D_1)\ \cap \ <N_0\iota_D(D)>_o
\end{equation*}
such that
\begin{equation}\label{msto}
m-m(st_0)\in M_{st_0}^{bd}(D_1)\cap \varphi_{st_0}^k(J_\ell(N_0) M).
\end{equation}
 By c) we have  $\mathcal H_{g,st_0} (  m(st_0))=
\mathcal H_{g,s} (  m(st_0))$ for $m\in M_{st_0}^{bd}(D_1)$.

e)  To end the proof   that $\mathcal H_{g,st_0} =
\mathcal H_{g,s}$ on  $M_{st_0}^{bd}(D_1)$ we use the $\mathfrak{C}_s$-uniform convergence of $(\mathcal H_{g,s ^{k} , J(N_0/s ^{k}N s ^{-k}) })_k$. We fix, for any $k \geq 1$, systems of representatives $J(N_0/(st_0)^{k}N_0(st_0)^{-k })$ and $J(N_0/s ^{k}N s ^{-k})$. We also choose a lattice $D_0 \subset D$ which is stable by $\ell^{-1}(L_{0}^{(2)})\cap L_{+}$ and such that $D_1 \subset D_0$.
We recall that
$M_{st_0}^{bd}(D_1) $ is compact (Prop. \ref{sfp} i)) and that $M_{st_0}^{bd}(D_1) \subset M_{st_0}^{bd}(D_0) \subset M_s^{bd}(D_0)$ by b).
For any
open $\Lambda(N_0)$-submodule  in the weak topology $M_0\subset  M$, there exists  a
common constant $k_0 \geq k_{g,s}^{(0)}\geq k_{g,st_0}^{(0)}$ (by c)) such that for $k\geq k_{0}$,
\begin{align}\label{cun}
\mathcal H_{g,(st_0)^{k} , J(N_0/(st_0)^{k}N (st_0)^{-k})}
&\in     \mathcal H_{g,st_0}  + E(M_{st_0}^{bd}(D_1), M_0 ) \\
\label{cuns}\mathcal H_{g,s ^{k} , J(N_0/s ^{k}N s ^{-k}) }  & \in   \mathcal H_{g,s}  +E(M_{st_0 }^{bd}(D_1), M_0 ) \ .
\end{align}
On the left hand side of \eqref{cun}, \eqref{cuns}, we have continuous endomorphisms
  of $M$.
   By Lemma \ref{4.6}, there exists an
integer $k_1\geq k_0$  such that   they send   $N_0 \varphi_{st_0 }^{k _1}(J_\ell(N_0)M) $  into
 $M_0$. Therefore, for $m\in M_{st_0}^{bd}(D_1)$, they send the element
  $m-m(st_0)$ associated to $k_1$ and  $J(N_0/((st_0)^{k_1}N_0(st_0)^{-k_1 })$ as in d)  \eqref{msto}   into $M_0$  hence
 $$
 \text{
$\mathcal H_{g,st_0}( m-m(st_0))$ and $\mathcal H_{g,s}(m-m(st_0))$  lie in $M_0$.
}
$$
By d)  we obtain that $H_{g,st_0}(m)-H_{g,s_2}(m)$ lies in $M_{0}$ for $m\in M_{st_0}^{bd}(D_1)$.
The statement follows since we chose $M_0$ to be an arbitrary open
neighborhood of zero in the weak topology of $M$.
\end{proof}

\begin{definition}\label{br} We define the transitive relation $s_1\leq s_2$ on $Z(L)_{\dagger \dagger}$ generated by
 $$s_1=s_{2} t_{0} \text{ for  } t_{0} \in   \ell^{-1}(L_{0}^{(2)})\cap Z(L)_{\dagger} \quad \text{ or} \quad s_1^{r_{1} }= s_{2}^{r_{2} } \ \text{for     positive  integers $r_{1}, r_{2}$.} $$
\end{definition}

Proposition \ref{powers} admit the following corollary.

\begin{corollary}\label{subset} Let $s_{1}, s_{2}\in Z(L)_{\dagger \dagger}$.
\begin{itemize}
\item[i)] When  $s_1\leq s_2$   we have
$M_{s_1}^{bd}\subseteq M_{s_2}^{bd}$ and
$\mathcal H_{g,s_{1}}=\mathcal H_{g,s _{2}}$  on $M_{s_1}^{bd}$.
\item[ii)] When the relation $\leq$  on $Z(L)_{\dagger \dagger}$ is  right filtered,   we have $\mathcal H_{g,s_1} =\mathcal H_{g,s_2}$   on $M_{s_1}^{bd}\cap M_{s_2}^{bd}$.
\end{itemize}
\end{corollary}
\begin{proof}
i) If $s_1 \leq s_2$ then there exists, by definition, a sequence $s_1 = s'_1 \leq s'_2 \leq \ldots \leq s'_m = s_2$ in $Z(L)_{\dagger \dagger}$ such that each pair $s'_i , s'_{i+1}$ satisfies one of the two conditions in Def.\ \ref{br}. Hence we may assume, by induction, that the pair $s_1, s_2$ satisfies one of these conditions, and we apply Prop.\ \ref{powers}.

 ii)  When there exists $s_{3}\in Z(L)_{\dagger \dagger}$ such that $s_{1}\leq s_{3}$ and $s_{2}\leq s_{3}$,  by i)  $M_{s_1}^{bd} $ and $M_{s_2}^{bd} $  are contained in $M_{s_3}^{bd}$ and
$\mathcal H_{g,s_1}=\mathcal H_{g,s_2}=\mathcal H_{g,s_3}$ on
   $M_{s_1}^{bd}\cap M_{s_2}^{bd}$.
\end{proof}

\begin{proposition}\label{spr}  We assume that  the relation $\leq$  on $Z(L)_{\dagger \dagger}$ is  right filtered.
   Then, the  intersection and the union
$$M^{bd}_\cap:= \bigcap_{s\in Z(L)_{\dagger \dagger}}M_{s}^{bd}\  \subset \
M_\cup^{bd}:= \bigcup_{s\in Z(L)_{\dagger \dagger}}M_{s}^{bd}
$$
  are  dense  \'etale $L_+ $-submodules of $M$ over $\Lambda (N_0)$.

 For $g\in N_0\overline P N_0$ the  endomorphisms $\mathcal H_g \in \End_{o} (M _\cup^{bd})$
 equal to $\mathcal H_{g,s}$ on $M_{s}^{bd}$ for each  $s\in Z(L)_{\dagger \dagger}$,  are well defined, stabilize $M^{bd}_\cap$
and satisfy the relations H1, H2, H3  of Prop. \ref{multiplicative}.
 \end{proposition}
\begin{proof}
 $M^{bd}_\cap$ is an $L_+$-submodule of $M$ over  $\Lambda (N_0)$  by Prop. \ref{Lbd} and Remark \ref{Landa}. It is   dense in $M$ by Prop. \ref{io} and Lemma \ref{lista}. The action of $L_+$ on  $M^{bd}_\cap$  is \'etale  because  $M^{bd}_\cap$ is $L_-$-stable.  When $\leq$ is  right filtered, $M_\cup^{bd}$ is a $\Lambda_\ell (N_0)$-module  by Cor. \ref{subset} i).   For the same reasons than for $M^{bd}_\cap$, it is an \'etale $L_+ $-submodule of $M$ over $\Lambda (N_0)$.

 By Cor. \ref{subset} the $\mathcal H_{g }$ are well defined and stabilize $M^{bd}_\cap$. They  satisfy the relations H1, H2, H3 of Prop. \ref{multiplicative} because the $\mathcal H_{g,s}$ satisfy them (Theorem \ref{the}).
\end{proof}

 We summarize our results and give our main theorem.

\begin{theorem}  \label{main}
For any $s\in Z(L)_{\dagger \dagger}$,  we have a faithful functor
$$
\mathbb Y_s: \mathcal M^{et}_{\Lambda_{\ell}(N_{0})}(L_+) \quad \to \quad \text { $G$-equivariant sheaves   on $G/P$} \ ,
 $$
 which associates to $M \in \mathcal M^{et}_{\Lambda_{\ell}(N_{0})}(L_+)$ the  $G$-equivariant  sheaf    $\mathfrak Y_s$ on $G/P$ such that $\mathfrak Y_s(\mathcal C_0)=M_s^{bd}$.

When the relation $\leq $ on $Z(L)_{\dagger \dagger}$ is right filtered, we have faithful functors
$$
\mathbb Y_\cap, \mathbb Y_\cup: \mathcal M^{et}_{\Lambda_{\ell}(N_{0})}(L_+) \quad \to \quad \text { $G$-equivariant sheaves   on $G/P$} \ ,
 $$
  which associate  to $M \in \mathcal M^{et}_{\Lambda_{\ell}(N_{0})}(L_+)$ the  $G$-equivariant  sheaves      $\mathfrak Y_\cap$ and $\mathfrak Y_\cup$ on $G/P$ with sections on $\mathcal C_0$ equal to  $\mathfrak Y_\cap (\mathcal C_0)=M^{bd}_\cap $  and  $\mathfrak Y_\cup (\mathcal C_0) =M_\cup^{bd}$.
\end{theorem}

\begin{proof} The existence of the functors results from Prop. \ref{spr}, Theorem \ref{the}, Prop. \ref{multiplicative}, and Remark \ref{G-extension}.

We show the  faithfulness of the functors.
 For a non zero  morphism $f:M\to M'$ in $\mathcal M^{et}_{\Lambda_{\ell}(N_{0})}(L_+)$, we have  $f(M_{\cap}^{bd})\neq 0$   because $f$ is continuous (\cite{SVig} Lemma 8.22) and $M_{\cap}^{bd}$ containing $\Lambda (N_0)\iota_D(D)$  is dense  (proof of Prop. \ref{list}). We deduce $\mathbb Y_\cap(f)\neq 0$ since it is nonzero on sections on $\mathcal C_0$. A fortiori  $\mathbb Y_s(f)\neq 0$, and $\mathbb Y_\cup(f)\neq 0$.
\end{proof}

\section{Connected reductive  split  group} \label{crsg}

 We explain how our results  apply to connected reductive groups.

 \bigskip a) Let  $F$ be  a locally compact non archimedean field of ring of integers $o_{F}$ and   uniformizer $p_{F}$.
Let $G$ be a connected reductive  $F$-group, let $S$ be a maximal $F$-split subtorus of $G$ and let $P$ be a  parabolic $F$-subgroup of $G$  with Levi component $L$ containing $S$ and unipotent radical $N$.
 Let   $X^{*}(S)$ be the group of characters of $S$, let  $\Phi_{L} $, resp. $\Phi$,    be the subset of roots of $S$ in $L$, resp.  $G$,  and let $\Phi_{+,N}$ be the subset of roots of $S$  in $N$ (we suppress the index $N$ if $P$ is a minimal parabolic $F$-subgroup of $G$).

 Let $s$ be any element of $S(F)$   such that  $\alpha (s) =1$ for  $\alpha \in \Phi_{L}$  and the $p$-valuation of $\alpha(s)\in F^{*}$ is positive for all  roots $\alpha\in \Phi_{+,N}$.  {\sl For any  compact open subgroup $N_{0}$ of $N(F)$,
the data $(P(F),L(F),N(F), N_{0},s)$ satisfy all the conditions  introduced in the section on \'etale $P_{+}$-modules (\ref{S3}), (\ref{2.4}), the assumptions  introduced in the
section \ref{fc}, and   in the section \ref{S9}.}

\bigskip b) We suppose that $P$ is a minimal parabolic $F$-subgroup. Let $W\subset N_{G}(L)$ be a system of representatives of the Weyl group $N_{G}(L)/L$ and let $w_{0}=w_{0}^{2}$ is the longest element of the Weyl group.
{\sl The data $(G(F),P(F),W)$ satisfy the assumptions  of the section \ref{S5}
on $G$-equivariant sheaves on $G/P$.}

\bigskip  c) We suppose until the end of this article that
$$\text{$F=\mathbb Q_{p}$,   $G $ is $\mathbb Q_{p}$-split and  $P$  is a  Borel $\mathbb Q_{p}$-subgroup.}
$$
 The Levi subgroup $L=T$ of $P$ is a split  $\mathbb Q_{p}$-torus. The  monoid  of dominant elements  and the submonoid without unit  of strictly dominant elements  are
\begin{align*}T(\mathbb Q_{p})_{+} &= \{t\in T(\mathbb Q_{p}), \ \alpha (t)\in \mathbb Z_p \ {\rm for \ all } \ \alpha \in \Delta\} \ , \\
  T(\mathbb Q_{p})_{++} &= \{t\in T(\mathbb Q_{p}), \ \alpha (t)\in p\mathbb Z_p -\{0\}\ {\rm for \ all } \ \alpha \in \Delta\} \ .
  \end{align*}
 With our former notation  $Z(L)=T(\mathbb Q_{p}), Z(L)_{\dagger \dagger}=T(\mathbb Q_{p})_{++}$.
     For each root $\alpha \in \Phi$,  let
\begin{equation}\label{ua}
u_{\alpha}: \mathbb Q_{p}\to N_{\alpha}(\mathbb Q_{p})  \ \ , \ \ tu_{\alpha}(x)t^{-1}= u_{\alpha}(\alpha(t) x) \ \ {\rm for}\  \ x\in \mathbb Q_{p}, t\in T( \mathbb Q_{p}) \ ,
\end{equation}
 be a  continuous isomorphism  from $\mathbb Q_{p}$ onto the root subgroup  $N_{\alpha} (\mathbb Q_{p})$ of $N (\mathbb Q_{p})$ normalized by $T(\mathbb Q_{p})$.
 We can write an element  $u\in N (\mathbb Q_{p})$ in the form
  \begin{equation*}
 u= \prod_{\alpha \in \Phi_{+}}u_{\alpha}(x_{\alpha})
\end{equation*}
 for any ordering of $\Phi_{+}$. The  coordinates $x_{\alpha}=x_{\alpha }(u)\in \mathbb Q_p$ of $u$ are determined by the ordering of the roots, but for a simple root $\alpha$, the coordinate
 \begin{equation}\label{la}
 x_{\alpha}: N (\mathbb Q_{p}) \to \mathbb Q_{p}
 \end{equation} is independent of the choice of the ordering, and satisfies $u_{\alpha}\circ x_{\alpha}=1$.
 We suppose, as we can, that the $u_{\alpha}$ have be chosen  such that
 the product
 $$N_0=\prod_{\alpha \in \Phi_+}u_\alpha(\mathbb Z_p) $$   is a group for some ordering  of $\Phi_+$. Then $N_0$ is the product of the $u_\alpha(\mathbb Z_p)=N_{\alpha} (\mathbb Z_p)$ for any ordering of $\Phi_+$.

 We choose a simple root $\alpha $.   We consider
the continuous homomorphisms
  $$\ell_{\alpha} \  : \ P(\mathbb Q_{p}) \to P^{(2)}(\mathbb Q_{p}) \ \ , \ \ \iota_{\alpha} :\ N(\mathbb Q_{p})^{(2)} \to N(\mathbb Q_{p}) \ \ , \ \ \ell_{\alpha} \circ \iota_{\alpha} = 1 \ ,$$
 defined by
 $$ \ell_{\alpha} (ut) := \begin{pmatrix} \alpha (t) & x_{\alpha }(u) \\ 0 & 1 \end{pmatrix} \ \ , \ \
 \iota_{\alpha}(u^{(2)}(x)) := u_{\alpha}(x) \ \ {\rm for } \ \ u^{(2)}(x) := \begin{pmatrix} 1 & x  \\ 0 & 1 \end{pmatrix} \ ,$$
 for $t\in T(\mathbb Q_{p}), u\in N(\mathbb Q_{p}), x\in \mathbb Q_{p}$.
They satisfy the functional equation
$$t\iota_{\alpha}(y) t^{-1}= \iota_{\alpha}(\ell_{\alpha}(t) y \ell_{\alpha} (t)^{-1})$$
for $y\in N(\mathbb Q_{p})  ^{(2)}$ and $t\in T(\mathbb Q_{p}) $.
{\sl The data $(N_{0},\ell_{\alpha} , \iota_{\alpha} )$ satisfies the assumptions introduced in the \ref{GM} and in the section \ref{S9}.}

 We consider  the binary relation $s_1\leq s_2 $ on $ T(\mathbb Q_{p})_{+ +}$  generated by
$$s_1 =s_2 s_0  \ \text{ with $ s_0 \in  T(\mathbb Q_{p})_+, \alpha(s_0)\in \mathbb Z_p^* $  \ , \ or  $  s_1^n=s_2 ^m$  with $n,m \geq 1$. }
$$

\begin{lemma} The  relation $s_1\leq s_2 $ on $ T(\mathbb Q_{p})_{++}$   is right filtered.
\end{lemma}
\begin{proof}
Let $\Delta=\{\alpha=\alpha_1,\ldots,  \alpha_n\}$.  The  image of
$T(\mathbb Q_p)_{++}$ by $A= (\val_p (\alpha_i (.))_{\alpha_i\in \Delta}$ is  contained in  $(\mathbb N-\{0\})^n$ and $s_1\leq s_2$ depends only on the cosets
$s_1T(\mathbb Q_{p})_0$ and $s_1T(\mathbb Q_{p})_0$, where
 $$ T(\mathbb Q_{p})_0 =\{t\in T(\mathbb Q_{p}), \ \alpha (t) \in \mathbb Z_p^*  \ {\rm for \ all } \ \alpha \in \Delta\} \ . $$

a) First we  assume  that, for any positive integer $k$,   there exists
$s_{[k]}\in  T(\mathbb Q_{p}) $  such $A(s_{[k]}) =(k,1, \ldots, 1)$.
 Then we have $s_{[k]}\leq s_{[k+1]} $, and $s\leq s_{[k(s)]} $ for   $s\in T(\mathbb Q_{p})_{++}$ with $k(s) = \val_p (\alpha (s))$. For any $s_1, s_2$ in $T(\mathbb Q_{p})_{++}$ we deduce that $s_1\leq  s_{[k(s_1)+k(s_2)]}$ and $s_2\leq  s_{[k(s_1)+k(s_2)]}$. Hence the relation $\leq$ on $ T(\mathbb Q_{p})_{++}$   is right filtered.

b) When $G$ is semi-simple and adjoint the dominant coweights $\omega_{\alpha_1}, \ldots, \omega_{\alpha_n}$  for $\Delta=\{\alpha=\alpha_1,\ldots,  \alpha_n\}$ form a basis of  $Y=\Hom (\mathbb G_m, T)$,  and
$A(T(\mathbb Q_p)_{++}) =(\mathbb N-\{0\})^n$. Hence
  $s_{[k]}$ exists for any $k\geq 1$.

c) When $G$ is semi-simple we consider the isogeny $\pi:G\to G_{ad}$ from $G$ onto the adjoint group $G_{ad}$  (\cite{Spr} 16.3.5). The image $T_{ad}$ of $T$ is  a maximal
split $\mathbb Q_p$-torus in  $G_{ad}$.
The isogeny
   gives an homomorphism $ T (\mathbb Q_p)\to  T_{ad} (\mathbb Q_p)$,
   inducing an injective map
 between the cosets
$$ T (\mathbb Q_p)_{++}/T(\mathbb Q_p)_{0}\  \to \ T_{ad}(\mathbb Q_p)_{++}/T_{ad} (\mathbb Q_p)_{0} $$
 respecting $\leq $, and
    such that for any $t_{ad}\in T_{ad}(\mathbb Q_p)$ there exists an integer $n \geq 1$ such that $ t_{ad}^n \in \pi(T (\mathbb Q_p))$.
    Given $s_1,s_2\in  T (\mathbb Q_p)_{++}$ there exists $s_{ad}\in T_{ad} (\mathbb Q_p)_{++}$ such that $\pi(s_1), \pi(s_2) \leq s_{ad}$ by b) and a). Let $n\geq 1$ such that  $s_{ad}^n=\pi(s_3)$ for $s_3\in T (\mathbb Q_p)$. We have $s_{ad}\leq s_{ad}^n$ hence $\pi(s_1), \pi(s_2) \leq \pi(s_3)$. This is equivalent to $s_1, s_2 \leq s_3$.

 d) When $G$ is reductive
 let  $\pi:G\to G'=G/Z^0 $ be the natural $\mathbb Q_p$-homomorphism from $G$ to the quotient of $G$ by its maximal split central torus $Z^0$.  The group $G'$ is semi-simple, $\pi(T)=T'$  is  a maximal
split $\mathbb Q_p$-torus in  $G'$,  $\pi|_T$ gives an exact sequence
$$
1\to Z_0(\mathbb Q_p) \to T(\mathbb Q_p) \to T'(\mathbb Q_p) \to 1 \ ,
$$
inducing a bijective map between the cosets
$$ T (\mathbb Q_p)_{++}/T(\mathbb Q_p)_{0}\  \to \ T'(\mathbb Q_p)_{++}/T'(\mathbb Q_p)_{0} $$
respecting $\leq$. By c),  $\leq$  is right filtered on $T'(\mathbb Q_p)_{++}$. We deduce that $\leq$  is right filtered on $T(\mathbb Q_p)_{++}$.
\end{proof}

By Theorem \ref{eq} and Theorem \ref{main}, we can associate functorially to  an \'etale $T_+$-module $D$ over $\mathcal O_{\mathcal E, \alpha}$ different  sheaves :

\begin{itemize}
\item For any $s\in T_{++}$,   a $G(\mathbb Q_p)$-equivariant sheaf $\mathfrak Y_s$  on $G(\mathbb Q_p)/P(\mathbb Q_p)$ with sections  on $\mathcal C_0$ equal to
 $\mathbb M(D)_s^{bd} $
   \item The  $G(\mathbb Q_p)$-equivariant sheaves $\mathfrak Y_\cap $ and $ \mathfrak Y_\cup$ on $G(\mathbb Q_p)/P(\mathbb Q_p)$ with sections on $\mathcal C_0$ equal to  $\cap_{s\in T_{++}}\mathbb M(D)_s^{bd} $ and $ \cup_{s\in T_{++}}\mathbb M(D)_s^{bd}$.
 \end{itemize}

In general $\mathbb M(D)$ is different from $ \cup_{s\in T_{++}}\mathbb M(D)_s^{bd}$, by the following proposition.

\begin{proposition} Let $M$ be  an  \'etale $T_+$-module $M$ over $ \Lambda_{\ell_\alpha}(N_0)$.  When  the root system of $G$ is irreducible of positive rank $rk(G)$, we have:

(i) If  $rk(G)=1$,  the $G(\mathbb Q_p)$-equivariant sheaf  on  $G(\mathbb Q_p)/P(\mathbb Q_p)$ with sections $M_{s}^{bd}$ over $\mathcal C_0$
does not depend on the choice of $s\in T_{++}$, and $M=M_{s}^{bd}$.

(ii) If $rk(G)>1 $,    a  $G(\mathbb Q_p)$-equivariant sheaf of $o$-modules $\mathfrak Y$ on  $G(\mathbb Q_p)/P(\mathbb Q_p)$   such that  $\mathfrak Y (\mathcal C_0) \subset M$ and $(u_\alpha(1)-1)$ is bijective on $\mathfrak Y (\mathcal C_0) $,  is zero.
 \end{proposition}

\begin{proof} We prove (i).  When $rk(G)=1$, then $\mathcal O_{\mathcal E}= \Lambda_{\ell_\alpha} (N_0)$ and $M=D$ is an \'etale  $T_+$-module over $\mathcal O_{\mathcal E}$.  With the same proof than in Prop. \ref{3}, we have $M_s^{bd}=M$ for any $s\in T_{++}$ and the integrals $\mathcal H_g$ for $g\in N_0 \overline P N_0$ do not depend on the choice of $s$.

 (ii) is equivalent to the property:   an \'etale $o[P_+]$-submodule $M'$ of $M$ wich is also
 a  $R=o[N_0][(u_\alpha(1)-1)^{-1}]$-submodule of $M$,  and is endowed with endomorphisms $\mathcal{H}_{g}\in \End_o(M)$, for all $g\in N_0\overline P(F) N_0$, satisfying the relations H1, H2, H3 (Prop. \ref{multiplicative}), is $0$.

a) Preliminaries. As $rk(G)\geq 2$ and the root system is irreducible, there exists a simple root $\beta$ such that $\alpha +\beta$ is a root. The elements  $n_\alpha:=u_\alpha(1)$ and $n_\beta:=u_\beta(1)$  do not commute. By the commutation formulas,
  $n_\alpha n_\beta = n_\beta n_\alpha h$  for some $h  \neq 1$ in  the group $H=\prod_\gamma N_{\gamma} (\mathbb Z_p)$ for all positive roots of the form $\gamma =i\alpha +j\beta\in \Phi_+$ with $i,j>0$. Note that $H$ is normalized by   $N_\alpha (\mathbb Z_p)$.
  Let $s\in T_{++}$. We have
 the expansion \eqref{writing}
  \begin{equation}\label{mk}
  (n_\alpha h -1)^{-k}= \sum_{u\in J(N_{\alpha}(\mathbb Z_p)H/ s N_{\alpha}(\mathbb Z_p)H s^{-1})} u\varphi_s(\psi_s(u^{-1}(n_\alpha h -1)^{-k}))
  \end{equation}
   in  $R$.
    We choose, as we can,  a lift $w_\beta$ of $s_\beta$ in the normalizer of  $T(\mathbb Q_p)$  such that
\begin{itemize}
\item[-] $w_\beta n_{\beta}\in n_{\beta} \overline P(\mathbb Q_p)$
\item[-] $w_\beta$ normalizes  the group $N_{\Phi_+-\beta}(\mathbb Z_p)=\prod_{\gamma }N_\gamma(\mathbb Z_p)$ for all positive roots $\gamma \neq \beta$.
\end{itemize}
The subset $N'_\beta(\mathbb Z_p)\subset N_\beta(\mathbb Z_p)$   of $u_\beta (b)$ such that   $w_\beta u_\beta (b) \in u_\beta (\mathbb Z_p) \overline P(\mathbb Q_p)$, contains $n_\beta$ but does not contain $1$.
The subset $U_{w_\beta }\subset N_{0}$ of   $u$ such that $w_\beta u \in
N_{0} \overline P(\mathbb Q_p)$ is equal to
$$U_{w_{\beta}} = N'_\beta(\mathbb Z_p)N_{\Phi_+-\beta}(\mathbb Z_p)=N_{\Phi_+-\beta}(\mathbb Z_p)N'_\beta(\mathbb Z_p)\ . $$
Hence $U_{w_\beta }=u U_{w_\beta }$, i.e. $w_\beta ^{-1}\mathcal{C}_0\cap \mathcal{C}_0=uw_\beta ^{-1}\mathcal{C}_0\cap \mathcal{C}_0 $,
 for any  $u\in N_{\Phi_+-\beta}(\mathbb Z_p)$.

b) Let $M'$ be an $R=o[N_0][(n_\alpha-1)^{-1}]$-module of $M$, which is also  an \'etale  $o[P_+]$-submodule, and is endowed with endomorphisms $\mathcal{H}_{g}\in \End_o(M)$, for all $g\in N_0\overline P(F) N_0$, satisfying the relations H1, H2, H3 (Prop. \ref{multiplicative}), and let $m\in M'$ be an arbitrary element.  We want to prove that $m=0$.

The idea of the proof is that, for $s\in T_{++}$, we have  $m=0$ if  $\mathcal{H}_{w_\beta}(n_\beta\varphi_s(m))=0$ and that $\mathcal{H}_{w_\beta}(n_\beta\varphi_s(m))=0$ because it is
infinitely divisible by
$n_\gamma-1$,  where $\gamma=s_\beta(\alpha)$.  An element in $M$ with this property is $0$ because $n_\gamma-1$ lies in the maximal ideal of
$\Lambda_{\ell_\alpha}(N_0)$.

 Let  $a\in \mathbb Z_p$. The product formula   in    Prop.\ \ref{corspecial}ii implies
\begin{align*}
\mathcal{H}_{w_\beta}\circ\mathcal{H}_{n_\alpha^a}
\circ\res(1_{w_\beta^{-1}\mathcal{C}_0\cap\mathcal{C}_0})=\mathcal{H}_{w_\beta n_\alpha^a}\circ\res(1_{w_\beta ^{-1}\mathcal{C}_0\cap\mathcal{C}_0})=\\
\mathcal{H}_{n_\gamma^aw_\beta }\circ\res(1_{w_\beta ^{-1}\mathcal{C}_0\cap\mathcal{C}_0})=\mathcal{H}_{n_\gamma^a}\circ\mathcal{H}_{w_\beta} \circ\res(1_{w_\beta ^{-1}\mathcal{C}_0\cap\mathcal{C}_0})
\end{align*}
since
$n_\alpha^{-a}w_\beta^{-1}\mathcal{C}_0\cap \mathcal{C}_0=w_\beta^{-1}\mathcal{C}_0\cap \mathcal{C}_0=w_\beta^{-1}n_\gamma^{-a}\mathcal{C}_0\cap \mathcal{C}_0$.
For all  $k\in \mathbb N$,  the elements
\begin{align}
m_k:=&(n_\alpha-1)^{-k}n_\beta\varphi_s(m)=n_{\beta}(n_\alpha
h  -1)^{-k}\varphi_s(m) \label{m_k}
\end{align}
   lie in the image of the idempotent
$\res(1_{w_\beta^{-1}\mathcal{C}_0\cap\mathcal{C}_0}) \in \End_o (M)$, because
  \begin{equation}\label{mkk}
  m_k= \sum_{u\in J(N_{\alpha}(\mathbb Z_p)H/ s N_{\alpha}(\mathbb Z_p)H s^{-1})} n_\beta u\varphi_s(\psi_s(u^{-1}(n_\alpha h -1)^{-k}m))
     \end{equation}
by \eqref{mk}, \eqref{m_k}.
   Therefore the product relations between  $\mathcal H_{w_\beta}, \mathcal H_{n_\alpha^a}$  and $\mathcal H_{n_\gamma^a}$  imply
\begin{align*}
\mathcal{H}_{w_\beta}(n_\beta\varphi_s(m))&=\mathcal{H}_{w_\beta}((n_\alpha-1)^km_k)=
\sum_{a=0}^k(-1)^{k-a}\binom{k}{a}\mathcal{H}_{w_\beta}\circ\mathcal{H}_{n_\alpha^a}(m_k)\\
&=\sum_{a=0}^k(-1)^{k-a}\binom{k}{a}\mathcal{H}_{w_\beta}\circ\mathcal{H}_{n_\alpha^a}\circ\res(1_{w_\beta^{-1}\mathcal{C}_0\cap\mathcal{C}_0})(m_k) \\
&= \sum_{a=0}^k(-1)^{k-a}\binom{k}{a}\mathcal{H}_{n_\gamma^a}\circ\mathcal{H}_{w_\beta}\circ\res(1_{w_\beta^{-1}\mathcal{C}_0\cap\mathcal{C}_0})(m_k)\\
&=(n_\gamma-1)^k\mathcal{H}_{w_\beta}(m_k)\ ,
\end{align*}
Hence $\mathcal{H}_{w_\beta}(n_\beta\varphi_s(m))=0$ since it is infinitely divisible by
$n_\gamma-1$ which lies in the maximal ideal of
$\Lambda_{\ell_\alpha}(N_0)$. We also have
\begin{equation*}
n_\beta\varphi_s(m)=\mathcal{H}_{1}\circ\res(1_{w_\beta^{-1}\mathcal{C}_0\cap\mathcal{C}_0})(n_\beta\varphi_s(m))=\mathcal{H}_{w_\beta}\circ\mathcal{H}_{w_\beta}(n_\beta\varphi_s(m))=0 \ .
\end{equation*}
As $n_\beta \circ \varphi_s\in \End_o(M')$ is injective, we deduce $m=0$.
\end{proof}

\begin{corollary}  There exists a   $G(\mathbb Q_p)$-equivariant sheaf on  $G(\mathbb Q_p)/P(\mathbb Q_p)$ with sections $M$ on $\mathcal C_0$ if and only if   $\mathrm{rk}(G)=1$.
\end{corollary}

\phantomsection

\end{document}